\documentclass[11pt]{article}
\usepackage[english]{babel}
\usepackage[a4paper, left=3.4cm, right=3.4cm, top=3cm]{geometry}

\usepackage{xfrac}

\usepackage{mathptmx}      

\usepackage{amsmath,amsthm}
\usepackage{amssymb}
\usepackage{amsfonts}
\usepackage{oldgerm}
\usepackage{mathrsfs}
\usepackage{comment}
\usepackage{tikz}
\usepackage{pgfplots}
\usepackage[scanall]{psfrag}
\usepackage{graphicx,pst-all,bm}
\usepackage{flushend}
\usepackage{subfig}
\usepackage{mdwlist}
\usepackage{multirow}
\usepackage{color}

\usepackage{longtable}
\usepackage{cuted}
\usepackage{verbatim}

\usepackage{siunitx}

\newcommand{\C}{\mathbb{C}}
\newcommand{\K}{\mathbb{K}}

\newcommand{\Q}{\mathbb{Q}}
\newcommand{\R}{\mathbb{R}}
\newcommand{\Z}{\mathbb{Z}}

\DeclareMathOperator{\im}{im \, }

\DeclareMathOperator{\diag}{diag}

\newcommand{\Real}{\mathbb{R}}
\newcommand{\Natu}{\mathbb{N}}
\newcommand{\pPi}{\mathnormal{\Pi}} 
\newcommand{\rank}{\operatorname{rank}}
\newtheorem{proposition}{Proposition}[section]
\newtheorem{theorem}[proposition]{Theorem}
\newtheorem{corollary}[proposition]{Corollary}
\newtheorem{lemma}[proposition]{Lemma}
\newtheorem{definition}[proposition]{Definition}
\newtheorem{remark}[proposition]{Remark}
\newtheorem{example}[proposition]{Example}

\newtheorem{conjecture}[proposition]{Conjecture}
\newtheorem{hypothesis}[proposition]{Hypothesis}
\newcommand{\corank}{\operatorname{corank}}

\renewcommand{\it}{\itshape}
\renewcommand{\rm}{\mathrm}
\pgfplotsset{compat=1.17}
\usepackage{authblk}

\begin{document}

\title{The common ground of DAE approaches.\\
\Large An overview of diverse DAE frameworks \\
emphasizing their commonalities.}

\author[1]{Diana Est\'evez Schwarz}
\author[2]{Ren\'e Lamour}
\author[2]{Roswitha M\"arz}

\affil[1]{\small Berliner Hochschule f\"ur Technik}
\affil[2]{\small Humboldt University of Berlin, Institute of Mathematics}

\maketitle

\begin{abstract}
We analyze different approaches to differential-algebraic equations with attention to the implemented rank conditions of various matrix functions. These conditions are apparently very different and certain rank drops in some matrix functions actually indicate a critical solution behavior. We look for common ground by considering various index and regularity notions from literature generalizing the Kronecker index of regular matrix pencils. In detail, starting from the most transparent reduction framework, 
we work out a comprehensive regularity concept with canonical characteristic values applicable across all frameworks and prove the equivalence of thirteen distinct definitions of regularity. 
This makes it possible to use the findings of all these concepts together.
Additionally, we show why not only the index but also these canonical characteristic values are crucial to describe the properties of the DAE.
\end{abstract}
\textbf{Keywords:} Differential-Algebraic Equation, Higher Index, Regularity, Critical Points, Singularities, Structural Analysis, Persistent Structure, Index Concepts, Canonical Characteristic Values
\medskip \\
\textbf{AMS Subject Classification:} 34A09, 34A12, 34A30, 34A34, 34-02

\setcounter{secnumdepth}{3}
\setcounter{tocdepth}{3}

\tableofcontents
\pagebreak
\section{Introduction}\label{sec:Introduction}

\begin{flushright}
\textsc{What proven concepts differ is remarkable,\\
but what they have in common is essential.}                                                                                              \end{flushright}
                                                                                                                                                                                            
 \emph{Who coined the term DAEs?} is asked in the engaging essay \cite{Simeon} and the answer is given there: Bill Gear. The first occurrence of the term \emph{Differential-Algebraic Equation} can be found in the title of Gear's  paper from 1971 \emph{Simultaneous numerical solution of differential-algebraic equations}\cite{Gear71} and in his book \cite{Gear71B} where he considers examples from electric circuit analysis. 
 The German term \emph{Algebro-Differentialgleichungssysteme} comes from physicists and electronics engineers and it is first found as a chapter title in the book \emph{Rechnergest{\"u}tzte Analyse in der Elektronik} from 1977, \cite{EMR77}, in which the above two works are already cited.
Obviously, electric circuit analysis accompanied by the diverse computer-aided engineering that was emerging at the time gave the impetus  for many developments in the following 50 years. Actually, there are several quite different approaches with a large body of literature, such as the ten volumes of the DAE-Forum book series, but still too few commonalities have been revealed. We would like to contribute to this, in particular by showing equivalences.
\bigskip

We are mainly focused on linear differential algebraic equations (DAEs) in standard form,
\begin{align}\label{1.DAE}
 Ex'+Fx=q,
\end{align}
in which $E,F:\mathcal I\rightarrow \Real^{m\times m}$ are sufficiently smooth, at least continuous, matrix functions on the interval $\mathcal I\subseteq\Real$  so that all the index concepts we look at apply\footnote{With regard to linearizations of nonlinear DAEs, we explicitly do not assume that $E, F$ are real analytic or from $C^{\infty}$.}. The matrix $E(t)$ is  singular for all $t\in\mathcal I$. 
\bigskip

If E and F are constant matrices, the regularity of the DAE means the regularity of the matrix pair $\{E,F\}$, i.e., $\det(sE+F)$ which is a polynomial in $s$ must not be identical zero. However, 
it must be conceded that, so far, for DAEs with variable coefficients, there are partially quite different definitions of regularity bound to the technical concepts behind them. 
Surely, regular DAEs have no freely selectable solution components and do not yield any consistency conditions for the inhomogeneities. But that's not all, certain qualitative characteristics of the flow and the input-output behavior are just as important, the latter especially with regard to the applications.
We are pursuing the question:
To what extent are the various rank conditions which support  DAE-index notions appropriate, informative and comparable? The answer results from 
an overview of diverse approaches to DAEs emphasizing their commonalities.
We hope that our analysis will also contribute to a harmonization of understanding in this matter. To our understanding, our main equivalence theorem from Section \ref{sec:MainTheorem} is a significant step toward this direction.
\medskip

In the vast majority of papers about DAEs, continuously differentiable solutions $x\in\mathcal C^{1}(\mathcal I,\Real^{m})$ are assumed, and smoother if necessary. On the other hand, since $E(t)$ is singular for every $t\in\mathcal I$, obviously only a part of the first derivative of the unknown solution is actually involved\footnote{For instance, the Lagrange parameters in DAE-formulations of mechanical systems do not belong to the differentiated unknowns.} in  the DAE \eqref{1.DAE}.
To emphasize this fact, the DAE \eqref{1.DAE} can be reformulated by means of a suitable factorization $E=AD$ as
\begin{align}\label{1.propDAE}
 A(Dx)'+Bx=q,
\end{align}
in which $B=F-AD'$. This allows the admission of only continuous solutions $x$ with continuously differentiable parts $Dx$. However, we do not make use of this possibility here. 
 Just as we
 focus on the original coefficient pair $\{E, F\}$ and smooth solutions in the present paper, we underline the identity,
\begin{align*}
 A(Dx)'+Bx=Ex'+Fx,\quad \text{ for }\; x\in\mathcal C^{1}(\mathcal I,\Real^{m}),
\end{align*}
being valid equally for each special factorization. In addition, we will highlight, that the auxiliary coefficient triple $\{A,D, B\}$ takes over the structural rank characteristics of $\{E, F\}$, and vice versa. With this we want to clear the frequently occurring  misunderstanding
that so-called \textit{DAEs with properly and quasi-properly stated leading term} are something completely different from  standard form DAEs.\footnote{A \textit{DAE with properly involved derivative} or \textit{properly stated leading term} is a DAE of the form \eqref{1.propDAE} with the properties $\im A=\im AD$, $\ker D=\ker AD$. We refer to \cite{CRR} for the general description and properties.}

Based on the realization that the \textsl{Kronecker index} is an adequate means to understand DAEs with constant coefficients,
we survey and compare different notions which generalize the  Kronecker index for regular matrix pairs.
We shed light on the concerns behind the concepts, but emphasize common features to a large extent 
as opposed to simply list them next to each other
 or to stress an otherness without further arguments.
We are convinced that especially
the basic rank conditions within  the various concepts prove to be an essential, unifying characteristic and gives the possibility of a better understanding and use.
\medskip

This paper is organized as follows.  
After clarifying important notions like solvability and equivalence transformations in Sections  \ref{s.Arrangements} and \ref{s.Equivalence}, we start introducing  a reference basic concept with its associated characteristic values, that depend on the rank of certain matrices in Section \ref{s.Basic&more}. This basic notion is our starting point to prove many equivalences. \\

The structure of the paper reflects that, roughly speaking, there are two types of frameworks to analyze DAEs:
\begin{itemize}
	\item Approaches based on the direct construction of a matrix chain or a sequence of matrix pairs without using the so-called derivative array. The basic concept and all concepts discussed in Section \ref{s.Solvab} are of this type. They turn out to be equivalent and lead to a common notion of regularity. This is also equivalent to tranformability into specifically structured standard canonical form.
	
	\item Approaches based on the derivative array are addressed in Section \ref{s.notions}. In this case, it turns out that some of these are equivalent to the basic concept, whereas others are different in the sense that weaker regularity properties are used.  The later ones lead to our notion of almost regular DAEs. 
\end{itemize}

\begin{center}
\begin{table}[ht]
 \begin{tabular}{|c | c|c|c|} 
 \hline
 & \multicolumn{1}{ c|}{without derivative array} & \multicolumn{1}{c|}{with derivative array} \\
 \hline
\parbox[t]{2mm}{\multirow{7}{*}{\rotatebox[origin=c]{90}{regularity}}}
 \parbox[t]{2mm}{\multirow{7}{*}{\rotatebox[origin=c]{90}{ }}}
 &  & \\ 
 & Basic (Sec. \ref{s.regular}) &\\
 & Elimination (Sec. \ref{subs.elimination}) & Regular Differentiation (Sec. \ref{subs.qdiff})\\
 & Dissection (Sec. \ref{subs.dissection}) & \\
 & Regular Strangeness (Sec. \ref{subs.strangeness})& Projector Based Differentiation (Sec. \ref{subs.pbdiff})\\
 & Tractability (Sec. \ref{subs.tractability})&\\
 &  & \\
\hline
\parbox[t]{2mm}{\multirow{7}{*}{\rotatebox[origin=c]{90}{regularity or}}}
 \parbox[t]{2mm}{\multirow{7}{*}{\rotatebox[origin=c]{90}{almost regularity}}}
&  & \\
 &  & \\
 &  &  Differentiation (Sec. \ref{subs.diff}) \\
 &  &  \\
 &  & Strangeness (Sec. \ref{subs.Hyp}) \\
 &  &\\
&  & \\
 \hline
 \end{tabular}
\caption{Overview of the discussed index notions. The different regularity properties are defined in Section \ref{subs.equivalence}.}
\label{t.overview}
 \end{table}
\end{center}

An overview of the approaches we discuss for linear DAEs can be found in Table \ref{t.overview}.  Illustrative examples for the different types of regularity are compiled in Section \ref{s.examples}. \\

All approaches use own characteristic values that correspond to ranks of matrices or dimensions of subspaces and in the end it turns out that, in case of regularity, they can be calculated with so-called canonical characteristic values and vice versa.\\

Section \ref{s.Notions} starts with a summary of all the obtained equivalence results in a quite extensive theorem with hopefully enlightening and pleasant content. Based on this, a discussion of the meaning of regularity completed by an inspection of related literature follows. \\

Finally, in Section \ref{s.nonlinearDAEs} we briefly outline the generalization of the discussed approaches to nonlinear DAEs with a view to linearizations.
To facilitate reading, some technical details  are provided in the appendix.

\section{Special arrangements for this paper}\label{s.Arrangements}
Throughout this paper the coefficients of the DAE \eqref{1.DAE} are  matrix functions $E, F:\mathcal I\rightarrow\Real^{m\times m}$ that are sufficiently smooth to allow the application of all the approaches discussed here, 
by convention of class $\mathcal C^m$, and $\mathcal C^{\mu}$, if applicable, if an index $\mu \leq m$ is already known, but not from $\mathcal C^{\infty}$ and the real-analytic function space. 
Our aim is to uncover the common ground between the various concepts, in particular the rank conditions.
We will not go into the undoubted differences between the concepts in terms of smoothness requirements here, which are very important, of course.  Please refer to the relevant literature.
\medskip

This is neither a historical treatise nor a comprehensive overview of approaches and results, but rather an attempt to reveal what is common to the popular approaches.
Wherever possible, we cite widely used works such as monographs and refer to the references therein for the classification of corresponding original works.

Our particular goal is the harmonizing comparison of the basic rank conditions behind the various concepts combined with the characterization of the class of regular  pairs $\{E,F\}$ or DAEs \eqref{1.DAE}.
Details regarding solvability statements within the individual concepts would go beyond the scope of this paper. Here we merely point out the considerable diversity of approaches.

While on the one 
hand, in many papers, from a rather functional analytical point of view, attention is paid to the lowest possible smoothness, suitable function spaces, rigorous solvability assertions, and precise statements about relevant operator properties such as surjectivity, continuity, e.g., \cite{GM86,CRR,Ma2014,Jansen2014}, we observe that,
on the one hand, solvability in the sense of the Definition \ref{d.solvableDAE} below is assumed and integrated into several developments from the very beginning, e.g., \cite{BCP89,KuMe2006,BergerIlchmann}.

We quote \cite[Definition 2.4.1]{BCP89}\footnote{See also Remark \ref{r.generalform} below.}:
\begin{definition}\label{d.solvableDAE}
 The system \eqref{1.DAE} is \emph{solvable} on the interval $\mathcal I$ if for every $m$-times differentiable $q$, there is at least one continuously differentiable solution to \eqref{1.DAE}. In addition, solutions are defined on all of $\mathcal I$ and are uniquely determined by their value at any $t\in\mathcal I$.
\end{definition}

\medskip

Here we examine and compare only those approaches whose characteristics do not change under equivalence transformations and which generalize the Kronecker index for regular matrix pairs. This rules out the so-called \emph{structural index}, e.g. \cite{Pantelides88,Pryce1998,RMB2000,PryceDAESA}.
\medskip

A widely used and popular means of investigating DAEs is the so-called \emph{perturbation index}, which according to \cite{HairerWanner} can be interpreted as a sensitivity measure in relation to perturbations of the given problem. 
For time-invariant coefficients $\{E,F\}$, the perturbation index coincides with the regular Kronecker index. 
We adapt \cite[Definition 5.3]{HairerWanner} to be valid for the linear DAE \eqref{1.DAE} on the interval $\mathcal I=[a,b]$:
\begin{definition}\label{d.perturbation}
 The system \eqref{1.DAE} has \emph{perturbation index} $\mu_p\in\Natu$ if  $\mu_p$ is the smallest integer such that for all functions $x:\mathcal I \rightarrow \Real^{m}$ having a defect $\delta = Ex'+Fx$ there exists an estimate
 \begin{align*}
  |x(t)|\leq c\{|x(a)|+ \max_{a\leq\tau\leq t}|\delta(\tau)|+\cdots+ \max_{a\leq\tau\leq t}|\delta^{(\mu_p-1)}(\tau)|\}, \quad t\in\mathcal I.
 \end{align*}
\end{definition}
The perturbation index does not contain any information about whether the DAE has a solution for an arbitrarily given $\delta$, but only records resulting defects. In the following, we do not devote an extra section to the perturbation index, but combine it with the proof of corresponding solvability statements and repeatedly involve it in the relevant discussions.
\medskip

We close this section with a comment on the index names below, more precisely on the various additional epithets used in the literature such as differentiation, dissection, elimination, geometric, strangeness, tractability, etc. We try to organize them and stick to the original names as far as possible, if there were any. In earlier works, simply the term \emph{index} is used, likewise \emph{local index} and \emph{global index}, other modifiers were usually only added in attempts at comparison, e.g., \cite{GHM,Mehrmann,RR2008}.
After it became clear that the so-called \emph{local index} (Kronecker index of the matrix pencil $\lambda E(t)+F(t)$ at fixed $t$) is irrelevant for the general characterization of time-varying linear DAEs, the term \emph{global index} was used in contrast. We are not using the extra label \emph{global} here, as all the terms considered here  could have this.

\section{Comments on equivalence relations}\label{s.Equivalence}

Equivalence relations and special structured forms are an important matter of the DAE theory from the beginning. Two pairs of matrix functions $\{E,F\}$ and $\{\bar E,\bar F\}$, and also the associated DAEs, are called \textit{equivalent}\footnote{In the context of the strangeness index \textit{globally equivalent}, e.g. \cite[Definition 2.1]{KuMe1996}, and \textit{analytically equivalent} in \cite[Section 2.4.22]{BCP89}. We underline that \eqref{1.Equivalence} actually defines a reflexive, symmetric, and transitive equivalence relation $\{E,F\}\sim \{\tilde E,\tilde F\}$.}, if there exist pointwise nonsingular, sufficiently smooth\footnote{$L$ is at least continuous, $K$ continuously differentiable. The further smoothness requirements in the individual concepts differ; they are highest when derivative arrays play a role.} matrix functions $L, K:\mathcal I\rightarrow \Real^{m\times m}$, such that
\begin{align}\label{1.Equivalence}
 \tilde E=LEK,\quad \tilde F=LFK+LEK'.
\end{align}
An equivalence transformation 
 goes along with  the premultiplication of \eqref{1.DAE} by $L$ and the coordinate change $x=K\tilde x$ resulting in the further DAE $\tilde E\tilde x'+\tilde F\tilde x=Lq$.
\medskip

It is completely the same whether one refers the equivalence transformation to the standard DAE \eqref{1.DAE} or to the version with properly involved derivative \eqref{1.propDAE} owing to the following relations:
\begin{align*}
 \tilde A&=LAK,\;\tilde D=K^{-1}DK,\;\tilde B=LBK+LAK'K^{-1}DK,\\
 \tilde A&=\tilde A\tilde D=\tilde E,\;\tilde B=\tilde F-\tilde E\tilde D'.
\end{align*}

 The DAE \eqref{1.DAE} is in \textit{standard canonical form} (SCF) \cite[Definition 2.4.5]{BCP89}, if
 \begin{align}\label{1.SCF}
  E=\begin{bmatrix}
     I_{d}&0\\0&N
    \end{bmatrix},\quad
F=\begin{bmatrix}
     \Omega&0\\0&I_{m-d}
    \end{bmatrix},
 \end{align}
and $N$ is strictly upper triangular.\footnote{Analogously, $N$ may also have strict lower triangular form.} 

The matrix function $N$ does not need to have constant rank or nilpotency index. Trivially, choosing 
\begin{align*}
 A=E,\; D=\diag\{I_{d},0,1,\ldots,1\},\; B=F,
\end{align*}
one obtains the form \eqref{1.propDAE}.
Obviously, a DAE in SCF decomposes into two essentially different parts, on the one hand a regular explicit ordinary differential equation (ODE) in $\Real^{d}$ and on the other some algebraic relations which require certain differentiations of components of the right-hand side $q$. More precisely, if $N^{\mu}$ vanishes identically, but $N^{\mu-1}$ does not, then derivatives up to the order $\mu-1$ are involved. The dynamical degree of freedom of the DAE in SCF is determined by the first part and equals $d$.
\medskip

In the particular case of constant $N$ and $\Omega$, the matrix pair $\{E,F\}$ in \eqref{1.SCF} has Weierstra{\ss}--Kronecker form \cite[Section 1.1]{CRR} or Quasi-Kronecker form \cite{BergerReis}, and the nilpotency index of $N$ is again called \emph{Kronecker index} of the pair $\{E,F\}$ and the matrix pencil $\lambda E+F$, respectively.\footnote{In general the Kronecker canonical form is complex-valued and $\Omega$ is in Jordan form. We refer to \cite[Remark 3.2]{BergerReis} for a plea not to call \eqref{1.SCF} a canonical form.}

The basic regularity notion \ref{d.2} below generalizes  regular matrix pairs (pencils) and their Kronecker index. Thereby, the Jordan structure of the nilpotent matrix $N$, in particular the characteristic values
$\theta_0\geq\cdots\geq\theta_{\mu-2}>\theta_{\mu-1}=0$, 
 \begin{align*}
  &\theta_0 \quad \text{number of Jordan blocks of order } \geq 2,\\
  &\theta_1 \quad \text{number of Jordan blocks of order } \geq 3,\\
  ...\\
  &\theta_{\mu-2}\quad \text{number of Jordan blocks of order }  \mu,\\
 \end{align*}
 play their role and one has $d=\rank E-\sum_{i=0}^{\mu-2}\theta_i $. Generalizations of these characteristic numbers play a major role further on.

\medskip
For readers who are familiar with at least one of the  DAE concepts discussed in this article, for a better understanding of the meaning of the characteristic values $\theta_i$ we recommend taking a look at Theorem \ref{t.Sum_equivalence} already now.

\section{Basic terms and beyond that}\label{s.Basic&more}
\subsection{What serves as our basic regularity notion}\label{s.regular}
In our view, the elimination-reduction approach to DAEs is the most immediately obvious and accessible with the least technical effort, which is why we choose it as the basis here. We largely use the representation from \cite{RaRh}.

We turn to the ordered pair $\{E,F\}$ of matrix functions $E,F:\mathcal I\rightarrow\Real^{m\times m}$ being sufficiently smooth, at least continuous, and consider the associated DAE
\begin{align}\label{DAE0}
 E(t)x'(t)+F(t)x(t)=q(t),\quad t\in \mathcal I,
\end{align}
as well as the accompanying time-varying subspaces in $\Real^{m}$,
\begin{align}\label{sub}
 \ker E(t),\quad S(t)=\{z\in \Real^m:F(t)z\in\im E(t)\},\quad t\in\mathcal I.
\end{align}
Let $S_{can}$ denote the so-called \emph{flow-subspace} of the DAE, which means that $S_{can}(\bar t)$ is the subspace containing the overall flow of the homogeneous DAE at time $\bar t$, that is, the set of all possible function values $x(\bar t)$ of solutions of the DAE $Ex'+Fx=0$\footnote{$S_{can}(\bar t)$ is also called \emph{linear subspace of initial values which are consistent at time }$\bar t$, e.g., \cite{BergerIlchmann}.},
\begin{align*}
 S_{can}(\bar t):=\{\bar x\in \Real^m: \text{there is a solution \;} x:(\bar t-\delta,\bar t+\delta)\cap\mathcal I\rightarrow \Real^m,\; \delta>0,\\
 \text{ of the homogeneous DAE such that } x(\bar t)=\bar x\},\quad \bar t\in\mathcal I.
\end{align*}

In accordance with various concepts, see \cite[Remark 3.4]{HaMae2023}, we 
agree on what \emph{regular} DAEs are, and show that then the time-varying flow-subspace $S_{can}(\bar t)$ is well-defined on all $\mathcal I$, and has constant dimension.

\begin{definition}\label{d.qualified}
 The pair  $\{E,F\}$ is called \emph{qualified} on $\mathcal I$ if 
 \begin{align*}
  \im [E(t) \,F(t)]=\Real^m,\quad \rank E(t)=r,\quad  t\in\mathcal I,
 \end{align*}
with integers $0\leq r\leq m$.
\end{definition}

\begin{definition}\label{d.prereg}
 The pair  $\{E,F\}$ and the DAE \eqref{DAE0}, respectively, are called \emph{pre-regular} on $\mathcal I$ if 
 \begin{align*}
  \im [E(t) \,F(t)]=\Real^m,\quad \rank E(t)=r,\quad \dim S(t)\cap \ker E(t)=\theta, \quad t\in\mathcal I,
 \end{align*}
with integers $0\leq r\leq m$ and $\theta\geq 0$.
Additionally, if $\theta =0$ and $r<m$, then the DAE is called \emph{regular with index one}, but if $\theta =0$ and $r=m$, then the DAE is called \emph{regular with index zero}.
\end{definition}
We underline that any pre-regular pair $\{E,F\}$ features three subspaces $S(t)$, $\ker E(t)$,  and $S(t)\cap \ker E(t)$ having constant dimensions $r$, $m-r$, and $\theta$, respectively. 
\medskip

We emphasize and keep in mind that now not only the coefficients are time dependent, but also the resulting subspaces. Nevertheless, we suppress in the following mostly the argument $t$, for the sake of better readable formulas. The equations and relations are then meant pointwise for all arguments.
\medskip

The different cases for $\theta=0$ are well-understood. 
A regular index-zero DAE is actually a regular implicit  ODE and $S_{can}=S=\Real^m,\, \ker E=\{0\}$. 
Regular index-one DAEs feature $S_{can}=S,\, \dim\ker E> 0$, e.g., \cite{GM86,CRR}. Note that $r=0$ leads to $S_{can}=\{0\}$.
All these cases are only interesting here as intermediate results.
\bigskip

We turn back to the general case, describe the flow-subspace $S_{can}$, and end up with a regularity notion associated with a regular flow.
 
The pair $\{E, F\}$ is supposed to be \emph{pre-regular}.
The first step of the reduction procedure from  \cite{RaRh} is then well-defined, we refer to \cite[Section 12]{RaRh} for the substantiating arguments.  In the first instance, we apply this procedure to homogeneous DAEs only.

We start by $E_0=E,\,F_0=F,\,m_0=m,\,r_{0}=r$, $\theta_0=\theta$, and consider the homogeneous DAE
\begin{align*}
 E_0x'+F_0 x=0.
\end{align*}
By means of a basis $Z_0:\mathcal I\rightarrow \Real^{m_0\times(m_0-r_0)}$ of $(\im E_0)^{\perp}=\ker E_0^{*}$ and a basis $Y_0:\mathcal I\rightarrow \Real^{m_0\times r_0}$ of $\im E_0$ we divide the DAE into the two parts
\begin{align*}
 Y_0^*E_0x'+Y_0^*F_0x=0,\quad Z_0^*F_0x=0.
\end{align*}
From  $\im[E_0,\,F_0]=\Real^m$ we derive that $\rank Z_0^*F_0= m_0-r_{0}$, and hence the subspace   
$S_{0}=\ker Z_0^*F_0$ has dimension $r_{0}$. Obviously, each solution of the homogeneous DAE must stay in the subspace $S_{0}$. Choosing a continuously differentiable basis 
 $C_0:\mathcal I\rightarrow \Real^{m_0\times r_0}$ of $S_{0}$, each solution of the DAE can be represented as $x=C_0 x_{(1)}$, with a function $x_{(1)}:\mathcal I\rightarrow \Real^{r_0}$ satisfying the DAE reduced to size $m_1=r_{0}$,
 \begin{align*}
 Y_0^*E_0C_0 x_{(1)}'+Y_0^*(F_0C_0+E_0C'_0)x_{(1)}=0.
\end{align*}
Denote $E_1=Y_0^*E_0C_0$ and $F_1=Y_0^*(F_0C_0+E_0C'_0)$ which have size $m_1\times m_1$.
The pre-regularity assures that $E_1$ has constant rank $r_{1}=r_0-\theta_0\leq r_{0}$. Namely,
we have
\begin{align*}
 \ker E_1=\ker E_0C_0=C^{+}_0 (\ker E_0\cap S_{0}),\quad \dim \ker E_1= \dim (\ker E_0\cap S_{0})=\theta_0.
\end{align*}
Here, $C_0(t)^+$ denotes the Moore-Penrose generalized inverse of $C_0(t)$.

Next we repeat the reduction step,
\begin{equation}\label{basic_reduction}
\begin{array}{rl}
 E_{i}&:=Y_{i-1}^*E_{i-1}C_{i-1},\quad F_{i}:= Y_{i-1}^*(F_{i-1}C_{i-1}+E_{i-1}C'_{i-1}),\\
 &Y_{i-1}, Z_{i-1}, C_{i-1} \text{ are smooth bases of the three subspaces}\\
 &\quad \im E_{i-1},\; (\im E_{i-1})^{\perp},\; \text{ and }\; S_{i-1}:=\ker Z_{i-1}^*F_{i-1},\\
 &\theta_{i-1}=\dim (\ker E_{i-1}\cap S_{i-1}),
 \end{array}
\end{equation}
supposing that the new pair 
$\{E_{i},F_{i}\}$ is  pre-regular again, and so on. The pair $\{E_{i},F_{i}\}$ has size $m_{i}:=r_{i-1}$ and $E_{i}$ has rank $r_{i}=r_{i-1}-\theta_{i-1}$.
This yields the decreasing sequence
$ m\geq r_{0}\geq \cdots\geq r_{j}\geq r_{j-1}\geq\cdots \geq 0 $ and rectangular matrix functions $C_i: \mathcal I \rightarrow \Real^{r_{i-1}\times r_i}$ with full column-rank $r_i$.
Denote by $\mu$ the smallest integer  such that either $r_{\mu-1}=r_{\mu}>0$ or $r_{\mu-1}=0$.
Then, it follows that $(\ker E_{\mu-1})\cap S_{\mu-1}=\{0\}$,  which means in turn  that
\begin{align*}
 E_{\mu-1}x_{(\mu-1)}'+F_{\mu-1}x_{(\mu-1)}=0
\end{align*}
represents a regular index-1 DAE. 

If  $r_{\mu-1}=0$, that is $E_{\mu-1}=0$, then $F_{\mu-1}$ is nonsingular due to the pre-regularity of the pair, which leads to $S_{\mu-1}=\{0\}$, $C_{\mu-1}=0$, and a zero flow 
$x_{(\mu-1)}(t)\equiv 0 $.
In turn  there is only the  identically vanishing solution 
\[x=C_0C_1\cdots C_{\mu-2}x_{(\mu-1)}=0\]
of the homogeneous DAE, and $C_0C_1\cdots C_{\mu-2}C_{\mu-1}=0$.

On the other hand,
 if $r_{\mu-1}=r_{\mu}>0$ then $x_{(\mu-1)}=C_{\mu-1}x_{(\mu)}$, $\rank C_{\mu-1}=r_{\mu-1}$, 
 and 
 $E_{\mu}$ remains nonsingular such that  the DAE 
\begin{align*}
 E_{\mu}x_{(\mu)}'+F_{\mu}x_{(\mu)}=0
\end{align*}
  is actually an implicit regular ODE living in $\Real^{m_{\mu}}$,\;$m_{\mu}=r_{\mu-1}$ and $S_{\mu}=\Real^{r_{\mu-1}}$.
Letting $C_{\mu}=I_{m_{\mu}}= I_{r_{\mu-1}}$, each solutions of the original homogeneous DAE \eqref{DAE0} has the form
\begin{align*}
x=Cx_{(\mu)}, \quad C:=C_0C_1\cdots C_{\mu-1}=C_0C_1\cdots C_{\mu-1}C_{\mu}
  :\mathcal I\rightarrow\Real^{m\times r_{\mu-1}},\; \rank C=r_{\mu-1}.
\end{align*}
Moreover, for each $\bar t\in \mathcal I$ and each $z\in\im C(\bar t)$, there is exactly one solution of the original homogeneous DAE passing through, $x(\bar t)=z$ which indicates that $\im C=S_{can}$.

As proved in \cite{RaRh}, the ranks $r=r_{0}> r_{1}>\cdots> r_{\mu-1}$ are independent of the special choice of the involved basis functions. In particular,
\begin{align*}
 d:=r_{\mu-1}=r-\sum_{i=0}^{\mu-2}\theta_i=\dim C
\end{align*}
appears to be the dynamical degree of freedom of the DAE.

The property of pre-regularity does not necessarily carry over to the subsequent reduction pairs, e.g.,\cite[Example 3.2]{HaMae2023}.

\begin{definition}\label{d.2a}
 The pre-regular pair  $\{E,F\}$ with $r<m$  and the associated DAE \eqref{DAE0}, respectively, are called \emph{regular} if there is an integer $\mu\in\Natu$ such that 
 the above reduction procedure \eqref{basic_reduction} is well-defined up to level $\mu-1$, each pair $\{E_{i},F_{i}\}$, $i=0,\ldots,\mu-1$, is pre-regular,  and if $r_{\mu-1}>0$ then $E_{\mu}$ is well-defined and  nonsingular, $r_{\mu}=r_{\mu-1}$. If $r_{\mu-1}=0$ we set $r_{\mu}=r_{\mu-1}=0$.
 
 The integer $\mu$ is called \emph{the index of the DAE \eqref{DAE0} and the given pair $\{E,F\}$}.
 The index $\mu$ and the ranks $r=r_{0}> r_{1}>\cdots > r_{\mu-1}=r_{\mu}$ are called \emph{characteristic values} of the pair and the DAE, respectively.
\end{definition}
By construction, for a regular pair it follows that $r_{i+1}= r_{i}-\theta_{i}$, $i=0,\ldots,\mu-1$. Therefore, in place of the above $\mu+1$ rank values $r_0,\ldots,r_{\mu}$, the following rank and the dimensions,
\begin{align}\label{theta}
 r \quad \text{and} \quad \theta_0 \geq\theta_1 \geq \cdots \geq \theta_{\mu-2} >\theta_{\mu-1}=0, 
\end{align}
\begin{align}\label{thetadef}
 \theta_{i}=\dim (\ker E_{i}\cap S_{i}),\; i\geq0,
\end{align}

can serve as characteristic quantities. Later it will become clear that these data also play an important role in other concepts, too, which  is the reason for the following definition equivalent to Definition \ref{d.2a}. 
\begin{definition}\label{d.2}
 The pre-regular pair  $\{E,F\}$, $E,F:\mathcal I\rightarrow \Real^{m\times m}$,  with $r=\rank E<m$,  and the associated DAE \eqref{DAE0}, respectively, are called \emph{regular} if there is an integer $\mu\in\Natu$ such that 
 the above reduction procedure is well-defined up to level $\mu-1$, with each pair $\{E_{i},F_{i}\}$, $i=0,\ldots,\mu-1$, being pre-regular,  and associated values \eqref{theta}.
 
 The  integer $\mu$ is called \emph{the index of the DAE \eqref{DAE0} and the given pair $\{E,F\}$}.
 The index $\mu$ and the values \eqref{theta} are called \emph{characteristic values} of the pair and the DAE, respectively.
\end{definition}
At this place we add the further  relationship,
\begin{align}\label{thetarank}
 \theta_{i}=\dim\ker \begin{bmatrix}
                      Y_{i}^{*}E_{i}\\Z_{i}^{*}F_{i}
                     \end{bmatrix}= m_{i} -\rank \begin{bmatrix}
                      Y_{i}^{*}E_{i}\\Z_{i}^{*}F_{i}
                     \end{bmatrix},\quad i=0,\ldots,\mu-1,
\end{align}
with which all quantities in \eqref{theta} are related to rank functions.
\begin{remark}\label{r.pencil}
 If $\{E,F\}$ is actually a pair of matrices $E,F\in \Real^{m\times m}$, then the pair is regular with index $\mu$ and characteristics $\theta_0\geq\cdots\geq\theta_{\mu-2}>\theta_{\mu-1}=0$, if and only if the matrix pencil is regular and the nilpotent matrix in its Kronecker normal form  shows
 \begin{align*}
  &\theta_0 \quad \text{Jordan blocks of order } \geq 2,\\
  &\theta_1 \quad \text{Jordan blocks of order } \geq 3,\quad\\
	&...\\
  &\theta_{\mu-2}\quad \text{Jordan blocks of order } \mu,
 \end{align*}
\end{remark}
\begin{remark}\label{r.regularity}
 As mentioned above, the presentation in this section mainly goes back to \cite{RaRh}. However, we have not taken up their notations \emph{regular} and \emph{completely regular} for the coefficient pairs  and \emph{reducible} and \emph{completely reducible} for DAEs, but that of other works, what we consider more appropriate to the matter.\footnote{In \cite{RaRh}, the coefficient pairs of DAEs which have arbitrary many solutions like \cite[Example 3.2 ]{HaMae2023} may belong to  \emph{regular} ones.}
 
 Not by the authors themselves, but sometimes by others, the index from \cite{RaRh} is also called \emph{geometric index}, e.g., \cite[Subsection 2.4]{RR2008}.
 
 An early predecessor version of this reduction procedure 
was already proposed and analyzed in \cite{Cis1982} under the name  \emph{elimination of the unknowns}, even for  more general pairs of rectangular matrix functions, see also Subsection \ref{subs.elimination}.
The regularity notion given in \cite{Cis1982} is consistent with Definition \ref{d.2}. 
Another very related such reduction technique has been presented and extended a few years ago under the name \emph{dissection concept} \cite{Jansen2014}. This notion of regularity also agrees with Definition \ref{d.2}, see Section \ref{subs.dissection}.

\end{remark}
\begin{theorem}\label{t.Scan}
 Let the DAE  \eqref{DAE0} be regular on $\mathcal I$ with index $\mu$ and  characteristic values \eqref{theta}.
 \begin{description}
  \item[\textrm{(1)}] Then the subspace $S_{can}(t)\subset \Real^m$ has  dimension $d=r-\sum_{i=0}^{\mu-2}\theta_i=r_{\mu-1}$ for all $t\in\mathcal I$, and
  the matrix function $C:\mathcal I\rightarrow \Real^{m\times d}$, $C=C_{0}\cdots C_{\mu-2}$, generated by the reduction procedure is a basis of $S_{can}$.
   \item[\textrm{(2)}]  The DAE features precisely the same structure on each subinterval $\mathcal I_{sub}\subset \mathcal I$.
 \end{description}
\end{theorem}
\begin{proof}
Regarding the relation $r_{i+1}= r_{i}-\theta_{i}$, $i=0,\ldots,\mu-2$ directly resulting from the reduction procedure, the assertion 
 is an immediate consequence of \cite[Theorem 13.3]{RaRh}.
\end{proof}
Two canonical subspaces varying with time in $\Real^m$ are associated with a regular DAE \cite{CRR,HaMae2023}.  The first one is the flow-subspace $S_{can}$. The second one is a unique pointwise complement $N_{can}$ to the flow-subspace, such that
\begin{align*}
 S_{can}(t)\oplus N_{can}(t)=\Real^m,\quad N_{can}(t)\supset \ker E(t),\quad t\in \mathcal I,
\end{align*}
and the initial condition $x(\bar t)-\bar x\in N_{can}(\bar t)$ fixes exactly one of the DAE solutions for each given\\ $\bar t\in \mathcal I,\, \bar x\in \Real^{m}$  without any consistency conditions for the right-hand side $q$ or its derivatives, \cite[Theorem 5.1]{HaMae2023}, also \cite{CRR}.

\begin{theorem}\label{t.solvability}
 If the DAE  \eqref{DAE0} is regular on $\mathcal I$ with index $\mu$ and  characteristics 
 \eqref{theta}, then the following assertions are valid:
 \begin{description}
  \item[\textrm{(1)}] The DAE is solvable at least for each arbitrary right-hand side $q\in C^{m}(\mathcal I,\Real^{m})$.
  \item[\textrm{(2)}] $d=r-\sum_{i=0}^{\mu-2}\theta_i=r_{\mu-1}$  is the dynamical degree of freedom.
   \item[\textrm{(3)}] The condition $r=\sum_{i=0}^{\mu-2}\theta_i$ indicates a DAE with zero degree of freedom\footnote{So-called \emph{purely algebraic} systems.} and $S_{can}=\{0\}$, i.e. $d=0$.
    \item[\textrm{(4)}] For arbitrary given $q\in C^{m}(\mathcal I,\Real^{m})$, $\bar t\in \mathcal I$, and $\bar x\in\Real^m$, the initial value problem
    \begin{align*}
     Ex'+Fx=q,\quad x(\bar t)=\bar x,
    \end{align*}
is uniquely solvable, if the consistency condition \eqref{cons2} in the proof below is satisfied. Otherwise there is no solution.
  \item[\textrm{(5)}] The DAE has perturbation index $\mu$ on each compact subinterval of $\mathcal I$.
 \end{description}
\end{theorem}
\begin{proof}
 \textrm (1): Given $q\in C^{m}(\mathcal I,\Real^{m})$ we apply the previous reduction now to the inhomogeneous DAE \eqref{DAE0}. We describe the first level only.  The general solution of the derivative-free part $Z_0^{*}F_0x=Z_0^{*}q$ of the given DAE reads now
 \begin{align*}
  x=(I-(Z_0^{*}F_0)^{+}Z_0^{*}F_0)x+ (Z_0^{*}F_0)^{+}Z_0^{*}F_0x= C_0x_{(1)}+ (Z_0^{*}F_0)^{+}Z_0^{*}q,
 \end{align*}
and inserting into $Y_0^{*}E_0x'+Y_0^{*}F_0x=Z_0^{*}q$ yields the reduced DAE $E_{1}x_{(1)}'+F_{1}x_{(1)}=q_{(1)}$, with
\begin{align*}
q_{(0)}=q,\;
 q_{(1)}=Y_0^{*}q_{(0)}-Y_0^{*}E_0((Z_0^{*}F_0)^{+}Z_0^{*}q_{(0)})'-Y_0^{*}F_0(Z_0^{*}F_0)^{+}Y_0^{*}q_{(0)}.
\end{align*}
Finally, using the constructed above matrix function sequence, each solution of the DAE has the form
\begin{align}
 x&=C_0x_{(1)}+(Z_0^*F_0)^{+}Z_0^*q_{(0)}
 =C_0(C_1x_{(2)}+(Z_1^*F_1)^{+}Z_1^*q_{(1)})+(Z_0^*F_0)^{+}Z_0^*q_{(0)}=\cdots\nonumber\\
 &=\underbrace{C_0C_1\cdots C_{\mu-2}}_{= C}x_{(\mu-1)}+p,\label{DAEsol}\\
p&=(Z_0^*F_0)^{+}Z_0^*q_{(0)}+C_0(Z_1^*F_1)^{+}Z_1^*q_{(1)}+\cdots+ C_0C_1\cdots C_{\mu-2}(Z_{\mu-1}^*F_{\mu-1})^{+}Z_{\mu-1}^*q_{(\mu-1)}, \nonumber \\\nonumber 
\quad &q_{(j+1)}=Y_{j}^{*}q_{(j)}-Y_{j}^{*}E_{j}((Z_{j}^{*}F_{j})^{+}Z_{j}^{*}q_{(j)})'-Y_{j}^{*}F_{j}(Z_{j}^{*}F_{j})^{+}Y_{j}^{*}q_{(j)},\; j=0,\ldots,\mu-2,\nonumber
\end{align}
in which $x_{(\mu-1)}$ is any solution of the regular index-one DAE
\[ E_{\mu-1}x_{(\mu-1)}'+F_{\mu-1}x_{(\mu-1)}=q_{(\mu-1)}. 
\]

Since $q$ and the coefficients are supposed to be smooth, all derivatives exist, and no further conditions with respect to $q$ will arise.

{\textrm (4)} Expression \eqref{DAEsol} yields $x(\bar t)= C(\bar t)x_{[\mu]}(\bar t)+p(\bar t)$. The initial condition $x(\bar t)=\bar x$ splits by means of the projector $\pPi_{can}(\bar t)$  onto $S_{can}(\bar t)$ along $N_{can}(\bar t)$ into the two parts 
\begin{align}
 \pPi_{can}(\bar t)\bar x= C(\bar t) x_{[\mu]}(\bar t)+ \pPi_{can}(\bar t)p(\bar t)\label{cons1},\\
( I-\pPi_{can}(\bar t))\bar x= (I-\pPi_{can}(\bar t))p(\bar t)\label{cons2}.
\end{align}
Merely part \eqref{cons1} contains the component $x_{(\mu)}(\bar t)$, which is to be freely selected  in $\Real^{r_{\mu-1}}$, and \\
 $x_{(\mu)}(\bar t)=C(\bar t)^+\pPi_{can}(\bar t)(\bar x-p(\bar t))$ is the only solution. 
 
 In contrast, \eqref{cons2} does not contain any free components. It is a strong consistency requirement and must be given a priori for solvability. Otherwise this (overdetermined) initial value problem fails to be solvable.

{\textrm (2),(3),(5)} are straightforward now, for details see \cite[Theorem 5.1]{HaMae2023}.
 \end{proof}
The following proposition comprises enlightening special cases which will be an useful tool to provide equivalence assertions later on. Namely, for given integers $\kappa \geq 2$, $d\geq 0$, $l=l_{1}+\cdots +l_{\kappa}$, $l_{i}\geq
 1$, $m=d+l$
we consider the pair $\{E,F\}$, $E,F:\mathcal I\rightarrow \Real^{m\times m}$, in special block structured form,
\begin{align}\label{blockstructure}
 E=\begin{bmatrix}
    I_{d}&\\&N
   \end{bmatrix},\quad
F=\begin{bmatrix}
    \Omega&\\&I_{l}
   \end{bmatrix}, \quad
N=\begin{bmatrix}
   0&N_{12}&&\cdots&N_{1\kappa}\\
   &0&N_{23}&&N_{2\kappa}\\
   &&\ddots&\ddots&\vdots\\
   &&&&N_{\kappa-1, \kappa}\\
   &&&&0
   \end{bmatrix},\\
   \text{with blocks}\quad N_{ij} \quad\text{of sizes}\quad l_{i}\times l_{j}.\nonumber
\end{align}
If $d=0$ then the respective parts are absent. All blocks are sufficiently smooth on the given interval $\mathcal I$. $N$ is strictly block upper triangular, thus nilpotent and $N^{\kappa}=0$. 

We set further $N=0$ for $\kappa=1$.
Obviously, then the pair $\{E,F\}$ is pre-regular with $r=d$ and $\theta_0=0$, and hence the DAE has index $\mu=\kappa=1$. Below we are mainly interested in the case  $\kappa\geq 2$.
\begin{proposition}\label{p.STform}
Let the pair $\{E,F\}$, $E,F:\mathcal I\rightarrow \Real^{m\times m}$ be given in the form \eqref{blockstructure} and $\kappa \geq 2$.
 \begin{description}
  \item[\textrm{(1)}] If the secondary diagonal blocks $N_{i, i+1}:\mathcal I\rightarrow \Real^{l_{i}\times l_{i+1}}$ in \eqref{blockstructure} have full column-rank, that is,
  \[\rank N_{i, i+1}=l_{i+1}, \quad i=1,\ldots,\kappa-1, 
   \]
then $l_{1}\geq \cdots\geq l_{\kappa}$ and the corresponding DAE is regular with index $\mu=\kappa$ and characteristic values
  \begin{align*}
   r=m-l_{1},\; \theta_{0}=l_{2}, \ldots,\,  \theta_{\mu-2}=l_{\mu}.
  \end{align*}
\item[\textrm{(2)}] If the secondary diagonal blocks $N_{i, i+1}$ in \eqref{blockstructure} have full row-rank, that is,
\[\rank N_{i, i+1}=l_{i}, \quad i=1,\ldots,\kappa-1,
\]
then $l_{1}\leq \cdots\leq l_{\kappa}$ and the corresponding DAE is regular with index $\mu=\kappa$ and characteristic values
  \begin{align*}
   r=m-l_{\mu},\; \theta_{0}=l_{\mu-1}, \ldots,\,  \theta_{\mu-2}=l_{1}.
  \end{align*}
 \end{description}
\end{proposition}
\begin{proof}
\textrm (1) Suppose the secondary diagonal blocks $N_{i, i+1}$ have full column-ranks $l_{i+1}$.
It results that $r=\rank E =d+l-l_1=m-l_1$ and $\theta_0=\dim S\cap\ker E=\dim(\ker N\cap\im N)= \rank N_{12}=l_2$, thus the pair is pre-regular. For deriving the reduction step we form the two auxiliary matrix functions
\begin{align*}
 \tilde N=\begin{bmatrix}
   N_{12}&&\cdots&N_{1\kappa}\\
  0 &N_{23}&&N_{2\kappa}\\
   &&\ddots&\vdots\\
   &&&N_{\kappa-1, \kappa}\\
   0&&\cdots&0
   \end{bmatrix}:\mathcal I\rightarrow\Real^{l\times (l-l_1)},\quad 
   \tilde E=\begin{bmatrix}
             I_d&\\&\tilde N
            \end{bmatrix}:\mathcal I\rightarrow\Real^{m\times (m-l_1)},
\end{align*}
which have full column rank, $l-l_1$ and $m-l_1$, respectively. By construction, one has $\im \tilde N=\im N$, $\im \tilde E=\im E$. The matrix function $C=\tilde E$ serves as basis of the subspace 
\begin{align*}
 S=\left\{ \begin{bmatrix}
      u\\v
     \end{bmatrix}\in \Real^{d+l}:v\in \im N
\right\}.
\end{align*}
Furthermore, with any smooth pointwise nonsingular matrix function $M:\mathcal I\rightarrow \Real^{(m-l_1)\times (m-l_1)}$, the matrix function $Y=\tilde EM$ serves as a basis of $\im E$. We will specify $M$ subsequently. 

Since $\tilde N^*\tilde N$ remains pointwise nonsingular, one obtains the relations
\begin{align*}
 \mathfrak A:=[
              \underbrace{0}_{l_1} \;\underbrace{\tilde N^*\tilde N}_{l-l_1}
             ]\,\tilde N
             = \tilde N^*\tilde N \,[
              \underbrace{0}_{l_1} \,I_{l-l_1}]\, \tilde N =  \tilde N^*\tilde N \; \mathring{N_1}
\end{align*}
with the structured matrix function
\begin{align*}
 \mathring{N_1}=\begin{bmatrix}
   0&N_{23}&&\cdots&N_{2\kappa}\\
   &0&N_{34}&&N_{3\kappa}\\
   &&\ddots&\ddots&\vdots\\
   &&&&N_{\kappa-1, \kappa}\\
   &&&&0
   \end{bmatrix} :\mathcal I\rightarrow \Real^{(l-l_1)\times(l-l_1)},\\
   \text{again with the full column-rank blocks}\; N_{ij}.
\end{align*}
We will show that the reduced pair $\{E_1,F_1\}$ actually features an analogous structure. We have
\begin{align*}
 E_1= Y^{*}EC=M^{*}\begin{bmatrix}
                    I_d&\\&\mathfrak A
                   \end{bmatrix}
=M^{*}\begin{bmatrix}
                    I_d&\\& \tilde N^*\tilde N \; \mathring{N_1}
                   \end{bmatrix}
                   =
                   M^{*}\begin{bmatrix}
                    I_d&\\& \tilde N^*\tilde N 
                   \end{bmatrix}
                   \begin{bmatrix}
                    I_d&\\&  \mathring{N_1}
                   \end{bmatrix},
\end{align*}
and 
 \begin{align*}
  F_1&= Y^{*}FC+ Y^{*}EC'=M^{*}\begin{bmatrix}
                    I_d&\\& \tilde N^*\tilde N 
                   \end{bmatrix}
                   \begin{bmatrix}
                    \Omega&\\& I_{l-l_1} +\mathring{N_1'}
                   \end{bmatrix}\\
                   &=
  M^{*}\begin{bmatrix}
                    I_d&\\& \tilde N^*\tilde N (I_{l-l_1} +\mathring{N_1'})
                   \end{bmatrix}
                   \begin{bmatrix}
                    \Omega&\\& I_{l-l_1}
                   \end{bmatrix}.                 
 \end{align*}
 Regarding that $I_{l-l_1}+\mathring{N_1'}$ is nonsingular, we choose 
\begin{align*}
 M^{*}=\begin{bmatrix}
                    I_d&\\& (I_{l-l_1} +\mathring{N_1'})^{-1}(\tilde N^*\tilde N )^{-1}
                   \end{bmatrix},
\end{align*}
which leads to 
\begin{align*}
E_1&=\begin{bmatrix}
                    I_d&\\& (I_{l-l_1}+\mathring{N_1'})^{-1}
                   \end{bmatrix}\begin{bmatrix}
                    I_d&\\&  \mathring{N_1}
                   \end{bmatrix}=
                   \begin{bmatrix}
                    I_d&\\& (I_{l-l_1}+\mathring{N_1'})^{-1} \mathring{N_1}
                   \end{bmatrix}=:
                   \begin{bmatrix}
                    I_d&\\& N_1
                   \end{bmatrix},\\
 F_1&=\begin{bmatrix}
                    \Omega&\\& I_{l-l_1}
                   \end{bmatrix}.
\end{align*}
By construction, see Lemma \ref{l.SUT1}, the resulting matrix function $N_1$ has again strictly upper triangular block structure and it shares its secondary diagonal blocks with those from $N$ (except for $N_{12}$), that is
\begin{align*}
 N_1=\begin{bmatrix}
   0&N_{23}&*&\cdots&*\\
   &0&N_{34}&&*\\
   &&\ddots&\ddots&\vdots\\
   &&&&N_{\kappa-1, \kappa}\\
   &&&&0
   \end{bmatrix} :\mathcal I\rightarrow \Real^{(l-l_1)\times(l-l_1)}.
\end{align*}
Thus, the new pair has an analogous block structure as the given one, is again pre-regular but now with $m_1=r=m-l_1 $, $r_1=m_1-l_2=m-l_1-l_2 $, $\theta_1=\rank N_{23}=l_3$.
Proceeding further in such a way we arrive at
the pair $\{E_{\kappa-2},F_{\kappa-2}\}$,
\begin{align*}
 E_{\kappa-2}=\begin{bmatrix}
               I_d&\\&N_{\kappa-2}
              \end{bmatrix},\quad
F_{\kappa-2}=\begin{bmatrix}
               \Omega&\\&I_{l_{\kappa-1}+l_{\kappa}}
              \end{bmatrix},\quad
              N_{\kappa-2}=\begin{bmatrix}
               0&N_{\kappa-1 ,\kappa}\\0&0
              \end{bmatrix},
\end{align*}
with $m_{\kappa-2}=m-l_1-\cdots-l_{\kappa-2}$, $r_{\kappa-2}=m-l_1-\cdots-l_{\kappa-1}=d+l_{\kappa}$, and $\theta_{\kappa-2}=\rank N_{\kappa-1,\kappa}=l_{\kappa}$, and the  final pair $\{E_{\kappa-1},F_{\kappa-1}\}$,
\begin{align*}
 E_{\kappa-1}=\begin{bmatrix}
               I_d&\\&0
              \end{bmatrix},\quad
F_{\kappa-1}=\begin{bmatrix}
               \Omega&\\&I_{l_{\kappa}}
              \end{bmatrix},\quad m_{\kappa-1}=d+l_{\kappa}, r_{\kappa-1}=d, \theta_{\kappa-1}=0,
\end{align*}
which completes the proof of the first assertion.

\textrm (2): We suppose now secondary diagonal blocks $N_{i, i+1}$ which have full row-ranks $l_i$, thus nullspaces of dimension $l_{i+1}-l_i$, $ i=1,\ldots,\kappa-1$. The pair $\{E,F\}$ is pre-regular and $r=\rank E=d+\rank N=d+l-l_{\kappa}= m-l_{\kappa}$, and $\dim \ker E\cap S=\dim \ker N\cap\im N=l_{1}+(l_2-l_1)+\cdots + (l_{\kappa -1}-l_{\kappa-2})= l_{\kappa-1} $, thus $\theta_0=l_{\kappa-1} $. The constant matrix function
\begin{align*}
 C=\begin{bmatrix}
    I_d&&&\\
    &I_{l_1}&&\\
    &&\ddots&\\
    &&&I_{l_{\kappa -1}}\\
    &&&0
   \end{bmatrix}
\end{align*}
serves as a basis of $S$ and also as a basis of $\im E$, $Y=C$. This leads simply to
\begin{align*}
 E_1=C^*EC=\begin{bmatrix}
            I_d&\\&N_1
           \end{bmatrix}, \quad 
  F_1=C^*FC=\begin{bmatrix}
            \Omega&\\&I_{l-l_{\kappa}}
           \end{bmatrix},         
\end{align*}
with $m_1=m-l_{\kappa}$, $r_1=m_1-l_{}$, and
\begin{align*}
 N_1=\begin{bmatrix}
   0&N_{12}&&\cdots&N_{1,\kappa-1}\\
   &0&N_{23}&&N_{2,\kappa-1}\\
   &&\ddots&\ddots&\vdots\\
   &&&&N_{\kappa-2, \kappa-1}\\
   &&&&0
   \end{bmatrix}.
\end{align*}
It results that $\theta_1=l_{\kappa-2}$. 
and so on.
\end{proof}

In Section \ref{sec:SCF} and Section \ref{subs.A_strictly} we go into further detail about these two structural forms from Proposition \ref{p.STform} and also illustrate there the difference to the Weierstraß–Kronecker form with a simple example.

\subsection{A specifically geometric view on the matter}\label{subs.degree}
A regular DAE living in $\Real^m$ can now be viewed as an embedded regular implicit ODE in $\Real^d$, which in turn uniquely defines a vector field on the configuration space $\Real^d$. 
Of course, this perspective has an impressive potential in the case of nonlinear problems, when smooth submanifolds replace linear subspaces, etc. We will give a brief outline and references in Section \ref{s.nonlinearDAEs} below. An important aspect hereby is that one first provides the manifold that makes up the configuration space, and only then examine the flow, which allows also  a flow that is not necessarily regular. In this context, the extra notion \emph{degree of the DAE} introduced by \cite[Definition 8]{Reich}\footnote{Definition \ref{d.degree} below.} is relevant. It actually measures the  degree of the embedding depth.

In the present section we concentrate on the linear case and do not use the special geometric terminology. Instead we adapt the notion so that it fits in with our presentation.
\medskip

Let us start by a further look at the basic procedure yielding a regular DAE.
In the second to last step of our basis reduction, the pair  $\{E_{\mu-1},F_{\mu-1}\}$ is pre-regular and $\theta_{\mu-1}=0$ on all $\mathcal I$. If thereby $r_{\mu-1}=0$ then there is no dynamic part, one has $d=0$ and $S_{can}=\{0\}$. This instance is of no further interest within the geometric context. 

However, the interest comes alive, if $r_{\mu-1}>0$. 
 Recall that by construction $r_{\mu-1}=r_0-\sum_{i=0}^{\kappa-2}\theta_i=d$.
In the regular case we see
\begin{align*}
 \im C_0\cdots C_{\mu-2}\supsetneqq  \im C_0\cdots C_{\mu-1}=\im C_0\cdots C_{\mu}, \quad r_{\mu-2}>r_{\mu-1}=r_{\mu}.
\end{align*}

If now the second to last pair would fail to be pre-regular, but would be
qualified with the associated rank function $\theta_{\mu-1}$ being positiv at a certain point $t_*\in\mathcal I$, and zero otherwise on $\mathcal I$, then the eventually  resulting last  matrix function $E_{\mu}(t)$ fails to remain nonsingular just at this critical point, because of $\rank E_{\mu}(t)=r_{\mu-1}-\theta_{\mu-1}(t)$. Nevertheless, we could state  $C_{\mu}=I_{r_{\mu-1}}$ and arrive at
\begin{align*}
 \im C_0\cdots C_{\mu-2}\supsetneqq  \im C_0\cdots C_{\mu-1}=\im C_0\cdots C_{\mu}, \quad r_{\mu-2}>r_{\mu-1}\geq r_{\mu}(t)
\end{align*}
Clearly, then the resulting ODE in $\Real^{r_{\mu-1}}$ and in turn the given DAE are no longer regular and one is confronted with a singular vector field.
\begin{example}\label{e.degree}
Given is the qualified pair with $m=2, r=1$,
 \begin{align*}
  E(t)=\begin{bmatrix}
        1&-t\\1&-t
       \end{bmatrix},\quad 
F(t)=\begin{bmatrix}
        2&0\\0&2
       \end{bmatrix},\quad t\in \Real,
 \end{align*}
yielding 
\begin{align*}
 Z_0&=\begin{bmatrix}
      1\\-1
     \end{bmatrix},\; 
     Z_0^*F_0=
\begin{bmatrix}
      2\\-2
     \end{bmatrix}, \;
  C_0=   \begin{bmatrix}
      1\\1
     \end{bmatrix},\; 
Y_0=   \begin{bmatrix}
      1\\1
     \end{bmatrix},\\ 
     E_1(t)&=2(1-t),\; F_1(t)=4,\;  m_1=r_0=1, \\     
  &\ker E_0(t) \cap  \ker (Z_0^*F_0)(t)=\{z\in\Real^2: z_1-tz_2=0, z_1=z_2\},
\end{align*}
and further $\theta_0(t)=0$ for $t\neq 1$, but $\theta_0(1)=1$. The homogeneous DAE has the solutions
\begin{align*}
 x(t)=\gamma (1-t)^2\begin{bmatrix}
                     1\\1
                    \end{bmatrix},\; t\in \Real, \quad \text{with arbitrary}\; \gamma\in \Real,
\end{align*}
which  manifests the singularity of the flow at point $t_*=1$. Observe that now the canonical subspace varies its dimension, more precisely,
\begin{align*}
 S_{can}(t_*)=\{0\},\quad S_{can}(t)=\im C_0,\; \text{ for all}\;t\neq t_*.
\end{align*}
\end{example}
\begin{definition}\label{d.degreelin}
The DAE given by the pair $\{E,F\}$, $E,F:\mathcal I\rightarrow\Real^{m\times m}$ 
has, if it exists, \emph{degree $s\in \Natu $}, if the reduction procedure in Section \ref{s.regular} is well-defined up to level $s-1$, the pairs  $\{E_i,F_i\}$, $i=0,\ldots,s-1,$  are pre-regular, the pair $\{E_s,F_s\}$ is qualified,
\begin{align*}
 \im C_0\cdots C_{s-1}\supsetneqq  \im C_0\cdots C_{s}, \quad r_{s-1}>r_{s},
\end{align*}
and $s$ is the largest such integer. The subspace 
$\im C_0\cdots C_{s}$ is called \emph{configuration space} of the DAE.
\end{definition}
We mention that $\im C_0\cdots C_{s}= C_0\cdots C_{s} (\Real^{r_{s}})$ and admit that, depending on the view, alternatively, $\Real^{r_{s}}$ can be regarded as the configuration space, too.
\medskip

If the pair $\{E,F\}$ is regular with index $\mu\in\Natu$, then its degree is $s=\mu-1$ and 
\begin{align*}
 \im C_0\cdots C_{\mu-2}\supsetneqq  \im C_0\cdots C_{\mu-1}=\im C_0\cdots C_{\mu}, \quad r_{\mu-2}>r_{\mu-1}=r_{\mu}.
\end{align*}

On the other hand, if the DAE has degree $s$ and $r_{s}=0$ then it results that $C_{s}=0$, in turn $\im C_0\cdots C_{s}=\{0\}$ and $\theta_{s}=0$.  Then the DAE is regular with index $\mu=s+1$ but the configuration space is trivial. As mentioned already, since the dynamical degree is zero, this instance is of no further interest in the geometric context. 
\medskip

Conversely, if the DAE has degree $s$ and $r_{s}>0$, then the pair $\{E_{s},F_{s}\}$ is not necessarily pre-regular but merely qualified such that, nevertheless, the next level $\{E_{s+1},F_{s+1}\}$ is well-defined, we can state $m_{s+1}=r_{s}$, $C_{s+1}=I_{r_{s}}$, and
\begin{align*}
 \rank E_{s+1}(t)&=m_{s+1}-\dim(\ker E_{s}(t)\cap \ker Z^*_{s}(t)F_{s}(t))
 =r_{s}-\theta_{s}(t),\quad t\in\mathcal I.
\end{align*}
It comes out that if $\theta_{s}(t)$ vanishes almost overall on $\mathcal I$, then a vector field with isolated singular points is given. If  $\theta_{s}(t)$ vanishes identically, then the DAE is regular.

This approach unfolds its potential especially for quasi-linear autonomous problems, see \cite{RaRh,Reich} and Section \ref{subs.nonlinearDAEsGeo}, however, the questions concerning the sensibility of the solutions with respect to the perturbations of the right-hand sides fall by the wayside.

\section{Further direct concepts without recourse to derivative arrays}\label{s.Solvab}

We are concerned here with the regularity notions and approaches from \cite{Cis1982,Jansen2014,KuMe2006,CRR} associated with the elimination procedure, the dissection concept, the strangeness reduction, and the tractability framework compared to Definition \ref{d.2}. 
The approaches in \cite{Cis1982,Jansen2014,KuMe2006,RaRh} are 
de facto special solution methods including reduction steps by elimination of variables and differentiations of certain variables.
In contrast,  the concept in \cite{CRR} aims at a structural projector-based decomposition of the given DAE in order to analyze them subsequently.

Each of the concepts is associated with a sequence of pairs of matrix functions, each supported by certain rank conditions that look very different. 
Thus also the regularity notions, which require in each case that the sequences are well-defined with well-defined termination, are apparently completely different. However,
at the end of this section, we will know that all these regularity terms agree with our Definition \ref{d.2}, and that the characteristics \eqref{theta} capture all the rank conditions involved.

When describing the individual method, traditionally the same characters are used to clearly highlight certain parallels, in particular, $\{E_j, F_j \}$ or $\{G_j, B_j \}$ for the matrix function pairs and  $r_j$ for the characteristic values.  Except for the dissection  concept, $r_j$ is the rank of the first pair member $E_j$ and $G_j$, respectively. 

 To avoid confusion we label the different characters with corresponding top indices $E$ (elimination), $D$ (dissection), $S$ (strangeness) and $T$ (tractability), respectively.
The letters without upper index  refer to the basic regularity in Section \ref{s.regular}. In some places we also give an upper index, namely $B$ (basic), for better clarity.
 \medskip
 
 Theorem \ref{t.equivalence} below  will provide the index relations $\mu^{E}=\mu^{D}=\mu^{T}=\mu^{S}+1=\mu^B$ as well as expressions of all $r_j^{E}$, $r_j^{D}$, $r_j^{S}$, and $r_j^{T}$ in terms of \eqref{theta}.
\bigskip

%

\subsection{Elimination of the unknowns procedure}\label{subs.elimination}

A special predecessor version of the  procedure described in  \cite{RaRh}
was already proposed and analyzed in \cite{Cis1982} and entitled by  \emph{elimination of the unknowns}, even for  more general pairs of rectangular matrix functions. Here we describe the issue already in our notation and confine the description to square matrix functions.

Let the pair $\{E,F\}$, $E,F:\mathcal I\rightarrow\Real^{m\times m} $,  be qualified in the sense of Definition \ref{d.qualified},  i.e. $\im [E(t)\;F(t)]=\Real^{m},\;t\in\mathcal I, $ and  $E(t)$ has constant rank $r$ on $\mathcal I$.

Let $T,T^c,Z$, and $Y $ represent bases of $\ker E, (\ker E)^{\perp}, (\im E)^{\perp}$, and $\im E $, respectively.
By scaling with $ [Y\, Z]^*$ one splits  the DAE
\begin{align*}
 Ex'+Fx=q
\end{align*}
 into the partitioned shape
\begin{align}
 Y^*Ex'+Y^*Fx&=Y^*q,\label{A.1}\\
 Z^*Fx&=Z^*q.\label{A.2}
\end{align}
Then the $(m-r)\times m$ matrix function  $ Z^*F$ features full row-rank $m-r$ and the subspace  $S=\ker Z^*F$ has dimension $r$. Equation \eqref{A.2} represents an underdetermined system. The idea is to provide its general solution in the following special way.

Taking a nonsingular matrix function  $K$ of size $m\times m$ such that $Z^*FK=: [\mathfrak A\, \mathfrak B]$, with $\mathfrak B:\mathcal I\rightarrow \Real^{(m-r)\times(m-r)}$ being nonsingular, the transformation  $x=K\tilde x$ turns  \eqref{A.2} into 
\begin{align*}
&Z^*FK\tilde x=:\mathfrak A u+\mathfrak B\tilde v=Z^*q,\quad \tilde x=\begin{bmatrix}
                                                                      u\\v
                                                                     \end{bmatrix}
 \\
&\text{yielding}\quad v=-\mathfrak B^{-1}\mathfrak A u + \mathfrak B^{-1}Z^*q.
\end{align*}
The further matrix function
\begin{align*}
 C:= K \begin{bmatrix}
     I_{r}\\-\mathfrak B^{-1}\mathfrak A
    \end{bmatrix} :\mathcal I\rightarrow \Real^{m\times r},
\end{align*}
has full column-rank $r$ on all $\mathcal I$ and serves as basis of $\ker Z^*F= S$. Each  solution of \eqref{A.2} can be represented in terms of $u$ as
\begin{align*}
 x=Cu+ p,\quad p:= K\begin{bmatrix}
                   0\\\mathfrak B^{-1}Z^*q
                  \end{bmatrix}.
\end{align*}
Next we insert this expression into \eqref{A.1}, that is,
\begin{align}\label{elimnew}
 Y^*ECu'+(Y^*FC + Y^*EC')u=Y^*q-Y^*Ep'-Y^*Fp.
\end{align}
Now the variable $v$  is eliminated and we are confronted with a new DAE with respect to $u$  living in $\Real^{r}$. By construction, it holds that
\begin{align*}
 \rank (Y^*EC)(t)= m-\dim (S(t)\cap \ker E(t))=:m-\theta(t).
\end{align*}
Therefore, the new matrix function has constant rank precisely if the pair $\{E, F\}$ is pre-regular such that $\theta$ is constant.

We underline again that the procedure in \cite{RaRh} and Section \ref{s.regular} allows for the choice of an arbitrary basis for $S$.
Obviously, the  earlier elimination procedure of \cite{Cis1982} can now be classified as its special version.

This way a sequence of matrix functions pairs $\{E_{j}^{E}, F_{j}^{E}\}$ of size $m_{j}^{E}$ , $j\geq 0$, starting from 
\[ m_{0}^{E}=m,\; r_{0}^{E}=r,\; E_{0}^{E}=E,\; F_{0}^{E}=F,
\]
and letting
\[ m_{j+1}^{E}=r_{j}^{E},\;  E_{j+1}^{E}=Y_{j}^*E_{j}^{E}C_{j},\; r_{j+1}^{E}=\rank E_{j+1}^{E},\; F_{j+1}^{E}=Y_{j}^*F_{j}^{E}C_{j}+ Y_{j}^*E_{j}^{E}C'_{j}.
\]
The corresponding regularity notion from \cite[p.\ 58]{Cis1982} is then:
\begin{definition}\label{d.Elim}
The DAE \eqref{DAE0} is called \emph{regular} on the interval $\mathcal I$ if the above process of dimension reduction is well-defined, i.e., at each level $[E_{j}^{E}\,F_{j}^{E}]=\Real^{m_{j}^{E}}$ and $E_{j}^{E}$ has constant rank $r^E_j$, and there is a number $\kappa$ such that either $E_{\kappa}^{E}$ is nonsingular or $E_{\kappa}^{E}=0$, but then $F_{\kappa}^{E}$ is nonsingular. 
\end{definition}
This regularity definition obviously fully agrees with  Definition \ref{d.2} in the matter and also with the name, but without naming the characteristic values. It is evident that
\begin{align}\label{elimchar}
 \kappa=\mu \quad \text{and}\quad r_{j}^{E}=\rank E_j^{E}= r-\sum_{i=0}^{j-1}\theta_i, \quad j=0,\ldots,\mu,
\end{align}
and each pair $\{E_{j}^{E},\,F_{j}^{E}\}$
 must be pre-regular.
The relevant solvability statements from \cite{Cis1982} match those in Section \ref{s.regular}.
\subsection{Dissection concept}\label{subs.dissection}
 A  decoupling technique has been presented and extended to apply to nonlinear DAEs quite recently under the name \emph{dissection concept} \cite{Jansen2014}. The intention behind this is to modify the nonlinear theory belonging to the projector based analysis in \cite{CRR}  by using appropriate basis functions along the lines of \cite{KuMe2006} instead of projector valued functions. This is, by its very nature, incredibly technical. We filter out the corresponding linear version here.
 
 Let the pair $\{E,F\}$, $E,F:\mathcal I\rightarrow\Real^{m\times m}$, be pre-regular with constants $r$ and $\theta$ according to Definition \ref{d.prereg}.

Let $T,T^c,Z$, and $Y $ represent bases of $\ker E, (\ker E)^{\perp}, (\im E)^{\perp}$, and $\im E $, respectively.
The matrix function $Z^*FT$ has size $(m-r)\times(m-r)$ and
\begin{align*}
 \dim \ker Z^*FT =T^+( \ker E \cap S)=\theta,\quad \rank Z^*FT =m-r-\theta=:a.
\end{align*}

By scaling with $ [Y\, Z]^*$ one splits  the DAE
\begin{align*}
 Ex'+Fx=q
\end{align*}
 into the partitioned shape
\begin{align}
 Y^*Ex'+Y^*Fx&=Y^*q,\label{A.1D}\\
 Z^*Fx&=Z^*q.\label{A.2D}
\end{align}
Owing to the pre-regularity, the $(m-r)\times m$ matrix function  $ Z^*F$ features full row-rank $m-r$. We keep in mind that $S=\ker Z^*F$ has dimension $r$.

The approach in \cite{Jansen2014} needs several additional splittings. Let $V,W$ be bases of  $(\im Z^*FT)^{\perp}$, and $\im Z^*FT$. By construction, $V$ has size $(m-r)\times a$ and $W$ has size $(m-r)\times \theta$.
One  starts with the  transformation
\begin{align*}
 x= \begin{bmatrix}
                       T^c& T
                      \end{bmatrix}\tilde x, \quad 
                      \tilde x=\begin{bmatrix}
                       \tilde x_1\\\tilde x_2
                      \end{bmatrix},\quad x=T^c\tilde x_1+ T\tilde x_2.
\end{align*}
The background is the associated possibility to suppress the derivative of the nullspace-part $T\tilde x_n$ similarly as in the context of properly formulated DAEs and to set $Ex'= ET^c\tilde x_1'+E{T^c}'\tilde x_1 +ET'\tilde x_2$, which, however, does not play a role in our context, where altogether continuously differentiable solutions are assumed. Furthermore, an additional partition of the derivative-free equation \eqref{A.2} by means of the scaling with $[V\,W]^*$ is applied, which results in the system
\begin{align}
Y^*ET^c\tilde x'_1 +Y^*(FT^c+E{T^c}')\tilde x_1+ Y^*(FT+E{T}')\tilde x_2&=Y^*q,\label{A3}\\
V^*Z^*FT^c\tilde x_1+V^*Z^*FT\tilde x_2&=V^*Z^*q,\label{A4}\\
W^*Z^*FT^c\tilde x_1 \hspace*{18mm} &=W^*Z^*q,\label{A5}.
\end{align}
The matrix function $W^*Z^*FT^c$ has full row-rank $\theta$ and $V^*Z^*FT$ has full row-rank $a$. Now comes another split. Choosing bases $G, H$ of $\ker W^*Z^*FT^c\subset\Real^{\theta}$ and $\ker V^*Z^*FT\subset\Real^{a}$, as well as bases of respective complementary subspaces, we transform 
\begin{align*}
 \tilde x_1= \begin{bmatrix}
                       G^c& G
                      \end{bmatrix}\bar x_1,\quad 
                      \bar x_1=\begin{bmatrix}
                       \bar x_{1,1}\\\bar x_{1,2}
                      \end{bmatrix},\quad \tilde x_1=G^c\bar x_{1,1}+ G\bar x_{1,2},\\
 \tilde x_2= \begin{bmatrix}
                       H^c& H
                      \end{bmatrix}\bar x_2,\quad 
                      \bar x_2=\begin{bmatrix}
                       \bar x_{2,1}\\\bar x_{2,2}
                      \end{bmatrix},\quad \tilde x_2=H^c\bar x_{2,1}+ H\bar x_{2,2}.                     
\end{align*}
Thus equations \eqref{A4} and \eqref{A5} are split into
\begin{align}
 V^*Z^*FT^c(G^c\bar x_{1,1}+ G\bar x_{1,2})+V^*Z^*FT H^c\bar x_{2,1}&=V^*Z^*q,\label{A6}\\
W^*Z^*FT^c G^c\bar x_{1,1} \hspace*{38mm} &=W^*Z^*q.\label{A7}
\end{align}
The matrix functions $V^*Z^*FT H^c$ and $W^*Z^*FT^c G^c$ are nonsingular each, which allows the resolution to $\bar x_{1,1}$ and $\bar x_{2,1}$. In particular, for $q=0$ it results that 
$\bar x_{1,1}= 0$ and $\bar x_{2,1}= \mathfrak E \bar x_{1,2}$, with 
\[\mathfrak E:=-(V^*Z^*FT H^c)^{-1}V^*Z^*FT^cG.
\]
Overall, therefore, the latter procedure presents again a transformation, namely
\begin{align*}
 x= K\bar x,\quad K=\begin{bmatrix}
    T^cG^c&T^cG&TH^c&TH
   \end{bmatrix},\quad \bar x=\begin{bmatrix}
   \bar x_{1,1}\\\bar x_{1,2}\\\bar x_{2,1}\\\bar x_{2,2}
   \end{bmatrix}
   \in \Real^{\theta}\times\Real^{r-\theta}\times\Real^{a}\times\Real^{\theta},
\end{align*}
and we realize that we have found again a basis of the subspace $S$, namely
\begin{align*}
 S=\im C,\quad C=K
 \begin{bmatrix}
    0&0\\I_{r-\theta}&0\\\mathfrak E&0\\0&I_{\theta}
                  \end{bmatrix}=
\begin{bmatrix}
 T^cG+TH^c\mathfrak E &\; TH
\end{bmatrix},
\end{align*}
 which makes the dissection approach a particular case of \cite{RaRh} and Section \ref{s.regular}. Consequently, the corresponding reduction procedure from there is well-defined for all regular DAEs in the sense of our basic Definition \ref{d.2}.
 \medskip
 
 In \cite{Jansen2014}  the approach is somewhat different. Again a sequence of matrix function pairs $\{E_{i}^{D},\,F_{i}^{D}\}$ is  built up starting from $E^{D}_0=E$, $F^{D}_0=F$. The construction of  $\{E_{1}^{D},\,F_{1}^{D}\}$ is closely related to the system given by \eqref{A3}, \eqref{A6}, and \eqref{A7}, 
  where the last two equations are solved with respect to $\tilde x_{1,1}$ and $\tilde x_{1,1}$ and these variables are replaced in \eqref{A3} accordingly. This leads to
  \begin{align*}
   E^{D}_1=\begin{bmatrix}
            0&Y^*ET^{c}G&0&0\\
            0&0&0&0\\
            0&0&0&0
            \end{bmatrix},\quad \rank E^{D}_1=\rank Y^*ET^{c}G= \rank G= r-\theta.
  \end{align*}
 In contrast to the basic procedure in Section \ref{s.regular} in which the dimension is reduced and variables are actually eliminated on each level, now all variables stay included and the original dimension $m$ is kept analogous to the strangeness concept in Section \ref{subs.strangeness}.
 We omit the further technically complex representation here and refer to \cite{Jansen2014}.
 It is evident that $\rank E^{D}_0>\rank E^{D}_1$ and so on.
 
 The characteristic values of the dissection concept are formally adapted to certain corresponding values of the tractability index framework. 
 It starts with $r^{D}_0=r$, and is continued in ascending order as the following definition from \cite[Definition 4.13, p. \ 83]{Jansen2014} says. 
 \begin{definition}\label{d.diss}
Let all basis functions exist and have constant ranks on $\mathcal I$ and let the sequence of the matrix function pairs be well-defined. The characteristic values of the DAE \eqref{DAE0} are defined as 
\begin{align*}
 r^{D}_0=r,\quad r^{D}_{i+1}=r^{D}_{i}+ a^{D}_{i}=r^{D}_{i}+\rank Z^*_{i}F^{D}_i T_i,\quad i\geq 0.
\end{align*}
If $r_0^{D}=r=m$ then the DAE is said to be regular with dissection index zero.
If there is an integer $\kappa\in \Natu$ and $r^{D}_{\kappa-1}<r^{D}_{\kappa}=m$ then the DAE is said to be \emph{regular with dissection index} $\mu^{D}=\kappa$.
The DAE is said to be \emph{regular}, if it is regular with any dissection index.
 \end{definition}
 In particular, in the first step one has
 \begin{align*}
  r^{D}_1= r+a=r+(m-r-\theta)=m-\theta=(m-r)+r-\theta =(m-r)+\rank E^{D}_1.
 \end{align*}
 
Owing to \cite[Theorem 4.25,p.101]{Jansen2014}, the tractability index (see Section \ref{subs.strangeness}) and the dissection index coincide, and also the corresponding characteristic values, that is,
\begin{align*}
 \mu^{D}=\mu^{T},\quad r_{i}^{D}=r_{i}^{T},\quad i=0,\ldots, \mu^{D}.
\end{align*}

\subsection{Regular strangeness index}\label{subs.strangeness}
The strangeness concept applies to rectangular matrix functions in general, but here we are interested in the case of square sizes only, i.e., $E,F:\mathcal I\rightarrow \Real^{m\times m}$.
Within the strangeness reduction framework the following five rank-values of the matrix function pair $\{E,F\}$  play their role, e.g., \cite[p. 59]{KuMe2006}:
\begin{align}
 r&=\rank E,\label{S1}\\
 a&=\rank Z^*FT, \;(\text{algebraic part})\label{S2}\\
 s&=\rank V^*Z^*FT^{c}, \;(\text{strangeness})\label{S3}\\
 d&=r-s, \;(\text{differential part})\label{S4}\\
 v&=m-r-a-s, \;(\text{vanishing equations})\label{S5}
\end{align}
whereby $T,T^{c}, Z,V$ represent orthonormal bases of $\ker E$, $(\ker E)^{\bot}$, $(\im E)^{\bot}$, and $(\im Z^*FT)^{\bot}$, respectively.
The strangeness concept is tied to the requirement that $r,a$, and $s$ are well-defined constant integers. Owing to \cite[Lemma 4.1]{HaMae2023},
the pair $\{E,F\}$ is pre-regular, if and only if the rank-functions \eqref{S1}-\eqref{S5} are constant and $v=0$. In case of pre-regularity, see Definition \ref{d.prereg}, one has
 \begin{align*}
  a=m-r-\theta,\quad s=\theta,\quad d=r-\theta.
 \end{align*}

Let the pair $\{E,F\}$ have constant rank values \eqref{S1}--\eqref{S5}, and $v=0$. We describe the related step from $\{E^{S}_0,F^{S}_0\}:=\{E,F\} $ to the next matrix function pair $\{E^{S}_1,F^{S}_1\}$.
Applying the  basic arguments of the strangeness reduction \cite[p.\ 68f]{KuMe2006}   the pair $\{E,F\}$ is equivalently transformed to $\{\tilde{E},\tilde F\}$,
 \begin{align*}
 \tilde{E}=\begin{bmatrix}
            I_s&&&\\&I_d&&\\
            &&0&\\
            &&&0
           \end{bmatrix},\quad   
 \tilde{F}=\begin{bmatrix}
            0&\tilde F_{12}&0&\tilde F_{14}\\
            0&0&0&\tilde F_{24}\\
            0&0&I_a&0\\
            I_s&0&0&0
           \end{bmatrix},        
 \end{align*}
with $d+s=r,\; a+s=m-r$. This means that the DAE is transformed into the intermediate form
\begin{align*}
 \tilde x_1' +\tilde F_{12}\tilde x_2 +\tilde F_{14}\tilde x_4 &= \tilde q_1,\\
 \tilde x_2'  +\tilde F_{24}\tilde x_4 &= \tilde q_2,\\
 \tilde x_3&=\tilde q_3,\\
  \tilde x_1&=\tilde q_4.
 \end{align*}
Replacing now in the first line $\tilde x_1'$ by  $\tilde q_4'$ leads to the 
new pair  defined as
\begin{align*}
 E^{S}_1=\begin{bmatrix}
            0&&&\\&I_d&&\\
            &&0&\\
            &&&0
           \end{bmatrix},\quad   
  F^{S}_1=\begin{bmatrix}
            0&\tilde F_{12}&0&\tilde F_{14}\\
            0&0&0&\tilde F_{24}\\
            0&0&I_a&0\\
            I_s&0&0&0
           \end{bmatrix}.        
 \end{align*}
Proceeding further  in this way, each pair $\{E^{S}_j, F^{S}_j\}$ must be supposed to be pre-regular for obtaining well-defined characteristic tripels $(r^S_j,a^S_j,s^S_j)$ and $v^S_j=0$. Owing to \cite[Theorem 3.14]{KuMe2006} these characteristics persist under equivalence transformations. The obvious relation $r^{S}_{j+1}=r^{S}_{j}-s^{S}_{j}$ guarantees that after a finite number of steps the so-called strangeness $s^{S}_{j}$ must vanish.
We adapt Definition 3.15 from \cite{KuMe2006} accordingly\footnote{The notion  \cite[Definition 3.15]{KuMe2006} is valid for more general rectangular matrix functions $E,F$. For quadratic matrix functions $E,F$ we are interested in here, it allows also nonzero  values $v_j^{S}=m-r_j^S-a_j^S-s_j^S$, thus  instead of pre-regularity of $\{E^{S}_j, F^{S}_j\}$, it is only required that $ r^S_j, a^S_j, s^S_j$ are constant on $\mathcal I$.}:
\begin{definition}\label{d.strangeness}
 Let each pair $\{E^{S}_j, F^{S}_j\}$, $j\geq 0$, be pre-regular and 
 \begin{align*}
  \mu^{S}=\min\{j\geq 0:s^{S}_{j}=0 \}.
 \end{align*}
Then the pair $\{E,F\}$ and the associated DAE are called \emph{regular with strangeness index} $ \mu^{S}$ and characteristic values $(r^S_j,a^S_j,s^S_j)$, $j\geq 0$.
In the case that $\mu^{S}=0$ the pair and the DAE are called \emph{strangeness-free}.
\end{definition}

Finally, if the DAE $Ex'+Fx=q$ is regular with strangeness index $ \mu^{S}$, this reduction procedure ends up with the strangeness-free pair
\begin{align}
 E^S_{\mu^S}=\begin{bmatrix}
  I_{d^S}&\\&0
 \end{bmatrix},\;
F^S_{\mu^S}=\begin{bmatrix}
  0&\\&I_{a^S}
 \end{bmatrix},\quad d^S:=d^S_{\mu^S},\; a^S:=a^S_{\mu^S},\; d^S+a^S=m,
\end{align}
and the transformed DAE 
 showing a simple form, which already incorporates its solution, namely
\begin{align*}
\tilde{\tilde x}'_1&=\tilde{\tilde q}_1,\\
 \tilde{\tilde x}_2&=\tilde{\tilde q}_2.
\end{align*}
 The function $\tilde{\tilde x}:\mathcal I\rightarrow \Real^m$ is a solution $x:\mathcal I\rightarrow \Real^m$ of the original DAE transformed by a pointwise nonsingular matrix function.
\medskip

 As a consequence of Theorem 2.5 from \cite{KuMe1996}, each  pair $\{E,F\}$ being regular with strangeness index $\mu^S$ can  be equivalently transformed into a pair  $\{\tilde E,\tilde F\}$,
 \begin{align}\label{SCFs}
  \tilde E=\begin{bmatrix}
            I_{d^S}&*\\0&N
           \end{bmatrix},\quad
\tilde F=\begin{bmatrix}
        *&0\\0&I_{a^S}
           \end{bmatrix},\quad d^S:=d^S_{\mu^S},\;  a^S:=a^S_{\mu^S},
 \end{align}
in which the matrix function $N$ is pointwise nilpotent with nilpotency index $\kappa=\mu^S +1$ and has size $a^S\times a^S$. $N$ is pointwise strictly block upper  triangular and the entries $N_{1,2}, \ldots, N_{\kappa-1,\kappa}$  have full row-ranks $l_1=s^S_{\mu^S-1},\ldots, l_{\kappa-1}=s^S_0$. Additionally, one has $l_{\kappa}=s^S_0+a^S_0 =m-r$, and $N$ has exactly the structure that is required in \eqref{blockstructure} and Proposition \ref{p.STform}(2). It results that each DAE having a well-defined regular strangeness index is regular in the sense of Definition \ref{d.2}.

\subsection{Tractability index}\label{subs.tractability}
The background of the tractability index concept is the projector based analysis which aims at an immediate characterization of the structure of the originally given DAE, its relevant subspaces and components, e.g., \cite{CRR}.  
In contrast to the reduction procedures with their transformations and built-in differentiations of the right-hand side, the original DAE is actually only written down in a very different pattern using the projector functions.
No differentiations are carried out, but it is only made clear which components of the right-hand side must be correspondingly smooth. This is important in the context of input-output analyses and also when functional analytical properties of relevant operators are examined \cite{Ma2014}. The decomposition using projector functions reveals the inherent structure of the DAE, including the inherent regular ODE. Transformations of the searched solution are avoided in this decoupling framework, which is favourable for stability investigations and also for the analysis of discretization methods \cite{CRR,HMT}.

As before we assume  $E,F:\mathcal I\rightarrow \Real^{m\times m}$  to be sufficiently smooth and the pair $\{E, F\}$ to be pre-regular. We choose any continuously differentiable projector-valued function $P$ such that
\[P:\mathcal I\rightarrow \Real^{m\times m},\quad P(t)^2=P(t),\; \ker P(t)=\ker E(t),\quad t\in \mathcal I,
\]
and regarding that $Ex'=EPx'=E(Px)'-EP'x$ for each continuously differentiable function $x:\rightarrow\Real^m$, we rewrite the DAE $Ex'+Fx=q$ as
\begin{align}\label{DAEP}
 E(Px)'+(F-EP')x=q.
\end{align}
\begin{remark}\label{r.AD}
 The DAE \eqref{DAEP} is a special version of a DAE with \emph{properly stated leading term} or properly involved derivative, e.g., \cite{CRR},
 \begin{align}\label{2.DAE}
 A(Dx)'+Bx=q,
\end{align}
which is obtained by a special  \emph{proper factorizations} of $E$,
 which are subject to the general requirements: $E=AD$, $A:\mathcal I\rightarrow \Real^{n\times m}$ is continuous, $D:\mathcal I\rightarrow \Real^{m\times n}$ is continuously differentiable, $B=F-AD'$, and 
 \begin{align*}
  \ker A\oplus \im D=\Real^{n},\; \ker D=\ker E,
 \end{align*}
whereby both subspaces $\ker A$ and $\im D$ have continuously differentiable basis functions.

As mentioned already above, a properly involved derivative makes sense, if not all components of the unknown solution are expected to be continuously differentiable, which does not matter here. In contrast, in view of  applications and numerical treatment 
 the model \eqref{2.DAE} is quite reasonable \cite{CRR}. 
\end{remark}
In order to be able to directly apply the more general results of the relevant literature, in the following we denote 
\begin{align*}
 P=:D,\quad G_0;=E,\quad B_0:=F-ED',\quad A:=E.
\end{align*}
Observe that the pair $\{G_0, B_0\}$ is pre-regular with constants $r$ and $\theta$ at the same time as $\{E, F\}$.
Now we build a sequence of matrix functions pairs starting from the pair $\{G_0, B_0\}$. 
Denote $N_0=\ker G_0$ and choose a second projector valued function $P_0:\mathcal I\rightarrow\Real^{m\times m}$, such that $\ker P_0=N_0$. With the complementary projector function  $Q_{0}:=I-P_{0}$ and $D^{-}:=P_{0}$ it results that
\begin{align*}
 DD^{-}D=D,\quad D^{-}DD^{-}=D^{-},\quad DD^{-}=P_{0},\quad D^{-}D=P_{0}.
\end{align*}
On this background we construct the following  sequence of matrix functions and associated projector functions:

Set $r^T_{0}=r=\rank G_0$ and $\pPi_{0}=P_{0}$
and build successively for $i\geq 1$,
\begin{align}
 G_{i}&=G_{i-1}+B_{i-1}Q_{i-1},\quad
 r^T_{i}=\rank G_{i},\label{2.Gi}\\
 \quad N_{i}&=\ker G_{i},\quad \widehat{N_{i}}=(N_{0}+\cdots+N_{i-1})\cap N_{i},\quad u^T_{i}=\dim \widehat{N_{i}},\nonumber
\end{align}
fix a subset $X_{i}\subseteq N_{0}+\cdots+N_{i-1}$ such that $\widehat{N_{i}}+X_{i}=N_{0}+\cdots+N_{i-1}$ and  choose then a projector function $Q_{i}:\mathcal I\rightarrow\Real^{m\times m}$ to achieve
\begin{align}\label{2.Qi}
 \im Q_{i}=N_{i},\quad X_{i}\subseteq\ker Q_{i},\quad P_{i}=I-Q_{i},\quad \pPi_{i}=\pPi_{i-1}P_i,
\end{align}
and then form
\begin{align}\label{2.Bi}
 B_{i}=B_{i-1}P_{i-1}-G_{i}D^{-}(D\pPi_{i}D^{-})'D\pPi_{i-1}.
\end{align}
By construction, the inclusions
\begin{align*} 
 \im G_{0}\subseteq \im G_{1}&\subseteq\cdots \im G_{k}\subseteq \Real^{m},\\
 \widehat{N_{1}}&\subseteq\widehat{N_{2}}\subseteq\cdots\subseteq\widehat{N_{k}},
\end{align*}
come off, which leads to the inequalities
\begin{align*}
 0\leq r^T_{0}&\leq r^T_{1}\leq \cdots\leq r^T_{k},\\
 0&\leq u^T_{1}\leq \cdots\leq u^T_{k}.
\end{align*}
The sequence $G_{0},\ldots, G_{k}$ is said to be \emph{admissible} if, for each $i=1,\ldots,k$,  the two rank functions $r^T_{i}$,  $u^T_{i}$ are constant,  $\pPi_{i}$ is continuous and $D\pPi_{i}D^{-}$ is continuously differentiable.
It is worth mentioning that the matrix functions $G_{0},\ldots,G_{k}$ of an admissible sequence are continuous and
the products $\pPi_{i}$ and $D\pPi_{i}D^{-}$ are projector functions again \cite{CRR}.
Moreover, if  $u^T_{k}=0$, then  $u^T_{i}=0$, for $i<k$. We refer to \cite[Section 2.2]{CRR} for further useful properties.
\begin{definition}{\cite[Section 2.2.2]{CRR}}\label{d.trac}
The smallest number $\kappa\geq 0$, if it exists, leading to an admissible matrix function sequence ending up with a nonsingular matrix function $G_{\kappa}$ is called the  \emph{tractability index (regular case)}\footnote{We refer to \cite[Sections 2.2.2 and 10.2.1]{CRR} for details and  more general notions including also nonregular DAEs.} of the pair $\{E,F\}$, and the DAEs \eqref{1.DAE} and \eqref{2.DAE}, respectively. It is indicated by $\kappa=: \mu^T$.
The associated characters
\begin{align}\label{2.characvalues}
 0\leq r^T_{0}\leq r^T_{1}\leq \cdots\leq r^T_{\kappa-1}< r^T_{\kappa}=m,\quad d^T=m-\sum_{i=0}^{\kappa-1}(m-r^T_{i}),
\end{align}
are called characteristic values of the pair $(E,F)$ and the DAEs \eqref{1.DAE} and \eqref{2.DAE}, respectively. The pair $(E,F)$ and the DAEs \eqref{1.DAE} and \eqref{2.DAE},
are called regular each. 
\end{definition}
By definition, if the DAE is regular, then  $r^T_{\mu^T}=m, u^T_{\mu^T}=0 $ and all rank functions $u^T_{i}$ have to be zero and play no further role here. The special possible choice of the projector functions $P, P_0,\ldots ,P_{\mu-1}$ does not affect regularity and the characteristic values \cite{CRR}.
\begin{remark}\label{r.Riaza}
An alternative way to construct admissible matrix function sequences  for the regular case if $u^T_i=0$,  $i\geq1$, is described in \cite[Section 2.2.4]{RR2008}. It avoids the explicit use of the nullspace projector functions onto $N_i$.  One starts with $G_0, B_0$, and $\pPi_0$ as above, introduces $M_0:=I-\pPi_0$, $G_1=G_0+B_0M_0$, and then for $i\geq 1$:
\begin{align*}
 &\text{choose a projector function }\; \pPi_i \; \text{ along }\; N_0\oplus\cdots\oplus N_i,\;\text{ with }\; \im \pPi_i \subseteq \pPi_{i-1} ,\\
 &B_i=(B_{i-1}-G_iD^-(D\pPi_iD^-)'D)\pPi_i,\\
 &M_i=\pPi_{i-1}-\pPi_i,\\
 &G_{i+1}=G_i+B_iM_i.
\end{align*}

\end{remark}

\begin{remark}\label{r.Tpairs}
If the pair $(E,F)$ is regular in the sense of Definition \ref{d.trac} then the subspace $S^T_j(t)$,
\begin{align*}
 S^T_j(t):=\{z\in\Real^m: B_j(t)z\in \im G_j(t)\}=\ker W^T_j(t)B_j(t),\quad W^T_j:=I-G_jG_j^+,
\end{align*}
has constant dimension $r_{j}$ on all $\mathcal I$. Moreover,
\begin{align*}
 \rank [G_j \; B_j]=\rank [G_j \; W^T_jB_j]= r^T_j+ m-r^T_j=m,\\
 \dim \ker G_{j+1}=\dim (\ker G_j\cap S^T_j)= m-r^T_{j+1}, \quad j=0,\ldots, \mu^T-1.
\end{align*}
All intermediate pairs $\{G_j,B_j\}$ are pre-regular. It is worth highlighting that in terms of the basic regularity notion\footnote{See Definition \ref{d.2} and Theorem \ref{t.equivalence}.} one has $\mu^T=\mu$ and 
\begin{align*}
 \dim (\ker G_j\cap S^T_j)= \theta_{j}, \quad j=0,\ldots, \mu-1.
\end{align*}

\end{remark}

The decomposition
\begin{align*}
 I_{m}=\pPi_{\mu^T-1}+Q_{0}+\pPi_{0}Q_{1}+\cdots+\pPi_{\mu^T-2}Q_{\mu^T-1}
\end{align*}
is valid and the
involved projector functions show constant ranks, in particular,
\begin{align}\label{2.ranks}
 \rank Q_{0}=m-r^T_{0},\;\rank \pPi_{i-1}Q_{i}=m-r^T_{i},\;i=1,\ldots,\mu^T-1,\; \rank \pPi_{\mu^T-1}=d^T.
\end{align}
\medskip

Let the  DAE \eqref{1.DAE} be regular with tractability index $\mu^T\in \Natu$
and characteristic values \eqref{2.characvalues}.
Then the admissible matrix functions and associated projector functions
provide a far-reaching decoupling of the DAE, which exposes the intrinsic structure of the DAE, for details see \cite[Section 2.4]{CRR}. In particular, the following  representation  of the scaled by $G_{\mu^T}^{-1}$ DAE was proved in \cite[Proposition 2.23]{CRR}):
\begin{align*}
 G_{\mu^T}^{-1}A(Dx)'+G_{\mu^T}^{-1}Bx&=G_{\mu^T}^{-1}q,\\
 G_{\mu^T}^{-1}A(Dx)'+G_{\mu^T}^{-1}Bx
 &=D^{-}(D\pPi_{\mu^T-1}x)'+G_{\mu^T}^{-1}B_{\mu^T}x\\&+
 \sum_{l=0}^{\mu^T-1}\{Q_{l}x-(I-\pPi_{l})Q_{l+1}D^{-}(D\pPi_{l}Q_{l+1}x)'+V_{l}D\pPi_{l}x\},
\end{align*}
with $V_{l}=(I-\pPi_{l})\{P_{l}D^{-}(D\pPi_{l}D^{-})'-Q_{l+1}D^{-}(D\pPi_{l+1}D^{-})'\}D\pPi_{l}D^{-}$.

Regarding the decomposition of the unknown function
\begin{align*}
 x=\pPi_{\mu^T-1}x+Q_{0}x+\pPi_{0}Q_{1}x+\cdots+\pPi_{\mu^T-2}Q_{\mu^T-1}x\\
\end{align*}
and several projector properties, we get
\begin{align}\label{Grundformel}
 G_{\mu^T}^{-1}A(Dx)'+G_{\mu^T}^{-1}Bx
 =&D^{-}(D\pPi_{\mu^T-1}x)'- \sum_{l=0}^{\mu^T-1}(I-\pPi_{l})Q_{l+1}D^{-}(D\pPi_{l}Q_{l+1}x)' \\
 &+G_{\mu^T}^{-1}B_{\mu^T}\pPi_{\mu^T-1}x+ \sum_{l=0}^{\mu^T-1}V_{l}D\pPi_{\mu^T-1}x\nonumber\\
 &+Q_{0}x
 + \sum_{l=0}^{\mu^T-1}Q_{l}\pPi_{l-1}Q_{l}x
 + \sum_{l=0}^{\mu^T-2}V_{l} \sum_{s=0}^{\mu^T-2}D\pPi_{s}Q_{s+1}x.\nonumber
\end{align}
The representation \eqref{Grundformel} is the base of two closely related versions of fine and  complete structural decouplings of the DAE \eqref{1.DAE} into the so-called \textit{inherent regular ODE} (and its compressed version, respectively),
\begin{align}\label{IRODE}
 (D\pPi_{\mu^T-1}x)'-(D\pPi_{\mu^T-1}D^-)'D\pPi_{\mu^T-1}x+D\pPi_{\mu^T-1}G_{\mu^T}^{-1}B_{\mu^T}D^-D\pPi_{\mu^T-1}x=D\pPi_{\mu^T-1}G_{\mu^T}^{-1}q,
\end{align}
and the extra part indicating and including all the necessary differentiations of $q$. It is worth mentioning that the explicit ODE \eqref{IRODE} is not at all affected from derivatives of $q$.

While the first decoupling version is a swelled system residing in a $m$-dimensional subspace of $\Real^{(\mu^T+1)m}$, the second version remains in $\Real^{m}$ and represents an equivalently transformed DAE\footnote{In the literature there are quite a few misunderstandings about this.}. More precisely,
owing to \cite[Theorem 2.65]{CRR}, each  pair $\{E,F\}$ being regular with tractability index $\mu^T$ can  be equivalently transformed into a pair  $\{\tilde E,\tilde F\}$,
 \begin{align}\label{SCFt}
  \tilde E=\begin{bmatrix}
            I_{d^T}&0\\0&N
           \end{bmatrix},\quad
\tilde F=\begin{bmatrix}
        \Omega&0\\0&I_{m-d^T}
           \end{bmatrix},
 \end{align}
in which the matrix function $N$ is pointwise nilpotent with nilpotency index $\kappa=\mu^T$  and has size $(m-d^T)\times (m-d^T)$. $N$ is pointwise strictly block upper  triangular and the entries $N_{1,2}, \ldots, N_{\kappa-1,\kappa}$  have full column-ranks $l_2=m-r^T_{1},\ldots, l_{\kappa}=m-r^T_{\kappa-1}$. Additionally, one has $l_{1}=m-r$, and $N$ has exactly the structure that is required in \eqref{blockstructure} and Proposition \ref{p.STform}(1).
 
The projector based approach sheds light on the role of several subspaces. In particular, the two canonical subspaces $S_{can}$ and $N_{can}$, see \cite{HaMae2023}, originate from this concept, e.g., \cite{CRR}. For regular pairs it holds that $N_{can}=N_0+\cdots+N_{\mu^T-1}$.

The following assertion provided in \cite{LinhMae,HaMae2023} plays its role when analyzing DAEs and its canonical subspaces.
\begin{proposition}\label{p.adjoint}
 If the DAE \eqref{1.DAE} is regular with tractability index $\mu^T$ and characteristics $0<r_0^T\leq\cdots<r^{T}_{\mu^T}=m$, then the adjoint DAE
 \begin{align*}
  -E^*y'+(F^*-{E^*}')y=0
 \end{align*}
is also regular with the same index and characteristics, and the  canonical subspaces $S_{can}, N_{can}$ and $S_{adj, can}, N_{adj, can}$,  are related by
\begin{align*}
 N_{can}=\ker C^*_{adj}E,\quad  N_{adj, can}=\ker C^*E^*,
\end{align*}
in which $C$ and $C_{adj}$ are bases of the flow-subspaces $S_{can}$ and $S_{adj, can}$, respectively.
\end{proposition}

\subsection{Equivalence results and other commonalities}\label{s.equivalence}
\begin{theorem}\label{t.equivalence}
Let $E, F:\mathcal I\rightarrow\Real^{m\times m}$ be sufficiently smooth and $\mu\in\Natu$.
The following assertions are equivalent in the sense that the individual characteristic values of each two of the variants are mutually uniquely determined.
\begin{description}
\item[\textrm{(1)}] 
The  pair $\{E,F\}$ is regular on $\mathcal I$ with index $\mu\in \Natu$ and  characteristics $r<m$,
$\theta_0=0$ if $\mu=1$, and, for $\mu>1$,
\begin{align*}
r<m,\quad \theta_0\geq\cdots\geq\theta_{\mu-2}>\theta_{\mu-1}=0.
\end{align*}
 \item[\textrm{(2)}] The strangeness index $\mu^S$ is well-defined for $\{E,F\}$ and regular, and $ \mu^S=\mu-1$. The associated characteristics are the tripels
 \begin{align*}
  (r^S_i,\;a^S_i, s^S_i ),\quad i=0,\ldots, \mu^S,\quad r^S_0=r, \quad \mu^S=\min\{i\in \Natu_0:s^S_i=0\}.
 \end{align*}
\item[\textrm{(3)}] The pair $\{E,F\}$ is regular with tractability index $\mu^T=\mu$ and characteristics 
\begin{align*}
 r^T_0=r,\quad r^T_0\leq \cdots\leq r^T_{\mu-1}<r^T_{\mu}=m.
\end{align*}
\item[\textrm{(4)}] The pair $\{E,F\}$ is regular with dissection  index $\mu^D=\mu$ and characteristics 
\begin{align*}
r^D_0=r,\quad r^D_0\leq \cdots\leq r^D_{\mu-1}<r^D_{\mu}=m.
\end{align*}
\item[\textrm{(5)}] 
The  pair $\{E,F\}$ is regular on $\mathcal I$ with elimination index $\mu^E=\mu$ and  characteristics $r<m$,
$\theta_0=0$ if $\mu=1$, and, for $\mu>1$,
\begin{align*}
r<m,\quad \theta_0\geq\cdots\geq\theta_{\mu-2}>\theta_{\mu-1}=0.
\end{align*}
\end{description}
\end{theorem}

\begin{proof}
 Owing to \cite[Theorem 4.3]{HaMae2023}, it remains to verify the implication {\textrm (3)}$\Rightarrow ${\textrm (1)}. A DAE being regular with tractability index $\mu^{T}$ and characteristics \eqref{trac}
is equivalent to a DAE in the form \eqref{blockstructure} with $\kappa=\mu^{T}$, $r=r_0^{T}$, and $l_i=m-r_{i-1}^{T}$ for $i=1,\ldots,\kappa $. Hence, by Proposition \ref{p.STform}, the DAE is regular with index $\mu=\kappa=\mu^{T}$ and characteristic values $r=r_0^{T}$, $m-\theta_0=r_1^{T},\ldots,m-\theta_{\mu-2}=r_{\mu-1}^{T}$, and $m=m-\theta_{\mu-1}=r_{\mu}^{T}$
 \end{proof}
Next we highlight the relations between the various characteristic values and trace back all of them to 
\begin{align*}
r<m,\quad \theta_0\geq\cdots\geq\theta_{\mu-2}>\theta_{\mu-1}=0.
\end{align*}
\begin{theorem}\label{t.indexrelation}
 Let the  pair $\{E,F\}$  regular on $\mathcal I$ with index $\mu\in \Natu$ and  characteristics $r<m$,
$\theta_0=0$ if $\mu=1$, and, for $\mu>1$,
\begin{align*}
r<m,\quad \theta_0\geq\cdots\geq\theta_{\mu-2}>\theta_{\mu-1}=0.
\end{align*}
 Then the following relations concerning the various characteristic values arise:
\begin{description}
 \item[\textrm{(1)}] The pair $\{E,F\}$ is regular with strangeness index $\mu^S=\mu-1$. The associated characteristics are
 \begin{align*}
  r^S_0&=r,\\
  s^S_i&=\theta_i,\\
d^S_i&=r^S_i-\theta_i  = r-\sum_{j=0}^{i} \theta_j,\\
  a^S_i&=m-r^S_i-\theta_i = m-r + \sum_{j=0}^{i-1} \theta_j - \theta_i,\\
  v^S_i&=0,\\
  r^S_{i+1}&=d^S_i = r-\sum_{j=0}^{i} \theta_j,\quad i=0,\ldots,\mu-1.
 \end{align*}
\item[\textrm{(2)}] The pair $\{E,F\}$ is regular with tractability index $\mu^T=\mu$ and characteristics 
\begin{align}\label{trac}
 r^T_0=r,\quad r^T_i=m-\theta_{i-1},\quad i=1,\ldots,\mu.
\end{align}
\item[\textrm{(3)}] The pair $\{E,F\}$ is regular with dissection  index $\mu^D=\mu$ and characteristics 
\begin{align*}
 r^D_0=r,\quad r^D_i=m-\theta_{i-1},\quad i=1,\ldots,\mu.
\end{align*}
\item[\textrm{(4)}] 
The  pair $\{E,F\}$ is regular on $\mathcal I$ with elimination index $\mu^E=\mu$ and  characteristics $r<m$,
$\theta_0=0$ if $\mu=1$, and, for $\mu>1$,
\begin{align*}
r<m,\quad \theta_0\geq\cdots\geq\theta_{\mu-2}>\theta_{\mu-1}=0.
\end{align*}
\end{description}
\end{theorem}
Thus, the statements of Theorems \ref{t.Scan} and \ref{t.solvability} apply equally to all concepts in this section. Every regular DAE with index $\mu\in\Natu$ is a solvable system in the sense of Definition \ref{d.solvableDAE}, and it has the pertubation index $\mu$.

\begin{remark}\label{r.inf}
 Obviously, for a regular  pair $\{E,F\}$ with index $\mu$, each of the above procedures is feasible up to infinity and will eventually stabilize. This can now be recorded by setting 
 \begin{align*}
  \theta_k:=0,\quad k\geq \mu.
 \end{align*}
Namely, in particular,
the strangeness index is well defined and regular, $\mu^S=\mu-1$,
 \begin{align*}
  r^S_0=r,\quad r^S_i=r^S_{i-1}-\theta_{i-1},\quad i=1,\ldots, \mu-1,\\
  s^S_i=\theta_i,\quad  i=0,\ldots, \mu-1,\\
  a^S_i=m-r^S_i-\theta_i,\quad  i=0,\ldots, \mu-1.
 \end{align*}
After reaching the zero-strangeness $s^S_{\mu-1}=0$ the corresponding sequence $\{E^S_i,F^S_i\}$ can be continued and for $i\geq\mu$ it becomes stationary \cite[p.\ 73]{KuMe2006},
\begin{align*}
  r^S_i=r^S_{\mu-1}=d^S=d,\quad i\geq \mu,\\
  s^S_i=0,\quad  i\geq \mu,\\
  a^S_i=m-r^S_{\mu-1}=m-d,\quad  i\geq \mu,
 \end{align*}
 which goes along with $\theta_i=0$ for $i\geq\mu$ and justifies the setting $\theta_k=0$ for $k\geq\mu$.
\end{remark}

\begin{corollary}\label{c.degree}
 The dynamical degree of freedom of a regular DAE is
 \begin{align*}
  d=r-\sum_{i=0}^{\mu -2}\theta_i=d^S=d^T=\dim S_{can}.
 \end{align*}
\end{corollary}
After we have recognized that the rank conditions in Definition \ref{d.2} are appropriate for a regular DAE, the question arises what rank violations can mean.

Based on the above equivalence statements, the findings of the projector-based analysis on regular and critical points, for instance in \cite{RR2008,CRR}  are generally valid. The characterization of critical and singular points presupposes a corresponding definition of regular points.

\begin{definition}\label{d.regpoint}
 Given is the pair $\{E,F\}$, $E,F:\mathcal I\rightarrow\Real^{m\times m}$. 
 The point $t_*\in\mathcal I$ is said to be a \emph{regular point} of the pair and the associated DAE, if there is an open neighborhood  $\mathcal U\ni t_*$ such that the pair restricted to  $\mathcal I\cap\mathcal U$, is regular. 
 Otherwise $t_*\in\mathcal I$ will be called \emph{critical or singular}.
 
 In the regular case the characteristic values \eqref{theta} are then also assigned to the regular point. 
 The set of all regular points within $\mathcal I$ will be denoted by $\mathcal I_{reg}$. 

A subinterval $\mathcal I_{sub}\subset \mathcal I$ is called regularity interval if all its points are regular ones.  
\end{definition}

We refer to \cite[Chapter 4]{RR2008} for a careful discussion and classification of possible critical points. Section \ref{s.examples} below  comprises a series of relevant but simple examples.

Critical points arise when rank conditions ensuring regularity are violated. We now realize that the question of whether a point is regular or critical can be answered independently of the chosen approach.
According to our equivalence result, critical points arise, if at all, then simultaneously in all concepts at the corresponding levels.
\medskip

When viewing a DAE as a vector field on a manifold, critical points are allowed exclusively in the very last step of the basis reduction,  with the intention of then being able to examine singularities of the flow, see Section \ref{subs.degree}.
The concept of geometric reduction basically covers regular DAEs and those with well-defined degree and configuration space, i.e. only rank changes in the very last reduction level are permitted.

\begin{remark}\label{Nonregular}
We end this section with an very important note:
The strangeness index and the tractability index are defined also for DAEs in rectangular size,with $E,F:\mathcal I\rightarrow\Real^{n\times m}$, $n\neq m$, but then they differ substantially from each other \cite{HM2020, CRR}. It remains to be seen whether and to what extent the above findings can be generalized.
\end{remark}

\subsection{Standard canonical forms}\label{sec:SCF}
DAEs in standard canonical form (SCF), that is,
\begin{align}\label{SCFDAE}
 \begin{bmatrix}
  I_d&0\\0&N(t)
 \end{bmatrix} x'(t)+
\begin{bmatrix}
  \Omega(t)&0\\0&I_a
 \end{bmatrix} x(t)=q(t),\quad t\in\mathcal I,
\end{align}
where $N$ is strictly upper (or lower) triangular, but it need not have constant rank or index, see \cite[Definition 2.4.5]{BCP89}, 
play a special role in the DAE literature \cite{BCP89,BergerIlchmann}. Their coefficient pairs represent  generalizations of the Weierstraß–Kronecker form\footnote{Quasi-Weierstraß form in \cite{BergerIlchmannTrenn,Trenn2013}} of matrix pencils. If $N$ is even constant, then the DAE is said to be in \emph{strong standard canonical form}. A DAE in SCF is also characterized by the simplest canonical subspaces which are even orthogonal to each other, namely 
\begin{align*}
 S_{can}=\im \begin{bmatrix}
              I_d\\0
             \end{bmatrix},\quad
N_{can}=\im \begin{bmatrix}
              0\\I_a
             \end{bmatrix}.
\end{align*}

DAEs being transformable into SCF are solvable systems in the sense of Definition \ref{d.solvableDAE}, but they are not necessary regular, see Examples \ref{e.1}, \ref{e.7} in Section \ref{s.examples}.  The critical points that occur here are called \emph{harmless} \cite{RR2008,CRR} because they do not generate a singular flow. We will come back to this below.

Furthermore, not all solvable systems can be transformed into SCF as Example \ref{e.2} below confirms. We refer to \cite{BCP89} and in turn to Remark \ref{r.generalform} below for the description of the general form of solvable systems.
\medskip

In Sections \ref{s.regular} and \ref{subs.tractability} we already have faced DAEs in SCFs with a special structure, which in turn represent narrower generalizations of the Weierstraß–Kronecker form. 
For given integers $\kappa \geq 2$, $d\geq 0$, $l=l_{1}+\cdots +l_{\kappa}$, $l_{i}\geq
 1$, $l=a$, $m=d+l$
 the pair $\{E,F\}$, $E,F:\mathcal I\rightarrow \Real^{m\times m}$, is structured as follows:
\begin{align}\label{blockstructureSCF}
 E=\begin{bmatrix}
    I_{d}&\\&N
   \end{bmatrix},\quad
F&=\begin{bmatrix}
    \Omega&\\&I_{l}
   \end{bmatrix}, \quad
N=\begin{bmatrix}
   0&N_{12}&&\cdots&N_{1\kappa}\\
   &0&N_{23}&&N_{2\kappa}\\
   &&\ddots&\ddots&\vdots\\
   &&&&N_{\kappa-1 \kappa}\\
   &&&&0
   \end{bmatrix},\\
   &\text{with blocks}\; N_{ij} \;\text{of sizes}\; l_{i}\times l_{j}.\nonumber
\end{align}
If $d=0$ then the respective parts are absent. All blocks are sufficiently smooth on the given interval $\mathcal I$. $N$ is strictly block upper triangular, thus nilpotent and $N^{\kappa}=0$. 
\medskip

The following theorem proves that and to what extent regular DAEs are distinguished by a uniform inner structure of the matrix function $N$ and thus of the canonical subspace $N_{can}$.
\begin{theorem}\label{t.SCF}
 Each regular DAE with index $\mu\in\Natu$ and characteristics 
 $r<m$,
$\theta_0=0$ if $\mu=1$, and, for $\mu>1$,
\begin{align*}
r<m,\quad \theta_0\geq\cdots\geq\theta_{\mu-2}>\theta_{\mu-1}=0,
\end{align*}
is transformable into a structured SCF \eqref{blockstructureSCF} where $\kappa=\mu$ and all blocks of the secondary diagonal have full column rank, that means,
\begin{align*}
\rank N_{12}&=l_{2}= m-r,\; \ker N_{12}=\{0\},\\
 \rank N_{i,i+1}&=l_{i+1}=\theta_{i-1},\; \ker N_{i,i+1}=\{0\},\quad i=1,\ldots,\mu-1,
\end{align*}
and the powers of $N$ feature constant rank, 
\begin{align*}
 \rank N&=r-d=\theta_0+\cdots+\theta_{\mu-2},\\
 \rank N^2&=\theta_1+\cdots+\theta_{\mu-2},\\
 &\cdots\\
 \rank N^{\mu-1}&=\theta_{\mu-2}.
\end{align*}
\end{theorem}
\begin{proof}
 Owing to Theorem \ref{t.equivalence} the DAE is regular with tractability index $\mu^T=\mu$ and the associated characteristics given by formula \eqref{trac}. By \cite[Theorem 2.65]{CRR}, each DAE being regular with tractability index $\mu$ can be equivalently transformed into a structured SCF, with N having the block upper triangular structure as in \eqref{blockstructureSCF}, $\kappa=\mu$,  $l_1=m-r, l_2=m-r^T_1,\ldots, l_{\kappa}=m-r^T_{\kappa-1}$. Now the assertion results by straightforward computations.
\end{proof}

Sometimes structured SCFs, in which the blocks on the secondary diagonal have full row rank, are more convenient to handle, as can be seen in the case of the proof of Proposition \ref{p.STform}, for example.
\begin{corollary}\label{c.SCT}
 Given is the  strictly upper block triangular matrix function with full row-rank blocks on the secondary block diagonal,
  \begin{align*}
\tilde N&=\begin{bmatrix}
   0&\tilde N_{12}&&\cdots&\tilde N_{1\kappa}\\
   &0&\tilde N_{23}&&\tilde N_{2\kappa}\\
   &&\ddots&\ddots&\vdots\\
   &&&&\tilde N_{\kappa-1 \kappa}\\
   &&&&0
   \end{bmatrix}:\mathcal I\rightarrow\Real^{l\times l},\\
   \text{with blocks}\; &\tilde N_{ij} \;\text{of sizes}\; \tilde l_{i}\times \tilde l_{j},\; 
   \rank \tilde N_{i,i+1}=\tilde l_{i}, \quad 1\leq\tilde l_1\leq \tilde l_2\leq\cdots\leq \tilde l_{\kappa}, \\&l=\sum_{i=1}^{\kappa}\tilde l_i,\; r_{\tilde N}=\rank \tilde N.
\end{align*}
Then the following two assertions are valid:
\begin{description}
 \item[\textrm{(1)}] 
The pair $\{\tilde N,I_l\}$ can be equivalently transformed to a pair $\{N,I_l\}$ with full column-rank blocks on the secondary block diagonal,
 \begin{align*}
N&=\begin{bmatrix}
   0&N_{12}&&\cdots&N_{1\kappa}\\
   &0&N_{23}&&N_{2\kappa}\\
   &&\ddots&\ddots&\vdots\\
   &&&&N_{\kappa-1 \kappa}\\
   &&&&0
   \end{bmatrix}:\mathcal I\rightarrow\Real^{l\times l},\\
   \text{with blocks}\; &N_{ij} \;\text{of sizes}\; l_{i}\times l_{j},\; 
  \rank N_{i,i+1}=l_{i+1}, \quad l_1\geq l_2\geq\cdots\geq l_{\kappa}\geq 1,\\&l=\sum_{i=1}^{\kappa} l_i,\;  r_{N}=\rank N,
\end{align*}
such there are pointwise nonsingular matrix functions $L,K:\mathcal I\rightarrow \Real^{l\times l}$ yielding
\begin{align}\label{NN}
 LNK=\tilde N,\; LK+LNK'=I_l.
\end{align}
Furthermore, both pairs $\{N,I_l\}$ and $\{\tilde N,I_l\}$ are regular with index $\mu=\kappa$ and characteristics
\begin{align*}
r_{N}= r_{\tilde N}&=l-\tilde l_{\mu}=l-l_1,\\
 \theta_0&=\tilde l_{\mu-1}=l_2,\\
  \theta_1&=\tilde l_{\mu-2}=l_3,\\
  &\cdots\\
   \theta_{\mu-2}&=\tilde l_{1}=l_{\mu}.
\end{align*}
\item[\textrm{(2)}] The pairs $\{E,F\}$ and $\{\tilde E,\tilde F\}$, given by 
\begin{align*}
 \tilde E=\begin{bmatrix}
    I_d&0\\0&\tilde N
   \end{bmatrix},\;
   \tilde F=\begin{bmatrix}
    \Omega &0\\0&I_l
   \end{bmatrix},\quad
   E=\begin{bmatrix}
    I_d&0\\0&N
   \end{bmatrix},\;
   F=\begin{bmatrix}
    \Omega &0\\0&I_l
   \end{bmatrix},\;
\end{align*}
with $N$ from {\textrm (1)} having full column-rank blocks on the secondary block diagonal,
are equivalent. Both pairs are regular with index $\mu=\kappa$ and characteristics $r=d+r_{N}$ and $\theta_0,\ldots, \theta_{\mu-2}$ from {\textrm (1)}.
\end{description}
\end{corollary}
\begin{proof} {\textrm (1):}
 The pair $\{\tilde N,I_l\}$ is regular with the characteristics 
 $r_{\tilde N}=l-\tilde l_{\mu},
 \theta_0=\tilde l_{\mu-1},
  \theta_1=\tilde l_{\mu-2},
  \ldots,
   \theta_{\mu-2}=\tilde l_{1}$ and $d=0$ owing to Proposition \ref{p.STform}(2). By Theorem \ref{t.SCF}
   it is equivalent to the pair $\{N,I_l\}$ which proves the assertion. The characteristic values are  provided by Proposition \ref{p.STform}.
   
 {\textrm (2):} By means of the transformation 
 \begin{align*}
   L=\begin{bmatrix}
    I_d&0\\0&\mathring L
   \end{bmatrix},\;
   K=\begin{bmatrix}
    I_d &0\\0&\mathring K
   \end{bmatrix}:\mathcal I\rightarrow\Real^{(d+l)\times(d+l)},
 \end{align*}
in which $\mathring L,\mathring K:\mathcal I\rightarrow \Real^{l\times l}$ represent the transformation from {\textrm (1)} we verify the equivalence by
\begin{align*}
 LEK&=\begin{bmatrix}
    I_d&0\\0&\mathring LN\mathring K
   \end{bmatrix}=
   \begin{bmatrix}
    I_d&0\\0&\tilde N
   \end{bmatrix}=\tilde E,\\
 LFK+LEK'&= \begin{bmatrix}
    \Omega&0\\0&\mathring L\mathring K
   \end{bmatrix}+
   \begin{bmatrix}
    I_d&0\\0&\mathring L N
   \end{bmatrix} 
    \begin{bmatrix}
    0&0\\0&\mathring K'
   \end{bmatrix} = \begin{bmatrix}
    \Omega&0\\0&\mathring L \mathring K+\mathring L N\mathring K'
   \end{bmatrix} =
   \begin{bmatrix}
    \Omega &0\\0&I_l
   \end{bmatrix}=\tilde F.
\end{align*}
The characteristic values are  provided by Proposition \ref{p.STform}.
\end{proof}

In case of constant matrices $\tilde N$ and $N$, $K$ is constant, too, and relation \eqref{NN} simplifies to the similarity transform $K^{-1}\tilde NK=N$.

\begin{example}\label{e.D}
Consider the following DAE in Weierstraß–Kronecker form:
\[
\left[\begin{array}{@{}ccc|cc|cc@{}}
	0 & 1 & 0 & 0 & 0 & 0 & 0  \\
	0 & 0 & 1 & 0 & 0 & 0 & 0  \\
	0 & 0 & 0 & 0 & 0 & 0 & 0\\ \hline
	0 & 0 & 0 & 0 & 1  & 0 & 0\\
	0 & 0 & 0 & 0 & 0 & 0 & 0 \\ \hline
	0 & 0 & 0 & 0 & 0 & 0 & 1 \\
	0 & 0 & 0 & 0 & 0 & 0 & 0
\end{array}\right]
\begin{bmatrix}
	x'_1\\
	x'_2 \\
	x'_3 \\
	x'_4 \\
	x'_5 \\
	x'_6 \\
	x'_7
\end{bmatrix}
+
\begin{bmatrix}
	x_1\\
	x_2 \\
	x_3 \\
	x_4 \\
	x_5 \\
	x_6 \\
	x_7
\end{bmatrix}=\begin{bmatrix}
	q_1\\
	q_2 \\
	q_3 \\
	q_4 \\
	q_5 \\
	q_6 \\
	q_7
\end{bmatrix}
\]
with  $m=7$, $r=4$, $d=0$, $\theta_0=3$, $\theta_1=1$, $\theta_2=0$.
\begin{itemize}
	\item An equivalent DAE with blockstructure \eqref{blockstructure} with full column rank secondary blocks is
	\[
	\left[\begin{array}{@{}ccc|ccc|cc@{}}
	0 & 0 & 0 & 1 & 0 & 0 & 0 \\
	0 & 0 & 0 & 0 & 1 & 0 & 0 \\ 
	0 & 0 & 0 & 0 & 0 & 1 & 0 \\ \hline
	0 & 0 & 0 & 0 & 0 & 0 & 0\\
	0 & 0 & 0 & 0 & 0 & 0 & 0 \\ 
	0 & 0 & 0 & 0 & 0 & 0 & 1 \\ \hline
	0 & 0 & 0 & 0 & 0 & 0 & 0
\end{array}\right]
\begin{bmatrix}
	x'_4\\
	x'_6\\
	x'_1 \\
	x'_5 \\
	x'_7 \\
	x'_2 \\
	x'_3
\end{bmatrix}
+ \begin{bmatrix}
	x_4\\
	x_6\\
	x_1 \\
	x_5 \\
	x_7 \\
	x_2 \\
	x_3
\end{bmatrix}=\begin{bmatrix}
	q_4\\
	q_6 \\
	q_1 \\
	q_5 \\
	q_7 \\
	q_2 \\
	q_3
\end{bmatrix}
\]
with $ \rank N_{1,2} = l_2=3$, $\rank N_{2,3}=l_3=1$, $l_1 \geq l_2 \geq l_3$. $\theta_0=3$, $\theta_1=1$, $\theta_2=0$. 
	\item An equivalent DAE with blockstructure \eqref{blockstructure} with full row rank secondary blocks is
\[
	\left[\begin{array}{@{}c|ccc|ccc@{}}
	0 & 1 & 0 & 0 & 0 & 0 & 0  \\ \hline
	0 & 0 & 0 & 0 & 1 & 0 & 0  \\
	0 & 0 & 0 & 0 & 0 & 1 & 0 \\ 
	0 & 0 & 0 & 0 & 0 & 0 & 1 \\ \hline
	0 & 0 & 0 & 0 & 0 & 0 & 0 \\
	0 & 0 & 0 & 0 & 0 & 0 & 0 \\
	0 & 0 & 0 & 0 & 0 & 0 & 0
\end{array}\right]
\begin{bmatrix}
	x'_1\\
	x'_2 \\
	x'_4 \\
	x'_6 \\
	x'_3 \\
	x'_5 \\
	x'_7
 \end{bmatrix}
+ \begin{bmatrix}
	x_1\\
	x_2 \\
	x_4\\
	x_6 \\
	x_3 \\
	x_5 \\
	x_7
\end{bmatrix}=\begin{bmatrix}
	q_1\\
	q_2 \\
	q_4 \\
	q_6 \\
	q_3 \\
	q_5 \\
	q_7
\end{bmatrix}
\]
with $ \rank N_{1,2} = l_1=1$, $\rank N_{2,3}=l_2=3$, $l_1 \leq l_2 \leq l_3$.  $\theta_0=3$, $\theta_1=1$, $\theta_2=0$. 
\end{itemize}
\end{example}
\begin{remark}\label{r.Scanform}
Theorem \ref{t.SCF} ensures that also each  pair with regular strangeness index is equivalently transformable into SCF. At this place it
should be added that the canonical form\footnote{Global canonical form in \cite{KuMe1994}} of regular pairs figured out in the context of the strangeness index \cite{KuMe1994,KuMe2006}  reads 
 
 \begin{align}\label{StrangeSCF}
 E=\begin{bmatrix}
    I_{d}&M\\&N
   \end{bmatrix},\quad
F&=\begin{bmatrix}
    \Omega&\\&I_{l}
   \end{bmatrix}, \quad
N=\begin{bmatrix}
   0&N_{12}&&\cdots&N_{1\kappa}\\
   &0&N_{23}&&N_{2\kappa}\\
   &&\ddots&\ddots&\vdots\\
   &&&&N_{\kappa-1 \kappa}\\
   &&&&0
   \end{bmatrix},\\
  M&=\begin{bmatrix}
     0&M_2&\cdots&M_{\kappa}
    \end{bmatrix},\nonumber
\end{align}
with full row-rank blocks $ N_{i,i+1}$ and $l=a=m-d$, $\kappa-1=\mu^S$.  In \cite[Theorem 3.21]{KuMe2006} one has even $\Omega=0$, taking into account that this is the result of the equivalence transformation
\begin{align*}
 LEK=\begin{bmatrix}
    I_{d}&K_{11}^{-1}M\\&N
   \end{bmatrix},\quad
LFK+LEK'&=\begin{bmatrix}
    0&\\&I_{l}
   \end{bmatrix}, 
  \end{align*} 
 in which $K_{11}$ is the fundamental solution matrix of the ODE $y'+\Omega y=0$. 
 Nevertheless this form fails to be in SCF if the entry $M$ does not vanish. This is apparently a technical problem caused by the special transformations used there.   
\end{remark}

\begin{remark}\label{r.SCFgeometric}
The structured SCF in Theorem \ref{t.SCF} makes the limitation of the geometric view from Section \ref{subs.degree} above and Section \ref{subs.nonlinearDAEsGeo} below obvious. These are regular DAEs with index $\mu$, 
degree $s=\mu-1$, and as figuration space serves  $\Real^d$ resp. $\im\begin{bmatrix}I_d\\0 \end{bmatrix}$. 
Of course, this enables the user to study the flow of the inherent ODE $u'+\Omega u=p $; however, the other part $Nv'+v=r$, which involves  the actual challenges from an application point of view, no longer plays any role.
\end{remark}

\section{Notions defined by means of derivative arrays }\label{s.notions}
\subsection{Preliminaries  and general features}\label{subs.Preliminaries}
Here we consider the DAE \eqref{1.DAE} on the given interval $\mathcal I\subseteq \Real$.
Differentiating the DAE  $k\geq 1$ times yields the inflated system
\begin{align*}
Ex^{(1)}+Fx&=q,\\
Ex^{(2)}+(E^{(1)}+F)x^{(1)}+F^{(1)}x&=q^{(1)},\\
Ex^{(3)}+(2E^{(1)}+F)x^{(2)}+(E^{(2)}+2F^{(1)})x^{(1)}+F^{(2)}x&=q^{(2)},\\
&\ldots\\
Ex^{(k+1)}+(kE^{(1)}+F)x^{(k)}+\cdots+(E^{(k)}+kF^{(k-1)})x^{(1)}+F^{(k)}x&=q^{(k)},
\end{align*}
or tightly arranged,
\begin{align}\label{1.inflated1}
 \mathcal E_{[k]}x'_{[k]}+ \mathcal F_{[k]}x 
=q_{[k]},
\end{align}
with the continuous matrix functions
$\mathcal E_{[k]}:\mathcal I\rightarrow \Real^{(mk+m)\times(mk+m)}$,\\
$\mathcal F_{[k]}:\mathcal I\rightarrow \Real^{(mk+m)\times m}$,
\begin{align}\label{1.GkLR}
\mathcal E_{[k]}=
\begin{bmatrix}
 E&0&&&\cdots&0\\
 E^{(1)}+F&E&&&\cdots&0\\
 E^{(2)}+2F^{(1)}&2E^{(1)}+F&&&\\
 \vdots&\vdots&&&\ddots&\\
 E^{(k)}+kF^{(k-1)}&\cdots&&& kE^{(1)}+F&E
\end{bmatrix},\quad
\mathcal F_{[k]}=
\begin{bmatrix}
 F\\
 F^{(1)}\\
 F^{(2)}\\
 \vdots\\
 F^{(k)}
\end{bmatrix},
\end{align}
and the variables and right-hand sides
\begin{align*}
x_{[k]}=
\begin{bmatrix}
 x\\
 x^{(1)}\\
 x^{(2)}\\
 \vdots\\
 x^{(k)}
\end{bmatrix},\qquad
q_{[k]}=
\begin{bmatrix}
 q\\
 q^{(1)}\\
 q^{(2)}\\
 \vdots\\
 q^{(k)}
\end{bmatrix}:\mathcal I\rightarrow\Real^{mk+m}.
\end{align*}
Set $\mathcal F_{[0]}=F,\, \mathcal E_{[0]}=E$,\ $x_{[0]}=x, \, q_{[0]}=q$, such that the DAE \eqref{1.DAE} itself coincides with
\begin{align}\label{1.inflated0}
 \mathcal E_{[0]}x'_{[0]}+\mathcal F_{[0]}x
=q_{[0]}.
\end{align}
By its design, the system \eqref{1.inflated1} includes all previous systems with lower dimensions,
\begin{align*}
 \mathcal E_{[j]}x'_{[j]}+\mathcal F_{[j]}x 
=q_{[j]},\quad j=0,\ldots,k-1,
\end{align*}
and the sets
\begin{align}\label{1.consitent1}
 \mathcal C_{[j]}(t)=\{z\in \Real^{m}: \mathcal F_{[j]}(t)z -q_{[j]}(t) \in \im \mathcal E_{[j]}(t)\},\quad t\in \mathcal I,\quad j=0,\ldots,k,
\end{align}
satisfy the inclusions
\begin{align}\label{1.consitent2}
 \mathcal C_{[k]}(t)\subseteq\mathcal C_{[k-1]}(t)\subseteq\cdots\subseteq \mathcal C_{[0]}(t)=\{z\in\Real^{m}:F(t)z-q(t)\in \im E(t)\}, \quad t\in \mathcal I.
\end{align}
Therefore, each smooth solution $x$ of the original DAE must meet the so-called constraints, that is,
\begin{align*}
 x(t)\in \mathcal C_{[k]}(t), \quad t\in\mathcal I.
\end{align*}
In the following, the rank functions $r_{[k]}:\mathcal I\rightarrow \Real$,
\begin{align}\label{eq:rankGR}
 r_{[k]}(t)=\rank \mathcal E_{[k]}(t),\quad t\in \mathcal J,\; k\geq 0,
\end{align}
and the projector valued functions $\mathcal W_{[k]}:\mathcal I\rightarrow \Real^{(mk+m)\times(mk+m)}$,
\begin{align}\label{eq:WGR}
\mathcal W_{[k]}(t)=I_{mk+m}- \mathcal E_{[k]}(t)\mathcal E_{[k]}(t)^{+},\quad t\in \mathcal J,\; k\geq 0,
\end{align}
will play their role,
and further the associated linear subspaces
\begin{align}\label{1.subspace1}
 S_{[k]}(t)=\{z\in \Real^{m}: \mathcal F_{[k]}(t)z \in \im \mathcal E_{[k]}(t)\}=\ker \mathcal W_{[k]}(t)\mathcal F_{[k]}(t) ,\quad t\in \mathcal I,\quad k\geq 0.
\end{align}
 Obviously, it holds that
\begin{align}\label{1.subspace2}
 S_{[k]}(t)\subseteq S_{[k-1]}(t)\subseteq\cdots\subseteq S_{[0]}(t)=\{z\in\Real^{m}:F(t)z\in \im E(t)\}, \quad t\in \mathcal I.
\end{align}
It should be emphasized that, if the rank function $r_{[k]}$ is constant, then the pointwise Moore-Penrose inverse ${\mathcal E_{[k]}}^{+}$ and the projector function $\mathcal W_{[k]}$ are as smooth as $\mathcal E_{[k]}$. Otherwise one is confronted with discontinuities.
\begin{remark}[A necessary regularity condition]\label{r.a1}
 One aspect of regularity is that the DAE \eqref{1.DAE} should be such that it has a correspondingly smooth solution to any $m$ times continuously differentiable function $q:\mathcal I\rightarrow \Real^{m}$. If this is so, all matrix functions 
 \begin{align*}
  \begin{bmatrix}
   \mathcal E_{[k]}&\mathcal F_{[k]}
  \end{bmatrix}:\mathcal I\rightarrow \Real^{(mk+m)\times(mk+2m)}
 \end{align*}
 must have full-row rank, i.e.,
  \begin{align}\label{eq:fullrank}
 \rank \begin{bmatrix}
   \mathcal E_{[k]}(t)&\mathcal F_{[k]}(t)
  \end{bmatrix}=mk+m,\quad t\in\mathcal I, \; k\geq 0.
 \end{align}
 If, on the contrary, condition \eqref{eq:fullrank} is not valid, i.e., there are a $\bar k$ and a $\bar t$ such that
 \begin{align*}
 \rank \begin{bmatrix}
   \mathcal E_{[\bar k]}(\bar t)&\mathcal F_{[\bar k]}(\bar t)
  \end{bmatrix}< m\bar k+m,
 \end{align*}
 then there exists a nontrivial $w\in\Real^{m\bar k+m}$ such that
 \begin{align*}
 w^{*}\begin{bmatrix}
   \mathcal E_{[\bar k]}(\bar t)&\mathcal F_{[\bar k]}(\bar t)
  \end{bmatrix}=0.
 \end{align*}
 Regarding the relation 
 \begin{align*}
   \mathcal E_{[\bar k]}(\bar t)x'_{[\bar k]}(\bar t)+\mathcal F_{[\bar k]}(\bar t)x(\bar t)=q_{[\bar k]}(\bar t)
 \end{align*}
one is confronted with the restriction $w^{*}q_{[\bar k]}(\bar t)=0$ for all inhomogeneities.
\end{remark}
\begin{remark}[Representation of $\mathcal C_{[k]}(t)$]\label{r.a2}
The full row rank condition  \eqref{eq:fullrank}, i.e. also
\begin{align}\label{eq:fullrank(t)}
 \im [\mathcal E_{[k]}(t)\, \mathcal F_{[k]}(t)]=\Real^{mk+m}
\end{align}
implies
\begin{align*}
 \im \underbrace{\mathcal W_{[k]}(t)[\mathcal F_{[k]}(t)\, \mathcal E_{[k]}(t)]}_{[\mathcal W_{[k]}(t)\mathcal F_{[k]}(t)\quad 0]}=\im \mathcal W_{[k]}(t),
\end{align*}
thus
 \begin{align}\label{eq:WRGL}
  \im \mathcal W_{[k]}(t)\mathcal F_{[k]}(t)=\im \mathcal W_{[k]}(t),
  \end{align}
and in turn 
\begin{align}
 \mathcal C_{[k]}(t)=S_{[k]}(t)+(\mathcal W_{[k]}(t)\mathcal F_{[k]}(t))^{+}\mathcal W_{[k]}(t)q_{[k]}(t),\label{eq:RepresC}\\
 \dim S_{[k]}(t)=m-\rank \mathcal W_{[k]}(t)=r_{[k]}(t)-mk.\label{eq:Represrank}
\end{align}
By representation \eqref{eq:RepresC}, $\mathcal C_{[k]}(t)$ appears to be an affine subspace of $\Real^{m}$ associated with $S_{[k]}(t)$.
\end{remark}
It becomes clear  that under the necessary regularity condition \eqref{eq:fullrank} the dimensions of the subspaces $S_{[k]}(t)$ are fully determined by the ranks of $\mathcal E_{[k]}(t)$ and vice versa. In particular, then $\dim S_{[k]}(t)$ is independent of $t$ if and only if $r_{[k]}(t)$ is so, a matter that will later play a quite significant role.
\medskip

If the DAE \eqref{1.DAE} is interpreted as in \cite{Chist1996,ChistShch} as a Volterra integral equation
\begin{align}\label{Int}
 E(t)x(t)+\int_a^t (F(s)-E'(s))x(s)){\mathrm ds}= c+\int_a^t q(s){\mathrm ds}
\end{align}
 then the inflated system created on this basis  reads
 \begin{align*}
  \mathcal{D}_{[k]}x_{[k]}=\begin{bmatrix}
                            -\int_a^t (F(s)-E'(s))x(s)){\mathrm ds}+ c+\int_a^t q(s){\mathrm ds}\\q_{[k-1]}
                           \end{bmatrix},
 \end{align*}
with the array function 
\begin{align}\label{1.Dk}
\mathcal{D}_{[k]} =
\begin{bmatrix}
	  E & 0 \\
	\mathcal F_{[k-1]}& \mathcal E_{[k-1]}
\end{bmatrix}:\mathcal I\rightarrow  \Real^{(m+mk)\times(m+mk)}.
\end{align}
To get an idea about the rank of $\mathcal D_{[k]}(t)$ we take a closer look at the time-varying subspace $\ker \mathcal D_{[k]}(t)$.
We have for $k\geq 1$ that 
\begin{align}
 \ker \mathcal D_{[k]}&=\left\{\begin{bmatrix}
                          z\\w
                         \end{bmatrix}\in \Real^{m}\times\Real^{mk}: Ez=0, \mathcal F_{[k-1]}z+ \mathcal E_{[k-1]}w=0
\right\}\nonumber\\
&=\left\{\begin{bmatrix}
            z\\w
     \end{bmatrix}\in \Real^{m}\times\Real^{mk}: z\in\ker E, \mathcal W_{[k-1]}\mathcal F_{[k-1]}z=0, \mathcal E_{[k-1]}^{+}\mathcal E_{[k-1]}w=- \mathcal E_{[k-1]}^{+}\mathcal F_{[k-1]}z
\right\}\nonumber\\
&=\left\{\begin{bmatrix}
            z\\w
     \end{bmatrix}\in \Real^{m}\times\Real^{mk}: z\in\ker E\cap S_{[k-1]}, \mathcal E_{[k-1]}^{+}\mathcal E_{[k-1]}w=-\mathcal E_{[k-1]}^{+}\mathcal F_{[k-1]}z
\right\},\label{1.kerDk}
\end{align}
and consequently,
\begin{align}
 \rank \mathcal D_{[k]}&= m-\dim(\ker E\cap S_{[k-1]})+r_{[k-1]}. \label{1.rankDk}
\end{align}

If $ \mathcal E_{[k]}$  has constant rank, then the projector functions $\mathcal W_{[k]}$ and the Moore-Penrose inverse $\mathcal E_{[k]}^{+}$ inherit the smoothness of $\mathcal E_{[k]}$.

The following proposition makes clear that, in any case, both $r_{[k]}(t)=\rank \mathcal E_{[k]}(t)$ and $\rank \mathcal D_{[k]}(t)$ as well as  $\dim S_{[k]}(t)$ and  $\dim (\ker E(t)\cap S_{[k]}(t))$, $t\in\mathcal I$,  are invariant under equivalence transformations.
\begin{proposition}\label{p.equivalenc}
 Given are two equivalent coefficient pairs $\{E,F\}$ and $\{\tilde E,\tilde F\}$, $\tilde E=LEK$, $\tilde F=LFK+LEK'$,  $E,F,L,K:\mathcal I\rightarrow \Real^{m\times m}$ sufficiently smooth, $L$ and $K$ pointwise nonsingular.
 
 Then, the inflated matrix function pair $\tilde{\mathcal E}_{[k]}:\mathcal I\rightarrow \Real^{(mk+m)\times(mk+m)}$,
$\tilde{\mathcal F}_{[k]}:\mathcal I\rightarrow \Real^{(mk+m)\times m}$ and the subspace $\tilde S_{[k]}$  related to $\{\tilde E,\tilde F\}$ satisfy the following:
 \begin{align*}
  \tilde{\mathcal E}_{[k]}&=\mathcal L_{[k]}\mathcal E_{[k]}\mathcal K_{[k]},\quad \tilde{\mathcal F}_{[k]}=\mathcal L_{[k]}\mathcal F_{[k]}K+\mathcal L_{[k]}\mathcal E_{[k]}\mathcal H_{[k]},\quad \mathcal H_{[k]}=
  \begin{bmatrix}
    K'\\\vdots\\K^{(k+1)}                                                                                                                                                                                                                                                                                              \end{bmatrix},\\
\tilde S_{[k]}&=K^{-1}S_{[k]}, \quad \tilde S_{[k]}\cap\ker \tilde E=K^{-1}(S_{[k]}\cap\ker E),
 \end{align*}
in which the 
 matrix functions $\mathcal L_{[k]},\mathcal K_{[k]}:\mathcal I\rightarrow \Real^{(m+mk)\times (m+mk}$ are uniquely determined by $L$ and $K$, and their derivatives, respectively. They are pointwise nonsingular and have lower triangular block structure,
\begin{align*}
 \mathcal L_{[k]}=\begin{bmatrix}
                   L&0&\cdots&0\\
                   \ast&L&\cdots&0\\
                   \vdots&&\ddots&0\\
                   \ast&&\cdots&L
                  \end{bmatrix},\quad
   \mathcal K_{[k]}=\begin{bmatrix}
                   K&0&\cdots&0\\
                   \ast&K&\cdots&0\\
                   \vdots&&\ddots&0\\
                   \ast&&\cdots&K
                  \end{bmatrix}=:
                  \begin{bmatrix}
                   K&0\\\mathcal K_{{[k]}\,{21}}&\mathcal K_{{[k]}\;{22}}
                  \end{bmatrix}.              
\end{align*}
\end{proposition}
\begin{proof}
The representation of $\tilde{\mathcal E}_{[k]}$ and $\tilde{\mathcal F}_{[k]}$ is given by a slight adaption of \cite[Theorem 3.29]{KuMe2006}. We turn to $\tilde S_{[k]}$.

$\tilde z\in \tilde S_{[k]}$ means $\tilde{\mathcal F}_{[k]}\tilde z\in \im \tilde{\mathcal E}_{[k]}$, thus $ \mathcal F_{[k]}\tilde z +\mathcal E_{[k]} \mathcal H_{[k]}\tilde z\in \im \mathcal E_{[k]}$, then also $ \mathcal F_{[k]}\tilde z \in \im \mathcal E_{[k]}$, that is, $K\tilde z\in S_{[k]}$.
Regarding also that $\tilde z\in \ker \tilde E$ means $K\tilde z\in \ker E$ we are done.
\end{proof}
 The following lemma gives a certain first idea about the size of the rank functions.
\begin{lemma}\label{l.rk}
 The rank functions  $r_{[k]}=\rank \mathcal E_{[k]}$ and $r^{\mathcal D}_{[k]}=\rank \mathcal D_{[k]}$, $k\geq 1$,  $r^{\mathcal D}_{[0]}= r_{[0]}=\rank E$, satisfy the inequalities
 \begin{align*}
  r_{[k]}(t)+r(t)\leq r_{[k+1]}(t)\leq r_{[k]}(t)+m,\quad t\in\mathcal I, \; k\geq0,\\
   r^{\mathcal D}_{[k]}(t)+r(t)\leq r^{\mathcal D}_{[k+1]}(t)\leq r^{\mathcal D}_{[k]}(t)+m,\quad t\in\mathcal I, \; k\geq0,
 \end{align*}
\end{lemma}
\begin{proof}
 The special structure of both matrix functions  satisfies the requirement of Lemma \ref{l.app2} ensuring the inequalities.
\end{proof}

The question of whether the ranks $r_{[i]}$ of the matrix functions $\mathcal E_{[i]}$ are constant will play an important role below. We are also interested in the relationships to the rank conditions associated with the Definition \ref{d.2}. We see points where these rank conditions are violated as critical points which require closer examination. In Section \ref{s.examples} below a few examples are discussed in detail to illustrate the matter.

\begin{lemma}\label{l.R1}
 Let the matrix functions $E,F:\mathcal I\rightarrow\Real^{m\times m}$ be such that, for all $t\in \mathcal I$, $\rank E(t)=r$, $\rank [E(t) F(t)]=m$. 
 Denote $\theta_{0}(t)=\dim(\ker E(t)\cap S_{[0]}(t))=\dim(\ker E(t)\cap \ker Z(t)^*F(t))$ in which $Z:\mathcal I\rightarrow\Real^{m\times(m-r)}$ is a  basis of $(\im E)^{\perp}$.
 
 Then it results that 
 \[r_{[1]}=\rank\mathcal E_{[1]}(t)=\rank\mathcal D_{[1]}(t)=m+r-\theta_{0}(t),\quad t\in \mathcal I,
  \]
  and both $\mathcal E_{[1]}$ and $\mathcal D_{[1]}$ have constant rank precisely if the pair is pre-regular.
\end{lemma}
\begin{proof}
 We consider the nullspaces of $\mathcal D_{[1]}$ and $\mathcal E_{[1]}$, that is
 \begin{align*}
 \ker \mathcal D_{[1]}&=\ker \begin{bmatrix}
                              E&0\\F&E
                             \end{bmatrix}
   =\{z\in\Real^{2m}:Ez_1=0, Fz_1+Ez_2=0 \}\\
   &=\{z\in\Real^{2m}:Ez_1=0, Fz_1\in \im E, E^+Ez_2=-E^+Fz_1 \}\\
   &=\{z\in\Real^{2m}:z_1\in \ker E\cap \ker Z^*F, E^+Ez_2=-E^+Fz_1 \},\\
  \ker \mathcal E_{[1]}&=\ker \begin{bmatrix}
                              E&0\\E'+F&E
                             \end{bmatrix}
   =\{z\in\Real^{2m}:Ez_1=0, (E'+F)z_1+Ez_2=0 \} \\
   &=\{z\in\Real^{2m}:Ez_1=0, (E'+F)z_1\in \im E, E^+Ez_2=-E^+(E'+F)z_1 \}\\
   &=\{z\in\Real^{2m}:z_1\in \ker E\cap \ker Z^*(E'+F), E^+Ez_2=-E^+(E'+F)z_1 \}.
 \end{align*}
Since $Z^*E'(I-E^+E)=-Z^*E(I-E^+E)'=0$ we know that $\ker E\cap \ker Z^*F=\ker E\cap \ker Z^*(E'+F)$ and hence
$\dim  \ker \mathcal E_{[1]}=\dim  \ker \mathcal E_{[1]}= \dim (\ker E\cap\ker Z^*F) +m-r=\theta_0+m-r$, thus $\rank  \mathcal E_{[1]}= 2m-(\theta_0+m-r)=m+r-\theta_0 $.
\end{proof}
\subsection{Array functions for DAEs being transformable into SCF and for regular DAEs}\label{subs.SCFarrays}
In this Section, we consider important properties of the array function $\mathcal E_{[k]}$ and $\mathcal D_{[k]}$ from \eqref{1.GkLR} and \eqref{1.Dk}. First of all we observe that both are special cases of the matrix function 
\begin{footnotesize}
\begin{align}\label{eq.arrayHk_SCF}
\mathcal H_{[k]}:=
\begin{bmatrix}
 E&0&&\cdots&&0\\
 \alpha_{2,1}E^{(1)}+ F&E&&&&\vdots\\
 \alpha_{3,1}E^{(2)}+\beta_{3,1}F^{(1)}&\alpha_{3,2}E^{(1)}+F&E&\\
 \vdots&\ddots&\ddots&\ddots&&0\\
 \alpha_{k+1,1}E^{(k)}+\beta_{k+1,k}F^{(k-1)}&\cdots& \alpha_{k+1,k-1} E^{(2)}+ \beta_{k+1,k-1} F^{(1)}& \alpha_{k+1,k} E^{(1)}+ F&&E
\end{bmatrix},
\end{align}
\end{footnotesize}
each with different coefficients $\alpha_{i,j}$ and $\beta_{i,j}$. We do not specify them, as they do not play any role later on.
\medskip

Let for a moment the given DAE be in SCF, see \eqref{1.SCF}, that is,
\begin{align*}
 E=\begin{bmatrix}
	I_d & 0 \\
	0 & N
\end{bmatrix}, \quad F =\begin{bmatrix}
	\Omega & 0 \\
	0 & I_{m-d}
	\end{bmatrix}, 
\end{align*}
with a strictly upper triangular matrix function $N$. We evaluate the nullspace of the corresponding matrix $\mathcal H_{[k]}(t)\in \Real^{(m+km)\times (m+km}$ for each fixed $t$, but drop the argument $t$ again.

Denote 
\begin{align*}
 z=\begin{bmatrix}
    z_0\\\vdots\\z_k
   \end{bmatrix}\in \Real^{(k+1)m},\; 
   z_j=\begin{bmatrix}
    x_j\\y_j
   \end{bmatrix}\in \Real^{m},\; x_j\in \Real^d,\; y_j\in \Real^{m-d},\; \begin{bmatrix}
    y_0\\\vdots\\y_k
   \end{bmatrix}=:y\in \Real^{(k+1)(m-d)}
\end{align*}
and evaluate the linear system $\mathcal H_{[k]}z=0$. 
The first block line gives 
\begin{align*}
 x_0=0, \quad Ny_0=0,
\end{align*}
 and the entire system decomposes in parts for $x$ and $y$.
All components  $x_j$ are fully determined and zero, and it results that $\mathcal{N}_{[k]}y=0$, with
\begin{align}\label{eq.arrayNk}
\mathcal N_{[k]}:=
\begin{bmatrix}
 N&0&&&\cdots&0\\
 I+\alpha_{2,1} N^{(1)}&N&&&&\vdots\\
 \alpha_{3,1} N^{(2)}&I+\alpha_{3,2}N^{(1)}&N&&\\
\vdots& \ddots&\ddots&\ddots& &\\
 \vdots& &\ddots&\ddots&\ddots&0\\
 \alpha_{k+1,1}N^{(k)}&\cdots&&\alpha_{k+1,k-2}N^{(2)}&I+ \alpha_{k+1,k}N^{(1)}&N
\end{bmatrix}.
\end{align}
This leads to the relations
\begin{align*}
 \rank \mathcal H_{[k]}=(k+1)d+\rank  \mathcal N_{[k]},\\
 \dim\ker \mathcal H_{[k]}= \dim\ker \mathcal N_{[k]},
\end{align*}
such that the question how $ \rank \mathcal H_{[k]}$ behaves can be traced back to $\mathcal N_{[k]}$. We have prepared relevant properties of $ \rank \mathcal N_{[k]}$ in some detail in  Appendix \ref{subs.A_strictly}, which enables us to formulate the following basic general results. 
Obviously, if the pair $\{E,F\}$ is transferable into SCF and $N$ changes its rank on the given interval, then $E$ and $\mathcal E_{[0]}=E$ do so, too. It may also happen that $N$ and in turn $\mathcal E_{[0]}=E$ show constant rank but further $\mathcal E_{[i]}$ suffer from rank changes, as Example \ref{e.5} confirms for $i=1$. Nevertheless, the subsequent matrix functions at the end have a constant rank as the next assertion shows. 
\begin{theorem}\label{t.rankSCF}
 If the pair $\{E,F\}$ is transferable into SCF with characteristics $d$ and $a=m-d$ then 
  \begin{description}
  \item[\textrm{(1)}] the 
 derivative array functions $\mathcal E_{[k]}$ and $\mathcal D_{[k]}$ become constant ranks for $k\geq a-1$, namely
 \begin{align*}
 r_{[k]}= \rank \mathcal E_{[k]}= \rank \mathcal D_{[k]}=km+d,\quad  k \geq a-1.
 \end{align*}
 \item[\textrm{(2)}] Moreover, 
\begin{align*}
\dim (\ker E \cap S_{[k]})=0,\quad k\geq a.
\end{align*}
 \end{description}
\end{theorem}
\begin{proof}
Owing to Proposition \ref{p.equivalenc} we  may turn to the SCF, which leads to
\begin{align*}
 \rank \mathcal D_{[k]}=\rank \mathcal E_{[k]}=\rank \mathcal H_{[k]}=(k+1)d+\rank \mathcal N_{[k]},
\end{align*}
and regarding Proposition \ref{prop.rank.Nk} we obtain
\begin{align*}
 \rank \mathcal H_{[k]}=(k+1)d+\rank \mathcal N_{[k]}=(k+1)d +ka+\rank N\tilde N_2 \cdots \tilde N_{k+1},
\end{align*}
in which $N\tilde N_2 \cdots \tilde N_{k+1}$ is a product of $k+1$ strictly upper triangular matrix functions  of size $a\times a$. Clearly, if $k\geq a-1$ then  $N\tilde N_2 \cdots \tilde N_{k+1}= 0$ and in turn 
\begin{align*}
 \rank \mathcal H_{[k]}=(k+1)d +ka=km+d.
\end{align*}
Now formula \eqref{1.rankDk} implies for $k\geq a$,
\begin{align*}
 \dim(\ker E\cap S_{[k]})&=m+r_{[k]}-\rank \mathcal D_{[k+1]}\\&=m+r_{[k]}-r_{[k+1]}=m+(km+d)-((k+1)m+d)=0.
\end{align*}
\end{proof}
It is an advantage of regular pairs that all associated matrix functions arrays have constant rank as we know from the following assertion.

\begin{theorem}\label{th.ranks}
Let the pair $\{E,F\}$ be regular on $\mathcal I$ with index $\mu $ and characteristic values $r$ and $\theta_0\geq\cdots\geq \theta_{\mu-2}>\theta_{\mu-1}=0$. Set $\theta_k=0$ for $k\geq\mu$. Then the
 following assertions are valid:
 \begin{description}
  \item[\textrm{(1)}] The 
 derivative array functions $\mathcal E_{[k]}$ and $\mathcal D_{[k]}$ have constant ranks, namely 
 \begin{align*}
 r_{[k]}= \rank \mathcal E_{[k]}= \rank \mathcal D_{[k]}=km+r-\sum_{i=0}^{k-1}\theta_i,\quad  k \geq 1.
 \end{align*}
 \item[\textrm{(2)}] In particular, $r_{[k]}= \rank \mathcal E_{[k]} = km+d$,\; $\dim\ker\mathcal E_{[k]}=m-d=a $,\;  if $k\geq\mu-1$.
\item[\textrm{(3)}] For $k\geq \mu$, there is a continuous function $H_k:\mathcal I\rightarrow\Real^{km\times km}$ such that the nullspace of $\mathcal E_{[k]}$ has the special form
\begin{align*}
 \ker \mathcal E_{[k]}=\{\begin{bmatrix}
                          z\\w
                         \end{bmatrix}\in\Real^{m+km}:z=0, H_kw=0\}.
\end{align*}
\item[\textrm{(4)}] $\dim S_{[k]}=r-\sum_{i=0}^{k-1}\theta_i, \; k\geq 1$, and  $\dim S_{[\mu-1]}= \dim S_{[\mu]}=d$.
\item[\textrm{(5)}] $S_{[\mu-1]}= S_{[\mu]}=S_{can}$.
\item[\textrm{(6)}] $ \dim(\ker E\cap S_{[k]})=\theta_k,\quad k\geq 0$.
 \end{description}
\end{theorem}
\begin{proof}
 {\textrm (1):} We note that this assertion is a straightforward consequence of 
 \cite[Theorem 3.30]{KuMe2006}. Nevertheless, we formulate here a more transparent direct proof based on the preceding arguments, which at the same time serves as an auxiliary means for the further proofs. For $\mu=1$ we are done by Lemma \ref{l.R1}, so we assume $\mu\geq 2$.

Each regular pair $\{ E,F \}$ with index $\mu \ge 2$ and characteristic values $r, \theta_0 \ge \cdots \ge \theta_{\mu-2} > \theta_{\mu-1} = 0$, features also  the regular tractability  index $\mu$ and can be equivalently transformed into the structured SCF \cite[Theorem 2.65]{CRR}
\begin{align}\label{N0}
 E=\begin{bmatrix}
    I_d&0\\0&N
   \end{bmatrix},\quad
F=\begin{bmatrix}
    \Omega&0\\0&I_a
   \end{bmatrix}
\end{align}
in which the matrix function $N$ is strictly block upper triangular with exclusively full  column-rank blocks on the secondary diagonal and $N^{\mu}=0$, in more detail, see Proposition \ref{p.STform}(1), 
\begin{align*}
N=\begin{bmatrix}
   0&N_{12}&&\cdots&N_{1,\mu}\\
   &0&N_{23}&&N_{2,\mu}\\
   &&\ddots&\ddots&\vdots\\
   &&&&N_{\mu-1,\mu}\\
   &&&&0
   \end{bmatrix},\\
   \text{with blocks}\; N_{ij} \; \text{ of sizes}\; l_{i}\times l_{j},\; \rank N_{i,i+1}=l_{i+1}.
\end{align*}
and $l_1=m-d-r, l_2=\theta_0,\ldots, l_{\mu}=\theta_{\mu-2}$. 
Since $\rank \mathcal E_{[k]}$ and $\rank \mathcal D_{[k]}$ are invariant with respect to equivalence transformations, we can turn to the array function $ \mathcal H_{[k]}$ applied 
to the structured SCF, and further to $ \mathcal N_{[k]}$.
Regarding the relation 
\begin{align*}
 \rank \mathcal H_{[k]}&=(k+1)d+\rank  \mathcal N_{[k]},\\
 \dim\ker \mathcal H_{[k]}&= \dim\ker \mathcal N_{[k]},
\end{align*}
we obtain by Proposition \ref{prop.rank.Nk}, formula \eqref{Nk_rank},
\begin{align*}
 \rank \mathcal H_{[k]}&=(k+1)d+\rank  \mathcal N_{[k]}=(k+1)d+k(m-d)+\rank N^{k+1}\\
 &=km+d+\rank N^{k+1}.
\end{align*}
Lemma \ref{l.Ncol} (with $l=m-d$) implies $\rank N^{k+1}= m-d-(l_1+\cdots l_{k+1})$, thus 
$\rank N^{k+1}= m-d-(m-r+\theta_0+\cdots+\theta_{k-1})= r-d-(\theta_0+\cdots+\theta_{k-1})$, and 
therefore
\begin{align*}
 \rank \mathcal H_{[k]}
 =km+r-\sum_{j=0}^{k-1}\theta_j.
\end{align*}
{\textrm(2):} This is a direct consequence of  {\textrm(1)}.

{\textrm(3):} This follows from Corollary \ref{c.Nk_-full}. 

{\textrm(4):} This is a consequence of relation \eqref{eq:Represrank} and the solvability properties provided by Theorem \ref{t.solvability}.

{\textrm(5):} This results from the inclusions $S_{[\mu]}\subseteq S_{[\mu-1]}$ and $S_{[\mu]}\subseteq S_{can}$ since all these subspaces have the same dimension, namely $d$.

{\textrm(6):} Next we investigate the intersection $S_{[k]} \cap \ker E$.

Applying {\textrm(1)} formula \eqref{1.rankDk} (which concerns the nullspace of $\mathcal D_{[k]}$) immediately yields
\begin{align*}
 \dim(\ker E\cap S_{[k-1]})=m+r_{[k-1]}-\rank \mathcal D_{[k]}=m+r_{[k-1]}-r_{[k]}=\theta_{k-1}.
\end{align*}
\end{proof}
\bigskip

\subsection{Differentiation index}\label{subs.diff}
The most popular idea behind the index of a DAE is to filter an explicit ordinary differential equation (ODE) with respect to $x$ out of the inflated system \eqref{1.inflated1}, a so-called \textit{completion ODE}, also \emph{underlying ODE}, of the form
\begin{align}\label{1.completionODE}
x^{(1)}+ Ax=f,
\end{align}
with a continuous matrix function $A:\mathcal I\rightarrow\Real^{m\times m}$. The  index of the DAE is the minimum number of differentiations needed to determine such an explicit ODE, e.g., \cite[Definition 2.4.2]{BCP89}.
At this point it should be emphasized that in early work the index type was not yet specified. It was simply spoken of the\emph{ index}. Only later epithets were used for distinction of various approaches. In particular in  \cite{GHM} the term \emph{differentiation index} is used  which is now widely practiced, e.g., \cite[Section 3.3]{KuMe2006}, \cite[Section 3.7]{RR2008}.
The following definition after \cite{Campbell87}\footnote{In \cite{BCP89}, this is the statement of \cite[Proposition 2.4.2]{BCP89}.} is the specification common today. 
\begin{definition}\label{d.diff}
 The smallest number $\nu\in\Natu$, if it exists, for which the matrix function $\mathcal E_{[\nu]}$ has constant rank and is smoothly $1$-full is called the \emph{differentiation index} of the pair $(E,F)$ and the DAE \eqref{1.DAE}, respectively.\footnote{For $1$-fullness we refer to the Appendix \ref{subs.A1}
} 
We then indicate the differentiation index by $\mu^{diff}=\nu$.
\end{definition}
If   $\mathcal E_{[\nu]}$ is smoothly $1$-full, then there is a nonsingular, continuous matrix function $\mathcal T$ such that
\begin{align}\label{1.1full}
 \mathcal T \mathcal E_{[\nu]}=\begin{bmatrix}
                                 I_{m}&0\\0&H_{\mathcal E}
                                \end{bmatrix},
\end{align}
and the first block-line of the inflated system \eqref{1.inflated1} scaled by $\mathcal T$ is actually an explicit ODE with respect to $x$, i.e.,
\begin{align*}
  x^{(1)}+(\mathcal T\mathcal F_{[\nu]})_{1}x 
=(\mathcal Tq_{[\nu]})_{1},
\end{align*}
with a continuous matrix coefficient $(\mathcal T\mathcal F_{[\nu]})_{1}:\mathcal I\rightarrow\Real^{m\times m}$.
Supposing a consistent initial value for $x$, that is, $x(t_0)=x_0\in \mathcal{C}_{[{\nu}]}(t_0)$\footnote{See representation \eqref{eq:RepresC}.}, the solution of the IVP for this ODE is a solution of the DAE \cite[Theorem 2.48]{
BCP89}.
\begin{proposition}\label{p.diffind}
The differentiation index remains invariant under sufficiently smooth equivalence transformations.
\end{proposition}
\begin{proof}
Let $\mathcal E_{[k]}$ have constant rank and be smoothly $1$-full such that \eqref{1.1full} is given. The transformed 
$\tilde{\mathcal E}_{[k]}$ has the same constant rank as $\mathcal E_{[k]}$.
Following \cite[Theorem 3.38]{KuMe2006}, with the notation of Proposition \ref{p.equivalenc}, we derive
\begin{align*}
\underbrace{\begin{bmatrix}
 K^{-1}&0\\
 -H_{\mathcal E}\mathcal K_{[k]\;21}K^{-1}&I
\end{bmatrix} \mathcal T\mathcal L_{[k]}^{-1}}_{=:\tilde{\mathcal T}}\tilde{\mathcal E}_{[k]}=
\begin{bmatrix}
 I&0\\0&H_{\mathcal E}\mathcal K_{[k]\;{22}}
\end{bmatrix}.
\end{align*}
$\tilde{\mathcal T}$ is pointwise nonsingular and continuous. The matrix function $\mathcal E_{[k]}$ and $\tilde{\mathcal E}_{[k]}$ are smoothly $1$-full simultaneously, which completes the proof.
\end{proof}
\begin{proposition}\label{p.index1}
The DAE \eqref{1.DAE} and the pair $\{E,F\}$ have differentiation index one, if and only if they are regular with index $\mu=1$ in the sense of Definition \ref{d.2}.
 The index-one case goes along with $S_{[0]}=S_{[1]}=S_{can}$ and $d=\dim S_{can}=\rank E=r$.
\end{proposition}
\begin{proof}
 Let $\mathcal E_{[1]}$ be smoothly $1-$ full,
 \[\mathcal E_{[1]}=\begin{bmatrix}
                     E&0\\E'+F&E
                    \end{bmatrix}.
 \]
Owing to Lemma \ref{l.app} there is a continuous matrix function $H:\mathcal I\rightarrow\Real^{m\times m}$ with constant rank such that
\begin{align}\label{index1}
 \ker \mathcal E_{[1]}=\{z\in\Real^{2m}: z_1=0, Hz_2=0  \}.
\end{align}
On the other hand we derive
\begin{align*}
 \ker \mathcal E_{[1]}&=\{z\in\Real^{2m}: Ez_1=0, (E'+F)z_1+Ez_2=0 \}\\
 &=\{z\in\Real^{2m}: Ez_1=0, (E'+F)z_1\in \im E,  Ez_2=-(E'+F)z_1 \}.
\end{align*}
Introduce the subspace $\tilde S:= \{w\in\Real^m: (E'+F)w\in \im E\}$. It comes out that 
\begin{align*}
 \ker \mathcal E_{[1]}
 =\{z\in\Real^{2m}: z_1\in \ker E\cap \tilde S,  Ez_2=-(E'+F)z_1 \}.
\end{align*}
Comparing with \eqref{index1} we obtain that the condition $\ker E\cap \tilde S=\{0\}$ must be valid, and hence
\begin{align*}
 \ker \mathcal E_{[1]}
 =\{z\in\Real^{2m}: z_1=0,  Ez_2=0 \},\quad \dim \ker \mathcal E_{[1]}=\dim \ker E.
\end{align*}
Then, in particular, $\rank E$ is constant and the projector functions $Q:=I-E^+E, W:=I-EE^+$
are as smooth as $E$. This leads to $WE'Q=-WEQ'=0$, thus $\ker E\cap \ker WF=\ker E\cap\tilde S=\{0\}$. Then the matrix function $E+WF$ remains nonsingular and $\im [E\,F]=\im [E\,WF]=\Real^m$.
Now it is evident that the pair $\{E,F\}$ is pre-regular with $\theta=0$ and furthermore  regular with index $\mu=1$.

In the opposite direction we assume the pair $\{E,F\}$ to be regular with index $\mu=1$. Then it is also regular with tractability index one and the matrix function $G_1=E+(F-EP')Q:\mathcal I\rightarrow\Real^{m\times m}$ remains nonsingular, $P:=I-Q$. With
\begin{align*}
 \mathcal T:=\begin{bmatrix}
              (I+QG_1^{-1}E'P)^{-1}&0\\
              -PG_1^{-1}E'(I+QG_1^{-1}E'P)^{-1}&I
             \end{bmatrix}
             \begin{bmatrix}
              P&Q\\Q-PG_1^{-1}F&P
             \end{bmatrix}
             \begin{bmatrix}
              G_1^{-1}&0\\0&G_1^{-1}
             \end{bmatrix}
\end{align*}
we obtain that
\begin{align*}
 \mathcal T\mathcal E_{[1]}=\begin{bmatrix}
                             I&0\\0&P
                            \end{bmatrix},
                            \end{align*}
and we are done.
\end{proof}
\begin{proposition}\label{p.diff}
 If the differentiation index $\mu^{diff}$ is well-defined for the pair $\{E,F\}$, then it follows that 
 \begin{description}
  \item[\textrm{(1)}] $\mathcal E_{[\mu^{diff}]}$ has constant rank $r_{[\mu^{diff}]}$.
   \item[\textrm{(2)}] The DAE has a solution to each arbitrary $q\in \mathcal C^{(m)}(\mathcal I,\Real^m)$ and the necessary solvability condition in Remark \ref{r.a1} is satisfied, that is,
   \[
    \rank [\mathcal E_{[k]} \,\mathcal F_{[k]}]= (k+1)m,\; k=0,\ldots, \mu^{diff}.
   \]
 \item[\textrm{(3)}] $\mathcal E_{[\mu^{diff}-1]}$ has constant rank $r_{[\mu^{diff}-1]}=r_{[\mu^{diff}]}-m$.
 \item[\textrm{(4)}] $S_{[\mu^{diff}-1]}=S_{[\mu^{diff}]}=S_{can}$.
 \item[\textrm{(5)}] $\ker E\cap S_{can}=\{0\}$.
 \end{description}
\end{proposition}
\begin{proof}
The issue{\textrm (1)} is already part of the definition.

{\textrm (2)}: The solvability assertion is evident and the necessary solvability condition is validated in  Remark \ref{r.a1}.

{\textrm (3)}: Owing to \cite[Lemma 3.6]{KuMe1996}  one has $\corank \mathcal E_{[\mu^{diff}]}=\corank \mathcal E_{[\mu^{diff-1}]}$ yielding $(\mu^{diff}+1)m-r_{[\mu^{diff}]}=\mu^{diff}m-r_{[\mu^{diff}-1]}$, and hence $r_{[\mu^{diff}-1]}=r_{[\mu^{diff}]}-m$.

{\textrm (4)}: Remark \ref{r.a2} provides the subspace dimensions $\dim S_{[\mu^{diff}-1]}=r_{[\mu^{diff}-1]}-(\mu^{diff}-1)m$ and $\dim S_{[\mu^{diff}]}=r_{[\mu^{diff}]}-\mu^{diff}m$. Regarding {\textrm (3)} this gives $\dim S_{[\mu^{diff}-1]}=\dim S_{[\mu^{diff}]}$. Due to the inclusion \eqref{1.subspace2} we arrive at $S_{[\mu^{diff}-1]}=S_{[\mu^{diff}]}$. It only remains to state that $S_{[\mu^{diff}]}=S_{can}$ by \cite[Theorem 2.4.8]{BCP89}.

{\textrm (5)}: This is a straightforward consequence of Assertion (4) and Lemma \ref{l.app}.
\end{proof}
\begin{theorem}\label{t.regulardiff}
 Let the pair $\{E,F\}$ and the DAE \eqref{1.DAE} be regular on $\mathcal I$ with index $\mu $ and characteristic values $r$ and $\theta_0\geq\cdots\geq \theta_{\mu-2}>\theta_{\mu-1}=0$. Then the
 differentiation index is well-defined, $\mu^{diff}=\mu$, and, additionally, the matrix functions $\mathcal E_{[k]}$ have constant ranks and the subspaces $S_{[k]}$ have constant dimensions.
\end{theorem}
\begin{proof}
 This is an immediate consequence of Proposition \ref{th.ranks}.
\end{proof}
In contrast to our basic index notion in Section \ref{s.regular}, the differentiation index allows for certain rank changes which is often particularly emphasized\footnote{There have been repeated scientific disputes about this.},e.g., \cite{BCP89,KuMe2006}.
In the special Examples \ref{e.1},  \ref{e.2},  \ref{e.7} in Section \ref{s.examples} below the rank of the leading matrix function  $E(t)$  changes,  nevertheless the differentiation index is well-defined. In Example \ref{e.5}  $E(t)$ has constant rank, but $r_{[1]}$ varies, but the DAE has differentiation index three on the entire given interval.
However,
it may well happen that a DAE having on $\mathcal I$ a well-defined differentiation index  features a different differentiation index on a subinterval. We consider the points where the rank changes to be \emph{critical points} for good reason.
\begin{remark}\label{r.generalform}
 The formation of the differentiation index approach is closely related to the search of a general form for solvable linear DAEs (in the sence of Definition \ref{d.solvableDAE}) with time-varying coefficients from the very beginning \cite{CamPet1983,Campbell87}. We quote \cite[Theorems 2.4.4 and 2.4.5]{BCP89} and a result from \cite{BergerIlchmann} for coefficients $E,F:\mathcal I\rightarrow \Real^m$.
 \begin{itemize}
  \item Suppose that $E, F$ are real analytic. Then \eqref{1.DAE} is solvable if and only if it is equivalent to a system in standard canonical (SCF) form \eqref{1.SCF} using real analytic coordinate changes\cite[Theorems 2.4.4]{BCP89}.
  \item Suppose that the DAE \eqref{1.DAE} is solvable on the compact interval $\mathcal I$. Then it is equivalent to the DAE in Campbell canonical form\footnote{This appreciatory name is introduced in \cite{KuMe2024}.}
 \begin{align*}
  \begin{bmatrix}
   I_d&G\\0&N
  \end{bmatrix}z'+
  \begin{bmatrix}
   0&0\\0&I_{m-d}
  \end{bmatrix}z=
  \begin{bmatrix}
   g\\h
  \end{bmatrix}
 \end{align*}
where  $Nz_2'+z_2=h$ has only one solution for each function $h$. Furthermore, there exists a countable family\footnote{As we understand it, this set is not necessarily countable, see Theorem \ref{t.M}.} of disjoint open intervals $\mathcal I^{\ell}$ such that $\cup \mathcal I^{\ell}$ is dense in $\mathcal I$ and on each $\mathcal I^{\ell}$, the system $Nz_2'+z_2=h$ is equivalent to one in standard canonical form  of the form $Mw'+w=f$ with $M$ structurally nilpotent\footnote{A square matrix $A$ is structurally nilpotent if and only if there is a permutation matrix $P$ such that $PAP^{-1}$ is strictly triangular, see \cite[Theorem 2.3.6]{BCP89}.} \cite[Theorems 2.4.5]{BCP89}.
\item Suppose an open interval $\mathcal I$. Then every system transferable into SCF with $\mathcal C^m$-coefficients is solvable.
 \end{itemize}
\end{remark}

Reviewing our examples in Section \ref{s.examples} we observe the following:
If the pair $\{E,F\}$ has on the interval $\mathcal I$ the differentiation index $\mu^{diff}$, then on each subinterval $\mathcal I_{sub}\in \mathcal I$ the differentiation index $\mu_{sub}^{diff}$ is also  well-defined, which can, however, be smaller than $\mu^{diff}$, which has an impact on the input-output behavior of the system.
Our next theorem captures the previous observations and generalizes them. 

Recall that we know from  Proposition \ref{p.index1} that a DAE having differentiation index one is regular with index one in the sense of Definition \ref{d.2} and vice versa. We are interested in what happens in the higher-index cases. The following assertion says that, for any DAE with well-defined differentiation index $\mu^{diff}$ on a compact interval $\mathcal I$, the subset of regular points $\mathcal I_{reg}$ is dense in $\mathcal I$ with uniform degree of freedom, but there might be subintervals on which the DAE features a strictly smaller differentiation index than $\mu^{diff}$.
\begin{theorem}\label{t.diffinterval}
 Let the pair $\{E,F\}$ and the DAE \eqref{1.DAE} be given on the compact interval $\mathcal I$ and have there the  differentiation index $\nu=\mu^{diff}\geq 2$.
 
 Then there is a partition of the interval $\mathcal I$ by a collection $\mathfrak S$ of open, non-overlapping subintervals\footnote{We apply Theorem \ref{t.M}  according to which the set of rank discontinuity points can also be over-countable. This is why we use the name \emph{collection} in contrast to a countable family.} such that
 \begin{align*}
  \overline{\bigcup_{\ell \in \mathfrak S}\mathcal I^{\ell}}=\mathcal I,\quad  \mathcal I^{\ell} \text{ open },\quad \mathcal I^{\ell_i}\cap \mathcal I^{\ell_j}=\emptyset \quad\text{for}\quad \ell_i\neq \ell_j,\quad \ell_i, \ell_j \in\mathfrak S,
 \end{align*}
 and the pair $\{E,F\}$ and the DAE \eqref{1.DAE} restricted to any subinterval $\mathcal I^{\ell}$ are regular in the sense of Definition \ref{d.2} with individual characteristics, 
 \[\mu^{\ell}\leq \mu^{diff},\; r^{\ell},\; \theta_0^{\ell}\geq \cdots \geq \theta^{\ell}_{\mu^{\ell}-2}>\theta^{\ell}_{\mu^{\ell}-1}=0,\quad 
\ell \in \mathfrak S,
\]
but necessarily with uniform degree of freedom $d$, which means 
 \begin{align*}
  d= d^{\ell}=r^{\ell}-\sum_{i=0}^{\mu^{\ell}-2}\theta_i^{\ell},\quad \ell \in\mathfrak S.
 \end{align*}
Furthermore, it holds that $\mu^{diff}=\max\{\mu^{\ell}: \ell \in\mathfrak S\}$.
\end{theorem}
\begin{proof}
Owing to \cite[Corollary 3.26]{KuMe2006} which is based on Theorem \ref{t.M} there is a decomposition of the compact interval $\mathcal I$ by open non-overlapping subintervals $\mathcal I^{\ell}$,  $\ell \in\mathfrak S$, such that the interval $\mathcal I$ is the closure of $\cup_{ \ell \in\mathfrak S}\mathcal I^{\ell}$, and the DAE has a  well-defined regular strangeness index  on each subinterval $\mathcal I^{\ell}$. In turn, by Theorem \ref{t.equivalence}, the DAE 
is regular on each subinterval $\mathcal I^{\ell}$  in the sense of Definition \ref{d.2}  with individual index $\mu^{\ell}$ and characteristics
\begin{align*}
 r^{\ell}, \quad\theta^{\ell}_0\geq \theta^{\ell}_1\geq\cdots\geq\theta^{\ell}_{\mu^{\ell}-2}>\theta^{\ell}_{\mu^{\ell}-1}=0, \quad d^{\ell}=r^{\ell}-\sum_{l=0}^{\mu^{\ell}-2}\theta^{\ell}_l.
\end{align*}
 As in Proposition \ref{th.ranks} we set $\theta^{\ell}_j=0$ for $j>\mu^{\ell}-1$.
Since the matrix functions $\mathcal E_{[\nu]}$ is pointwise $1$-full and  has constant rank $r_{[\nu]}$ on the overall $\mathcal I$, owing to Proposition \ref{th.ranks} we have $\mu_{\ell} \leq \nu$ on each subinterval $\mathcal I^{\ell}$  and 
\begin{align*}
 r_{[\nu-1]}&=(\nu-1)m+ r^{\ell}-\theta_0^{\ell} -\theta_1^{\ell}-\cdots-\theta_{\nu-2}^{\ell}-\theta_{\nu-1}^{\ell} = (\nu-1)m+d^{\ell},\\
 r_{[j]}&=j m+ d^{\ell},\quad j\geq \nu-1.
\end{align*}
Therefore, the values $d^{\ell}$ are equal on all subintervals, $d^{\ell}=d$.  
Denote $\kappa=\max\{\mu^{\ell}:{\ell}\in \mathfrak S\}$, $\kappa\leq\nu$, and observe that
$r_{[\kappa]}=\kappa m+d$ on each subinterval $\mathcal I^{\ell}$ , thus $r_{[\kappa]}\leq\kappa m+d$ on all $\mathcal I$.

Owing to Proposition \ref{p.diff}, on all $\mathcal I$ it holds that $S_{[\nu]}=S_{can}$, 
$\dim S_{[can]}=\dim S_{\nu}=r_{[\nu]}-\nu m=d$. The inclusion $S_{[\kappa]}(t)\supseteq S_{can}(t)$ which is given for all $t\in\mathcal I$, implies $r_{[\kappa]}-\kappa m \geq d$ on all $\mathcal I$. This leads to $r_{[\kappa]}=\kappa m+d$ on all $\mathcal I$ and  $S_{[\kappa]}=S_{can}$ as well. Finally, $\mathcal E_{[\kappa]}$ has constant rank on the whole intervall, and additionally, regarding again Proposition \ref{p.diff}, it results that $\ker E\cap S_{[\kappa]}=\ker E\cap S_{can}=\{0\}$.
This implies  $\kappa=\nu$, since $\nu$ is the smallest such integer.
\end{proof}
\begin{remark}\label{r.Chist}
To a large extend similar results are developed in the monographs \cite[Chapter 3]{Chist1996} and \cite[Chapter 2]{ChistShch} using operator theory. To the given DAE that is represented as operator equation of order one, $\Lambda_{1}x=Ex'+Fx=q$  a left regularization operator $\Lambda_{\nu}=A_{\nu}\frac{d^{\nu}}{dt^{\nu}}+\cdots +A_{1}\frac{d^{1}}{dt^{1}}+A_{0}$ is constructed, if possible, by evaluating  derivative arrays $\mathcal D_{{k}}$ as introduced in Subsection \ref{subs.Preliminaries} above and using generalized inverses, such that
$ \Lambda_{\nu}\circ\Lambda_{1}x=x'+ Bx$. The minimal possible number $\nu$ is called (\cite[p.\ 85]{Chist1996}) \emph{non-resolvedness index}, and this is the same as the differentiation index. In \cite{Chist1996,ChistShch} the SCF is renamed to central canonical form. Instead of solvable systems in the sense of Definition \ref{d.solvableDAE}, DAEs that have a \emph{general Cauchy-type solution} now form the background, see \cite[p.\ 110]{Chist1996}.
\end{remark}

\subsection{Regular differentiation index by geometric approaches}\label{subs.qdiff}
In concepts that assume certain continuous projector-valued functions, especially where geometric ideas play a role,  one finds a somewhat restricted or qualified by additional rank conditions index understanding. In \cite{Griep92}, based on the rank theorem, a modified version of the differentiation index is given, which is closely related to the differential-geometric concepts in \cite{Reich,RaRh}. Indeed, the presentation and index definition in \cite{Griep92} is a more analytical notation of the differential-geometric concept in \cite{Reich}, and this version fits well with the rest of our presentation. As before, we are dealing with linear DAEs.

Basically, the derivative array functions $\mathcal E_{[k]}$ introduced in Section \ref{subs.Preliminaries} are assumed to feature constant ranks $r_{[k]}$ for all $k$. Due to the rank theorem there are smooth pointwise nonsingular matrix functions $U_{[k]}, V_{[k]}:\mathcal I\rightarrow \Real^{(km+m)\times(km+m)}$ providing the factorization
\begin{align*}
 \mathcal E_{[k]}= U_{[k]}\bar{P}_{[k]}V_{[k]},\quad \bar{P}_{[k]}:=\diag( I_{r_{[k]}},0,\ldots, 0 )\in \Real^{(km+m)\times(km+m)}.
\end{align*}
Then, letting $\bar{Q}_{[k]}=I-\bar{P}_{[k]}$ we form the projector functions
\begin{align*}
 R_{[k]}&=U_{[k]}\bar{P}_{[k]}U_{[k]}^{-1}\quad \text{onto}\quad \im \mathcal E_{[k]},\\
 W_{[k]}&=U_{[k]}\bar{Q}_{[k]}U_{[k]}^{-1}\quad \text{along}\quad \im \mathcal E_{[k]},\\
 Q_{[k]}&=V_{[k]}^{-1}\bar{Q}_{[k]}V_{[k]}\quad \text{onto}\quad \ker\mathcal E_{[k]},\\
 P_{[k]}&=V_{[k]}^{-1}\bar{P}_{[k]}V_{[k]}\quad \text{along}\quad \ker\mathcal E_{[k]},
\end{align*}
and turn to the equation
\begin{align}\label{G1}
 \mathcal E_{[k]}x'_{[k]}+\mathcal F_{[k]}x=q_{[k]},
\end{align}
which is divided into the two parts,
\begin{align}
 \mathcal E_{[k]}x'_{[k]}+R_{[k]}\mathcal F_{[k]}x&=R_{[k]}q_{[k]},\label{G2}\\
 W_{[k]}\mathcal F_{[k]}x&=W_{[k]}q_{[k]}.\label{G3}
\end{align}
Applying the factorization one obtains the reformulation of \eqref{G2} to
\begin{align}\label{G4}
V_{[k]}^{-1}\bar{P}_{[k]}V_{[k]}x'_{[k]}+  V_{[k]}^{-1}\bar{P}_{[k]}U_{[k]}^{-1}(\mathcal F_{[k]}x-q_{[k]}) =0.
\end{align}
Regarding \eqref{G3} one has $\mathcal F_{[k]}x-q_{[k]}= U_{[k]}\bar{P}_{[k]}U_{[k]}^{-1} (\mathcal F_{[k]}x-q_{[k]})$, 
thus  $V_{[k]}^{-1}U_{[k]}^{-1} (\mathcal F_{[k]}x-q_{[k]})= V_{[k]}^{-1}\bar{P}_{[k]}U_{[k]}^{-1} (\mathcal F_{[k]}x-q_{[k]})$, 
and \eqref{G4} becomes 
\begin{align}
V_{[k]}^{-1}\bar{P}_{[k]}V_{[k]}x'_{[k]}= -  V_{[k]}^{-1}U_{[k]}^{-1}(\mathcal F_{[k]}x-q_{[k]}),\nonumber\\
x'_{[k]}=V_{[k]}^{-1}\bar{Q}_{[k]}V_{[k]}x'_{[k]}- V_{[k]}^{-1}U_{[k]}^{-1}(\mathcal F_{[k]}x-q_{[k]}). \label{G5}
\end{align}
Coming from \eqref{G1} we deal now with the equation
\begin{align}\label{G6}
 \mathcal E_{[k]}(t)y+\mathcal F_{[k]}(t)x=q_{[k]}(t),\quad t\in \mathcal I,
\end{align}
where $y\in\Real^{(k+1)m}$ and $x\in R^m$ are placeholders for $x'_{[k]}(t)$ and $x(t)$.

Denote by $\tilde C_{[k]}$ the so-called \emph{constraint manifold of order} $k$, which contains exactly all pairs $(t,x)$ for which equation \eqref{G6} is solvable with respect to $y$, that is
\begin{align*}
\tilde C_{[k]} &=\{(x,t)\in \Real^m\times\mathcal I: W_{[k]}(t)(\mathcal F_{[k]}(t)x-q_{[k]}(t))=0\}\\
&=\{(x,t)\in \Real^m\times\mathcal I: W_{[k]}(t)\mathcal F_{[k]}(t)x =W_{[k]}(t)q_{[k]}(t))\}\\
&=\{(x,t)\in \Real^m\times\mathcal I: x\in C_{[k]}(t)\},
\end{align*}
with $C_{[k]}(t)$ from \eqref{1.consitent1} which represent  the fibres at $t$ of the constraint manifold $\tilde C_{[k]}$. The inclusion chain
\begin{align*}
 \tilde C_{[0]}\supseteq \tilde C_{[1]}\supseteq\cdots \supseteq\tilde C_{[k]}
\end{align*}
is obviously valid.
For each $(x,t)\in \tilde C_{[k]}$ we form the manifold $M_{[k]}(x,t)\subseteq \Real^{(k+1)m}$
of all $y\in\Real^{m+km}$ solving the equation \eqref{G6}. Regarding the representation \eqref{G5} we know that $M_{[k]}(x,t)$ is an affine subspace parallel to $\ker \mathcal E_{[k]}(t)$ and it depends linearly on $x$:
\begin{align}
 M_{[k]}(x,t)&=\{y\in\Real^{m+km}: y=z- (U_{[k]}V_{[k]})^{-1}(t)(\mathcal F_{[k]}(t)x-q_{[k]}(t)), z\in\ker \mathcal E_{[k]}(t)\}\nonumber\\
 &=\ker \mathcal E_{[k]}(t)+ \{- (U_{[k]}V_{[k]})^{-1}(t)(\mathcal F_{[k]}(t)x-q_{[k]}(t))\}.\label{G7}
\end{align}
Using the truncation matrices
\begin{align*}
 \hat{T}_{[k]}&=[I_{km}\; 0]\in\Real^{km\times(m+km)},\\
 {T}_{[k]}&= \hat{T}_{[1]}\cdots \hat{T}_{[k]}=[I_{m}\; 0]\in\Real^{m\times(m+km)},
\end{align*}
the inclusions
\begin{align*}
 M_{[0]}(x,t)&\supseteq T_{[1]}M_{[1]}(x,t)\supseteq\cdots\supseteq T_{[k]}M_{[k]}(x,t),\\
 \ker E(t)= \ker \mathcal E_{[0]}(t)&\supseteq T_{[1]}\ker \mathcal E_{[1]}(t)\supseteq\cdots\supseteq T_{[k]}\ker \mathcal E_{[k]}(t),
\end{align*}
are provided in \cite{Griep92}. 
Each DAE solution proceeds within the constraint manifolds of order $k\geq 0$, and we have
\begin{align*}
 x(t)\in C_{[k]}(t),\quad x'_{[k]}(t)\in M_{[k]}(x(t),t),\quad t\in \mathcal I,\quad k\geq 0.
\end{align*}
The corresponding index definition from \cite[Section 3]{Griep92} reads:
\begin{definition}\label{d.G1}
The equation \eqref{1.DAE} is called a  DAE  with \emph{regular differentiation  index $\nu$} if all $\mathcal E_{[j]}$ feature constant ranks, $T_{[\nu]} M_{[\nu]}(x,t)$ is a singleton for all $(x,t)\in \tilde C_{[\nu]}$, and $\nu$ is the smallest integer with these properties.
We then indicate the regular differentiation index by $\nu=:\mu^{rdiff}$.
\end{definition}
From representation \eqref{G7} it follows that $T_{[\nu]} M_{[\nu]}(x,t)$ is a singleton exactly if $T_{[\nu]} \ker \mathcal E_{[\nu]}=\{0\}$, thus $T_{[\nu]}Q_{[\nu]}=0$.
\medskip

With the resulting vector field $v(x,t):=- T_{[\nu]}(U_{[\nu]}V_{[\nu]})^{-1}(t)(\mathcal F_{[\nu]}(t)x-q_{[\nu]}(t) $,
the DAE \eqref{1.DAE} having the regular differentiation  index $\nu$ may be seen as vector field on a manifold, that is,
\begin{align*}
   x'(t)=v(x(t),t),\quad (x(t),t)\in \tilde C_{[\nu]}.
\end{align*}

It must be added here  that in early works like \cite{Griep92,Reich} no special epithet was given to the index term. It was a matter of specifying the idea formulated in \cite{Gear88} that the index of a DAE is determined as the smallest number of differentiations necessary to filter out from the inflated system a well-defined explicit ODE.
In particular, in \cite{Griep92} there is only talk about an \emph{index-$\nu$ DAE}, without the epithet \emph{regular}, but in \cite{Reich} regularity is central and  the characterization of the DAE as \emph{regular} is particularly emphasized, and so we added here the label \emph{regular differentiation} to differ from other notions, specifically also from the differentiation index in Subsection \ref{subs.diff}. 
Based on closely related index concepts, some variants of index transformation\footnote{Also called index reduction.} are discussed  in \cite{Griep92,Reich}, i.e., for a given DAE, a new DAE with an index lower by one is constructed. 
We pick out the respective basic idea 
from \cite{Reich,Griep92} which is in turn closely related to the geometric reduction in \cite{RaRh}. 

Given is the  DAE \eqref{1.DAE} with a  pair $\{E,F\}$ featuring the regular differentiation index $\mu^{rdiff}=\nu$. Then $\{E,F\}$ is pre-regular with $r=r_{[0]}$ and $\theta= m+r-r_{[1]}$, see Lemma \ref{l.R1}.  Let $W$ and $P_S$ be the ortho-projector functions with $\ker W=\im E$ and $\im P_S=\ker WF=S$. We represent $P_{S}=I-(WF)^+WF$. 

Differentiating the derivative-free part $WFx-Wq=0$ leads to $(WF)'x-(Wq)'=-WFx'$, and in turn to
$x'=P_{S}x'+(WF)^+WFx'=P_{S}x'-(WF)^+((WF)'x-(Wq)')$. Inserting this into the DAE \eqref{1.DAE} yields
\begin{align*}
 EP_Sx'+(F-E(WF)^+(WF)'x=q-E(WF)^+(Wq)'.
\end{align*}
Regarding that $P_{S}'=-{(WF)^{+}}'WF-(WF)^+(WF)'$ and $WFx=Wq$ we arrive at 
\begin{align}\label{R1}
 EP_Sx'+(F+EP'_S)x=q-E((WF)^+Wq)'.
\end{align}
We quote \cite[Theorem 12]{Griep92}: The transfer from the DAE \eqref{1.DAE} to the DAE \eqref{R1} reduces the (regular differentiation) index by 1.

Next we show the close connection to the basic reduction step described in Section \ref{s.regular}. By means of a smooth basis $C$ of the subspace  $S$  we represent the above projector function $P_{S}$ by  $P_S=CC^+$, \\$C^+=(CC^*)^{-1}C^*$, and rewrite \eqref{R1} as 
\begin{align}\label{R2}
 EC(C^+x)'+(F+EC'C^+)x=q-E((WF)^+Wq)'.
\end{align}
Letting $y=C^+x$, so that $x=P_Sx=CC^+x=Cy$, we obtain
\begin{align*}
 ECy'+(FC+EC')y=q-E((WF)^+Wq)',
\end{align*}
and finally, using a basis $Y$ of $\im E$ as in Section \ref{s.regular},
\begin{align}\label{R3}
 Y^*ECy'+Y^*(FC+EC')y=Y^*(q-E((WF)^+Wq)'),
\end{align}
which illuminates the consistency of \eqref{R1} with the basic reduction step in \cite{RaRh} and  Section \ref{s.regular}.
%

\begin{theorem}\label{t.rdiff} The following assertions are valid:
\begin{description}
 \item[\textrm{(1)}] If the DAE \eqref{1.DAE} is regular with index $\mu\geq 1$ in the sense of Definition \ref{d.2} then it has also the  regular differentiation index $\mu^{rdiff}=\mu$, and vice versa, and the characteristic values are related by
 \begin{align*}
  r_{[0]}&= r,\\
 r_{[1]}&=m+ r-\theta_0,\\
 r_{[2]}&=2m+ r-\theta_0 -\theta_1,\\
 \cdots\\
 r_{[\nu-2]}&=(\nu-2)m+ r-\theta_0 -\theta_1-\ldots-\theta_{\nu-3},\\
 r_{[\nu-1]}&=(\nu-1)m+ r-\theta_0-\theta_1-\ldots-\theta_{\nu-2},\\
 r_{[j]}&=j m+ d,\quad j\geq \mu-1.
 \end{align*}
 \item[\textrm{(2)}] If the DAE has regular differentiation index $\mu^{rdiff}$ then it
 has also the differentiation index $\mu^{diff}=\mu^{rdiff}$.
 \item[\textrm{(3)}] If the DAE has regular differentiation index $\mu^{rdiff}$ on the interval $\mathcal I$ then, on each subinterval $\mathcal I_{sub}\subset\mathcal I$,  it shows the same 
 regular differentiation index $\mu^{rdiff}$.
\end{description} 
\end{theorem}
\begin{proof}
{\textrm (1):} Owing to Proposition \ref{p.index1} there is nothing to do in the index-$1$ case.  If the DAE is regular with index $\mu\geq 2$ in the sense of Definition \ref{d.2} then it has regular differentiation index $\mu^{rdiff}=\mu$ as an immediate consequence of Proposition \ref{th.ranks} and Lemma \ref{l.app}.

Contrariwise, let the DAE \eqref{1.DAE} have regular differentiation index $\mu^{rdiff}=\nu\geq 2$. Then the pair $\{E,F\}$ is pre-regular, and by \cite[Theorem 12]{Griep92} the DAE 
\begin{align}\label{R4}
 EP_Sx'+(F+EP'_S)x=p,\quad p:=q-E((WF)^+Wq)',
\end{align}
has regular differentiation index $\nu-1$.
Using smooth bases $Y, Z, C$, and $D$ of $\im E, (\im E)^{\perp}, \ker Z^*F$, and $ (\ker Z^*F)^{\perp}$, respectively,
we form the pointwise nonsingular matrix functions
\begin{align*}
 K=\begin{bmatrix}
    C&D
   \end{bmatrix}, \; L=\begin{bmatrix}
    Y^*\\(Z^*FD)^{-1}Z^*
   \end{bmatrix},
\end{align*}
scale the DAE \eqref{R4} by $L$ and transform $x=K\tilde x=: C\tilde x_C+D\tilde x_D$, which leads to an equivalent DAE of the form $\tilde E\tilde x'+\tilde F\tilde x=Lp$ with coefficients
\begin{align*}
\tilde E&=LEP_{S}K=\begin{bmatrix}
    Y^*EC&0\\0&0
   \end{bmatrix},\\
   \tilde F&=L(F+EP_{S}')K+LEP_{S}K'=LFK+LE(P_{S}K)'=\begin{bmatrix}
    Y^*FC+Y^*EC'&Y^*FD\\0&I
   \end{bmatrix}.
\end{align*}
The resulting DAE reads in detail
\begin{align}
 Y^*EC\tilde x'_C+(Y^*FC+Y^*EC')\tilde x_C +Y^*FD\tilde x_D&=Y^*p, \label{R5}\\
 \tilde x_D&=(Z^*FD)^{-1}Z^*q.\label{R6}
\end{align}
As a DAE featuring  regular differentiation index $\nu-1$, the DAE \eqref{R5}, \eqref{R6} is pre-regular. 
Observe that
\begin{align*}
 \ker \tilde E\cap\ker \tilde W\tilde F=\left\{\begin{bmatrix}
                                           u\\v
                                          \end{bmatrix}\in\Real^m: u\in \ker Y^*EC, (Y^*FC+Y^*EC')u\in \im Y^*EC, v=0\right\},
\end{align*}
which allows to restrict the further investigation to the inherent part
\begin{align}\label{R7}
 \underbrace{Y^*EC}_{=E_1}\tilde x'_C+\underbrace{(Y^*FC+Y^*EC')}_{=F_1}\tilde x_C &=Y^*p-Y^*FD (Z^*FD)^{-1}Z^*q, 
\end{align}
which has also regular differentiation index $\nu-1$, and $\dim \ker E_1\cap \ker Z^*_1F_1=  \dim  \ker \tilde E\cap\ker \tilde W\tilde F$.
If $\nu=2$ we are done.
If $\nu>2$ then we repeat the whole procedure and provide this way a basic sequence of pairs $\{E_k,F_k\}$.

{\textrm (2):} Lemma \ref{l.app} makes this evident.

{\textrm (3):} This is given by the construction.
\end{proof}
We observe that the regular differentiation index is well-defined, if and only if the (standard) differentiation index is well-defined, and, additionally, all preceding $\mathcal E_{[j]}$ have constant ranks. 
\begin{remark}\label{r.diff}
If $\{E, F\}$ has differentiation index $\mu^{diff}=:\nu$ the matrix function  $\mathcal E_{[\nu]}$ has constant rank, and, due to  Lemma \ref{l.app}, it holds that $T_{[\nu]}Q_{[\nu]}=0$ such that the above formula \eqref{G5} immediately provides an underlying  ODE in the form
\begin{align*}
 x'= T_{[\nu]}x'_{[\nu]}=- T_{[\nu]}V_{[\nu]}^{-1}U_{[\nu]}^{-1}(\mathcal F_{[\nu]}x-q_{[\nu]}),
\end{align*}
without the predecessors $\mathcal E_{[j]}$ having to have constant rank and without the background of geometric reduction. 
\end{remark}
%
 
\subsection{Projector based differentiation index for initialization}\label{subs.pbdiff}

The index concept developed in \cite{EstLam2016Decoupling,EstLamNewApproach2018,EstLamDecoupling2020} 
 has its origin in the computation of consistent initial values and was initially intended as a reinterpretation of the differentiation index.
Although it has therefore not yet had an own name,
in this article we will denote this index concept the \textit{projector based differentiation index}.  Roughly speaking the projector based differentiation index is reached, as soon as by differentiation we have found sufficient (hidden) constraints. For an appropriate description, orthogonal projectors are used to decouple different components of $x$.
\medskip

Let $P=E^{+}E$, $Q=I-P$, and $W_{0}=I-EE^{+}$ be the orthoprojector functions onto $(\ker E)^{\bot}$, $\ker E$, and $(\im E)^{\bot}$. Given that $E(t)$ has constant rank on $\mathcal I$, these projector functions are as smooth as $E$ is itself. Then we decompose the unknown $x=Px+Qx$ and rewrite the DAE as proposed in \cite{GM86},
\begin{align}\label{1.modDAE}
 E(Px)'+(F-EP')x=q.
\end{align}
All solutions of the homogeneous DAE with $q=0$ reside within 
the timevarying subspace of $\Real^{m}$
\begin{align}\label{1.S0}
 S_{0}=\{z\in \Real^{m}:Fz\in \im E\}=\ker W_{0}F =S_{[0]}.
\end{align}
The  DAE \eqref{1.modDAE}  splits into the following two equations:
\begin{align}
 P(Px)'+E^{+}(F-EP')(Px+Qx)&=E^{+}q,\label{1.ODE}\\
 W_{0}FQx&=-W_{0}FPx+W_{0}q.\label{1.alg}
\end{align}
Obviously, if the second equation \eqref{1.alg} uniquely determines $Qx$ in terms of $Px$ and $q$, then replacing $Qx$ in \eqref{1.ODE} by the expression resulting from \eqref{1.alg} yields an explicit ODE for $Px$. This actually happens on condition that $S_{0}\cap\ker E=\{0\}$ is given, 
which indicates regular index-1 DAEs as it is well-known \cite{GM86,CRR}. For higher-index DAEs this condition  is no longer met.
In the context of the \textit{projector based differentiation index} one aims for extracting the needed information concerning $Qx$ from the inflated system.

Thereby, the further matrix functions $\mathcal B_{[k]}:\mathcal I\rightarrow \Real^{(mk+m)\times(mk+m)}$,
\begin{align}\label{1.Bk}
\mathcal{B}_{[k]} =
\begin{bmatrix}
	  P & 0 \\
	\mathcal F_{[k-1]}& \mathcal E_{[k-1]}
\end{bmatrix}.
\end{align}
plays its role. For $k\geq 1$ we evaluate the inflated system 
\begin{align*}
 \mathcal E_{[k-1]}x'_{[k-1]}+ \mathcal F_{[k-1]}x= q_{[k-1]}.
\end{align*}

Introducing the variable $\omega=Qx$ and decomposing $x=Qx+Px=\omega+Px$ yields the system
\begin{align*}
P\omega&=0,\\
 \mathcal E_{[k-1]} x'_{[k-1]}+\mathcal F_{[k-1]}\omega &= q_{[k-1]}-\mathcal F_{[k-1]}Px,
\end{align*}
that is,
\begin{align}\label{1.B}
 \mathcal B_{[k]}\begin{bmatrix}
                   \omega\\x'_{[k-1]}
                  \end{bmatrix}=
\begin{bmatrix}
                   0\\q_{[k-1]}-\mathcal F_{[k-1]}Px
                  \end{bmatrix}.
\end{align}
If $\mathcal B_{[k]}$ is smoothly $1$-full,
\begin{align*}
 \mathcal T_{\mathcal B} \mathcal B_{[k]}=\begin{bmatrix}
                                 I&\\
                                 0&H_{\mathcal B}
                                \end{bmatrix},
\end{align*}
then the first block-row of system \eqref{1.B} multiplied by $\mathcal T_{\mathfrak B}$ reads
\begin{align*}
 \omega=\left(\mathcal T_{\mathcal B}\begin{bmatrix}
                   0\\q_{[k-1]}-\mathcal F_{[k-1]}Px
                  \end{bmatrix} \right)_{1},
\end{align*}
which is actually a representation of $Qx=\omega$ in terms of $Px$ and $q_{[k-1]}$ we are looking for. Inserting this expression into equation \eqref{1.ODE} leads to the explicit ODE for the component $Px$,
\begin{align*}
 (Px)'-P'Px+E^{+}(F-EP')(Px+\omega)=E^{+}q,
\end{align*}
and, supposing a consistent initial value for $Px$, eventually to a solution  $x=Px+Qx$ of the DAE.

\begin{remark}
Recall that the regular differentiation index focuses on the 1-fullness condition of $\mathcal E_{[k]}$ for its first $m$ rows corresponding to $x'$. In contrast, we want to emphasize that the projector based differentiation index focuses on the on the 1-fullness condition of $\mathcal B_{[k]}$ for its first $m$ rows, that correspond to $x$. 
\end{remark}

To get an idea about the rank of $\mathcal B_{[k]}(t)$ we point out its  connection to the matrix $\mathcal D_{[k]}(t)$ defined in \eqref{1.Dk}: 
\begin{align}
 \ker \mathcal B_{[k]}(t) &= \ker \mathcal D_{[k]} \nonumber \\
&=\left\{\begin{bmatrix}
            z\\w
     \end{bmatrix}\in \Real^{m}\times\Real^{mk}: z\in\ker E\cap S_{[k-1]}, \mathcal E_{[k-1]}^{+}\mathcal E_{[k-1]}w=-\mathcal E_{[k-1]}^{+}\mathcal F_{[k-1]}z
\right\},\label{1.kerBk}
\end{align}
for  the subspace $S_{[k-1]}$ defined in \eqref{1.subspace1}, and consequently,
\begin{align}
 \rank \mathcal B_{[k]}&= m-\dim(\ker E\cap S_{[k-1]})+r_{[k-1]} . \label{1.rankBk}
\end{align}
In addition, regarding the inclusions \eqref{1.subspace2}, 
we recognize immediately the inclusions
\begin{align*}
 S_{0}\cap\ker E=S_{[0]}\cap\ker E\supseteq S_{[1]}\cap\ker E\supseteq\cdots\supseteq S_{[k-1]}\cap\ker E\supseteq S_{[k]}\cap\ker E,
 \end{align*}
that for the projector $\mathcal W_{[k]}$ from \eqref{eq:WGR} and
\begin{align*}
 \rho_{k}:=\rank\begin{bmatrix}
                         P\\\mathcal W_{[k]}\mathcal F_{[k]}
                        \end{bmatrix}=\rank\begin{bmatrix}
                         E\\\mathcal W_{[k]}\mathcal F_{[k]}
                        \end{bmatrix}=
                       m-\dim (\ker E\cap S_{[k]})
\end{align*}
obviously yield the inequalities
 \begin{align*}
\rho_{0}\leq \rho_{1}\leq\cdots\leq \rho_{k-1}\leq\rho_{k}\leq m.
\end{align*}

\begin{definition}[\cite{EstLam2016Decoupling}]\label{d.pbdiffA}
 If there is a number $\nu\in\Natu$ such that the matrix functions $\mathcal E_{[0]},\ldots,\mathcal E_{[\nu-1]}$  have constant ranks, the rank functions $\rho_{0},\ldots,\rho_{\nu-1}$ are constant, too, and
 \begin{align*}
  \rho_{\nu-2}<\rho_{\nu-1}=m,
 \end{align*}
then $\nu$
 is called the \emph{projector based differentiation index} of the pair $\{E,F\}$ and the DAE \eqref{1.DAE}, respectively. We indicate it by $\nu=:\mu^{pbdiff}$.
\end{definition}

Having in mind the known index-1 criterion $S_{[0]}\cap\ker E=\{0\}$ 
we recognize  the condition $S_{[\mu^{pbdiff}-1]}\cap\ker E=\{0\}$ to characterize the projector based differentiation index in general.

If the index $\mu^{pbdiff}$ is well-defined, then owing to Lemma \ref{l.app}, the matrix function  $\mathcal B_{[\mu^{pbdiff}]}$ is smoothly $1$-full and has constant rank $r_{\mathcal B}=m+r_{[\mu^{pbdiff}-1]}$. Additionally, then $E$ and $\mathcal B_{[i]}$, $i=1,\ldots,\mu^{pbdiff}$, have  constant ranks, too.
Conversely, if  there is a $\nu\in\Natu$ such that $E$ and $\mathcal B_{[i]}$, $i=1,\ldots,\nu$, have  constant ranks, and $\mathcal B_{[\nu]}$ is $1-$full, and $\nu$ is the smallest such integer, then the rank functions 
$\rho_i=\rank \mathcal B_{[i+1]}-r_{[i]}$, $i=0,\ldots,\nu-1$, are constant, and $\rho_{\nu-2}<\rho_{\nu-1}=m$.

This makes it obvious that
the following alternative definition, that  is equivalent to Definition \ref{d.pbdiffA}, may also be considered:

\begin{definition}\label{d.pbdiffB}
 If there is a $\nu\in\Natu$  such that the matrix functions $E$,  $\mathcal B_{[1]},\ldots,\mathcal B_{[\nu]}$  have constant ranks $r^{\mathcal B}_{[0]}:=r, r^{\mathcal B}_{[1]},\ldots,r^{\mathcal B}_{[\nu]}$, respectively, and $ \nu$ is the smallest number  for which the matrix
function $\mathcal B_{[\nu]}$ is smoothly 1-full, 
then $\nu$
 is called the \emph{projector based differentiation index} of the pair $\{E,F\}$ and the DAE \eqref{1.DAE}, respectively. Again, we use the notation $\nu=:\mu^{pbdiff}$.
\end{definition}
Regarding again Lemma \ref{l.app}, this means precisely that
\begin{align*}
 r^{\mathcal B}_{[i]}< r_{[i-1]}+m, \; i=1,\ldots, \nu-1,\;  r^{\mathcal B}_{[\nu]}= r_{[\nu-1]}+m.
\end{align*}


\begin{remark}
With regard to the computation of consistent initial values, if the index is $\mu^{pbdiff}$, then for $k=\mu^{pbdiff}$ according to \cite{EstLamNewApproach2018} a uniquely determined consistent initial value $x_0 \in S_{[k-1]}(t_0)$ can be computed as solution of
\begin{eqnarray*}
 \mbox{ minimize } && \left\| P(t_0) (x_0 - \alpha)\right\|_2 \\
\mbox{ subject to } && \mathcal W_{[k-1]} \mathcal F_{[k-1]}(t_0)x_0  =   \mathcal W_{[k-1]}  q_{[k-1]}(t_0),
\label{eq:RegSysConsInit}
\end{eqnarray*}
for a given guess $\alpha$. This solution can also be computed as solution $x_0$ of the optimization problem
\begin{eqnarray*}
\mbox{ minimize } && \left\| P(t_0) (x_0 - \alpha)\right\|_2 \\
\mbox{ subject to } &&\mathcal F_{[k-1]}(t_0)x_0 + \mathcal E_{[k-1]}(t_0)w =    q_{[k-1]}(t_0)
\label{eq:OptimizationConsInit}
\end{eqnarray*}
 for  a vector $w \in \Real^{km}$ that is not uniquely determined, cf. \eqref{1.kerBk} and the results for the solvability from \cite{EstLamNewApproach2018}. There, the convenience of considering the orthogonal projector $P$ instead of the matrix $E$ for the objective function is discussed, that led to the consideration of $\mathcal B_{[k]}$ instead of $\mathcal D_{[k]}$. Moreover, for $k> \mu^{pbdiff}$, the last optimization problem permits the additional computation of consistent Taylor coefficients as parts of $w$.
\end{remark}

\begin{proposition}\label{p.pbdiff}
The projector based differentiation index remains invariant under sufficiently smooth equivalence transformations.
\end{proposition}
\begin{proof}
 Owing to Proposition \ref{p.equivalenc} and \eqref{1.rankBk}, the rank functions $\rank \mathcal B_{[k]}$ and $\rho_{k}$ are invariant under sufficiently smooth equivalence transformations, which makes the assertion evident.
\end{proof}

We emphasize again that the \textit{ projector based differentiation index} focuses on a  1-full condition for a different matrix than the \textit{regular differentiation index}. 
 However, if the pair $\{E,F\}$ is regular on the interval $\mathcal I$, we can show that they turn out to be equivalent.

\begin{theorem}\label{t.projector_based_diff}
 Let the pair $\{E,F\}$ be regular on $\mathcal I$ with index $\mu $ and characteristic values $r$ and \\ $\theta_0\geq\cdots\geq \theta_{\mu-2}>\theta_{\mu-1}=0$. 
  Then the projector based
 differentiation index is well-defined and coincides with the regular differentiation index i.e. $\mu^{rdiff}=\mu=\mu^{pbdiff}$. Moreover, 
\[
\rank \mathcal B_{[k]} = \rank \mathcal D_{[k]} = \rank \mathcal E_{[k]}=km+r-\sum_{i=0}^{k-1}\theta_i,\quad  k\geq 1,
\]
 \begin{align*}
  \rho_k&=m-\dim (\ker E\cap S_{[k]})=m-\theta_k,\quad k=0,\ldots ,\mu-1, 
	\end{align*}
and in particular
\begin{align*}
	\rho_{\mu-1}&=m, \quad \theta_{\mu-1} = 0, \quad  \dim (\ker E\cap S_{[\mu-1]})=\left\{0\right\}.
 \end{align*}
\end{theorem}

\begin{proof}
This follows directly from $\ker \mathcal B_{[k]} = \ker \mathcal D_{[k]}$, the definition of $\rho_k$ and Proposition \ref{th.ranks}.
\end{proof}

A closer look onto the matrix
\begin{eqnarray}
\begin{bmatrix}
	{\mathcal F}_{[k-1]}   & {\mathcal E}_{[k-1]}
\end{bmatrix}
\label{eq:DefGk}
\end{eqnarray}
and  the orthogonal projectors $Q$ and $P$ permits an orthogonal  decoupling of the different components of $x$ with further orthogonal projectors\footnote{That is why in this paper we have chosen the label \textit{projector based differentiation index} for this DAE approach.}. We briefly summarize these results from \cite{EstLam2016Decoupling} and \cite{EstLamDecoupling2020}.

\begin{itemize}

\item To decouple the $Q$-component for $k=1, \ldots, \mu$, 
 the
projector $T_{k}$ is defined as the orthogonal projector onto
\[
\ker E \cap S_{[k-1]} =
\ker \begin{bmatrix}
	  P  \\
	\mathcal W_{[k-1]} \mathcal F_{[k-1]}
\end{bmatrix} =: \im T_{k}.
\]
Consequently, $T_{k}x$ corresponds to the part of the $Q$-component that,  after $k$-1 differentiations, cannot yet be represented as a function of $(Px,t)$.

\item To characterize the different parts of the $P$-component, 
 the  matrix $\mathcal F_{[k-1]}$ is further splitted into $\mathcal F_{[k-1]}  P$ and $\mathcal F_{[k-1]}Q$, such that 
\[
\begin{bmatrix} 
\mathcal F_{[k-1]}  P && \mathcal F_{[k-1]}  Q && \mathcal E_{[k-1]}
\end{bmatrix}
\]
is considered instead of \eqref{eq:DefGk}.
With this decoupling, the orthogonal projector $\mathcal V_{[k]}$ with
\[
\ker \mathcal V_{[k-1]}= \im \begin{bmatrix} \mathcal F_{[k-1]} Q &\quad & \mathcal E_{[k-1]}\end{bmatrix}
\]
is defined, permitting finally to define the orthogonal projector $V_k$ onto
\[
\ker \begin{bmatrix}
	  Q  \\
	 \mathcal V_{[k-1]} \mathcal F_{[k-1]}
\end{bmatrix} =: \im V_k.
\]
By definition, $V_k x$ represents the part of P-component that is not determined by the constraints resulting after $k$-1 differentiations, such that $d=\rank V_{\mu} = \rank V_{\mu-1} $ holds. 

\end{itemize}

To determine the rank of $T_k$ and $V_k$ we will use the fact that $\rank \mathcal W_{[k-1]} \mathcal F_{[k-1]} $  is the number of explicit and hidden constraints resulting after $k-1$ differentiations, and that with \eqref{eq:WRGL}, \eqref{eq:Represrank} and Theorem \ref{th.ranks} it holds
\begin{align}
\rank \mathcal W_{[k-1]} \mathcal F_{[k-1]}  = \rank \mathcal W_{[k-1]} =m -\dim S_{[k-1]} = m-r + \sum_{i=0}^{k-2} \theta_i. \label{eq:rankW_[k-1]}
\end{align}

\begin{proposition} \label{rankTkVk}
For every regular pair $\{E, F\}$ on $\mathcal I$ with index $\mu$ it holds
\[
\rank T_{k} = \theta_{k-1}, \quad \rank V_k = r-\sum_{i=0}^{k-1} \theta_i.
\] 
\end{proposition}
\begin{proof} 
For $T_k$, the assertions follows directly from the definition. For $ \rank V_k$, we use \eqref{eq:rankW_[k-1]} to obtain
\begin{align*}
 r- \sum_{i=0}^{k-2} \theta_i &= \dim \ker \mathcal W_{[k-1]} \mathcal F_{[k-1]} =  \dim \ker \begin{bmatrix}
I_m&  0 \\
0	 & \mathcal W_{[k-1]} \mathcal F_{[k-1]}
 \end{bmatrix} \begin{bmatrix}
	 Q & P \\
	P & Q
 \end{bmatrix} \\
&= \dim \ker  \begin{bmatrix}
	 Q & P \\
\mathcal	W_{[k-1]} \mathcal F_{[k-1]}P  & \mathcal W_{[k-1]} \mathcal F_{[k-1]}Q 
 \end{bmatrix}
\end{align*}
and have a closer look to the nullspace of the last matrix
\begin{align*}
& \left\{\begin{bmatrix}
                          z_1\\z_2
                         \end{bmatrix}\in \Real^{2m}: \quad Qz_1=0, \quad  Pz_2=0, \quad  \mathcal W_{[k-1]} \mathcal F_{[k-1]}Pz_1 +\mathcal W_{[k-1]} \mathcal F_{[k-1]}Qz_2=0
\right\}\nonumber\\
&= \left\{\begin{bmatrix}
                          z_1\\z_2
                         \end{bmatrix}\in \Real^{2m}: \quad \begin{matrix}
												 z_1 \in \ker Q \cap \ker \mathcal V_{[k-1]} \mathcal F_{[k-1]}, \\
												 Pz_2=0, \quad \mathcal W_{[k-1]} \mathcal F_{[k-1]}z_2 =- \mathcal W_{[k-1]} \mathcal F_{[k-1]}Pz_1
												\end{matrix}
\right\} .
\end{align*}
Consequently, it holds
\[
\dim \ker \mathcal W_{[k-1]} \mathcal F_{[k-1]} = \dim \left\{\begin{bmatrix}
                          z_1\\z_2
                         \end{bmatrix}\in \Real^{2m}: \quad \begin{matrix}
												 z_1 \in \ker Q \cap \ker \mathcal V_{[k-1]} \mathcal F_{[k-1]} = \im V_k, \\
												 z_2 \in \ker P \cap  \ker \mathcal W_{[k-1]} \mathcal F_{[k-1]}=\im T_k												\end{matrix}
\right\},
\]
leading to
\[
r- \sum_{i=0}^{k-2} \theta_i = \rank V_k + \rank T_k, 
\]
i.e.
\[
\rank V_k = r- \sum_{i=0}^{k-2} \theta_i -  \rank T_k =r- \sum_{i=0}^{k-1} \theta_i.
\]
\end{proof}

This means that the $m-r + \sum_{i=0}^{k-2} \theta_i$ linearly independent constraints from
\[
\mathcal W_{[k-1]} (t)\mathcal F_{[k-1]}(t)x  =   \mathcal W_{[k-1]}(t)  q_{[k-1]}(t)
\]
uniquely determine $(I-T_k-V_k)x$ as a function of $(V_kx,t)$.  \\

For the orthogonal projector  $\Pi:= V_{\mu}= V_{\mu-1} $
with $\rank \Pi = d = r-\sum_{i=0}^{\mu-2} \theta_i $,  $\Pi x$ represents the part of $P$-components that is not determined by the constraints and can be used to formulate an orthogonally projected explicit ODE
\begin{eqnarray}
(\Pi x)'-\Pi'(\Pi x) + \Pi C(t) (\Pi x) = \Pi c(t)
\label{eq:Proj_Pi_ODELinDAE}
\end{eqnarray}
for suitable $C(t)$ and $c(t)$, cf.\ \cite{EstLamDecoupling2020}. The remaining components $(I-\Pi)x$ can then be computed accordingly with the constraints. Note that $\Pi$ is not orthogonal to $S_{[\mu-1]}$ in general, since $\mathcal V_{[\mu-1]}$ does not coincide with $\mathcal W_{[\mu-1]}$ in general.
\medskip

Note further that, by definition,  $T_k=QT_k=T_kQ$ as well as
\[
T_{k+1}T_k = T_k T_{k+1}= T_{k+1}, \quad V_{k+1}V_k = V_k V_{k+1}= V_{k+1}, \quad T_{k_1} V_{k_2} =0
\]
holds, cf.\ \cite{EstLamDecoupling2020}. Therefore $(V_k - V_{k+1})$ is an orthogonal projector as well, fulfilling $\rank (V_k - V_{k+1}) = \theta_k$ and $(V_k - V_{k+1})=P(V_k - V_{k+1})=(V_k - V_{k+1})P$. 

\begin{remark}\label{r.mod1}
It is opportune to mention that $\rank\mathcal B_{[k]}$ serves as proven monitor for indicating singular points by means of the  algorithms from \cite{EstLamInitDAE}. In  \cite{ELMRoboticArm2020} several simple examples are discussed and in  \cite{EstLam2017Discovering} the well-known nonlinear benchmark robotic arm is analyzed in detail.\\
Indeed, for many applications, the projector $T_k$ is constant. If this is not the case, the changes in $T_{k}$  may provide an indication of which entries of $E$ or $F$ lead to a change of $\theta_{k-1}$. Comparing the obtained projectors at a regular and a singular point, critical parameter combinations or model errors may be identified, cf. Example \ref{e.5}, Example \ref{e.robotic_arm} and \cite{ELMRoboticArm2020}.
\end{remark}

\subsection{Strangeness index via  derivative array}\label{subs.Hyp}
DAEs with differentiation index zero are (possibly implicit) regular ODEs,  they are well-understood and of no interest in our context here. Further, a DAE showing differentiation index one is a priori\footnote{See Proposition \ref{p.index1}} a regular DAE with index $\mu=1$ in the sense of Definition \ref{d.2}, and  it is rather unreasonable\footnote{In view of different properties such as stability behavior and numerical handleability.} here to change to an  underlying ODE. So the question arises as to whether one should even look for an index-1 DAE instead of a regular ODE in general.
 \medskip
 
 The aim is now to filter out a regular index-one DAE,  more precisely, a strangeness-free DAE from an inflated system instead of the underlying ODE in the context of the differentiation-index.
Again we consider the DAE
\begin{align*}
 Ex'+Fx=q
\end{align*}
and try to find an associated new DAE with the same unknown function $x$ in the partitioned form 
\begin{align}
 \hat E_1x'+\hat F_1x&=\hat q_1,\label{5.5.1}\\
 \hat F_2x&=\hat q_2,\label{5.5.2}
\end{align}
 which is strangeness-free resp.\ regular with index zero or index one in the sense of Definition \ref{d.2}.
 We assume the given DAE to have a well-defined differentiation index, say $\nu:=\mu^{diff}$. By means of Proposition \ref{p.diff} we obtain the constant numbers $d=r_{[\nu]}-\nu m=\dim S_{can}$ and $a=m-d$ such that
 \begin{align*}
 r_{[\nu-1]}=\rank \mathcal E_{[\nu-1]}&=(\nu-1) m+d= \nu m-a,\\
 \dim\ker \mathcal E_{[\nu-1]}&=a,\\
 \im [ \mathcal E_{[\nu-1]}  \mathcal F_{[\nu-1]}]&=\Real^{\nu m}.
 \end{align*}
Then we form a smooth full-column-rank function $Z:\mathcal I\rightarrow \Real^{\nu m\times a}$ such that $Z^*\mathcal E_{[\nu-1]}=0$ and thus $\ker Z^*\mathcal F_{[\nu-1]}= S_{[\nu-1]}$ has constant dimension $m-a=d$. Recall that $S_{[\nu-1]}=S_{can}$. 
Let $\mathcal C:\mathcal I\rightarrow \Real^{m\times a}$ define a smooth basis of the subspace $S_{[\nu-1]}$, so that $Z^*\mathcal F_{[\nu-1]}C=0$. The matrix function $EC$ has full column-rank $d$ due to Proposition \ref{p.diff}. Therefore, 
with any matrix function $Y:\mathcal I\rightarrow\Real^{m\times a}$ forming a basis of $\im EC$ we obtain a nonsingular product $Y^*EC$.

Letting in \eqref{5.5.1}, \eqref{5.5.2}
\begin{align*}
 \hat E_1=Y^*E,\quad &\hat F_1=Y^*F, \quad \hat q_1=Y^*q,\\
 &\hat F_2=Z^*\mathcal F_{[\nu-1]}, \quad \hat q_2=Z^*q_{[\nu-1]},
\end{align*}
we receive  $\hat{\theta}_0=\dim (\ker \hat E_1\cap \ker \hat F_2) =0$ since
\begin{align*}
 \ker \hat E_1\cap \ker \hat F_2&=\{z\in\Real^m: Y^*Ez=0, z\in S_{[\nu-1]}\}\\
 &=\{z\in\Real^m: z=Cw, Y^*ECw=0\}=\{0\},
\end{align*}
and hence we are done.

This matter was developed as part of the index concept and formally tied into a so-called hypothesis. We quote \cite[Hypothesis 3.48]{KuMe2006} in a form adapted to our notation.

\begin{hypothesis}[\textbf{Strangeness-Free-Hypothesis (SF-Hypothesis)}]\label{SHyp}

Given are sufficiently smooth matrix functions $E,F:\mathcal I\rightarrow\Real^{m\times m}$.

There exist integers $\hat\mu, \hat a,$ and $\hat d=m-\hat a$ such that the inflated pair $\{\mathcal E_{[\hat\mu]},\mathcal F_{[\hat\mu]}\}$ associated with the given pair $\{E, F\}$ has the following properties:
\begin{description}
 \item[\textrm{(1)}] $\rank \mathcal E_{[\hat\mu]}(t)=(\hat\mu+1)m-\hat a,\; t\in \mathcal I$,  
 such that there is a smooth matrix function $Z:\mathcal I\rightarrow\Real^{(\hat\mu m+m)\times\hat a}$ 
 with full column-rank  $\hat a$  
 on $\mathcal I$ 
 and $Z^*\mathcal E_{[\hat\mu]}=0$.
\item[\textrm{(2)}] For $\hat F_2:= Z^* \mathcal F_{[\hat\mu]}$ one has $\rank \hat F_2(t)=\hat a, t\in \mathcal I$, such that there is a smooth matrix function \\
$C:\mathcal I\rightarrow\Real^{m\times\hat d}$  with constant rank $\hat d$ on $\mathcal I$ and $\hat F_2C=0$.
\item[\textrm{(3)}] $\rank E(t)C(t)=\hat d,\;t\in \mathcal I$,  such that there is a smooth matrix function $Y:\mathcal I\rightarrow\Real^{m\times\hat d}$ with constant rank $\hat d$ on $\mathcal I$, and, for $\hat E_1:= Y^*E$, one has $\rank \hat E_1(t)=\hat d,\;t\in\mathcal I$.
\end{description}
\end{hypothesis}
The next definition simplifies \cite[Definition 4.4]{KuMe2006} for linear DAEs.
\begin{definition}\label{d.HypStrangeness}
Given are matrix functions $E,F:\mathcal I\rightarrow\Real^{m\times m}$. The smallest value of $\bar \mu$ such that the SF-Hypothesis \ref{SHyp} is satisfied is called the \emph{strangeness index} of the pair $\{E,F\}$ and of the DAE \eqref{1.DAE}. If $\bar \mu=0$ then the DAE is called \emph{strangeness-free}.
\end{definition}
Obviously, since the SF-Hypothesis is satisfied for DAEs having a well-defined differentiation index, it is even more satisfied for regular DAEs in the sense of Definition \ref{d.2} which also cover all DAEs featuring the regular strangeness index from Section \ref{subs.strangeness}.

\begin{remark}\label{r.HypStrange}
 Also the notion \emph{regular strangeness index} is sometimes used in the context of the SF-Hypothesis, e.g., \cite[p.\ 1261]{Baum}, \cite[p.\ 154]{KuMe2006} with the reasoning that a differential-algebraic operator  somehow (see \cite[Section 3.4]{KuMe2006})  associated to the strangeness-free reduced system \eqref{5.5.1}, \eqref{5.5.2} is a continuous bijection.
 Unfortunately, this is not a viable argument, because it says far too little about the nature of the original DAE and its associated operator\footnote{We refer to \cite{HaMae2023} and the references therein for basics on differential-algebraic operators.} $Tx=Ex'+Fx$. All differentiations are analytically assumed in advance and available from the derivative array. 
 
 In addition, the term \emph{regular strangeness index} is already used for the regular case of the original strangeness index (see Definition \ref{d.strangeness} and footnote).
 
The SF-Hypothesis is associated  with the fact that a DAE with differentiation index one always contains a regular index-1 DAE, cf. Proposition \ref{p.index1}.
\end{remark}

It is claimed in \cite[Theorem 3.50]{KuMe2006} that, if the pair $\{E, F\}$ has differentiation index $\mu^{diff}\geq1$ on a compact interval then the SF-Hypothesis is satisfied with $\hat\mu=\mu^{diff}-1$, $\hat a=a$, and $\hat d=d$. 

Conversely, according to \cite[Corollary 3.53]{KuMe2006}, if the pair $\{E, F\}$ satisfies  the SF-Hypothesis then it features a well-defined differentiation index, and  
$\mu^{diff}=\hat\mu+1$ applies if $\hat\mu$ is minimal.

\subsection{Equivalence issues}\label{subs.equivalence}
Let us summarize the most relevant results of the present section concerning the equivalence. We start with a well-known fact. 
\begin{theorem}\label{t.diff2}
Let  $E,F:\mathcal I\rightarrow \Real^{m\times m}$ be sufficiently smooth on the compact interval $\mathcal I$. The following two assertion are equivalent:
\begin{description}
 \item[\textrm{(1)}] The differentiation index $\mu^{diff}$ of the pair  $\{E,F\}$ on the interval $\mathcal I$ is well-defined according to Definition \ref{d.diff}.
  \item[\textrm{(2)}] The pair  $\{E,F\}$ satisfies the Strangeness Hypothesis \ref{SHyp} on the interval $\mathcal I$ with  strangeness index  $\hat{\mu}$ according to Definition \ref{d.HypStrangeness}.
\end{description}
If these statements are valid, then  $\mu^{diff}=\hat{\mu}+1$ and $\hat d=d=\dim S_{can}$ and $\dim \ker \mathcal E_{[\hat \mu]}=\hat a=a=m-d$.
\end{theorem}
\begin{proof} 
The direction (1) $\Rightarrow$ (2) immediately results from Section \ref{subs.Hyp} for an arbitrary interval. For the more complicated proof of (1) $\Leftarrow$ (2) we refer to \cite[Corollary 3.53]{KuMe2006}, cf.\ Section \ref{subs.Hyp}.
\end{proof}
The other index concepts considered in the present section require additional rank conditions. It turns out that they are equivalent among each other and comprise just the regular DAEs in the sense of the basic Definition \ref{d.2}.
\begin{theorem} \label{t.equivalence_array}
Let $E, F:\mathcal I\rightarrow\Real^{m\times m}$ be sufficiently smooth, $\mu\in\Natu$.

The following assertions are equivalent in the sense that the individual characteristic values of each two of the variants are mutually uniquely determined.
\begin{description}
\item[\textrm{(1)}] 
The  pair $\{E,F\}$ is regular on $\mathcal I$ with index $\mu\in \Natu$, according to Definition \ref{d.2}.
\item[\textrm{(2)}] The DAE \eqref{1.DAE} has regular differentiation index $\mu^{rdiff}=\mu$.
\item[\textrm{(3)}] The DAE \eqref{1.DAE} has projector-based differentiation index $\mu^{pbdiff}=\mu$.
\item[\textrm{(4)}] The DAE has differentiation index $\mu^{diff}=\mu$ and, additionally, the rank functions $r_{[k]},\ k<\mu^{diff}$, are constant.
\item[\textrm{(5)}] The DAE  fulfills the Hypothesis \ref{SHyp} with  $\hat{\mu}=\mu-1$ and, additionally, the rank functions $r_{[k]},\ k<\hat{\mu}$, are constant.
\end{description}
\end{theorem}
\begin{proof}
The equivalence of \textrm{(1)} and \textrm{(2)} has been shown in Theorem \ref{t.rdiff}  \textrm{(1)}.

The equivalence of \textrm{(2)} and \textrm{(4)} follows from the equivalence of \textrm{(1)} and \textrm{(4)} in  Lemma \ref{l.app}, since in both cases all ranks are assumed to be constant.

The equivalence of \textrm{(4)} and \textrm{(5)} is a consequence of Theorem \ref{t.diff2}.

\textrm{(1)} implies \textrm{(3)} by Theorem \ref{t.projector_based_diff}. 

For the last step of the proof of equivalences we verify that \textrm{(3)} implies \textrm{(5)}. Let  \textrm{(3)} be given, so that $r_{[i]}=\rank \mathcal E_{[i]}$ is constant, $i=0,\ldots,\mu-1$, $\ker E\cap S_{[\mu-1]}=\{0\}$, and the necessary  solvability condition \eqref{eq:fullrank} is satisfied.
We set $\hat{\mu}:=\mu-1$, $\hat a:=\mu m-r_{[\mu-1]}$, $\hat d:=m-\hat a$ and show that the Hypothesis \ref{SHyp} with these values is satisfied.

First,  it results that $\dim (\im \mathcal E_{[\mu-1]})^{\perp}=\hat a$ and there is a smooth basis  $Z:\mathcal I\rightarrow \Real^{\mu m\times\hat a}$ of $(\im \mathcal E_{[\mu-1]})^{\perp}$
such that $Z^*\mathcal E_{[\mu-1]}=0$.

Next, regarding the condition \eqref{eq:fullrank} we evaluate
\[
 \rank Z^*\mathcal F_{[\mu-1]}=m-\dim \ker Z^*\mathcal F_{[\mu-1]}=m-\dim S_{[\mu-1]}=\mu m-r_{[\mu-1]}=\hat a.
\]
Finally, since $\dim S_{[\mu-1]}=\hat d$, with a smooth matrix function $C:\mathcal I\rightarrow \Real^{m\times\hat d}$ forming a basis of $\dim S_{[\mu-1]}$, we obtain a product $EC$ that feature full column-rank $\hat d$. Namely, it holds
\[
 \ker EC=\{z\in\Real^{\hat d}:Cz\in\ker E\}=\{z\in\Real^{\hat d}:Cz\in\ker E\cap S_{[\mu-1]}\}=\{0\},
\]
and the Hypothesis \eqref{SHyp} is satisfied, and thus statement \textrm{(5)}. 
\end{proof}

We now indicate how the individual characteristic values depend on those of the base concept. Because of the equivalence, this allows to determine the relationships between the values of any two concepts providing regular DAEs.
\begin{theorem}\label{t.theta_relation_array}
 Let the  pair $\{E,F\}$ be regular on $\mathcal I$ with index $\mu\in \Natu$ and  characteristics $r<m$,
$\theta_0=0$ if $\mu=1$, and, for $\mu>1$,
\begin{align*}
r<m,\quad \theta_0\geq\cdots\geq\theta_{\mu-2}>\theta_{\mu-1}=0,\quad d=r-\sum_{j=0}^{\mu-2}\theta_{j}.
\end{align*}
Then the three array functions $\mathcal E_{[k]},\mathcal D_{[k]}$, and $\mathcal B_{[k]}$ feature shared constant ranks,
\[
r_{[k]}=\rank \mathcal E_{[k]} = \rank \mathcal D_{[k]} = \rank \mathcal B_{[k]},
\]
and the following relations concerning the characteristic values arise:
\begin{align*}
 r_{[k]} &=km+r-\sum_{j=0}^{k-1}\theta_{j},\quad k=1,\ldots, 
 \end{align*}
 in particular,
 \begin{align*}
 r_{[\mu-1]} &= (\mu-1)m+r-\sum_{j=0}^{\mu-2}\theta_{j}=(\mu-1)m+d,&&
 r_{[\mu]} = \mu m+r-\sum_{j=0}^{\mu-1}\theta_{j}=\mu m+d,
\end{align*}
and, moreover,
 \begin{align*}
  \rho_k=m-\dim (\ker E\cap S_{[k]})=m-\theta_k,\quad &k=0,1, \ldots \ ,\\
  \rank T_{k} = \theta_{k-1},\quad  \rank V_k = r-\sum_{j=0}^{k-1} \theta_{j},\quad &k=0,1, \ldots .
	\end{align*}
and, conversely,
\begin{align*}
 r&=r_{[0]},\\
 \theta_0&=m+r_{[0]}-r_{[1]},\\
 &\cdots \\
 \theta_{\mu-2}&=m+r_{[\mu-2]}-r_{[\mu-1]},\\
  \theta_{\mu-1}&=m+r_{[\mu-1]}-r_{[\mu]},\\
  d&=r_{[\mu-1]}-(\mu-1)m=r_{[\mu]}-\mu m.
\end{align*}
\end{theorem}
\begin{proof}
The relations for the characteristic values follow from Theorems \ref{t.rdiff} and \ref{t.projector_based_diff}. 
\end{proof}
\begin{corollary}\label{c.crit}
Regular and critical points in the sense of the Definition \ref{d.regpoint} are independent of the specific approach. 
\end{corollary}

We emphasize that, for the approaches gathered in Theorems \ref{t.equivalence_array} and \ref{t.theta_relation_array} that capture regular DAEs,  constant $r$ and $\theta_i$ are mandatory. In contrast, the two concepts recorded in Theorem \ref{t.diff2} allow changes of $r$ as well as the $\theta_i$, as long as $\mu$ and $d$ remain constant. This motivates the following two definitions.

\begin{definition} \label{d.harmless}
Given are $E,F:\mathcal I\rightarrow\Real^{m\times m}$. 
 The critical\footnote{See Definition \ref{d.regpoint}.} point $t_*\in\mathcal I$ is said to be a \emph{ harmless critical point}\footnote{Note that this definition is consistent with  that in \cite{Dokchan2011,CRR,RR2008} which is formulated in more specific terms of the projector-based analysis.} of the pair the pair $\{E,F\}$ and the associated DAE \eqref{1.DAE}, if there is an open neighborhood  $\mathcal U\ni t_*$ such that the DAE  restricted to $\mathcal I\cap \mathcal U$  is solvable in the sense of Definition \ref{d.solvableDAE}. 
\end{definition}
Harmless critical points only become apparent with less smooth problems and in the input/output behavior of the systems.
For smooth problems, we quote from \cite[p. 180]{RR2008}: \textit{the local behavior around a harmless critical point is entirely analogous to the one near a regular point}. That is precisely why they are called that. The Examples \ref{e.2}, \ref{e.5} and \ref{e.7} in the next section are to confirm this, see also \cite[Section 2.9]{CRR}.

By Theorem \ref{t.diffinterval}, for a DAE featuring a well-defined differentiation index on a compact interval, the set of regular points is dense and all critical points are harmless.

\begin{definition} \label{d.almost-reg}
Given are $E,F:\mathcal I\rightarrow\Real^{m\times m}$. The  pair $\{E,F\}$, and the associated DAE \eqref{1.DAE} are \emph{ almost regular} if all points of  $ \mathcal I$ are regular points in the sense of Definition \ref{d.regpoint} or harmless critical points in the sense of Definition \ref{d.harmless} and the regular points are dense in $\mathcal I$.
\end{definition}

\section{A selection of simple examples to illustrate possible critical points}\label{s.examples}

We use a few simple examples to illustrate several critical points that can arise with DAEs.
For a deeper insight we refer to \cite{RR2008}.

\subsection{Serious singularities}

\begin{example}[$r$ constant, $\theta_0$ changes]\label{e.degree_cont}
 Recall Example \ref{e.degree} from Section \ref{subs.degree}
  \begin{align*}
  E(t)=\begin{bmatrix}
        1&-t\\1&-t
       \end{bmatrix},\quad 
F(t)=\begin{bmatrix}
        2&0\\0&2
       \end{bmatrix},\quad t\in \Real.
 \end{align*}
 We know already that the homogeneous DAE has the nontrivial solution
 \begin{align*}
 x(t)=\gamma (1-t)^2\begin{bmatrix}
                     1\\1
                    \end{bmatrix},\; t\in \Real, \quad \text{with }\; \gamma\in \Real,
\end{align*}
and $t_\star = 1$ is obviously a singular point of the flow, because of
\begin{align*}
 \ker E(t) \cap S_0(t) = \{z \in \Real^2: z_1-t z_2=0, z_1 = z_2\},\text{ i.e., } \theta_0(t) = \begin{cases}
                                                                                                 0 &t \neq 1\\
                                                                                                 1 &t = 1
                                                                                                \end{cases}
\end{align*}
\begin{itemize}
 \item The framework of the tractability index with
 \begin{align*}
  A:=E, D=\begin{bmatrix}
        1&-t\\0&0
       \end{bmatrix},G_0=A D = E, Q_0=\begin{bmatrix}
                            0 & t\\0 & 1
                           \end{bmatrix},B_0=F-A D' = \begin{bmatrix}
                                                          2 & -1\\0 & 1
                                                        \end{bmatrix}
 \end{align*}
leads to
\begin{align*}
 G_1=G_0 + B_0 Q_0 = \begin{bmatrix}
                        1 & t-1\\
                        1 & -t+1
                      \end{bmatrix}, \det G_1 = 2(1-t), r_1(t) = \begin{cases}
                                                                  2 & t \neq 1\\
                                                                  1 & t = 1
                                                                 \end{cases}.
\end{align*}
This indicates $t_\star = 1$ as critical point.
\item By Lemma \ref{l.R1} we obtain $r_{[1]}(t) = m+r - \theta_0(t) = 3 -\theta_0(t)$ . 
\end{itemize}
\end{example}

\begin{example}[$r$ changes, $\theta_0$ constant]\label{e.rank_drop_E} This example is a special case of  \cite[Example 2.69]{CRR}. We consider the pair 
\[
E=\begin{bmatrix}0 & \alpha\\
                 0&0
                 \end{bmatrix},  \quad 
								 F=\begin{bmatrix}
                   \beta & \gamma\\
                    1    &  1
                  \end{bmatrix}, \quad m=2,
	\]
and $\alpha, \beta, \gamma: \mathcal{I} \rightarrow \Real$ are smooth functions. $\alpha^2 + (\beta - \gamma)^2 > 0$ ensures that $\rank [E(t), F(t)] \equiv 2$ and $\alpha(t) = 0$ requires $\beta(t) \neq \gamma(t)$.\\
We have $r(t) = \begin{cases}
                 1 & \alpha(t) \neq 0\\
                 0 & \alpha(t) = 0
                \end{cases}\quad 
$,  $\ker E(t) = \begin{cases}
                 \im \begin{bmatrix}
                      1 \\ 0
                     \end{bmatrix}
& \alpha(t) \neq 0\\
                 \Real^2 & \alpha(t) = 0
                \end{cases}$
and\\
$S_0(t) =\{z \in \Real^m: F(t)z \in \im E(t)\} =\begin{cases}
                 \im \begin{bmatrix}
                      1 \\ -1
                     \end{bmatrix}
& \alpha(t) \neq 0,\\
                 \{0\} & \alpha(t) = 0.
                \end{cases}$\\
Surprisingly we obtain for $\theta_0(t) = \dim (\ker E(t) \cap S_0(t)) = 0$ for all $t$.\\
A simple reformulation  of  $Ex'+Fx = q$ shows the equations
\begin{align*}
 \alpha x_2' + (\gamma-\beta) x_2 &= q_1-\beta q_2,\\
 x_1 &= -x_2 + q_2.
\end{align*}
We consider the particular case that $\gamma(t)-\beta(t) \equiv M \neq 0$ constant, $q \equiv 0$ and $\alpha(t)=t$, which leads to the singular scalar homogeneous ODE for $x_2$
\begin{align}\label{solution_x2}
 t x_2'(t) + M x_2(t) =0.
\end{align}
The solution is $x_2(t) = c\, t^M$ with an arbitrary real constant $c$, see Figure \ref{fig:solM}.
\begin{figure}
\includegraphics[width=6.5cm]{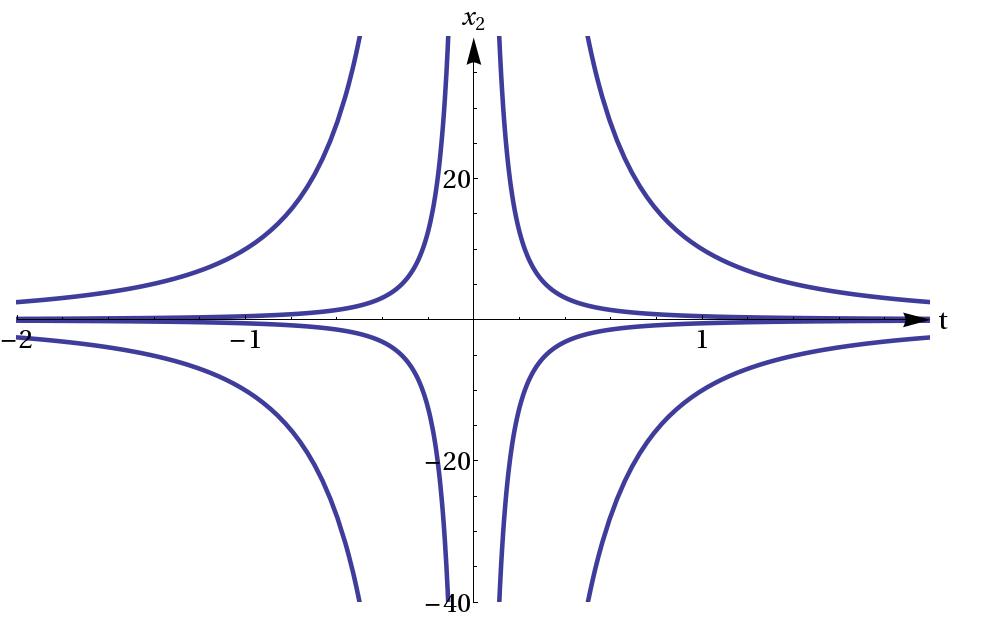}
\includegraphics[width=6.5cm]{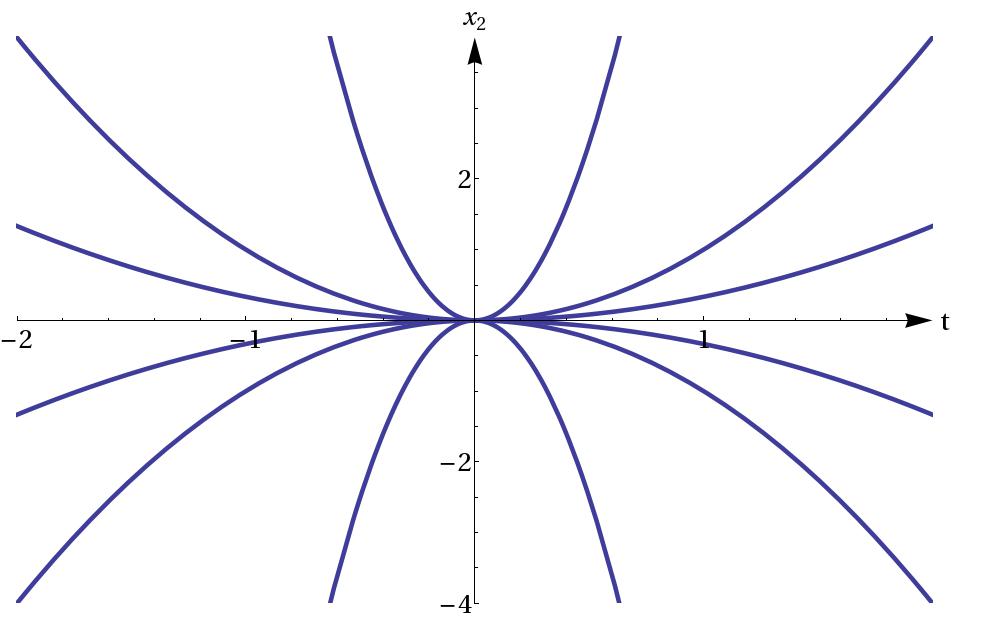}
\caption{Solution of \eqref{solution_x2} for $M>0$ and $M<0$}
\label{fig:solM}
\end{figure} 
\end{example}
%
\begin{example}[$r$ constant, $\theta_0$ constant, $\theta_1$ changes]\label{e.6}
Given a smooth function $\beta:\mathcal I\rightarrow\Real$, we investigate the pair $\{E,F\}$,
 \begin{align*}
 E(t)=\begin{bmatrix}
    1&0&0\\0&1&0\\0&0&0
   \end{bmatrix},\quad
   F(t)=\begin{bmatrix}
    0&0&\beta(t)\\1&1&0\\1&0&0
   \end{bmatrix}.
\end{align*}
 A look at the associated DAE brightens the type of singularity. The DAE reads
\begin{align*}
 x_1'+\beta x_3&=q_1,\\
 x_2'+x_1+x_2&=q_2,\\
 x_1&=q_3,
\end{align*}
which can be rearranged to
\begin{align*}
 x_1&=q_3,\\
 x_2'+x_2&=q_2-q_3,\\
 \beta x_3&=q_1-q_3'.
\end{align*}
It is now evident that, if $\beta$ has zeros, then the DAE is no longer solvable for all sufficiently smooth rigt-hand sides $q$. 

From a more general point of view,
 the pair $\{E,F\}$ is pre-regular with $m=3$. $r=2$ and $\theta_0=1$ and the singularity can be detected by different approaches. 
 
 We start with the basic approach, letting
  \begin{align*}
 Y_0(t)=\begin{bmatrix}
    1&0\\0&1\\0&0
   \end{bmatrix},\quad
   C_0(t)=\begin{bmatrix}
    0&0\\1&0\\0&1
   \end{bmatrix},
\end{align*}
 the first step in the basic reduction procedure leads to the new pair
 \begin{align*}
 E_1(t)=\begin{bmatrix}
    0&0\\1&0
   \end{bmatrix},\quad
   F_1(t)=\begin{bmatrix}
    0&\beta(t)\\1&0
   \end{bmatrix}.
\end{align*}
The new pair $\{E_1,F_1\}$ is pre-regular if and only if the function $\beta$ has no zeros, and then one has $\theta_1=0$, and hence the pair $\{E,F\}$ is regular with index two.

In contrast, if    $\beta(t_*)=0$ at a point $t_*\in\mathcal I$, we are confronted with $\theta_1(t_*)=1$ and $\rank[E_1(t_*) F_1(t_*)]=1<m$, and the pair $\{E_1,F_1\}$ fails to be pre-regular. In turn, the original pair $\{E,F\}$ is no longer regular.

The tractability framework \eqref{2.Gi} leads to 
\begin{align*}
 G_0=E, B_0=F, Q_0=\begin{bmatrix}
                    0&&\\
                    &0&\\
                    &&1
                   \end{bmatrix}, G_1=\begin{bmatrix}
                                      1& 0 & \beta\\
                                      0& 1& 0\\
                                      0 & 0 & 0
                                      \end{bmatrix},
\end{align*}
which immediately indicates zeros of $\beta$ as critical point, too, because of
\begin{align*}
 N_1 \cap N_0 =\{z \in \Real^3: z_1=0, z_2=0, \beta z_3 =0\},
\end{align*}
i.e., $u_1^T(t)= \begin{cases}
                  0 & \beta \neq 0\\
                  1 & \beta = 0
                 \end{cases}
$,\quad which is called "B-singularity`` in \cite[Definition 2.75]{CRR} and \cite[p.\ 144]{RR2008}.

Next we consider the first array functions
 \begin{align*}
 \mathcal E_{[1]}(t)=\begin{bmatrix}
    1&0&0&0&0&0\\0&1&0&0&0&0\\0&0&0&0&0&0\\0&0&\beta(t)&1&0&0\\1&1&0&0&1&0\\1&0&0&0&0&0
   \end{bmatrix},\quad
   \mathcal F_{[1]}(t)=\begin{bmatrix}
    0&0&\beta(t)\\1&1&0\\1&0&0\\0&0&\beta'(t)\\0&0&0\\0&0&0
   \end{bmatrix}.
\end{align*}
and compute $\rank \mathcal E_{[1]}(t)=4=m+r-\theta_0$ independently of how the function $\beta$ behaves. However, the necessary solvability requirement $\rank[\mathcal E_{[1]}(t), \mathcal F_{[1]}(t)]=6$ is  satisfied only if $\beta(t)\neq 0$, but otherwise one has  $\rank[\mathcal E_{[1]}(t_*), \mathcal F_{[1]}(t_*)]=5$. 

\end{example}
%

\begin{example}[$r$ constant, $\theta_0$ constant, $\theta_1$ changes]\label{e.3}
 Given is the pair $\{E, F\}$ with $m=2$,
 \begin{align*}
  E(t)=\begin{bmatrix}
        1&-1\\1&-1
       \end{bmatrix},\quad 
  F(t)=\begin{bmatrix}
        2&0\\0&t+2
       \end{bmatrix}, \quad t\in \mathcal I=[-1,1],      
 \end{align*}
such that $E(t)$ has constant rank $r=1$ and $\rank [E(t), F(t)]=m$. Following the basic reduction procedure in Section \ref{s.regular} we choose and find 
\begin{align*}
 Z_0(t)=\begin{bmatrix}
      1\\-1
     \end{bmatrix},\;
Y_0(t)=\begin{bmatrix}
      1\\1
     \end{bmatrix},\;
     C_0(t)=\begin{bmatrix}
      \frac{t+2}{2}\\1
     \end{bmatrix},\\ 
     \ker E(t)\cap\ker (Z^*_0F)(t)=\{z\in\Real^2: z_1=z_2, tz_2=0\},
\end{align*}
and 
\begin{align*}
  E_1(t)=(Y^*EC_0)(t)=t,\; F_1(t)=(Y^*FC_0)(t)+(Y^*EC'_0)(t)= 2t+5.
\end{align*}
The  pair $\{E, F\}$ fails to be pre-regular on $\mathcal I$ because of 
\begin{align*}
 \theta_0(t)=\Bigg\lbrace \quad\begin{matrix}
                               0& \text{ if }t\neq 0, \\ 
                               1& \text{ if }t= 0\;,
                                     \end{matrix}
\end{align*}
however, it is pre-regular and regular with index one on the subintervals $[-1,0)$ and $(0,1]$. The corresponding DAE,
\begin{align*}
 x_1'-x_2'+2x_1&=q_1,\\ x_1'-x_2'+(t+2)x_2&=q_2,
\end{align*}
reads in slightly rearranged form as
\begin{align*}
-tx_2+2(x_1-x_2)&=q_1-q_2,\\
 (x_1-x_2)'+\frac{2}{t}(t+2)(x_1-x_2)&=q_2-\frac{1}{t}(t+2)(q_1-q_2).
\end{align*}
Having a solution of the ODE for the difference $x_1-x_2$ we find the original solution components by $x_1=\frac{1}{2}(q_1-(x_1-x_2)')$ and $x_2=x_1-(x_1-x_2)$.
No doubt, $t_*=0$ is a critical point causing a singular inherent ODE of the DAE.
We refer to \cite{Koch}, where this example also comes from, for the specification of bounded solutions.

Inspecting the array functions
\begin{align*}
 \mathcal E_{[1]}=\begin{bmatrix}
                   1&-1&0&0\\1&-1&0&0\\2&0&1&-1\\0&t+2&1&-1
                  \end{bmatrix},\quad
   \mathcal E_{[2]}=\begin{bmatrix}
                   1&-1&0&0&0&0\\1&-1&0&0&0&0\\2&0&1&-1&0&0\\0&t+2&1&-1&0&0\\
       0&0&2&0&            1&-1\\0&2&0&t+2&1&-1
                  \end{bmatrix},                
\end{align*}
we see that 
\begin{align*}
  \rank\mathcal E_{[1]}(t)&=\Bigg\lbrace\quad\begin{matrix}
                                 2m-1=3& \text{if }t\neq 0 \\ 
                                 2m-2=2& \text{if }t= 0\\
                                     \end{matrix}\\ 
 \rank\mathcal E_{[2]}(t)&=\Bigg\lbrace\quad\begin{matrix}
                                 3m-1=5& \text{if }t\neq 0 \\ 
                                 3m-2=4& \text{if }t= 0,\;\\
                                    \end{matrix}
\end{align*}
and the rank of the array functions also indicates this point $t_*=0$ as critical.
\end{example}

\begin{example}[$r$ and $\theta_0$ change]\label{e.4}
Given a smooth function $\alpha:\mathcal I\rightarrow\Real$, we investigate the pair $\{E,F\}$,
 \begin{align*}
 E(t)=\begin{bmatrix}
    0&\alpha(t)&0\\0&0&1\\0&0&0
   \end{bmatrix},\quad
   F(t)=\begin{bmatrix}
    -6&0&0\\0&1&0\\1&0&1
   \end{bmatrix},\quad t\in\mathcal I,
\end{align*}
 and the associated DAE living in $\Real^m$, $m=3$,
 \begin{align*}
  \alpha x_2'-6x_1&=q_1,\\
  x_3'+x_2&=q_2,\\
  x_1+x_3&=q_3.
 \end{align*}
Rearranging the DAE as
\begin{align*}
  \alpha x_2'+6x_3&=q_1+6q_3,\\
  x_3'+x_2&=q_2,\\
  x_1+x_3&=q_3,
 \end{align*}
 we immediately know that the fact whether the function $\alpha$ has zeros or even disappears on subintervals is essential. Note that $\rank \, [E(t) F(t)]=m=3$ for all $t\in \mathcal I$, but $\rank E(t)=2$ if $\alpha(t)\neq 0$ and otherwise $\rank E(t)=1$. Obviously, 
 on subintervals where $\alpha(t)$ does not vanish, we see a regular index-one DAE with characteristics $r=2, \theta_0=0$ and $d=2$.  In contrast, if the function $\alpha$ vanishes on a subinterval then there a regular index-two DAE results with $r=1, \theta_0=1, \theta_1=0$ and $d=0$. There is no doubt that points with zero crossings of $\alpha$ are critical and may cause singularities of the solution flow, see Figure \ref{fig:solMATHEMATICA}. 
 It should be emphasized that the rank conditions associated with regularity Definition \ref{d.2} exclude such critical points on regularity intervals and the corresponding reduction procedure reliably recognizes them.
 
 Investigating the corresponding low level array functions, $\mathcal E_{[0]}=E$, $\mathcal F_{[0]}=F$, 
 \begin{align*}
 \mathcal E_{[1]}=\begin{bmatrix}
                   0&\alpha&0&0&0&0\\0&0&
1&0&0&0\\0&0&0&0&0&0\\-6&\alpha'&0&0&\alpha&0 \\0&1&0&0&0&1 \\1&0&1&0&0&0                \end{bmatrix},\quad
  \mathcal F_{[1]}=\begin{bmatrix}
                    -6&0&0\\0&1&0\\1&0&1\\0&0&0\\0&0&0\\0&0&0
                  \end{bmatrix}                
 \end{align*}
 \begin{align*}
 \mathcal E_{[2]}=\begin{bmatrix}
                   0&\alpha&0&0&0&0&0&0&0\\0&0&
1&0&0&0&0&0&0\\0&0&0&0&0&0&0&0&0\\-6&\alpha'&0&0&\alpha&0 &0&0&0\\0&1&0&0&0&1 &0&0&0\\1&0&1&0&0&0 &0&0&0\\0&\alpha''&0&-6&2\alpha'&0&0&\alpha&0 \\0&0&0&0&1&0&0&0&1\\0&0&0&1&0&1&0&0&0              \end{bmatrix},\quad
  \mathcal F_{[2]}=\begin{bmatrix}
                    -6&0&0\\0&1&0\\1&0&1\\0&0&0\\0&0&0\\0&0&0\\0&0&0\\0&0&0\\0&0&0
                  \end{bmatrix},               
 \end{align*}
 we find that the ranks of the derivative arrays also indicate those critical points, namely
 \begin{align*}
   \rank\mathcal E_{[0]}(t)&= \begin{cases}
                               m-1=  2& \text{if }\alpha(t)\neq 0 \\ 
                               m-2=  1& \text{if }\alpha(t)= 0
                                     \end{cases}\\                                     
   \rank\mathcal E_{[1]}(t)&=\begin{cases}
                                 2m-1=5& \text{if }\alpha(t)\neq 0 \\ 
                                 2m-2=4& \text{if }\alpha(t)= 0,\alpha'(t)\neq 0\\
                                 2m-3=3& \text{if }\alpha(t)= 0,\alpha'(t)=0, \alpha''(t)\neq 0
                                     \end{cases}\\ 
 \rank\mathcal E_{[2]}(t)&=\begin{cases}
                                 3m-1=8& \text{if }\alpha(t)\neq 0 \\ 
                                 3m-2=7& \text{if }\alpha(t)= 0,\alpha'(t)\neq 0\\
     3m-3=6& \text{if }\alpha(t)= 0,\alpha'(t)=0, \alpha''(t)\neq 0, \alpha''(t)\neq 6\\
     3m-3=6& \text{if }\alpha(t)= 0,\alpha'(t)=0, \alpha''(t)= 0\\
     3m-4=5& \text{if }\alpha(t)= 0,\alpha'(t)=0, \alpha''(t)= 6 \;.
                                    \end{cases}
  \end{align*}
	
	\begin{figure}
 \includegraphics[width=6.5cm]{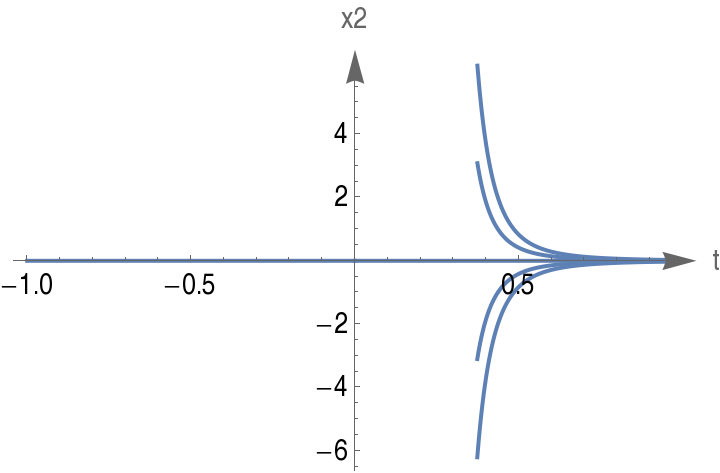}%
\includegraphics[width=6.5cm]{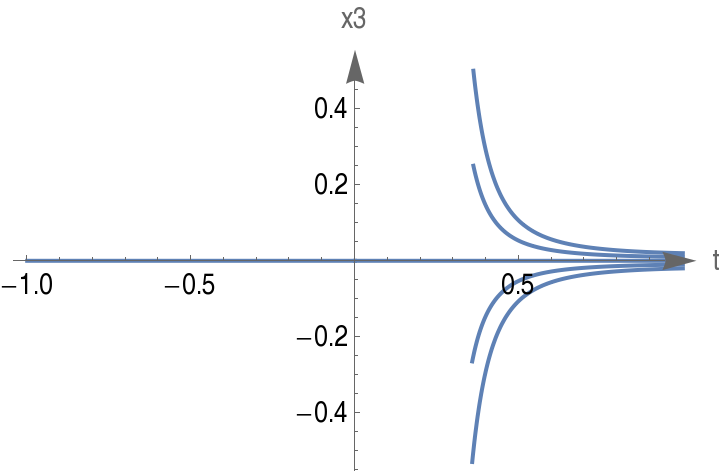}

\caption{Solutions for $x_2$ and $x_3$ computed with MATHEMATICA, Version 13 for \\ $\alpha(t)=\begin{cases}
                              0&\text{ for }t\in (-\infty,0]\\
                             t^{4}&\text{ for }t\in (0,\infty)
                              \end{cases}$, $q_1=q_2=q_3=0$ in Example \ref{e.4}. The difficulty to plot the solution around $0$ is due to the singularity.}
\label{fig:solMATHEMATICA}
\end{figure}

\end{example}
\subsection{Harmless critical points}
The first example of the present section, that shows an almost regular DAE, is of particular interest because $r$ and $d$ are constant, while $\theta_0$ and $\theta_1$ change. To our knowledge such an circumstance was not discussed in literature before, since harmless critical points were usually tight to rank changes of $E$. Therefore, for this example we illustrate in detail how four different approaches identify critical points.\\
The three other examples of the section are classical cases discussed in the literature showing rank changes of $E$.
\begin{example}[$r$ constant, $\theta_0$ and $\theta_1$ change]\label{e.5}
 Given is the pair $\{E, F\}$ with $m=4$, and smooth functions $\alpha, \beta:\mathcal I\rightarrow\Real$,
 \begin{align*}
  E(t)=\begin{bmatrix}
        0&1&\alpha(t)&0\\0&0&0&\beta(t)\\0&0&0&1\\0&0&0&0
       \end{bmatrix},\quad
  F(t)=\begin{bmatrix}
        1&0&0&0\\0&1&0&0\\0&0&1&0\\0&0&0&1
       \end{bmatrix}, \quad t\in \mathcal I,    
 \end{align*}
such that $E(t)$ has constant rank $r=2$ and $\rank \, [E(t) F(t)]=m$.

The associated DAE reads
\begin{align*}
 x_2'+\alpha x_3'+x_1&=q_1,\\
 \beta x_4'+x_2&=q_2,\\
 x_4'+x_3&=q_3,\\
 x_4&=q_4.
\end{align*}
For each sufficiently smooth right-hand side this DAE possesses the unique solution,
\begin{align*}
 x_1&=q_1-q_2'-\alpha q_3'+\beta'q_4'+(\alpha+\beta)q_4'',\\
 x_2&=q_2-\beta q_4',\\
 x_3&=q_3- q_4',\\
 x_4&=q_4,
\end{align*}
that is, the DAE is a solvable system in the sense of Definition \ref{d.solvableDAE} with zero dynamical degree of freedom. 
It can be checked immediately that in the sense of Definition \ref{d.2} the DAE is regular with index $\mu=3$ and $d=0$ on all subintervals  where $\alpha+\beta$ has no zeros, and it is regular with index $\mu=2$ and $d=0$ on all subintervals  where $\alpha+\beta$ vanishes identically.
\medskip

We analyse this example with four different approaches. All of them lead to the same values for the characteristics $\theta$:
\begin{align}\label{e5.theta}
\begin{array}{cccc}
                     & \theta_0  & \theta_1  &\theta_2\\ 
 \alpha+\beta \neq 0 &    1  & 1 & 0\\
 \alpha+\beta = 0 &    2  & 0&
\end{array}
\end{align}
\medskip

 \textbf{Basic reduction procedure}, cf.\ \eqref{basic_reduction}. We have $E_0=E$ and $F_0=I$. We obtain for a basis of $\ker E_0^\star$  $Z_0 = \begin{bmatrix}
                                                0 & 0\\
                                                1 & 0\\
                                                -\beta & 0\\
                                                0 & 1                                                                                                                                                                    \end{bmatrix}$ and a basis of $\im E_0$ $Y_0 = \begin{bmatrix}
                                                                                               1 & 0\\
                                                                                               0 & \beta\\
                                                                                               0 & 1\\
                                                                                               0 & 0
                                                                                              \end{bmatrix}$.
$S_0 = \ker Z_0^\star F_0 = \{z \in \Real^4: z_2 = \beta z_3, z_4 = 0 \}$. $C_0 = \begin{bmatrix}
                                                                                1 & 0\\
                                                                                0 & \beta\\
                                                                                0 & 1\\
                                                                                0 & 0
                                                                               \end{bmatrix}
$ forms a basis of $S_0$.\\
The next reduction step serves $E_1 = Y_0^\star E_0 C_0 = \begin{bmatrix}
                                                           0 & \alpha + \beta \\
                                                           0 & 0
                                                          \end{bmatrix}$ and 
$F_1 = Y_0^\star F_0 C_0 +  Y_0^\star E_0 C'_0 = \begin{bmatrix}
                                                  1 & \beta' \\
                                                  0 & 1 + \beta^2
                                                 \end{bmatrix}$.\\
To determine $\theta_0$ we investigate $\ker E_0 \cap \ker Z_0^\star F_0 = \{ z \in \Real^4: (\alpha + \beta) z_3 = 0, z_4 = 0, z_2 = \alpha z_3\}$.\\
We have now to continue with two different cases.
\begin{itemize}
 \item $\alpha+\beta \neq 0$:
 It results that $\ker E_0 \cap \ker Z_0^\star F_0 = \{ z \in \Real^4: z_3 = 0, z_4 = 0, z_2 = 0\} = \im \begin{bmatrix}
                                                                                                          1\\0\\0\\0
                                                                                                         \end{bmatrix}$
and therefore $\theta_0 =1$.\\
The new pair $[E_1, F_1]$ is pre-regular. $Z_1 = \begin{bmatrix}
                                                  0\\
                                                  1
                                                 \end{bmatrix}
$, $Y_1 = \begin{bmatrix}
                                                  1\\
                                                  0
                                                 \end{bmatrix}
$, and $C_1 = \begin{bmatrix}
               1\\
               0
              \end{bmatrix}
$, 
i.e., $E_2 = Y_1^\star E_1 C_1 = 0$ and $F_2 = Y_1^\star F_1 C_1 = 1$.\\
$\ker E_1 \cap \ker Z_1^\star F_1 =\{z \in \Real^2: z_2 = 0\}$, which means that $\theta_1 = 1$.\\
The last reduction step delivers the pre-regular pair $E_2 = 0$ and $F_2 = 1$, which leads to $\theta_2 = 0$ and therefore $\mu = 3$.
 \item $\alpha+\beta = 0$:
 In this case we obtain for $\ker E_0 \cap \ker Z_0^\star F_0 = \{ z \in \Real^4: z_4 = 0, z_2 = \alpha z_3\} = \im \begin{bmatrix}
            1 & 0\\
            0 & \alpha\\
            0 & 1\\
            0 & 0
  \end{bmatrix}
$ and therefore $\theta_0 = 2$.\\
The next matrix pair is $E_1 = 0$ and the nonsingular $F_1 = \begin{bmatrix}
                                                           1 & \beta'\\
                                                           0 & 1 + \beta^2
                                                          \end{bmatrix}
$. $[E_1, F_1]$ is pre-regular and we obtain $\theta_1 = 0$ and therefore $\mu = 2$.
\end{itemize}
\medskip


\textbf{ Projector based analysis (tractability index) with the related matrix chain, cf.\ \eqref{2.Gi}}.
\begin{equation*}
 G_0 = E,\quad B_0 = F - E D',\quad Q_0 =
 \begin{bmatrix}
  1 & 0                      & 0                    & 0\\
  0 & 0 & -\alpha & 0\\
  0 & 0  & 1      & 0\\
  0 &                       0&                     0& 0
 \end{bmatrix},\quad D = P_0 = I-Q_0.
\end{equation*}

\begin{equation*}
 G_1=G_0+B_0 Q_0 =
 \begin{bmatrix}
  1 & 1                      & \alpha - \alpha'                   & 0\\
  0 & 0 & -\alpha & \beta\\
  0 & 0  & 1      & 1\\
  0 &                       0&                     0& 0
 \end{bmatrix}.
 \end{equation*}
 To determine an admissible nullspace projector $Q_1$ we analyse $\ker G_1$.
 \begin{align*}
  \ker G_1 &= \{z \in \Real^4: z_1+z_2+(\alpha - \alpha')z_3 = 0,-\alpha z_3+\beta z_4 = 0,z_3+z_4 = 0\}\\
           &= \{z \in \Real^4: z_1+z_2+(\alpha - \alpha')z_3 = 0,(\alpha +\beta) z_4 = 0,z_3+z_4 = 0\}
 \end{align*}
Also here we have to continue with two different cases.
\begin{itemize}
 \item $\alpha+\beta \neq 0$:
 
 We obtain that $\ker G_1 =\{z \in \Real^4: z_4 = 0, z_3 = 0, z_1+z_2 = 0\} = \im \begin{bmatrix}
                                                                                   -1 \\1\\0\\0
                                                                                  \end{bmatrix}
$ and $\rank G_1 = 3$. An admissible nullspace projector is
 $Q_1=
 \begin{bmatrix}
  0 & -1                      & -\alpha                    & 0\\
  0 & 1  & \alpha & 0\\
  0 & 0 &  0      & 0\\
  0 &                       0&                     0& 0
 \end{bmatrix}$
and $\Pi_1 = P_0P_1 = P_0(I-Q_1)=  \begin{bmatrix}
                                    0&&&\\
                                    &0&&\\
                                    &&0&\\
                                    &&&1
                                   \end{bmatrix}
$.\\
The next matrix chain level starts with $B_1 = B_0P_0 -G_1D^-(D\Pi_1D^-)'D\Pi_0 $ $=B_0P_0 -G_1D^-\Pi_1'D\Pi_0 =  B_0P_0$ and we obtain
\begin{align*}
G_2&=G_1+B_1 Q_1=
\begin{bmatrix}
        1&1&\alpha - \alpha'&0\\
        0&1&         0&\beta\\
        0&0&         1&1\\
        0&0&         0&0
       \end{bmatrix},\\
\ker G_2 &= \{z \in \Real^4: z_1+z_2 + (\alpha-\alpha')z_3=0, z_2+\beta z_4 = 0, z_3+z_4 = 0\} \\
&= \im \begin{bmatrix}
  \alpha-\alpha'+\beta\\
    -\beta\\
      -1\\
      1
\end{bmatrix} \text{ and } \rank G_2 = 3.
\end{align*}
 As an admissible nullspace projector we choose
$Q_2 =
\begin{bmatrix}
  0&0&0 &\alpha -\alpha'+ \beta\\
        0&0&         0&-\beta\\
        0&0&         0&-1\\
        0&0&         0&1
\end{bmatrix}$
and \\ $\Pi_2 = \Pi_1(I-Q_2) = 0$.  The nonsingular matrix 
\begin{align*}
 G_3 = G_2+B_2 Q_2= 
 \begin{bmatrix}
        1&1&\alpha - \alpha'&0\\
        0&1&         0&\beta\\
        0&0&         1&1\\
        0&0&         0&1
       \end{bmatrix},
\end{align*}
which indicates that the index is 3.
\item $\alpha+\beta = 0$:
For this case  we obtain the nullspace\\
$\ker G_1 = \{z \in \Real^4: z_1+z_2+(\alpha+\alpha')z_3=0, z_3+z_4=0 \} = \im \begin{bmatrix}
                                             \alpha+\alpha' & -1\\                                                                                                                         
                                             0 & 1\\
                                             -1 & 0\\
                                             1 & 0
                                             \end{bmatrix}
$ and $\rank G_1 = 2$. As an admissible nullspace projector we can choose
\begin{equation*}
 Q_1=
 \begin{bmatrix}
  0 & -1                      & -\alpha                   & \alpha'\\
  0 & 1  & \alpha         & \alpha\\
  0 & 0  &  0    & -1\\
  0 &                       0&                     0& 1
 \end{bmatrix}
\end{equation*}
and obtain because of $\Pi_1 = 0$ as the next matrix chain element the nonsingular matrix 
\begin{align*}
 G_2 = 
 \begin{bmatrix}
        1&1&\alpha-\alpha' &0\\
        0&1&         0&\beta\\
        0&0&         1&1\\
        0&0&         0&1
       \end{bmatrix},
\end{align*}
which indicates an index 2 DAE.\\
\end{itemize}
For the characteristics we have, cf.\ \eqref{trac}, $\theta_{i-1} = m- \rank G_i$, 
leading to \eqref{e5.theta}.
\medskip

\textbf{Differentiation index concept}.
 Inspecting the array functions
 \setcounter{MaxMatrixCols}{15}
\begin{align*}
 \mathcal E_{[1]}=\begin{bmatrix}
                   0&1&\alpha&0&0&0&0&0\\0&0&0&\beta&0&0&0&0\\0&0&0&1&0&0&0&0\\0&0&0&0&0&0&0&0\\
                    1&0&\alpha'&0&0&1&\alpha&0\\0&1&0&\beta'&0&0&0&\beta\\0&0&1&0&0&0&0&1\\0&0&0&1&0&0&0&0
                  \end{bmatrix},\\
   \mathcal E_{[2]}=\begin{bmatrix}
    0&1&\alpha&0&0&0&0&0&0&0&0&0\\
    0&0&0&\beta&0&0&0&0&0&0&0&0\\
    0&0&0&1&0&0&0&0&0&0&0&0\\
    0&0&0&0&0&0&0&0&0&0&0&0\\
                    1&0&\alpha'&0&0&1&\alpha&0&0&0&0&0\\
                    0&1&0&\beta'&0&0&0&\beta&0&0&0&0\\
                    0&0&1&0&0&0&1&0&0&0&0&0\\
                    0&0&0&1&0&0&0&0&0&0&0&0\\
  0&0&\alpha''&0&1&0&2\alpha'&0&0&1&\alpha&0\\
                    0&0&0&\beta''&0&1&0&2\beta'&0&0&0&\beta\\
                    0&0&0&0&0&0&1&0&0&0&0&1\\
                    0&0&0&0&0&0&0&1&0&0&0&0                  
                  \end{bmatrix},                
\end{align*}
we find
\begin{align*}
 r_{[1]}= \rank\mathcal E_{[1]}(t)&=\Bigg\lbrace\quad\begin{matrix}
                                 5& \text{if }\alpha(t)+\beta(t)\neq 0 \\ 
                                 4& \text{if }\alpha(t)+\beta(t)= 0,\\
                                     \end{matrix}
\end{align*}
 but, in contrast,  $\mathcal E_{[2]}(t)$ does not undergo any rank changes,
 \begin{align*}
 \dim \ker \mathcal E_{[2]}(t)= 4,\quad \rank\mathcal E_{[2]}(t)= 3m-4=8.
\end{align*}

Using for $\theta_i = m + r_{[i]}-r_{[i+1]}$, (cf.\ Theorem \ref{t.theta_relation_array}), the characteristic values $\theta$ are the same as in \eqref{e5.theta}.
Nevertheless, this DAE has a well-defined differentiation index being equal three. We add that the DAE according to \cite[Chapter 9]{CRR} is quasi-regular with an index less than or equal to three.
\medskip

\textbf{Projector based differentiation concept}. Starting from
\begin{align*}
\ker E(t)= \im \begin{bmatrix}
	1 & 0 \\0 & -\alpha \\ 0 &  1 \\0 & 0
\end{bmatrix}, \quad Q=\begin{bmatrix}
	1 & 0 & 0 & 0 \\
	0 & \frac{\alpha^2}{1+\alpha^2 } & \frac{-\alpha}{1+\alpha^2 } & 0 \\
	0 & \frac{-\alpha}{1+\alpha^2 } & \frac{1}{1+\alpha^2 } & 0 \\
	0 & 0 & 0 & 0
\end{bmatrix}, \quad  P=\begin{bmatrix}
	0 & 0 & 0 & 0 \\
	0 & \frac{1}{1+\alpha^2 } & \frac{\alpha}{1+\alpha^2 } & 0 \\
	0 & \frac{\alpha}{1+\alpha^2 } & \frac{\alpha^2}{1+\alpha^2 } & 0 \\
	0 & 0 & 0 & 1
\end{bmatrix}
\end{align*}
we recognize
\begin{align*}
\ker Q= \ker \begin{bmatrix}
	1 & 0 & 0 & 0\\
	0 & -\alpha & 1 & 0
\end{bmatrix}, \quad \ker P = \ker \begin{bmatrix}
	0 & 0 & 0 & 1\\
	0 & 1 & \alpha & 0
\end{bmatrix}.
\end{align*}

On the one hand, we have
\[
\im \begin{bmatrix} \mathcal F_{[0]} Q &\quad & \mathcal E_{[0]}\end{bmatrix} = \im \begin{bmatrix}  Q &\quad & E \end{bmatrix} = \im \begin{bmatrix}
1 & 0& 0\\
0 & -\alpha & \beta \\
0 & 1 & 1 \\
0 & 0 & 0	
\end{bmatrix},
\]
such that $\mathcal V_{[0]}$ and its rank depend on whether $\alpha = -\beta$ is given or not. On the other hand the explicit constraints are
\begin{align*}
x_2-\beta(t) x_3 &= q_2-\beta(t) q_3, \\
x_4 &= q_4,
\end{align*}
such that
\begin{align*}
\ker E \cap S_{[0]} = \ker \begin{bmatrix}
	P \\
\mathcal W_{[0]} \mathcal F_{[0]}
\end{bmatrix} = \ker \begin{bmatrix}
	0 & 0 & 0 & 1\\
	0 & 1 & \alpha & 0  \\ \hline
	0 & 1 & -\beta & 0  \\
	0 & 0 & 0 & 1
\end{bmatrix}.
\end{align*}
Again, the dimension of this space depends on $\alpha+\beta$.
Therefore, we consider two cases:
\begin{itemize}
	\item $\alpha +\beta \neq 0 $: Then
	\[
	\ker E \cap S_{[0]} = \ker \begin{bmatrix}
	P \\
\mathcal W_{[0]} \mathcal F_{[0]}
\end{bmatrix} = \im \begin{bmatrix}
	1 \\
	0 \\
	0 \\
	0 \\
\end{bmatrix}, \quad T_1= \begin{bmatrix}
	1 & 0 & 0 & 0 \\
	0 & 0 & 0 & 0 \\
	0 & 0 & 0 & 0 \\
	0 & 0 & 0 & 0 \\
\end{bmatrix},
	\]
	and
	\[
	\ker \begin{bmatrix}
	Q \\
\mathcal V_{[0]} \mathcal F_{[0]}
\end{bmatrix} = \ker \begin{bmatrix}
1 & 0 & 0 & 0\\
	0 & -\alpha & 1 & 0 \\ \hline
	0 & 0 & 0 & 1 
\end{bmatrix}, \quad V_1= \begin{bmatrix}
	0 & 0 & 0 & 0 \\
	0 & \frac{1}{1+\alpha^2 } & \frac{\alpha}{1+\alpha^2 } & 0 \\
	0 & \frac{\alpha}{1+\alpha^2 } & \frac{\alpha^2}{1+\alpha^2 } & 0 \\
	0 & 0 & 0 & 0
\end{bmatrix}.
\]
The next steps lead to $V_2=0$, $T_2=T_1$, $V_3=0$,  $T_3=0$, such that
\[
r=2,\quad \theta_0=1, \quad \theta_1=1, \quad \theta_2=0, \quad \mu=3, \quad d=0.
\]
	\item $\alpha+\beta=0$: Then
	\[
	\ker E \cap S_{[0]} = \ker \begin{bmatrix}
	P \\
\mathcal W_{[0]} \mathcal F_{[0]}
\end{bmatrix} = \im \begin{bmatrix}
	1 & 0 \\
	0 & -\alpha \\
	0 &  1 \\
	0 &  0 \\
\end{bmatrix}, \quad T_1= Q,
\]
\end{itemize}
and
	\[
	\ker \begin{bmatrix}
	Q \\
\mathcal V_{[0]} \mathcal F_{[0]}
\end{bmatrix} = \ker \begin{bmatrix}
1 & 0 & 0 & 0\\
	0 & -\alpha & 1 & 0 \\ \hline
	0 & 1 & \alpha & 0\\
	0 & 0 & 0 & 1 
\end{bmatrix}, \quad V_1= 0.
\]
The next step leads to $V_2=0$, $T_2=0$, such that
\[
r=2, \quad \theta_0=2, \quad \theta_1=0, \quad \mu=2, \quad d=0.
\]

In summary, all approaches lead to the same characteristic values \eqref{e5.theta} and reveal the same critical points at the zeros of $\alpha+\beta$.

We realize that we have a solvable DAE here, although all the procedures show critical points and in particular not all derivative-array functions have constant rank.  
By Definition \ref{d.harmless}, these crical points are harmless.
From the solution representation we recognize the precise dependence of the solution on the derivatives of the right-hand side $q$. 
Accordingly, a sharp perturbation index three is only valid on the subintervals where $\alpha+\beta$ has no zeros, and perturbation index two on subintervals where $\alpha+\beta$  is identically zero. 
This is very important when it comes to minimal smoothness and the input-output behavior from a functional analysis perspective.
\end{example}

\begin{example}[\cite{BCP89}, $d=1$, $r$ and $\theta_0$ change, in SCF]\label{e.1}\hfill
Given the function $\alpha(t)=\begin{cases}
                              0&\text{ for }t\in[-1,0)\\
                             t^{3}&\text{ for }t\in [0,1]
                              \end{cases}$\quad
we consider the pair $\{E,F\}$,
\begin{align*}
 E(t)=\begin{bmatrix}
    1&0&0\\0&0&\alpha(t)\\0&0&0
   \end{bmatrix},\quad
   F(t)=\begin{bmatrix}
    1&0&0\\0&1&0\\0&0&1
   \end{bmatrix},\quad t\in\mathcal I=[-1,1],
\end{align*}
and the associated DAE \eqref{1.DAE},
\begin{align*}
 x_1'+x_1&=q_1,\\
 \alpha x_3'+x_2&=q_2,\\
 x_3&=q_3.
\end{align*}

By straightforward  evaluations we know that the DAE has differentiation index two on the entire interval $[-1,1]$, but differentiation index one on the subinterval $[-1,0]$. 

Obviously the DAE forms a solvable system in the sense of Definition \ref{d.solvableDAE} with dynamical degree of freedom $d=1$ on the entire interval $[-1,1]$. 

The DAE is regular with index two in the sense of Definition \ref{d.2} on the subinterval $(0,1]$, and regular with index one on $[-1,0]$. Similarly, the perturbation index equals one on $[-1,0]$, but two on each closed subinterval of $(0,1]$.

Observe  that $\mathcal E_{[0]}(t)=E(t)$ changes the rank at $t=0$, but $\rank \mathcal E_{[1]}(t)=4$, $\mathcal E_{[2]}(t)=7$, $t\in [-1,1]$, and $r(t)-\theta(t)=d=1$ is constant.
\end{example}


\begin{example}[\cite{BCP89} Example 2.4.3, $d$ constant, $r$ and $\theta_0$ change, not transferable into SCF]\label{e.2}\hfill
For the functions 
\[
\alpha(t)=\begin{cases}
                                0& \text{ for } t\in [-1,0)\\
                                t^{3}&\text{ for } t\in [0,1]
                               \end{cases}
 \text{ and }\quad 
 \beta(t)=\begin{cases}
            t^{3}&\text{ for } t\in [-1,0)\\
            0&\text{ for }t\in [0,1]
           \end{cases}
\]           
we
consider the DAE \eqref{1.DAE} with the coefficients
\begin{align*}
 E(t)=\begin{bmatrix}
   0&\alpha(t)\\\beta(t)&0
   \end{bmatrix},\quad
   F(t)=\begin{bmatrix}
    1&0\\0&1
   \end{bmatrix},\quad t\in\mathcal I=[-1,1].
\end{align*}
To each arbitrary smooth right-hand side $q$, the DAE has the unique solution 
\begin{align*}
 x_1&=q_1-\alpha q_2',\\
 x_2&=q_2-\beta q_1',
\end{align*}
so that it is solvable with zero dynamical degree of freedom. The DAE has obviously perturbation index two.

Observe  that $\mathcal E_{[0]}(t)=E(t)$ has a rank drop at $t=0$, but $\rank \mathcal E_{[1]}(t)=2$, $\rank \mathcal E_{[2]}(t)=4$, $t\in [-1,1]$.
The DAE has differentiation index two on the entire interval $[-1,1]$ and also on each subinterval.

In contrast, the basic reduction procedure from Section \ref{s.regular} indicates the point $t=0$ as critical. The DAE is regular with index two and $r=1, \theta_0=1, \theta_1=0, d=0$ on both subintervals $[-1,0)$ and $(0,1]$.
\end{example}

\begin{example}[$d=0$, $r$ changes, index 1 or 3, in SCF]\label{e.7}
With this example in SCF we illustrate that a change of $r$ leads to an in- or decrease the local index on subintervals that differ more than one.
Given the function\\
$\alpha(t)=\begin{cases}
                               0 &\text{ for } t \in [-1,0)\\
                               t^{3} &\text{ for } t\in [0,1]
                              \end{cases}$\quad
 we
consider the pair $\{E,F\}$,
\begin{align*}
 E(t)=\begin{bmatrix}
    0&\alpha(t)&0\\0&0&\alpha(t)\\0&0&0
   \end{bmatrix},\quad
   F(t)=\begin{bmatrix}
    1&0&0\\0&1&0\\0&0&1
   \end{bmatrix},\quad t\in\mathcal I=[-1,1],
\end{align*}
and the associated DAE \eqref{1.DAE},
\begin{align*}
 \alpha x_2'+x_1&=q_1,\\
 \alpha x_3'+x_2&=q_2,\\
 x_3&=q_3.
\end{align*}

By straightforward  evaluations we know that the DAE has differentiation index three on the entire interval $[-1,1]$, but differentiation index one on the subinterval $[-1,0]$. 

The DAE forms a solvable system in the sense of Definition \ref{d.solvableDAE} with zero dynamical degree of freedom $d=0$ on the entire interval $[-1,1]$. 

The DAE is regular with index three, $r=2, \theta_0=1, \theta_1=1, \theta_2=0$, and $d=0$ in the sense of Definition \ref{d.2} on the subinterval $(0,1]$, and it is regular with index one, $r=0, \theta_0=0$, and $d=0$ on $[-1,0]$. Similarly, the perturbation index equals one on $[-1,0]$, but three on each closed subinterval of $(0,1]$.

The rank of  $\mathcal E_{[0]}(t)$ and $\mathcal E_{[1]}(t)$ changes  at $t=0$ but $\mathcal E_{[2]}(t)$, $\mathcal E_{[3]}(t)$ have constant rank each. More precisely, we have 
\begin{align*}
\dim \ker\mathcal E_{[0]}(t)&=\Bigg\lbrace\quad\begin{matrix}
                                 1& \text{if }\alpha(t)\neq 0 \\ 
                                 3& \text{if }\alpha(t)= 0,\\
                                     \end{matrix}\\
 \dim \ker\mathcal E_{[1]}(t)&=\Bigg\lbrace\quad\begin{matrix}
                                 2& \text{if }\alpha(t)\neq 0 \\ 
                                 3& \text{if }\alpha(t)= 0,\\
                                     \end{matrix}\\
          \dim \ker\mathcal E_{[2]}(t) =  \dim \ker\mathcal E_{[3]}(t)&=3, \quad t\in\mathcal I=[-1,1].
\end{align*}
\end{example}

\subsection{A case study}\label{e.struc}
With the following case study we emphasize that monitoring the index of  DAEs is not sufficiently informative. For a deeper understanding of their properties, all characteristic values $r$ and $\theta_i$ should be considered.\\

For identity matrices $I_i\in \Real^{m_i \times m_i}$, $i=1,2,3$ and constant strictly upper triangular matrices $N_2 \in \Real^{m_2 \times m_2}$, $N_3 \in \Real^{m_3 \times m_3}$ of the special form
\[
	N_i=\begin{bmatrix}
		0 & 1 & \cdots & 0\\
		0 & 0 & \ddots \\
		\vdots &  & \ddots & 1\\
		0 & 0 & \cdots  & 0
	\end{bmatrix}, \quad i=2,3,
\]
 let us consider DAEs of the form
\[
\begin{bmatrix}
\alpha_1(t) I_d & 0 & 0\\
 0 & 	\alpha_2(t) N_1 & 0 \\
	0 & 0 & \alpha_3(t) N_2 
\end{bmatrix}\begin{bmatrix}
	x_1' \\
	x_2'\\
	x_3' 
\end{bmatrix} + 
\begin{bmatrix}
\beta_1(t) I_1 & 0 & 0\\
 0 & 	\beta_2(t) I_2 & 0 \\
	0 & 0 & \beta_3(t) I_3 
\end{bmatrix}\begin{bmatrix}
	x_1 \\
	x_2\\
	x_3 
\end{bmatrix} =\begin{bmatrix}
	q_1 \\
	q_2\\
	q_3 
\end{bmatrix},
\]
with $m=m_1+m_2+m_3$ and smooth functions $\alpha_i, \beta_i: \mathcal I \rightarrow \Real$. 
\begin{itemize} 
	\item We focus first on the functions $\beta_i$:
\begin{itemize}
\item Zeros of $\beta_1(t)$ are not critical at all.
	\item If $\beta_2(t_*)$ or $\beta_3(t_*)$ are zero for $t_* \in \mathcal I$, then $\left[ E(t_*) \ F(t_*)\right]$ has a trivial row. Then $\left\{E,F\right\}$ is not qualified on $\mathcal I$, cf. Definition \ref{d.qualified} and the necessary solvability condition \eqref{eq:fullrank} is violated as in Example \ref{e.6}.
	\end{itemize}
\item Let us suppose now that $\beta_2$ and $\beta_3$ have no zeros and focus on $\alpha_1$:
\begin{itemize}
	\item Zeros of $\alpha_1(t)$ obviously cause a singular ODE $\alpha_1(t)x_1' + \beta_1(t)x_1=q_1$.
	\item If $\alpha_1$ has no zeros, then the degree of freedom is constant $d=m_1$ regardless of whether the $\alpha_i$, $i=2,3$, have zeros or not.
	\end{itemize}
\item Let us suppose now that $\alpha_1$, $\beta_2$ and $\beta_3$ have no zeros and focus on $\alpha_2$, $\alpha_3$:
\begin{itemize}
	\item For $\alpha_2(t)\neq 0$, $\alpha_3(t) \neq 0$ for all $ t \in \mathcal I$, the DAE is regular  with index  $\mu=\max \left\{m_2, m_3  \right\}$.
	\item For $m_2\geq m_3$ and $\alpha_2(t)\neq 0$  for all $ t \in \mathcal I$, the DAE has differentiation index $\mu=m_2$ and all points $t_*$ such that $\alpha_3(t_*)=0$, but $\alpha_3$ does not identically vanish in a neighborhood of $t_*$, are harmless critical points.
	\item For $m_2 > m_3$, all points $t_*$ such that $\alpha_2(t_*)=0$, but $\alpha_2$ does not identically vanish in a neighborhood of $t_*$, are harmless critical points. 
	
\item  For $m_2 > m_3$,	if $\alpha_2$ vanishes identically on a subinterval $\mathcal I_*\subset\mathcal I$, then  the DAE restricted to this subinterval has differentiation index $\mu\leq m_3<m_2$. 
\item In general it may happen, if both $\alpha_2$ and $\alpha_3$ vanish on a subinterval, that there the index reduces to one.
\end{itemize}
\end{itemize}

In general, for $\alpha_i(t)\neq0$, $\beta_j(t)\neq 0$ for $i=1,2,3$ and $j=2,3$, by construction it holds
\begin{align*}
r&=m_1+(m_2-1)+(m_3-1),\\
\mu&=\max \{m_2,m_3\}, \\
\theta_i &=
\begin{cases}
2 & \text{ for } i \leq \min\{m_2-2,m_3-2\}, \\
1 &\text{ for }  \min\{m_2-2,m_3-2\} < i \leq \mu-2, \\
0 & \text{ else},
\end{cases}\\
d&= m_1. 
\end{align*} 
For instance, for $m_1=4, m_2=2, m_3=3$ this means
\begin{align*}
r&=7,\ \quad 
&\mu&=3, \ \quad\
&\theta_0 &= 2,\\
\theta_1 &= 1,\ \quad
&\theta_2 &=0,\ \quad
&d&= 4.
\end{align*} 
In general, zeros of $\alpha_i$ or $\beta_j$ imply a change of these characteristic values.\\

For general linear DAEs, the components are intertwined in a complex manner, such that it is not possible to guess directly whether zeros of some coefficients in $\{E,F\}$ are critical or not. However, monitoring these characteristic values may provide crucial indications.



\section{Main equivalence theorem concerning regular DAEs and further comments on regularity and index notions}\label{s.Notions}
\subsection{Main equivalence theorem concerning regular DAEs}\label{sec:MainTheorem}

We finally summarize the main equivalence results of the preceding sections in  statements for regular pairs $\{E,F\}$ and associated DAEs \eqref{DAE0} on the interval $\mathcal I$. Each of the involved equivalent DAE frameworks is based on a number of specific characteristics. The  characteristic values correspond to requirements for constant dimensions of subspaces or ranks of matrix functions. 
Recall that we always assume for regularity that the matrix functions $E, F : \mathcal I\rightarrow \Real^{m\times m}$ are sufficiently smooth and $E$ has constant rank $r$. 
\medskip

In the previous chapters, we precisely described the relationships between the individual characteristics of each concept and the $\theta$-values \eqref{theta},
\[
  \theta_{0}\geq\cdots\geq\theta_{\mu-2}>\theta_{\mu-1}=0,
\]
introduced in our basic regularity Definition \ref{d.2}.
The following main theorem emphasizes the universality of the $\theta$-characteristic, which is why we take the liberty of calling them and $r$ {\it canonical characteristics}, which for regular DAEs in particular indicate  the dynamical degree of freedom $d=\dim S_{can}$  and the inner structure of the canonical subspace $N_{can}$, cf.\ Remark \ref{r.pencil}. We claim and highlight that the characteristics $\theta_i$ can be found regardless of which DAE concept we start from.
\medskip

 Let us briefly consider the easier cases $\mu=0,1$:
\begin{itemize}
	\item For constant $r=m$ the regular DAE is a regular implicit ODE, and the index is $\mu=0$. We can interpret this as $\theta_i=0$ for all $i$ and  $d=m$.
  \item For constant $r<m$ the condition $\theta_0=0$ is equivalent to $\mu=1$. We interpret this as  $\theta_i=0$ for all $i$ and $d=r<m$. Then we have a regular index-one DAE and a strangeness-free DAE, respectively.
\end{itemize}

Owing to Proposition \ref{p.index1} we know that the DAE associated to the pair $\{E, F\}$ has differentiation index one if and only if the pair is  regular with index one in the sense of Definition \ref{d.2}. Moreover, all  further index-1 concepts are consistent with each other, too, which is well-known. To enable simpler formulations, we now turn to the case that the index is higher than or equal to two.

For regular DAEs with index $\mu\geq 2$ we have seen that $\theta_0>0$ follows by definition, as will be specified in the following. For the sake of uniformity, in general for $i>\mu-2$ we set $\theta_i=0$, cf.\ Remark \ref{r.inf}.
 

\begin{theorem}[Main Equivalence Theorem]\label{t.Sum_equivalence}
Let $E, F:\mathcal I\rightarrow\Real^{m\times m}$ be sufficiently smooth, and $\mu\in \Natu$, $\mu\geq2$.

\textbf{Part A (Equivalence):}
The following 13 assertions are equivalent in the sense that the characteristic values (constants) of each of the statements can be uniquely reproduced by those of any other statement:
\begin{description}
\item[\textrm{(0)}] 
The pair $\{E,F\}$ is regular on $\mathcal I$ with index $\mu\in \Natu$ according to Definition \ref{d.2}, with the associated characteristics 
\[
 r=\rank E,\ \theta_{0}\geq\cdots\geq\theta_{\mu-2}>\theta_{\mu-1}=0,\quad \theta_i:=\dim (\ker E_i\cap S_i).
\]
\item[\textrm{(1)}] 
 The  pair $\{E,F\}$ is regular on $\mathcal I$ with index $\mu\in \Natu$ according to Definition \ref{d.2a}, with the  associated characteristics  
\[
r=r_0>r_1> \ldots > r_{\mu-1}=r_{\mu}, \quad r_i:=\rank E_i.
\]
\item[\textrm{(2)}] 
The  pair $\{E,F\}$ is regular on $\mathcal I$ with elimination index $\mu^E=\mu$ according to Definition \ref{d.Elim}, with the associated characteristics
\[
r=r^E_0>r^E_1> \ldots >r^E_{\mu-1}=r^E_{\mu},\quad r^E_i:=\rank E^E_i.
\]
\item[\textrm{(3)}] The pair $\{E,F\}$ is regular on $\mathcal I$ with dissection  index $\mu^D=\mu $ according to Definition \ref{d.diss}, with the associated characteristics
\begin{align*}
 r=r^D_0\leq r^D_1\leq \cdots \leq r^D_{\mu-1}<r^D_{\mu} =m,\quad r^D_i:=\rank E^D_i=r^D_{i-1}+a^D_{i-1}.
\end{align*}
 \item[\textrm{(4)}] The  pair $\{E,F\}$ is regular on $\mathcal I$ with strangeness index $\mu^S=\mu-1$ according to Definition \ref{d.strangeness}, with associated characteristic tripels 
 \[
  (r^S_i,a^S_i,s^S_i),\ i=0,\ldots,\mu-1,\;r_0^S=r,\, s^S_{\mu-1}=0.
 \]
\item[\textrm{(5)}] The pair $\{E,F\}$ is regular on $\mathcal I$ with tractability index $\mu^T=\mu$ according to Definition \ref{d.trac}, with associated characteristics 
\[
 r=r^T_0\leq r^T_1\leq\dots\leq r^T_{\mu-1}<r^T_{\mu}=m,\quad r^T_i=\rank G_i.
\]
\item[\textrm{(6)}]
The  pair $\{E,F\}$ can be equivalently transformed to block-structured Standard Canonical Form \eqref{blockstructure}, in which the nilpotent matrix function $N(t)=(\tilde N_{ij}(t))^{\mu}_{i,j=1}, \,\tilde N_{i,j}:\mathcal I\rightarrow\Real^{\tilde l_i\times \tilde l_j}$, is block-upper-triangular, has nilpotency index $\mu$ on $\mathcal I$,
 and has full row-rank blocks on the secondary block diagonal, with characteristics
\[ 1\leq\tilde l_1\leq\cdots\leq\tilde l_{\mu},\; \tilde l_i=\rank \tilde N_{i,i+1}, \,i=1,\ldots,\mu-1,\, \tilde l_{\mu}=m-r, r:=\rank E.
\]
\item[\textrm{(7)}]
The  pair $\{E,F\}$ can be equivalently transformed to block-structured Standard Canonical Form \eqref{blockstructure}, in which the nilpotent matrix function $N(t)=(N_{ij}(t))^{\mu}_{i,j=1}, \, N_{i,j}:\mathcal I\rightarrow\Real^{l_i\times l_j}$, is block-upper-triangular, has nilpotency index $\mu$ on $\mathcal I$,
 and has full column-rank blocks on the secondary block diagonal, with characteristics
\[ l_1\geq\cdots\geq l_{\mu},\; l_{1}=m-r,\,  l_{i+1}=\rank  N_{i,i+1}, \,i=1,\ldots,\mu-1, r:=\rank E.
\]
\item[\textrm{(8)}] For the  pair $\{E,F\}$ the DAE \eqref{1.DAE} has regular differentiation index $\mu^{rdiff}=\mu$ on $\mathcal I$ according to Definition \ref{d.G1}, with associated characteristics 
\[
 r_{[0]}=r,\, r_{[i]}< r_{[i-1]}+m,\, i=1,\ldots, \mu-2,\, r_{[\mu-1]}= r_{[\mu]}+m,\: r_{[i]}:=\rank \mathcal E_{[i]}.
\]
\item[\textrm{(9)}] For the  pair $\{E,F\}$ the DAE \eqref{1.DAE} has projector-based differentiation index $\mu^{pbdiff}=\mu$ on $\mathcal I$ according to Definition \ref{d.pbdiffA}, with associated characteristics  
\[
 r_{[0]}=r,\, r_{[i]}< r_{[i-1]}+m,\, i=1,\ldots, \mu-2,\,\: r_{[i]}:=\rank \mathcal E_{[i]},
\]
\[
 \rho_{0}\leq\cdots\leq \rho_{\mu-2}<\rho_{\mu-1}=m, \;\rho_{i}:=m-\dim \ker E\cap S_{[i]}.
\]

\item[\textrm{(10)}] For the  pair $\{E,F\}$ the DAE \eqref{1.DAE} has projector-based differentiation index $\mu^{pbdiff}=\mu$ on $\mathcal I$ according to Definition \ref{d.pbdiffB}, with associated characteristics
\[
 r^{\mathcal B}_{[0]}=r,\, r^{\mathcal B}_{[i]}< r^{\mathcal B}_{[i-1]}+m,\, i=1,\ldots, \mu-2,\, r^{\mathcal B}_{[\mu-1]}= r^{\mathcal B}_{[\mu]}+m,\: r^{\mathcal B}_{[i]}:=\rank \mathcal B_{[i]}.
\]
 
\item[\textrm{(11)}] For the  pair $\{E,F\}$ the DAE \eqref{1.DAE} has differentiation index $\mu^{diff}=\mu$ on $\mathcal I$ according to Definition \ref{d.diff}
 and, addionally, also the matrix functions $\mathcal E_{[i]}$, $i<\mu$,  feature constant on $\mathcal I$, so that  the characteristics are
 \[
 r_{[0]}=r,\, r_{[i]}< r_{[i-1]}+m,\, i=1,\ldots, \mu-2,\, r_{[\mu-1]}= r_{[\mu]}+m,\: r_{[i]}:=\rank \mathcal E_{[i]}.
\]
 \item[\textrm{(12)}] For the  pair $\{E,F\}$ the DAE \eqref{1.DAE} satisfies the Strangeness-Free-Hypothesis \eqref{SHyp} on $\mathcal I$ with $\hat{\mu}=\mu-1$   with associated characteristics $\hat a$ and $\hat d=m-\hat a$,
 and, addionally,
 also the matrix functions $\mathcal E_{[i]}$, $i<\hat{\mu}$,  feature constant on $\mathcal I$, so that  the characteristics are
 \[
 r_{[0]}=r,\, r_{[i]}< r_{[i-1]}+m,\, i=1,\ldots, \mu-2,\, r_{[\mu-1]}= \mu m+\hat a,\: r_{[i]}:=\rank \mathcal E_{[i]}.
\]
\end{description}
\textbf{Part B (Relations between characteristics):}

 Let the DAE \eqref{1.DAE} be regular with index $\mu$ in the sense of Definition \ref{d.2} or one of the equivalent statements of \textbf{Part A}. Then the following relations concerning the diverse characteristic values  are valid:
\begin{align*}
r_0 &= r_{0}^E=r_0^S=r_0^D=r_0^T=r,\\
 l_{1}&= m-r, \quad   \quad \tilde{l}_{\mu}=m-r,
\end{align*}
and for $i=1, \ldots, \mu-1$
\begin{align*}
r_i&= r_{i}^E=r_i^S=r-\sum_{j=0}^{i-1} \theta_j, \\
r_i^D&=r_i^T=\rho_{i-1}=m-\theta_{i-1},\\
s_i^S&= \theta_i, \quad 
a_i^S=m-r+\sum_{j=0}^{i-1} \theta_j-\theta_i,\\
l_{i+1}&=\theta_{i-2}, \quad \tilde{l}_i= \theta_{\mu-i-1},\\
r_{[i]}&=\rank \mathcal B_{[i]} = \rank \mathcal D_{[i]} = \rank \mathcal E_{[i]}=km+r-\sum_{j=0}^{i-1}\theta_{j}.
\end{align*}
Conversely, for $i=0, \ldots, \mu-1$,  we obtain
\begin{align*}
\theta_i&=r_i-r_{i+1}=s_i^S=r_i^S-r_{i+1}^S=r_i^E-r_{i+1}^E=m-r_{i+1}^T=m-r_{i+1}^D \\
 &=m-\rho_i =r_{[i]}-r_{[i+1]}+m,
\end{align*}
which leads to  $\theta_{\mu-1}=0$, and
\begin{align*}
\theta_i &=l_{i+2}=\tilde{l}_{\mu-i-1}, \quad \mbox{for} \quad i=0, \ldots, \mu-2. 
\end{align*}
\textbf{Part C (Description of resulting main features):}

Let the DAE \eqref{1.DAE} be regular with index $\mu$ in the sense of Definition \ref{d.2} or one of the equivalent statements of \textbf{Part A}. Then the following descriptions concerning the dynamical degree of freedom $d$ of the DAE and the number of constraints $a$ are given:
\begin{align*}
 d&=r-\sum_{i=0}^{\mu-2} \theta_i,\\
 a&=m-d=m-r+\sum_{i=0}^{\mu-2} \theta_i,\\
 a&=\hat a=\mu m-r_{[\mu-1]}=\sum_{i=1}^{\mu}l_i=\sum_{i=1}^{\mu}\tilde{l_i},\\ d&=\hat d=m-\hat a = r_{[\mu-1]}-(\mu-1)m=r_{[\mu]}-\mu m.
\end{align*}
\end{theorem}
\begin{proof}

\textbf{Part A:}
 
 \begin{itemize}
	 \item  The equivalence of two of the five statements \textrm{(0)}, \textrm{(2)}, \textrm{(3)}, \textrm{(4)},  and \textrm{(5)}, is given by Theorem \ref{t.equivalence}. As main means of achieving this serve \cite[Theorem 4.3]{HaMae2023} and Proposition \ref{p.STform}.

\item The equivalence of two of the statements  \textrm{(0)},  \textrm{(8)}, \textrm{(9)},  \textrm{(11)},  \textrm{(12)} has been provided by Theorem \ref{t.equivalence_array}. An important step for this proof is the equivalence of  \textrm{(0)} and  \textrm{(8)} that has been verified by Theorem \ref{t.rdiff} \textrm{(1)}.
 
\item The equivalence of  \textrm{(0)} and \textrm{(1)} has been  verified in Section \ref{s.regular} right after Definition \ref{d.2a}.

\item  The equivalence of  \textrm{(0)} and \textrm{(6)} as well as that of \textrm{(0)} and \textrm{(7)}
are implications of Theorem \ref{t.SCF} and Proposition \ref{p.STform}.

\item  The equivalence of  \textrm{(9)} and  \textrm{(10)} has been shown in Section \ref{subs.pbdiff} right after Definition \ref{d.pbdiffA}.

 \end{itemize}
 
 \textbf{Part B and Part C:}
 These are straightforward summaries of the related findings of Theorem \ref{t.indexrelation}, Corollary \ref{c.SCT}, Theorem \ref{t.theta_relation_array}, and Corollary \ref{c.degree}.
\end{proof}

Obviously, each one of the equivalent statement \textrm{(0)}-\textrm{(12)} from Theorem \ref{t.Sum_equivalence} implies both statements \textrm{(1)} and \textrm{(2)} from Theorem \ref{t.diff2}, but  the reverse direction does not apply, cf.\ discussion in Section \ref{subs.equivalence} and examples in Section \ref{s.examples}.

Indeed, \textrm{(1)}-\textrm{(2)} from Theorem \ref{t.diff2} only require  constant $r_{[\mu]}$ and $d$. 
\bigskip

In case of regularity the dynamical degree of freedom $d=r-\sum_{i=0}^{\mu-2} \theta_i$ is precisely  the dimension 
of the flow-subspace of the DAE $S_{can}$.  This basic canonical subspace is characterized in different manners in  literature, whereas  we emphasized
\begin{itemize}
	\item $S_{can} = \im C$ (see Section \ref{s.regular}),
	\item $S_{can} = \im \Pi_{can}$ (see Section \ref{subs.tractability}),
	\item $S_{can} =S_{[\mu-1]}$ (see Section \ref{subs.SCFarrays}).
\end{itemize} 
The second canonical subspace $N_{can}$ is defined to be the so-called canonical complement to the flow-subspace,
\[
  S_{can}(t)\oplus N_{can}(t)=\Real,\quad t\in \mathcal I.
\]
Actually $N_{can}$ accommodates important information about the structure of the DAE and the necessary differentiations. This is closely related to the perturbation index. Only an in a way  uniform inner structure of that part of the DAE  which resides in the subspace  $N_{can}$  ensures a uniform perturbation index over the given interval. We know the representations:
\begin{itemize}
	\item $N_{can} = \ker C^*_{adj}E$ (see Proposition \ref{p.adjoint}),
	\item $N_{can} = \ker \Pi_{can}=\ker \Pi_{\mu-1}= N_0+N_1+\cdots+N_{\mu-1}$ (see Section \ref{subs.tractability}).
\end{itemize} 
If the DAE is in standard canonical form
\begin{align}\label{SCFDAEneu}
 \begin{bmatrix}
  I_d&0\\0&N(t)
 \end{bmatrix} x'(t)+
\begin{bmatrix}
  \Omega(t)&0\\0&I_a
 \end{bmatrix} x(t)=q(t),\quad t\in\mathcal I,
\end{align}
where $N$ is strictly upper triangular, then the canonical subspaces are simply
\begin{align*}
 S_{can}=\im \begin{bmatrix}
              I_d\\0
             \end{bmatrix},\quad
N_{can}=\im \begin{bmatrix}
              0\\I_a
             \end{bmatrix},
\end{align*}
and the behaviour of the particular solution components proceeding in $N_{can}$ is  governed by the properties of the matrix function $N$. Clearly, if $N$ is constant, which means that the DAE is in strong standard canoncal form, then the DAE is regular
with index $\mu$ and characteristics $\theta_0\geq\cdots\geq\theta_{\mu-2}>\theta_{\mu-1}=0$, $\mu$ is the nilpotency index of the matrix $N$ and the Jordan normal form of the matrix $N$ shows\footnote{Compare also Remark \ref{r.pencil} and Section \ref{sec:SCF}.}
 \begin{align*}
  &\theta_0 \quad \text{Jordan blocks of order } \geq 2,\\
  &\theta_1 \quad \text{Jordan blocks of order } \geq 3,\quad\\
	&...\\
   &\theta_{\mu-3}\quad \text{Jordan blocks of order }  \geq \mu-1,\\
  &\theta_{\mu-2}\quad \text{Jordan blocks of order } \mu.
 \end{align*}
Eventually, each DAE being transformable into strong standard canonical form is regular, and its characteistics are determined by the structure of the nilpotent matrix $N$. 
This gives the characteristic  values
\[
 r  \text{ and }\quad \theta_0\geq\cdots\geq \theta_{\mu-2}>\theta_{\mu-1}=0
\]

a rather phenomenological background apart from special approaches and the further justification to call them canonical. The latter is all the more important if our following conjecture proves correct.
\medskip

\begin{conjecture}\label{conjecture} Let $E, F:\mathcal I\rightarrow\Real^{m\times m}$ be sufficiently smooth. If the pair $\{E,F\}$ is regular with index $\mu\geq2$ then it is transformable into strong standard canonical form and the Jordan normal form of the matrix $N$ is exactly made of
 \begin{align*}
  &m-r-\theta_0 \quad \text{Jordan blocks of order } 1,\\
  &\theta_0-\theta_1 \quad \text{Jordan blocks of order } 2,\\
  &\theta_1-\theta_2 \quad \text{Jordan blocks of order } 3,\quad \\
	&...\\
   &\theta_{\mu-3}-\theta_{\mu-2}\quad \text{Jordan blocks of order } \mu-1,\\
  &\theta_{\mu-2}\quad \text{Jordan blocks of order } \mu.
 \end{align*}
 \end{conjecture}

\subsection{What is regularity supposed to mean?}\label{subs.regularity}
As we have seen, our formal basic definition of regularity in Section \ref{s.regular} agrees with the view of many other authors. We keep in mind that the characteristic values in all corresponding concepts are derived from specific rank functions and represent constant rank requirements.
The equivalence results then allows the simultaneous application of all the corresponding different concepts. 

We associate regularity of a linear DAE $Ex'+Fx=q$, $E,F:\mathcal I\rightarrow\Real^{m}$, with the following five criteria ensuring a regular flow combined with a homogenous dependency on the input function $q$ on the given interval :
\begin{description}
\item[\textrm{(1)}] The homogenous DAE $Ex'+Fx=0$ has a finite-dimensional solution space $S_{can}(t), t\in\mathcal I,$ with constant dimension $d=\dim S_{can}(t), t\in\mathcal I$, which serves as dynamical degree of freedom of the system.
\item[\textrm{(2)}] The DAE possesses a solution at least for each $q\in \mathcal C^{m}(\mathcal I,\Real^{m})$.
\item[\textrm{(3)}] Regularity, the index $\mu$, and the canonical characteristic values 
\begin{align}\label{char}
 r<m,\; \theta_{0}\geq\cdots\geq\theta_{\mu-2}>\theta_{\mu-1}=0,\; d=r-\sum_{k=0}^{\mu-2}\theta_k,
\end{align}
persist under equivalence transformations.
 \item[\textrm{(4)}] Regularity generalizes the Kronecker structure of regular matrix pencils.
 \item[\textrm{(5)}] If the DAE is regular with canonical characteristic values \eqref{char} on the given interval $\mathcal I$, then its restriction to any subinterval of $\mathcal I$ is regular with the same characteristics.
\end{description}
In other words, a regular linear DAE is characterized by its two canonical subspaces $S_{can}$ and $N_{can}$, both featuring a homogenous structure on the given interval.
\medskip

A point $t_*\in\mathcal I$ is called a \emph{regular point} of the DAE, if there is an open interval $\mathcal  I_*$ containing $t_*$ such that the DAE is regular on $\mathcal I_*\cap\mathcal I$. Otherwise the point $t_*$ is called a \emph{critical point}. 
If the DAE is regular on the interval $\mathcal I$, then all points of  $\mathcal I$ are regular points.
It should not go unmentioned that there are
also other names for this, including \emph{singular point} in \cite[Chapter 4]{RR2008} and \emph{exceptional point} in \cite[Page 80]{KuMe2006}. We prefer the term \emph{critical} because in our opinion it leaves it open whether there are strong singularities in the flow or harmless critical points in solvable systems that only become apparent at rigorously low smoothness properties as pointed out, e.g., in \cite[Section 2.9]{CRR}, see also Definition \ref{d.harmless}. We refer to \cite[Chapter 4]{RR2008} for a careful investigation and classification of points failing to be regular.
\medskip

The \emph{class of regular linear DAEs} comprises exactly all those DAEs that can be transformed equivalently into a structured SCF from Theorem \ref{t.SCF}, that is, the inner algebraic  structure of the subspace $N_{can}(t),\; t\in \mathcal I$, which represents the pointwise canonical complement to the flow-subspace  $S_{can}(t)$, does not vary with time.
Correspondingly, there is a homogenous structure on the whole interval $\mathcal I$  of how the solutions depend on derivatives of the right-hand side $q$. In particular, the perturbation index does not change if one turns to subintervals.
\bigskip

The \emph{class of almost regular linear DAEs} established in Definition \ref{d.almost-reg} coincides with the class of solvable DAEs by Definition \ref{d.solvableDAE}, see also Remark \ref{r.generalform}, and it 
is actually more capacious than the class of regular linear DAEs, which are all solvable, of course. Almost regular systems possess a well-defined differentiation index $\mu^{diff}$ and they satisfy the SF-Hypothesis with $\hat{\mu}=\mu^{diff}-1$. 
They feature a dense set of regular points.  More precisely, they satisfy the above regularity issues {\textrm (1)}, {\textrm (2)}, and {\textrm (4)}. Issue {\textrm (5)} is not valid.  Almost regular  systems allow for so-called harmless critical points which do not affect the flow in sufficiently smooth problems.
Instead of Issue {\textrm (3)} one has merely a flow-subspace $S_{can}(t)$ of constant dimension $d$ and a pointwise canonical complement  $N_{can}(t)$ of constant dimension $a=m-d$, but the inner algebraic structure of the latter is no longer constant. The perturbation index may vary on subintervals.

\subsection{Regularity, accurately stated initial condition, well-posedness, and ill-posedness}\label{subs.posedness}
Let $\{E,F\}$, $E,F:\mathcal I=[a,b]\rightarrow\Real^m$, be a regular pair with index $\mu$ and canonical characteristics $r$ and $\theta_0\geq\cdots\geq\theta_{\mu-2}>\theta_{\mu-1}=0$. The DAE has the dynamical degree of freedom $d$. Obviously, for fixing a special solution from the flow one needs precisely $d$ scalar requirements but also a right way to frame them.
We consider the initial value problem (IVP)
\begin{align}
 Ex'+Fx&=q,\label{IVP1}\\
 Gx(a)&=\gamma,\quad \gamma\in \im G,\label{IVP2}
\end{align}
with sufficiently smooth data and the matrix $G\in\Real^{s\times m}$, $s\geq d$, $\rank G=d$.

The initial condition \eqref{IVP2} for the DAE \eqref{IVP1} is said to be \emph{accurately stated}, e.g. \cite[Section 5]{HaMae2023}, \cite[Definition 2.3]{LMW}, if there is a solution $x_*$,  each IVP with slightly perturbed initial condition,
\begin{align}
 Ex'+Fx&=q,\label{IVP3}\\
 Gx(a)&=\gamma+\varDelta\gamma,\quad \varDelta\gamma\in\im G,\label{IVP4}
\end{align}
has a unique solution $x$, and the  inequality
\begin{align*}
 \max_{t\in [a,b]}|x(t)-x_*(t)|\leq K |\varDelta\gamma|
\end{align*}
is valid with a constant $K$. As pointed out in \cite{HaMae2023}, the two canonical time-varying subspaces, the flow-subspace $S_{can}$ and its canonical complement $N_{can}$, which are for a long time established in the context of the projector based analysis, e.g., \cite{CRR}, are well-defined for regular DAEs and the initial condition \eqref{IVP2} is accurately stated, exactly if 
\begin{align}\label{IVP5}
 \ker G=N_{can}(a).
\end{align} 
This assertion is evident, if one deals with a DAE in SCF, that is,
\begin{align}
 u'+\Omega u&=f,\nonumber\\
 Nv'+v&=g,\label{IVP6}
\end{align}
and
\begin{align*}
 E=\begin{bmatrix}
    I_d&0\\0&N
   \end{bmatrix},\;
F=\begin{bmatrix}
    \Omega&0\\0&I_a
   \end{bmatrix},\; S_{can}= \im \begin{bmatrix}
    I_d\\0
   \end{bmatrix}, \; N_{can}= \im \begin{bmatrix}
    0\\I_a
   \end{bmatrix},\; G=\begin{bmatrix}
    I_d&0
   \end{bmatrix}.
\end{align*}
To achieve easier insight we suppose a constant $N$ in \eqref{IVP6}. Denoting by $P_{N^i}$ a projector matrix along the nullspace of the power $ N^i$, such that $N^i=N^iP_{N^i}$, the unique solution $v_*$ corresponding to $g$ is given by formula
\begin{align}\label{IVP7a}
 v_*=g+\sum_{i=1}^{\mu-1}(-N)^i(P_{N^i}g)^{(i)},
\end{align}
which precisely indicates all involved derivatives of the right-hand side $g$. It becomes evident that each initial requirement for $v_*(a)$ would immediately passed to consistency condition for the righthand side $g$ and its derivatives. Condition \eqref{IVP5} has to prevent this.
\bigskip

If the initial conditions are stated accurately, small changes only have a low effect on the solution. Unfortunately, this can look completely different if the right side of the DAE itself is changed. The smallest changes in $g$ can cause huge amounts in the solution. We take a closer look to the initial value problem with perturbed DAE,
\begin{align}
 Ex'+Fx&=q+\varDelta q,\label{IVP8}\\
 Gx(a)&=\gamma,\quad \gamma\in \im G,\nonumber
\end{align}
and suppose $\ker G=N_{can}$. Again we turn to the SCF, which decomposes the original problem into an IVP for a regular ODE living in $\Real^d$ and an index-$\mu$-DAE with zero-dynamical degree of freedom,
\begin{align}
 u'+\Omega u&=f+\varDelta f,\quad u(a)= \gamma, \nonumber\\
 Nv'+v&=g+\varDelta g.\label{IVP9}
\end{align}
Again, to easier understand what is going on, we suppose a constant $N$. Then the unique solution of \eqref{IVP9} reads
\begin{align}\label{IVP7b}
 v=g+\varDelta g+\sum_{i=1}^{\mu-1}(-N)^i(P_{N^i}(g+\varDelta g))^{(i)}.
\end{align}
According to the traditional definition of Hadamard, an operator equation $Ty=z$ with a linear bounded operator $T$ between Banach spaces $Y$ and $Z$ is called \emph{well-posed} if $T$ is a bijection, and \emph{ill-posed} otherwise. In the well-posed case, there is a unique solution $y\in Y$ to each arbitrary $z\in Z$, and the inverse $T^{-1}$ is bounded, too,
such that, if $z$ tends  to $z_*$ in $Z$, then $y:=T^{-1}z$ tends in $Y$ necessarily to  $y_*:=T^{-1} z_*$, and 
\begin{align*}
 \lVert y-y_*\rVert_{Y} \leq \lVert T^{-1}\rVert_{Z\rightarrow Y} \; \lVert z-z_*\rVert_{Z}.
\end{align*}
Whether a problem is well-posed or ill-posed depends essentially on the choice of the function spaces Y and Z, which, however, should be practicable with respect to applications. It is of no use if errors to be investigated are simply ignored by artificial topologies, see \cite[Chapter 2]{Ma2014} for a discussion in the DAE context. 
\medskip

What about the operator $L:Y\rightarrow Z$ representing the IVP \eqref{IVP1}, \eqref{IVP2}  with 
accurately stated homogeneous initial condition, that is $\gamma=0$. It makes sense to set
\begin{align*}
 Lx&=(Ex)'-E'x+Fx,\quad x\in Y\\
 &Y=\{y\in \mathcal C([a,b],\Real^m): Ey\in \mathcal C^1([a,b],\Real^m), Gy(a)=0\},\\
 &Z= \mathcal C([a,b],\Real^m).
\end{align*}
If the DAE has index $\mu=1$, then $N_{can}=\ker G=\ker E$, in turn
\begin{align*}
 Y=\{y\in \mathcal C([a,b],\Real^m): Ey\in \mathcal C^1([a,b],\Real^m), Ey(a)=0\},
\end{align*}
and $L$ is actually a bijection, e.g., \cite{GM86,CRR,Ma2014}. A particular result is then the inequality
\begin{align}\label{IVP10}
  \max_{t\in [a,b]}|x(t)-x_*(t)|\leq M \max_{t\in [a,b]}|\varDelta q(t)|,
\end{align}
for $x_*$ and $x$ corresponding to $q$ and $q+\varDelta q$, respectively.
For higher-index cases, that is $\mu\geq 2$ and $N\neq 0$ in the SCF, the situation is more complex. 
Then the operator $L$ is by no means surjective anymore. Its range $\im L\subset Z$ is a proper, nonclosed subset so that $L$ is not a Fredholm operator either. This can be recognized by the simplified case $E=N, F=I$, $d=0$, $x=v$. Regarding that 
\begin{align*}
\rank P_{N} =\rank N=\sum_{j=0}^{\mu-1}\theta_j,\quad \rank P_{N^i} = \rank N^i=\sum_{j=i-1}^{\mu-1}\theta_j,\quad i=2,\ldots,\mu,
\end{align*}
formula \eqref{IVP7a} leads to the representation
\begin{align*}
 \im L=\{g\in\mathcal C([a,b],\Real^m): P_{N^i}g\in\mathcal C^i([a,b],\Real^m),\; i=1,\ldots , \mu-1\},
\end{align*}
which rigorously describes in detail all involved derivatives. This representation also reveals that the DAE index is only one important aspect, but that the exact structure can only be described by all the canonical characteristic values together. Moreover,
from \eqref{IVP7a}, \eqref{IVP7b} it results that 
 \begin{align*}
 v-v_*=\varDelta g+\sum_{i=1}^{\mu-1}(-N)^i(P_{N^i}(\varDelta g))^{(i)},
\end{align*}
which makes an inequality like \eqref{IVP10} impossible.
\bigskip

We emphasize that an IVP \eqref{IVP1}, \eqref{IVP2} with accurate initial conditions is a well-posed problem in this setting only in the index-one case. In the case of a DAE \eqref{IVP1} with a higher index $\mu\geq2$, an ill-posed problem generally arises. Thus, a higher-index regular DAE always has a twofold character, it is on the one hand a well-behaved dynamic system and on the other hand an ill-posed input-output problem.

To give an impression of this ill-posedness, we will mention a  
very simple example elaborated in \cite[Example 2.3]{Ma2014}, in which $E$ and $F$ are constant $5\times 5$ matrices, the matrix pencil is regular with index four, the input is $q=0$, the corresponding output is $x_*=0$, 
the perturbation $\varDelta q$  has size $\epsilon n^{-1}$ and tends in $Z$ to zero for $n$ tending to $\infty$, but the corresponding difference $x-x_*$ shows size $\epsilon n^2$, growths unboundedly for increasing $n$, and tends in $Y$  by no means to zero.

For more profound mathematical analyses, for instance in \cite{HMT} to provide instability thresholds, individual  topologies specially adapted to $\im L$ can be useful. 
Then one enforce a bounded bijection $L:Y\rightarrow \tilde Z$  by setting $\tilde Z=\im L$ and equipping $\tilde Z$ with an appropriate norm to become a Banach space, however, you need a precise description of $\im L$ for that. Of course, in this decidedly peculiar setting the problem becomes well-posed in this solely  theoretical sense \cite{Ma2014}.
\bigskip

At this point it must also be mentioned that in some areas of mathematics the question of continuous dependencies is completely ignored and yet the term \emph{formally well-posed} is used, e.g. \cite[p.\ 298]{SeilerZerz}. 
Unfortunately, this can lead to considerable misunderstandings, as the above mentioned simple example \cite[Example 2.3]{Ma2014} shows already. 

The aim of \cite{SeilerZerz} is to make results from the algebraic theory of linear systems usable for DAEs, in particular methods of symbolic computation together with the theory of Gr{\"o}bner bases to provide formally well-posed problems. 
Among others it is figured out with a size-three Hessenberg DAE that the first Gr{\"o}bner index $\gamma_1$ recovers the strangeness index $\mu^S$ and the differentiation index by $\gamma_1=\mu^{diff}-1$. 
Moreover, \cite[Proposition 5.3]{SeilerZerz} provides the relation $\mu_p\leq\gamma_1+1$ as upper bound of the perturbation index for DAEs being not under-determined. 
Clearly, it is well-known that $\mu_p=\mu=\mu^T=\mu^S+1$  for regular linear DAEs, and $\mu_p=\mu^{diff}$ for DAEs being solvable in the sense of Definition \ref{d.solvableDAE}. However, for DAEs being over-determined the differentiation index is not defined, but strangeness index and  tractability index are quite different, \cite{HM2020,CRR}. It seems to be open whether any of them is recovered by an Gr{\"o}bner index.

\subsection{Other views on regularity that we do not adopt}\label{subs.views}
In early work, when less was known about DAEs, regularity was associated with special technical requirements.
But there are also different views on the matter in current research. We pick out just a few of the different versions.

At the beginning it was assumed that the so-called local pencils $\lambda E(t)+F(t)$,  for fixed $t$, and their Kronecker index were relevant for the characterization not only of DAEs with constant coefficients. The associated term is the so-called \emph{local index}. In contrast, then the further index terms were given the suffix \emph{global}. Today it goes without saying that our index terms in this sense are of a global nature, i.e. related to an interval, and we avoid the \textit{global} suffix. 
\bigskip

\textbullet\quad In the famous monograph \cite[Page 23]{BCP89} regularity of the DAE $Ex'+Fx=q$ is considered as regularity of the associated \emph{local matrix pencils} $\lambda E(t)+F(t), t\in \mathcal I$.
The intention behind this is to obtain feasible numerical integration methods. However, it is ibidem pointed out that this regularity does not imply solvability in the sense of Definition \ref{d.solvableDAE}.
For instance, the DAE with coefficients
\begin{align*}
 E(t)=\begin{bmatrix}
       -t&t^2\\-1&t
      \end{bmatrix},\quad 
F(t)=\begin{bmatrix}
       1&0\\0&1
      \end{bmatrix},\quad t\in\mathcal I=[-1,1],
\end{align*}
features solely regular local pencils, $\det(\lambda E(t)+F(t))=1, t\in \mathcal I$. However, it can be checked that, given an arbitrary smooth function $\gamma:\mathcal I\rightarrow \Real$, 
\[
x_*(t)=\gamma(t)\begin{bmatrix}t\\1\end{bmatrix}, \; t\in\mathcal I 
\]
solves the homogenous DAE. This specific pair $\{E,F\}$ is pre-regular ($r=1, \theta=1$) but not regular in our context, since the next pair $\{E_1,F_1\}$ is no longer pre-regular, $\im[E_1 \, F_1]=\{0\}\neq \Real^r$.

Apart from the inappropriate definition of regularity, the focus at the time on numerical methods bore many fruit and, with the exclamation "Differential/algebraic equations are not ODEs", \cite{Petzold1982}, generated a great deal of interest in DAEs.
\bigskip                                                                                                                   
  
\textbullet\quad With a completely different intention, namely to obtain theoretical solvability statements, in the  monographs \cite{Boya1980,BDLCH1989} it is assumed  for regularity, among other things, that there is a number $c$ such that $cE(t)+F(t)$ is non-singular for all $t \in \mathcal I$. Even for the pair
\begin{align*}
 E(t)=\begin{bmatrix}
       1&t\\0&0
      \end{bmatrix},\quad 
F(t)=\begin{bmatrix}
       0&0\\1&t
      \end{bmatrix},
\end{align*}
that we recognize today as regular with index two and characteristic $r=1, \theta_0=1, \theta_1=0$, this requirement is obviously not met. Nonetheless, in \cite[Chapter 5]{Boya1980}, \cite[Chapter 2]{BDLCH1989} devoted to this kind of regular system  some statements about solvability are provided. However, these results are very restricted by the then current working tools such as the Drazin inverse and the like.
\bigskip

\textbullet\quad In \cite{BergerIlchmann} regularity is understood to mean equivalent transformability to SCF. The class of DAEs that can be transformed into SCF is more comprehensive than the regular DAEs defined by Definition \ref{d.2}, and also includes DAEs with harmless critical points, which we have excluded for good reason. On the other hand, the class of DAEs being transformable into an SCF is only a subclass of almost regular DAEs according to the Definition \ref{d.almost-reg}, which contains DAEs featuring further harmless critical points.
\bigskip

\textbullet\quad We quote from \cite[p.\ 154]{KuMe2006}, in which Hypothesis 3.48 is the local name for the SF-Hypothesis \ref{SHyp}:  

``...Hypothesis 3.48 is the correct way to define a regular differential-algebraic equation. Regularity here is to be understood as follows.  A linear problem satisfying Hypothesis 3.48 fixes a strangeness-free differential-algebraic equation.... The underlying differential-algebraic operator ... with appropriately chosen spaces, is invertible and the inverse is continuous.''

This definition de facto declares the solvable DAEs to be regular, against what we already argued in the Subsection \ref{subs.regularity}. What is more, the attempt to justify this is somewhat confusing. Surely, strangeness-free DAEs, i.e. index-zero or index-one problems, are well-posed operator equations in natural spaces, which is well known at least through \cite{GM86,CRR,Ma2014}, see also Subsection \ref{subs.posedness} above. In detail, the original DAE $Ex'+Fx=q$ which satisfies the SF-Hypothesis\footnote{See Subsection \ref{subs.Hyp}} is remodelled to
\begin{align*}
 Y^*Ex'+Y^*Fx&=Y^*q,\\
 Z^*\mathcal F_{[\nu-1]}x&=Z^*q_{[\nu-1]}=:p.
\end{align*}
But to analyze the original equation, perturbations of the right-hand side $q+\varDelta q$ are appropriate, which leads to $p+Z^*\varDelta q_{[\nu-1]}$ within the transformed DAE. This fact is not recognized and only continuous functions $\varDelta p$ are applied, which actually hides the DAE structure.
In the context of local statements on nonlinear DAEs, e.g. \cite[p.\ 164]{KuMe2006} this seems to be even more critically.
\bigskip

\textbullet\quad  Recently, in the textbook \cite{KuMe2024}, that is the second edition of \cite{KuMe2006}, the following regularity definition (that we adapted to our notation) is emphasized as central notion.\\

\cite[Definition 3.3.1]{KuMe2024}: The pair $\{E,F\}$ and the corresponding DAE $Ex'+Fx=q$ are called regular if
\begin{enumerate}
	\item the DAE is solvable for every sufficiently
 smooth $q$ ,
 \item  the solution is unique for every $t_0 $ in a compact interval $\mathcal I$ and every consistent initial condition
 $x_0 \in \R^m$\footnote{ In \cite{KuMe2024} $x_0 \in \C^m$ is considered.}  given at $t_0\in \mathcal I$,
 \item the solution depends smoothly on $q$ , $t_0$, and $x_0$.
\end{enumerate}
 
The items (1) and (2) capture the solvable systems, see Definition \ref{d.solvableDAE}, and almost regular DAEs, see Definition \ref{d.almost-reg}, since harmless critical points are allowed. However, item (3) is inappropriate:
\begin{itemize}
\item On the one hand, for DAEs continuous dependency or even smooth dependency on consistent values $x_0$ cannot be assumed in general, not even for index-one DAEs. For arbitrary perturbations $\Delta_0 \in \R^m$, the vector $x_0+\Delta_0$  is not necessarily consistent. Indeed, the perturbation $\Delta_0$ has to be restricted accordingly such that $\Delta_0 \in S_{can}(t_0)$.  
\item On the other hand, as sketched in Section \ref{subs.posedness}, starting from corresponding spaces of continuous and differentiable functions, continuous dependence of the DAE solutions on the inhomogeneity $q$ is exclusively given for index-one problems (e.g.\ \cite{GM86}, \cite{CRR}, \cite{Ma2014}). For higher-index DAEs, surjectivity needs sophisticated, strongly problem-specific function spaces, see. \cite{Ma2014}.  
\end{itemize}

\textbullet\quad In the monograph devoted to algebro-differential operators having finite-dimensional nullspaces \cite{Chist1996} DAEs are organized in the framework of a ring $\mathcal M_
A$  of linear differential and integral operators acting in $\mathcal C^{\infty}$. In this context, the first order operator $\Lambda_{1}x=Ex'+Fx$ is called \emph{regular} if $E(t)=I$ on the given interval. If it exists, the minimal order $\nu$ of a differential operator $\Lambda_{\nu}$ belonging to the ring $\mathcal M_A$,  which serves as \emph{left regularization operator} for $\Lambda_1$  such that $\Lambda_{\nu}\circ \Lambda_{1}$ is regular, is called \emph{non-resolvedness index} of the operator $\Lambda_{1}$, see also \ref{r.Chist}. This is nothing else than the differentiation index.


%
\section{About nonlinear DAEs}\label{s.nonlinearDAEs}
There are numerous important studies of nonlinear DAEs that are based on special structural requirements, in particular, such as for simulating multibody system dynamics and circuit modeling, 
which are not to be reflected here in detail. We refer to \cite{Arnold,Riaza} for overviews.
Here we look solely at general unstructured non-autonomous DAEs
\begin{align}\label{N.0}
 f(t,x(t),x'(t))=0,\quad t\in \mathcal I,
\end{align}
given by a sufficiently smooth function $f:\mathcal I_f\times \mathcal D_{f}\times \Real^m\rightarrow \Real^m $,\; $\mathcal I_f\times \mathcal D_f\subseteq \Real\times\Real^m $ open,
 and ask about their general properties.
 The only exception to the restriction of the structure, which is allowed in some cases below, is the assumption that  the subspace $\ker D_yf(t,x,y)$ is independent of the variables $y$ or $(x,y)$. If circumstances require, a change to the augmented form 
 \begin{align*}
   f(t,x(t),y(t))=0,\; y(t)-x'(t)=0 \quad t\in \mathcal I.
 \end{align*}
 can be made\footnote{However, this is accompanied by an increase in the index!}.
 
As we have already seen with linear DAEs, certain subspaces play a crucial role, which will also be the case here. However,
 while in the linear case the subspaces only depend on one parameter, namely $t\in\Real $, there are now the parameters $(t,x)\in \Real^{m+1}$ and more. These subspaces are handled by means of both projector functions and base functions. 
 If such a subspace depends only on $t$ varying on a given interval $\mathcal I$ and has there constant dimension, then both smooth basic functions and smooth projector functions defined on the entire interval are available. 
 In contrast, if a subspace depends on a variable $z\in \Real^n$, $n\geq 2$, and has constant dimension, then there are globally defined smooth projector functions, but smooth base functions only exist locally, e.g., \cite[Sections A3, A4]{CRR}.
 This must be taken into account. It facilitates the use of projector functions, which are usually more difficult to determine in practice, and it makes the practically often easier way of 
 dealing with bases in theory more complicated.
 \bigskip
 
It should be recalled that, in the case of nonlinear ODEs, the uniqueness of the solutions is no longer guaranteed with merely continuous vector fields.   A  vector field of class $\mathcal C^1$ is locally Lipschitz and hence ensures uniqueness.
The same applies to vector fields on smooth manifolds.
This must be taken into account when it comes to  a regularity notion, which should include uniqueness of solutions to the corresponding initial value problems.

We can only give a rough sketch of the approaches for nonlinear DAE and confine ourselves to the rank conditions used in each case. 
\medskip

In the present section, in order to achieve better clarity in all of the very different following approaches, we use the descriptions $g'(t)$ but also $\frac{\rm d}{\rm dt}g(t)$ for the derivatives of a function $g(t)$ of the independent variable $t\in \Real$.
For the partial derivatives of a function $g(u,v,t)$ depending on several independent variables $u\in\Real^n,v\in\Real^k,t\in\Real$ we apply the description $g_u(u,v,t)$, $g_v(u,v,t)$, $g_t(u,v,t)$, but also 
$D_ug(u,v,t)$, $D_vg(u,v,t)$, $D_tg(u,v,t)$.

\subsection{Approaches by means of derivative arrays}\label{subs.arrays}
The approaches via derivative arrays and geometric reduction procedures have been developed for nonlinear DAEs almost from the beginning \cite{Gear88,BCP89,Griep92,Reich,RaRh,KuMe2006,ChistShch,EstLam2016Decoupling}.  To treat the DAE
\begin{align}\label{N.1}
 f(t,x(t),x'(t))=0,\quad t\in \mathcal I,
\end{align}
on $\mathcal I\times \mathcal D\times \Real^m \subseteq \mathcal I_f\times \mathcal D_{f}\times \Real^m$ 
, one forms the derivative array functions  \cite{Gear88,BCP89,Griep92,KuMe2006,ChistShch}
\begin{align}\label{N.2}
 \mathfrak F_{[k]}(t,x,\underbrace{x^1,\cdots,x^{k+1}}_{y_{[k]}})=\begin{bmatrix}
                      f_{[0]}(t,x,x^1)\\f_{[1]}(t,x,x^1,x^2)\\ \vdots \\f_{[k]}(t,x,x^1,\cdots,x^{k+1})
                     \end{bmatrix}\in \Real^{(k+1)m}, \\
 \text{for} \quad (t,x)\in \mathcal I\times \mathcal D,\quad
  \begin{bmatrix}
     x^1\\\vdots\\x^{k+1}
    \end{bmatrix}=:y_{[k]}
                 \in \Real^{km+m},\quad   k\geq 0,  \nonumber
\end{align}
in which
 \begin{align*}                    
  f_{[0]}(t,x,x^1)&=f(t,x,x^1),\\
  f_{[j]}(t,x,\underbrace{x^1,\cdots,x^{j+1}}_{y_{[j]}})&={f_{[j-1]}}'_t(t,x,\underbrace{x^1,\cdots,x^{j}}_{y_{[j-1]}})+{f_{[j-1]}}'_x(t,x,y_{[j-1]})x^1+\sum_{i=1}^{j}{f_{[j-1]}}'_{x^i}(t,x,y_{[j-1]})x^{i+1},
\end{align*}
such that, for each arbitrary smooth reference function $x_*:\mathcal I\rightarrow \Real^m$ whose graph runs in  $\mathcal I\times \mathcal D$ it results that 
\begin{align*}
 \mathfrak F_{[k]}(t,x_*(t),\underbrace{x_*^{(1)}(t),\cdots,x_*^{(k+1)}(t)}_{x_{*[k]}'})&=\begin{bmatrix}
                       f(t,x_*(t),x_*'(t))\\\frac{{\rm d}}{{\rm d}t}f(t,x_*(t),x_*'(t))\\ \vdots \\\frac{{\rm d}^{k}}{{\rm d}t^{k}} f(t,x_*(t),x_*'(t))
                     \end{bmatrix}.
\end{align*}
By construction the sets $\mathfrak L_{[{k}]}$, 
\begin{align}\label{N.L}
 \mathfrak L_{[{k}]}=\{(t,x,y_{[k]})\in \mathcal I\times \mathcal D\times\Real^{(k+1)m}:\mathfrak F_{[k]}(t,x,y_{[k]})=0\},\quad k\geq 0,
\end{align}
house the extended graphs of all smooth solutions $x_{*}:\mathcal I\rightarrow \mathcal D$ of the DAE \eqref{N.1}.
The partial Jacobians of the matrix function $\mathfrak F_{[k]}$ with respect to $y_{[k]}$ and to $x$ show the  structure, 
\begin{align*}
 \mathcal E_{[k]}= 
 D_{y_{[k]}}\mathfrak F_{[k]}
 =\begin{bmatrix}
                                                  f_{x^1}&&&\\
                                                  *&f_{x^1}&&\\
                                                  &&\ddots&\\
                                                  *&\cdots&*&f_{x^1}
                                                 \end{bmatrix},\quad
 \mathcal F_{[k]}= D_{x}{\mathfrak F_{[k]}} =\begin{bmatrix}
                                                  f_{x}\\
                                                  *\\
                                                  \vdots\\
                                                  *
                                                 \end{bmatrix}.
\end{align*}
Using again the arbitrary reference function $x_*$ being not necessarily a solution and denoting 
\begin{align*}
 E_*(t)=f_{x^1}(t,x_*(t),x_*'(t)),\;  F_*(t)=f_{x}(t,x_*(t),x_*'(t)),\; t\in\mathcal I, 
\end{align*}
one arrives at  matrix functions as introduced by \eqref{1.GkLR} in Section \ref{subs.Preliminaries}, namely
\begin{align*}
 \mathcal E_{*[k]}(t)&=\mathcal E_{[k]}(t,x_*(t),x'_{*[k]}(t))=\begin{bmatrix}
                                                  E_*(t)&&&\\
                                                  *&E_*(t)&&\\
                                                  &&\ddots&\\
                                                  *&\cdots&*&E_*(t)                                                 \end{bmatrix}, \\
  \mathcal F_{*[k]}(t)&=\mathcal F_{[k]}(t,x_*(t),x'_{*[k]}(t))=\begin{bmatrix}
                                                  F_*(t)\\
                                                  *\\
                                                  \vdots\\
                                                  *                                                 \end{bmatrix}.                                                 
\end{align*}
This opens up the option of using linearization for handling and tracing back questions to the linear case.
Here too, as in the linear case, there are different views on rank conditions for the Jacobians. 
As we will see below, again, in the concepts of the (standard) differentiation index and of the strangeness index there is no need for the Jacobians $\mathcal E_{[k]}$ with lower $k$ to have constant rank. 
In contrast, for the regular differentiation index and projector based differentiation index, each of these Jacobians is explicitly supposed to have constant rank which allows to use parts of the geometric operation equipment such as manifolds, tangent bundles etc.
\bigskip

\subsubsection{Differentiation index}

What we now call the differentiation index was originally simply called the index and was introduced into the discussion in \cite[p.\ 39]{Gear88} as follows:
\textit{Consider  \eqref{N.2} as a system of equations  in the \emph{separate} dependent variables $x^1,\ldots, x^{k+1}$, and solve for these variables as functions of $x$ and $t$ considered as \emph{independent} variables. 
If it is possible to solve  for $x^1$ for some finite $k$, then the index, $\mu$, is defined as smallest $k$ for which \eqref{N.2} can be solved for $x^1(x,t)$.} 
We quote the corresponding definition \cite[Definition 2.5.1]{BCP89}  that incorporates this idea:
\begin{definition}\label{d.N1}
 The \emph{index} $\nu$ of the DAE \eqref{N.1} is the smallest integer $\nu$ such that $\mathfrak F_{[\nu]}$ uniquely determines the variable $x^1$ as a continuous function of $(x,t)$.
\end{definition}
Unfortunately, this definition is rather vague, which triggered a lively discussion at the time.
There were subsequently a series of attempts at a more precise definition partly with a variety of new terms, see e.g. \cite{CaGear95}. We will come back to this below when dealing with the perturbation index, see also Example \ref{e.simeon}. 
It should also not go unmentioned that what we call \emph{regular} differentiation index below was also simply called index in \cite{Griep92}, and it was explicitly intended as an adjustment of the index notion in \cite{Gear88}.

We underline that in \cite{BCP89} the above definition \cite[Definition 2.5.1]{BCP89} is immediately followed by a proposition \cite[Proposition 2.5.1]{BCP89} from which we learn that the matter of the standard differentiation index of nonlinear DAEs can be traced back to properties of the partial Jacobians as follows:
\begin{proposition}\label{p.N1}\cite{BCP89}
 Sufficient conditions for
 \begin{align}\label{N.3}
 \mathfrak F_{[k]}(t,x,y_{[k]})=0
 \end{align}
to uniquely determine $x^1$ as a continuous function of $(x,t)$ are that the Jacobian matrix of $ \mathfrak F_{[k]}$ with respect to  $y_{[k]}$  is $1-$full with constant rank and \eqref{N.3} is consistent.
\end{proposition}

In the meantime, this has become established as one of the possible definitions being at the same time the straightforward generalization of Definition \ref{d.diff}:
\begin{definition}\label{d.N1a}
 The \emph{index} $\nu$ of the DAE \eqref{N.1} is the smallest integer $\nu$ such that $\mathcal E_{[\nu]}$ is $1-$full with constant rank and \eqref{N.3}  is consistent.
\end{definition}
\bigskip

\subsubsection{Strangeness index}

There is also a straightforward generalization of the SF-Hypothesis \ref{SHyp} with the strangeness index.
We adapt \cite[Hypothesis 4.2]{KuMe2006} and \cite[Definition 4.4]{KuMe2006} to our notation.

\begin{hypothesis}[\textbf{Strangeness-Free-Hypothesis for nonlinear DAEs}]\label{SHypN}
There exist integers $\hat \mu, \hat a$, and $\hat d=m-\hat a$ such that the set
\begin{align*}
 \mathfrak L_{[\hat{\mu}]}=\{(t,x,y_{[\hat{\mu}]})\in \mathcal I\times \mathcal D\times\Real^{(\hat{\mu}+1)m}:\mathfrak F_{[\hat{\mu}]}(t,x,y_{[\hat{\mu}]})=0\}
\end{align*}
associated with $f$ is nonempty and such that for every $(t,x,y_{[\hat{\mu}]})\in \mathfrak L_{[\hat{\mu}]}$ there exist a (sufficiently small) 
neighborhood in $\mathcal I_f\times \mathcal D_f\times\Real^{(\hat{\mu}+1)m}$ in which the following properties hold:
\begin{description}
 \item[\textrm{(1)}] We have $\rank \mathcal E_{[\hat\mu]}(t,x,y_{[\hat\mu]})=(\hat\mu+1)m-\hat a$ on $\mathfrak L_{[\hat{\mu}]}$
 such that there is a smooth matrix function $Z$ of size $((\hat\mu+1)m)\times\hat a$ and pointwise maximal rank, satisfying
 $Z^*\mathcal E_{[\hat\mu]}=0$ on $\mathfrak L_{[\hat{\mu}]}$.
\item[\textrm{(2)}] We have $\rank Z^*(t,x,y_{[\hat\mu]})\mathcal F_{[\hat\mu]}(t,x,y_{[\hat\mu]})=\hat a$ such that there exists a smooth matrix function $C$ of size $m\times\hat d$, 
and pointwise maximal rank, satisfying $Z^*\mathcal E_{[\hat\mu]}C=0$.
\item[\textrm{(3)}] We have $\rank f'_{x^1}(t,x,x^1)C(t,x,y_{[\hat\mu]})=\hat d$ such that there exists  a smooth matrix function $Y$ of size $m\times\hat d$ and pointwise maximal rank, satisfying  $\rank Y^* f'_{x^1}C=\hat d$.
\end{description}
\end{hypothesis}

\begin{definition}\label{d.HypStrangenessN}
Given a DAE as in \eqref{N.1}, the smallest value of $\bar \mu$ such that $f$ satisfies the  SF-Hypothesis \ref{SHypN} is satisfied is called the \emph{strangeness index} of \eqref{N.1}. 
If $\bar \mu=0$ then the DAE is called strangeness-free.
\end{definition}
\bigskip

\subsubsection{Regular differentiation index}
Next we follow the  ideas of \cite{Griep92} to a nonlinear version of the regular differentiation index discribed for linear DAEs in Subsection \ref{subs.qdiff}, 
which will turn out to be closely related to the geometric reduction and thus to the geometric index.
Suppose that the partial Jacobians $\mathcal E_{[k]}$  have constant rank for all $k\geq 0$. The set
\begin{align*}
 \tilde C_{[k]}=\{ (t,x)\in \mathcal I\times\mathcal D: \exists  y_{[k]}\in\Real^{km+m},\mathfrak F_{[k]}(t,x,y_{[k]})=0 \}
\end{align*}
is called \emph{constraint manifold of order $k$}, and to each $(t,x)\in \tilde C_{[k]}$ one obtains the manifold $M_{[k]}(t,x)$ and its tangent space given by
\begin{align*}
 M_{[k]}(t,x)=\{y_{[k]}\in \Real^{km+m}: \mathfrak F_{[k]}(t,x,y_{[k]})=0 \},\; TM_{[k]}(t,x; y_{[k]})=\ker \mathcal E_{[k]}(t,x,y_{[k]}).
\end{align*}
The following generalizes Definition \ref{d.G1} via linearization.
\begin{definition}\label{d.regulardiff}
 The DAE \eqref{N.1} has \emph{regular differentiation index} $\nu$ if the partial Jacobians $\mathcal E_{[k]}$  have constant rank for $k\geq 0$,  $\tilde C_{[\nu]}$ is non-empty,  
 $T_{[\nu]} M_{[\nu]}(t,x)$\footnote{As in Section \ref{subs.qdiff} $T_{[\nu]}=[I_m 0\cdots 0]\in \Real^{m\times(m+m\nu)}$ is a truncation matrix.} is a singleton for each $(t,x)\in \tilde C_{[\nu]}$, 
 and $\nu$ is the smallest such integer.
\end{definition}
We note that, analogously to Subsection \ref{subs.qdiff}, $T_{[\nu]} M_{[\nu]}(t,x)$ is a singleton, if and only if $T_{[\nu]}\ker  \mathcal E_{[\nu]}(t,x,y_{[\nu]})=0$ for all $y_{[\nu]}\in M_{[\nu]}$.

The main intention in \cite{Griep92} is  giving index reduction procedures a rigorous background. In particular, owing to \cite[Theorem 16]{Griep92}, the transfer from the DAE \eqref{N.1} to
\begin{align}\label{N.5}
 (I-&W(t,x(t),x'(t)))f(t,x(t),x'(t)) \nonumber \\
 +&W(t,x(t),x'(t)){D_xf(t,x(t),x'(t))x'(t)+D_tf(t,x(t),x'(t))}=0,\quad t\in \mathcal I, 
\end{align}
subject to the initial restriction $f(t_0,x(t_0),x'(t_0))=0$, reduces the (regular differentiation) index by one. Thereby, $W(t,x,y)$ denotes the orthoprojector function along $\im D_yf(t,x,y)$. 
\medskip

Finally, in this segment it should be mentioned that in \cite{ChenTrenn} for autonomous quasi-linear DAEs (with $\mathcal C^{\infty}$ functions) the version of the (regular) differentiation index from \cite{Griep92}  was recast in a rigorous geometric language and shown to be consistent with the geometric index, cf. Remark \ref{r.RaRh} below.
\bigskip

\subsubsection{Projector-based differentiation index}

Next we turn to the concept associated with the projector based  differentiation index. Supposing that the nullspace $\ker f_{x^1}(t,x,x^1)$ is actually independent of the variables $(x,x^1)$ and 
does not change its dimension. With the orthoprojector $P(t)$ along $\ker f_{x^1}(t,x,x^1)$ it results that
\begin{align*}
 f(t,x,x^1)=f(t,x,P(t)x^1),\quad (t,x,x^1)\in \mathcal I\times \mathcal D\times\Real^m, 
\end{align*}
and the DAE \eqref{N.1} rewrites to
\begin{align*}
 f(t,x(t),(Px)'(t)-P(t)x(t))=0,\quad t\in \mathcal I.
\end{align*}
This makes clear that the given DAE accommodates an equation  $(Px)'(t)=\phi(x(t),t)$. The idea now is to extract only the remaining component $(I-P(t))x$ in terms of $(P(t)x,t)$ from the derivative array.
 As in the linear case, the further matrix functions $\mathcal B_{[k]}:\mathcal I\rightarrow \Real^{(mk+m)\times(mk+m)}$,
\begin{align*}
\mathcal{B}_{[k]} =
\begin{bmatrix}
	  P & 0 \\
	\mathcal F_{[k-1]}& \mathcal E_{[k-1]}
\end{bmatrix}
\end{align*}
play their role here.
\begin{definition}\label{d.pbdiff}
 The DAE \eqref{N.1} has \emph{projector-based differentiation index} $\nu$ if the matrix functions $\mathcal B_{[k]}$  have constant rank for $k\geq 0$, $\nu$ is the smallest integer such that $\mathcal B_{[\nu]}$ is $1$-full.
\end{definition}

In contrast to Definition \ref{d.N1a}, we do not assume the consistency of \eqref{N.3} in the above definition. Nevertheless, to compute consistent initial values, of course
\begin{align*}
 \mathfrak F_{[\nu-1]}(t,x,y_{[\nu-1]})=0
 \end{align*}
 has to be consistent.

\bigskip

\subsection{Geometric reduction}\label{subs.nonlinearDAEsGeo}

The geometric reduction procedures are intended right from the start for nonlinear DAEs 
\cite{Rhe84,Reich,Griep92,RR2008}. While \cite{Griep92}, see Definition \ref{d.regulardiff} above, still uses 
a rather analytical representation with the rank theorem as background, \cite{Reich} and \cite{Rhe84,RaRh} use the means of geometry more consistently. 
The  explicit rank conditions from \cite{Griep92} become inherent components of the corresponding terms, 
the rank theorem is replaced by the subimmersion theorem etc.\footnote{We recommend \cite[Section 3.3]{RR2008} for a nice roundup.} 
A clear presentation of the issue for autonomous DAEs is given in \cite{RaRh}.
Here we follow the lines of \cite{Reich} whose depiction of nonautonomous  DAEs fits best with the rest of the material in our treatise.
\medskip

The stated intention of \cite{Reich} is to elaborate a concept of regularity for general non-autonomous DAEs \eqref{N.1} considered on the open connected set $\mathcal I\times \mathcal D\times \Real^m$.

Recall some standard terminology for this aim. For a $\rho$-dimensional differentiable manifold $M$ we consider the tangent bundle $TM$ of $M$ and  the tangent space $T_zM$ of $M$ at $z\in M$.  
We deal with manifolds $M$ being embedded in $\Real\times\Real^m$, that is, 
sub-manifolds of $\Real\times\Real^m$, and we can accordingly assume that $T_zM$ has been identified with a $\rho$-dimensional linear subspace of $\Real\times\Real^m$.

Denote by $\pi_1:\Real\times\Real^m\rightarrow \Real$ the projection onto the first factor in $\Real\times\Real^m$, and let $\mathcal J$ be the open set of $\Real$ with $\pi_1(M)=\mathcal J$. 
Introduce further the restriction $\pi:M\rightarrow\mathcal J$ of $\pi_1$ to $M$, that is $\pi=\pi_1|M$.
Then the tripel $(M,\pi,\mathcal J)$ is a sub-bundle of $\Real\times\Real^m$, if $\pi(T_{(t,x)}M)=\Real$ for all $(t,x)\in M$. 
The manifolds $M(t)\subset\Real^m$ defined by $\{t\}\times M(t)=(M\cap\{t\}\times \Real^m )$  are called the fibres of $M$ at $t\in\mathcal J$.
\bigskip

The origin of the following regularity notion is \cite[Definition 2]{Reich}. We have included the explicit requirement for $\mathcal C^1$ classes to ensure the uniqueness of solutions to initial value problems. 
According to the context, we believe that this was originally intended.
\begin{definition}\label{d.regularReich}
The DAE  \eqref{N.1} is called \emph{regular} if there is a unique sub-bundle $(\mathcal C,\pi,\mathcal I)$  of $\Real\times\Real^m$ and a unique vectorfield $v:\mathcal C\rightarrow \Real^m$ on $\mathcal C$, both of class $\mathcal C^1$, such that a differentiable mapping $x:I\subset \mathcal I\rightarrow \Real^m$ is a solution of the vectorfield $v$ if and only if $x$ is a solution of the given DAE.

The manifold $\mathcal C$ is called \emph{configuration space} and $v$ the \emph{corresponding vector field}.
\end{definition}

A technique is stated in \cite{Reich} by means of which the configuration space and the corresponding vector field for a given DAE can be obtained. In more detail, one starts from the set 
\begin{align*}
  L =\{(t,x,p)\in \mathcal I\times \mathcal D\times\Real^m: f(t,x,p)=0\}.
\end{align*}
A differentiable map $x:I\subseteq\mathcal I\rightarrow \Real^m$ is a solution of the DAE if and only if 
\begin{align*}
 (t,x(t),x'(t))\in L\quad\text{for all}\;  t\in I.
\end{align*}
We form the new set
\begin{align*}
 \mathcal C_1=\pi_{1,2}(L)\subseteq \Real\times\Real^m,
\end{align*}
where $\pi_{1,2}: \Real\times \Real^m\times \Real^m\rightarrow \Real\times \Real^m$ is the projection onto the first two factors in $\Real\times \Real^m\times \Real^m$. This set reflects algebraic constraints on the solution of the DAE. 
If the triple $(\mathcal C_1,\pi,\mathcal I)$ is a differentiable sub-bundle  of $\Real\times\Real^m$, then the differentiable map $x:I\rightarrow \Real^m$ is a solution of the DAE if and only if
\begin{align*}
 (t,x(t),x'(t))\in L\cap S\mathcal C_1\subseteq L\quad\text{for all}\;  t\in I.
\end{align*}
Thereby, $S\mathcal C_1=\bigcup_{(t,x)\in \mathcal C_1}\{(t,x)\}\times S_{(t,x)}\mathcal C_1$ and $S_{(t,x)}\mathcal C_1=\{\pi\in\Real^m:(1,p)\in T_{(t,x)}\mathcal C_1\}$ denote the so-called\footnote{This notion is motivated by the fact that the variable $t$ in \eqref{N.1} satisfies the differential equation $t'=1$. 
In autonomous DAEs, the variable $t$ is absent in \eqref{N.1} and the set $\mathcal C_1$ such that one operates by the usual tangent bundle $T\mathcal C_1$ and tangent space $T_x\mathcal C_1$ at $x$.}
\emph{restricted tangent bundle} of $\mathcal C_1$ and \emph{restricted tangent space} of $\mathcal C_1$ at $(t,x)$, respectively. Then we form the next set
\begin{align}
 \mathcal C_2=\pi_{1,2}( L\cap S\mathcal C_1).
\end{align}
This procedure leads to a sequence of sub-manifolds $\mathcal C_k$ of $\Real\times\Real^m$ associated with the DAE. We quote \cite[Definition 8]{Reich} and mention the modification for linear DAEs in Section \ref{subs.degree} to compare with.
\begin{definition}\label{d.degree}
 Let $L$ be the corresponding set of the given DAE \eqref{N.1}. We define a family $\{ \mathcal C_k\}_{k=0,\ldots, s}$ of submanifolds $\mathcal C_k$ of $\Real\times\Real^m$ by the recursion
 \begin{align*}
  \mathcal C_0&=\mathcal I\times\mathcal D,\\
  \mathcal C_k&= \pi_{1,2}( L\cap S\mathcal C_{k-1}),\quad k=0,\ldots, s,
 \end{align*}
where $s$ is the largest non-negative integer such that the triples $(\mathcal C_k, p, \mathcal I)$ are differentiable sub-bundles and $\mathcal C_{s-1}\neq\mathcal C_s$. In case of $\mathcal C_1=\mathcal I\times\mathcal D$ we define $s=0$. We call the family $\{ \mathcal C_k\}_{k=0,\ldots s}$ the \emph{family of constraint manifolds} and the integer $s$ the \emph{degree of the given DAE}.
\end{definition}
It is shown that the degree, if it is well-defined, satisfies $s\leq m$. Furthermore, by \cite[Theorem 9]{Reich}, the DAE \eqref{N.1} is regular, if it has degree $s$ and, additionally, for each fixed $(t,x)\in \mathcal C_{s}$, the set
\begin{align*}
L\cap\{(t,x)\} \times S_{(t,x)}C_{s} 
\end{align*}
contains exactly one element, $L\cap SC_s$ is of class $\mathcal C^1$, and, for all $(t,x,p)\in L\cap SC_s$, $\dim C_s=\rank \pi_{1,2}T_{(t,x,p)}(L\cap SC_s)$.
Then, $\mathcal C=\mathcal C_{s}$ is the configuration space and the vectorfield is uniquely defined by
\begin{align*}
 (t,x,v(t,x))\in L\cap S\mathcal C.
\end{align*}

Owing to (\cite[Theorem 12]{Reich}) regularity of the DAE is ensured by  properties of the reduced derivative arrays given by
\begin{align}\label{N.4}
 \mathfrak G_{[k]}(t,x,p)&=\begin{bmatrix}
                      g_{[0]}(t,x,p)\\g_{[1]}(t,x,p)\\ \vdots \\g_{[k]}(t,x,p)
                     \end{bmatrix}, \quad  (t,x,p)\in \mathcal I\times \mathcal D\times\Real^m,
\end{align}
 for $k\geq 0$, with                   
 \begin{align*}                    
  g_{[0]}(t,x,p)&=f(t,x,p),\\
  g_{[j]}(t,x,p)&=W_{[j-1]}(t,x,p)(D_t g_{[j-1]}(t,x,p)+D_x g_{[j-1]}(t,x,p)p),
\end{align*}
in which $W_{[j-1]}(t,x,p)$ denotes a projector along $\im D_pg_{[j-1]}(t,x,p)$. 
Aiming for smooth projector functions the Jacobian $ D_pg_{[j-1]}(t,x,p)$ is supposed to have constant rank. In contrast to the array functions introduced in Section \ref{subs.Preliminaries}, not all equations are differentiated, but only those that are actually needed, namely the so-called derivative-free equations on each level. This leads to overdetermined systems of equations with regard to the variables $(x,p)$ which then have to be consistent.
 This tool is often used in structured DAEs, e.g., in multibody dynamics. A comparison with the index transformation from \eqref{N.1} to \eqref{N.5} shows that this makes sense in general.
\medskip

Using the sets 
\begin{align*}
 L_k=\{(t,x,p)\in\mathcal I\times\mathcal D\times\Real^m: \mathfrak G_{[k]}(t,x,p)=0 \}, \; k\geq 0,
\end{align*}
the above constrained manifolds $\mathcal C_k$ can be represented by
\begin{align*}
 \mathcal C_{k}=\pi_{1,2}(L_k)=\{(t,x)\in \mathcal I\times\mathcal D: \exists  p\in\Real^m,\mathfrak G_{[k-1]}(t,x,p)=0 \},
\end{align*}
which is verified in \cite{Reich}.

Owing to \cite[Theorem 12]{Reich} the DAE \eqref{N.1} is regular if there is a non-negative integer $\nu$ such that the matrix functions ${D_{p}\mathfrak G_{[k]}}$ and $[{D_{x}\mathfrak G_{[k]}} \; {D_{p}\mathfrak G_{[k]}} ]$ have constant ranks for $0\leq k<\nu$, and the row echelon form of ${D_{p}\mathfrak G_{[\nu]}}$
is $\begin{bmatrix}
     I_{m}\\0
    \end{bmatrix}$ independent of $(t,x,p)\in \mathcal J\times\mathcal D\times \Real^m$.
Then the DAE \eqref{N.1} has regular differentiation index $\nu$ if this is the smallest such non-negative integer.

\begin{remark}
 Let us briefly turn to the linear DAE 
 \begin{align}\label{lin}
  Ex'+Fx=0.
 \end{align}
 We have here
 \begin{align*}
  L&=\{(t,x,p)\in \mathcal I\times\Real^m\times\Real^m:E(t)p+F(t)x=0\},\\
  \mathcal C_1&=\{(t,x)\in \mathcal I\times\Real^m:F(t)x\in \im E(t)\},\\
   \mathcal C_1(t)&=\{x\in \Real^m:F(t)x\in \im E(t)\}=: S(t),\\
   S_{(t,x)}\mathcal C_1 &=\{p\in\Real^m: p=P_S(t)p+D_tP_S(t)x\},\\
   S\mathcal C_1&=\{(t,x,p)\in\mathcal I\times\Real^m\times\Real^m:E(t)p+F(t)x=0, p\in S_{(t,x)}\mathcal C_1\},
 \end{align*}
in which $P_S(t):=I-(WF)^+WF$ denotes the same projector function onto $S(t)=\mathcal C_1(t)$  as used in Subsection \ref{subs.qdiff} since the set $S(t)=\ker W(t)F(t)$ from Subsection \ref{subs.qdiff} obviously coincides with $C_1(t)$.

If $x:\mathcal I\rightarrow\Real^m$ is a solution of the DAE, then it holds that
\begin{align*}
 x(t)&=P_S(t)x(t),\\
 x'(t)&=D_tP_S(t)x(t)+P_S(t)x'(t),\\
 (t,x(t),x'(t))&=(t,x(t),D_tP_S(t)x(t)+P_S(t)x'(t))\in L\cap \mathcal C_1, \quad t\in \mathcal I.
\end{align*}
This leads to the new DAE 
\begin{align*}
 EP_Sx'+(F+EP_S')x=0,
\end{align*}
and this procedure is shown in \cite{Reich} to reduce the degree by one. We underline that the same
procedure is applied in Subsection \ref{subs.qdiff} for obtaining \eqref{R1}. Furthermore, as pointed out in Subsection \ref{subs.qdiff}, there is a close relationship with our basic reduction in Section \ref{s.regular}, see formulas \eqref{R2} and 
\eqref{R3}.
Supposing the DAE \eqref{lin} to be pre-regular in the sense of Definition \ref{d.prereg} we find that $\rank EP_S= r-\theta$ which sheds further light on the connection to the basic reduction in Section \ref{s.regular}.
\end{remark}
\begin{remark}\label{r.RaRh}
 In \cite[Chapter IV]{RaRh} the geometric reduction of quasilinear autonomous DAEs (but class $\mathcal C^{\infty}$),
 \begin{align}\label{RR0}
  E(x)x'+h(x)=0,
 \end{align}
given by functions $E\in \mathcal C^{\infty}(\mathcal D,\Real^m\times\Real^m),\;
 h\in \mathcal C^{\infty}(\mathcal D,\Real^m)$, $\mathcal D\subseteq \Real^m$ open, 
is best revealingly developed in the spirit of reduction of manifolds. Among other things, the corresponding  dimensions are also specified, which more clearly emphasizes the connection with our basic reduction procedure in section \ref{s.regular} that has its antetype in \cite[Chapter II]{RaRh}.
 
 First, a general procedure to specify the associated configuration space is created. Starting from a smooth $\bar r_0$-dimensional submanifold $\mathcal C_0$ of  $\Real^n$, $T\mathcal C_0$  becomes a smooth $2\bar r_0$- dimensional  submanifold of $T\Real^n=\Real^n\times\Real^n$. For any given functions $E_0\in \mathcal C^{\infty}(\mathcal C_0,\Real^n\times\Real^m),\;
 h_0\in \mathcal C^{\infty}(\mathcal C_0,\Real^m)$, set
 \begin{align*}
   f_0(x,p)&=E_0(x)p+h_0(x)\in \Real^m,\quad (x,p)\in T\mathcal C_0,
 \end{align*}
and  form 
 \begin{align*}
  L_0&=\{(x,p)\in T\mathcal C_0: f_0(x,p)=0\},\\
  \mathcal C_1&=\pi(L_0)=\pi|_{L_0}(L_0),
 \end{align*}
 with the projection $\pi:\Real^m\times\Real^m\rightarrow\Real^m$ onto the first factor.

The following two assumptions play a crucial role in this approach:
\begin{description}
 \item[A1:] The set $L_0$ is a smooth $\bar r_0$-dimensional submanifold of $T\mathcal C_0$ with tangent space $T_{(x,p)}L_0=\ker T_{(x,p)}f_0$ for every $(x,p)\in L_0$.
  \item[A2:] There exists a nonnegative integer $\bar r_1\leq \bar r_0$ such that $\rank E_0(x)|_{T_x\mathcal C_0}=\bar r_1$ for all $x\in \mathcal C_0$.
\end{description}
In particular, these two assumption ensure that $\pi|_{L_0}:L_0\rightarrow \Real ^m$ is a subimmersion with rank $\bar r_1$ and an open mapping into $\mathcal C_1$. 
In turn, $\mathcal C_1$ is a smooth $\bar r_1$-dimensional submanifold of both  $\mathcal C_0$ and $\Real^m$.

This shows how  manifolds $\mathcal C_i$ and $L_i$, $i\geq0$, can be defined inductively. We introduce 
\begin{align*}
   f_1(x,p)&=E_0(x)p+h_0(x)\in \Real^m,\quad (x,p)\in T\mathcal C_1
 \end{align*}
and form
\begin{align*}
  L_1&=\{(x,p)\in T\mathcal C_1: f_1(x,p)=0\}= T\mathcal C_1\cap L_0,\\
  \mathcal C_2&=\pi(L_1)=\pi|_{L_1}(L_1),
 \end{align*}
and so on. If the requirements of the above two assumptions hold at each step,  one obtains sequences of manifolds $L_0\supset L_1\supset\cdots\supset L_i\supset\dots $ and  $\mathcal C_0\supset \mathcal C_1\supset\cdots\supset \mathcal C_i\supset\dots $, where $L_{i+1}$ is an $\bar r_{i+1}$-dimensional submanifold of  $L_{i}$ and $\mathcal C_{i+1}$ is an $\bar r_{i+1}$-dimensional submanifold of  $T\mathcal C_{i}$, satisfying the relation 
\begin{align*}
 \mathcal C_{i+1}=\pi(L_i),\quad L_{i+1}=T\mathcal C_{i+1}\cap L_{i}.
\end{align*}
The pair of manifolds $(\mathcal C_0, L_0)$ is said to be \emph{completely reducible}  if the sequence is well-defined up to infinity, and hence $\bar r_0\geq \bar r_1\geq\cdots\geq \bar r_i\geq \cdots $.
The nonincreasing sequence of integers must eventually stabilize. Owing to \cite[Theorem]{RaRh}, if $L_{\nu}\neq\emptyset$ and $\bar r_{\nu}=\bar r_{\nu+1}$ for some integer $\nu\geq 0$, then $L_j=L_{\nu}$, $\bar r_j=\bar r_{\nu}$, for $j\geq \nu$, and $\mathcal C_j=\mathcal C_{\nu+1}$ for $j\geq \nu+1$.
\medskip

The described reduction applies to the DAE \eqref{RR0} by letting 
\begin{align*}
 E_{0}(x)=E(x), \; h_0(x)=h(x), \; \mathcal C_0=\mathcal D,\; \bar r_0=m,\; L_0= f_0^{-1}(0).
\end{align*}
By \cite[Definition 24.1]{RaRh} the quasilinear DAE \eqref{RR0}
has \emph{(geometric)\footnote{In \cite{RaRh} the suffix \emph{geometric} is still missing, it was added later in \cite[Section 3.4.1]{RR2008} to distinguish it from other terms.} index} $\nu$,   $0\leq\nu\leq m$, if  the pair $(\mathcal C_0, L_0)$  
is completely reducible and has index $\nu$ with $L_{\nu}\neq\emptyset$.

 A DAE \eqref{RR0} with well defined geometric index features locally existing and unique solutions \cite[Theorem 24.1]{RaRh}, and hence regularity.
\medskip

 At this point, it makes sense to compare once again with linear DAEs. In \cite[Chapter II]{RaRh} the DAE $Ex'+Fx=q$ is called \emph{completely reducible} in the given interval, if 
 our basic reduction procedure described in Section \ref{s.regular} and starting from $E_0=E, F_0=F$ is well-defined up to infinity, with constants $r_{-1}:=m$, $r_j:=\rank E_j$, $j\geq 0$.
 The smallest integer $0\leq\nu\leq m$ such that $r_{\nu-1}=r_{\nu}$ is the \emph{(geometric) index} of the DAE. Then $E_{\nu}$ remains nonsingular, and $\rank [E_j\;F_j]= r_{j-1}$, $j\geq 0$. This means that complete reducibility is the same as regularity in the sense of Definition \ref{d.2a}. On the other hand, if $E$ and $F$ are constant matrices, then this becomes a special case of the above autonomous geometric reduction with $r_j=\bar r_{j-1}$ for all $j\geq 0$.
\end{remark}
%

\subsection{Direct approaches without using array functions}
Direct concepts without recourse to derivative arrays should be possible 
starting from the fact that derivatives of a function can not contain information, which is not already present in the function itself. Derivative arrays or their restricted versions are not longer used here. Instead, sequences of special matrix functions including several  projector functions are pointwise build on the given function and their domain. 

In the context of the dissection and tractability index  more general equations,
\begin{align}\label{NT1}
 g(t,x(t),\frac{\rm d}{\rm dt}\varphi(t,x(t)))=0,
\end{align}
given by the two  functions $g:\mathcal I_g\times\mathcal D_g\times\Real^n\rightarrow\Real^m$ and $\varphi:\mathcal I_g\times\mathcal D_g\rightarrow\Real^n$, $n\leq m$, are investigated.
This has advantages both in terms of an extended solution concept and corresponding strict solvability statements with lower smoothness \cite{CRR,Jansen2014}. Since we deal with smoothness more generously here and assume $C^1$ solutions, this equation can also be written in standard form
\begin{align}\label{NT2}
 f(t,x(t),x'(t)):=g(t,x(t),\varphi_t(t,x(t))+\varphi_x(t,x(t))x'(t))=0.
\end{align}
For each smooth reference function $x_*:\mathcal I\rightarrow\Real^m$ not necessarily being a solution, but residing in the  definition domain $\mathcal I_g\times \mathcal D_g$  it results that
\begin{align*}
 f(t,x_*(t),x'_*(t))&=f(t,x_*(t),\varphi_t(t,x_*(t))+\varphi_x(t,x_*(t))x'_*(t))\\
 &=g(t,x_*(t),\frac{\rm d}{\rm dt}\varphi(t,x_*(t))).
\end{align*}
Below, linear DAEs 
\begin{align}\label{NT3}
 A_*(t)(D_*x)'(t)+B_*(t)x(t)=q(t), \quad t\in \mathcal I,
\end{align}
with coefficients
\begin{align*}
 A_*(t)=g_y(t,x_*(t),\frac{\rm d}{\rm dt}\varphi(t,x_*(t))),\; D_*(t)=\varphi_x(t,x_*(t)),\;
 B_*(t)=g_x(t,x_*(t),\frac{\rm d}{\rm dt}\varphi(t,x_*(t))),
\end{align*}
which arise from linearizations of the nonlinear DAE \eqref{NT1} along the given reference function, play an important role. Roughly speaking, we will decompose the domain $\mathcal I_g\times\mathcal D_g$ into certain so-called regularity regions so that all linearizations along smooth reference functions residing in one and the same region are regular with uniform index and  characteristic values. Then the borders of a maximal regularity region are critical points.
\medskip

The DAE \eqref{NT1} has a so-called \emph{properly involved derivative}, if the decomposition
\begin{align}\label{NT5}
 \ker g_y(t,x,y)\oplus\im \varphi_x(t,x)=\Real^n,\quad t\in \mathcal I_g,\;x\in \mathcal D_g,\; y\in \Real^n,
\end{align}
is valid and both matrix function $g_y$ and $\varphi_x$ feature constant rank $r$. 

At this place it is worth mentioning that there are weaker versions, namely the quasi-properly involved derivative in \cite[Chapter 9 ]{CRR} admitting certain rank drops of $g_y$ and the semi-properly involved derivative in \cite[]{Jansen2014} requiring constant ranks but merely $\im g_y=\im g_y \varphi_x$ instead of \eqref{NT5}.

The simplest version already applied in \cite{GM86} starts from the standard form DAE
\[
 f(t,x(t),x'(t))=0
\]
and supposes that the partial Jacobian $f_{x^1}(t,x,x^1)$ has constant rank $r$ and $\ker f_{x^1}(t,x,x^1)=N(t)$ is independent of the variables $x$ and $x^1$. 
Using a smooth projector function $P:\mathcal I\rightarrow \Real^{m\times m}$ such that $\ker P(t)=N(t)$
we set $n=m$, $\varphi(t,x)=P(t)x$. and $g(t,x, y)=f(t,x,y-P'(t)x)$. 
Then one has $\varphi_x(t,x)=P(t)$, $g_y(t,x, y)=f_{x^1}(t,x,y-P'(t)x)$, and $\ker g_y(t,x, y)=N(t)$, and hence we arrive at a DAE with properly involved derivative. 
\medskip

We turn back to the general case \eqref{NT1} and suppose a properly involved derivative. Analogously to  Section \ref{subs.tractability} for linear DAEs we associate to the DAE \eqref{NT1} a sequence of matrix functions built pointwise now for $t\in\mathcal I_g$,  $x\in\mathcal D_g$, $x^1\in\Real^m$. It will provide relevant information about the DAE, quite comparable to the array functions above.  We start letting
\begin{align*}
 D(t,x)&:=\varphi_x(t,x),\\
A(t,x,x^1)&:= g_y(t,x,\varphi_t(t,x)+D(t,x)x^1),\\
G_0(t,x,x^1)&:=A(t,x,x^1)D(t,x),\\
B_0(t,x,x^1)&:=g_x(t,x,\varphi_t(t,x)+D(t,x)x^1).
\end{align*}
Let $P_0(t.x)\in \Real^{m\times m}$ denote a smooth projector such that $\ker P_0(t,x)=\ker D(t,x)=:N_0(t,x)$ and 
\begin{align}\label{NT4}
 Q_0(t,x)=I -P_0(t,x),\quad \pPi_0(t,x)=P_0(t,x),
\end{align}
and introduce the generalized inverse $D(t,x,x^1)^-$ being uniquely determined by the four relations
\begin{align*}
D(t,x,x^1)^-D(t,x)D(t,x,x^1)^-&=D(t,x,x^1)^-,\\
D(t,x)D(t,x,x^1)^-D(t,x)&=D(t,x),\\
 D(t,x,x^1)^-D(t,x)&=P_0(t,x),\\
 \ker D(t,x)D(t,x,x^1)^-&=\ker A(t,x,x^1).
\end{align*}
Since the derivative is properly involved, it holds that $\ker G_0(t,x,x^1)=\ker D(t,x)=N_0(t,x)$. We form
\begin{align*}
 G_1(t,x,x^1)&;=G_0(t,x,x^1)+B_0(t,x,x^1)Q_0(t,x),\\
 N_1(t,x,x^1)&:=\ker G_1(t,x,x^1),\\
 \widehat{N_1}(t,x,x^1)&:=N_1(t,x,x^1)\cap N_0(t,x),
\end{align*}
and choose projector functions $Q_1, P_1, \pPi_1: \mathcal I_g\times \mathcal D_g \times\Real^{m\times m}$ such that pointwise
\begin{align*}
 \im Q_1&=N_1,\quad \ker Q_1\supseteq X_1, \quad \text{with any complement}\; X_1\subseteq N_0, \; N_0=\widehat N_1\oplus X_1,\\
 P_1&=I-Q_1,\; \pPi_1=\pPi_0P_1.
\end{align*}
We are interested in a smooth matrix function $G_1$ and require constant rank. From the case of linear DAEs we know of the necessity to incorporate the derivative of $D\pPi_1D^-$ into the next expressions. Instead of the time derivative $(D\pPi_1D^-)'$ in the linear case, we now use the total derivative  in jet variables $[D\pPi_1D^-]'$ given by
\begin{align*}
 [D\pPi_1D^-]'(t,x,x^1,x^2)&:=(D\pPi_1D^-)_t(t,x,x^1)+(D\pPi_1D^-)_x(t,x,x^1)x^1\\
 &+(D\pPi_1D^-)_{x^1}(t,x,x^1)x^2.
\end{align*}
The subsequent matrix function $B_1= B_0P_0-G_1D^-[D\pPi_1D^-]'D\pPi_0$ depends now on the variables $t,x,x^1$, and $x^2$. On each following level of the sequence  a new variable comes in owing to the involved total derivative. 
 
Now we are ready to adapt \cite[Definition 3.21]{CRR} and \cite[Definition 3.28]{CRR}   concerning admissible matrix function sequences and regularity. Both are straightforward generalizations of the linear case discussed in Section \ref{subs.tractability} above.
\begin{definition}\label{d.NTadmissible}
 Let $\mathfrak G\subseteq \mathcal I_g\times\mathcal D_g$ be open  and connected set. For given level $\kappa\in \Natu$, we call the sequence $G_0,\ldots,G_{\kappa}$ an \emph{admissible matrix function sequence} associated with the DAE \eqref{NT1} on the set  $\mathfrak G$, if it is built by the rule:
 \begin{align*}
  G_i&=G_{i-1}+B_{i-1}Q_{i-1}:\mathfrak G\times\Real^{im}\rightarrow\Real^m,\quad r_i^T=\rank G_i,\\
  B_{i}&=B_{i-1}P_{i-1}-G_{i}D^-[D\pPi_{i}D^-]'D\pPi_{i-1}:\mathfrak G\times\Real^{(i+1)m}\rightarrow\Real^m,\\
  &\quad N_{i}=\ker G_{i},\quad \widehat{N_{i}}=(N_0+\cdots+N_{i-1})\cap N_{i},\quad u_{i}^T=\dim \widehat{N_i},\\
  & \text{fix a complement}\; X_{i}\;\text{ such that}\; N_0+\cdots+N_{i-1}=\widehat{N_{i}}\oplus X_{i},\\
  &\text{choose a smooth projector function}\; Q_{i}\;\text{such that}\; \im Q_{i}=N_{i},\;\ker Q_{i}\supseteq X_{i},\\
  &\text{set}\; P_{i}=I-Q_{i},\; \pPi_{i}=\pPi_{i-1}P_{i},\\
  i&=1,\ldots,\kappa-1,\\
  G_{\kappa}&=G_{\kappa-1}+B_{\kappa-1}Q_{\kappa-1}:\mathfrak G\times\Real^{\kappa m}\rightarrow\Real^m,\quad r_{\kappa}^T=\rank G_{\kappa},
 \end{align*}
 and, additionally, all the involved functions $r_i$ and $u_i$ are constant.
\end{definition}
The total derivative used here reads in detail:
\begin{align*}
 [D\pPi_{i}D^-]'(t,x,x^1,\ldots,x^{i+1})&=(D\pPi_{i}D^-)_t(t,x,x^1,\ldots,x^{i})+(D\pPi_{i}D^-)_x(t,x,x^1,\ldots,x^{i})x^1\\
 &+\sum_{j=1}^{i}(D\pPi_{i}D^-)_{x^j}(t,x,x^1,\ldots,x^{i})x^{j+1}.
\end{align*}

At this point, the general agreement of this work on the smoothness of the given data ensures also  the existency of these derivatives.
Then the required smooth projector functions actually exist due to the demanded constancy of the ranks $r_i^T$ and the dimensions $u_i^T$. 

The inclusions 
\begin{align*}
 \im G_i\subseteq \im G_{i+1},\quad r_i\leq r_{i+1},
\end{align*}
are meant point by point and result immediately by the construction.
We refer to \cite[Section 3.2]{CRR} for further useful properties. In particular, it is possible to determine the projector functions $Q_i$ in such a way that the $\pPi_{i}$ and $\pPi_{i-1}Q_i$ are pointwise symmetric projector functions \cite[p.\ 205]{CRR}. 
\begin{definition}\label{d.NTregularity}
 Let $\mathfrak G\subseteq \mathcal I_g\times\mathcal D_g$ be open  and connected set.
 The DAE \eqref{NT1} is said to be \emph{regular on $\mathfrak G$  with tractability index $\mu\in\Natu$} if there is an admissible matrix function sequence reaching a pointwise nonsingular matrix function $G_{\mu}$ and $r_{\mu-1}^T<r_{\mu}^T=m $. 
 The rank values
 \begin{align}\label{NT6}
  r=r_0^T\leq\cdots\leq  r_{\mu-1}^T<r_{\mu}^T=m
 \end{align}
are said to be \emph{characteristic values} of the DAE. 

The set $\mathfrak G$ is called \emph{regularity region} of the DAE with associated index $\mu$ and characteristics \eqref{NT6}.\footnote{One can also understand $\mathfrak G^{[\mu]}=\mathfrak G\times \Real^{\mu m}$ as a regularity region. We refer to \cite[Section 3.8]{CRR} for a relevant refinement of the definition.}

If $\mathfrak G$ has the structure $\mathfrak G = \mathcal I \times \mathcal G $, $\mathcal I , \mathcal G$ open, then simply $\mathcal G$ is called regularity region, too.
\end{definition}
\begin{definition}\label{d.NTregpoint}
The point $(\bar t,\bar x)\in \mathcal I_g\times\mathcal D_g$ is called a \emph{regular point} of the DAE, if there is an open neighborhood $\mathfrak G\ni (\bar t,\bar x)$,  $\mathfrak G\subseteq  \mathcal I_g\times\mathcal D_g$  being a regularity region.\footnote{In case of  a regularity region $\mathfrak G^{[\mu]}=\mathfrak G\times \Real^{\mu m}$ we speak of \emph{regular jets} $(\bar t,\bar x, \bar x^1,\ldots,\bar x^{\mu})$.} Otherwise, the point is called a \emph{critical point} of the DAE.

If this $\mathfrak G$ has the structure $\mathfrak G = \mathcal I \times \mathcal G $, $\mathcal I , \mathcal G$ open, then simply $\bar{x}$ is called regular point, too.
\end{definition}
Regularity goes along with $u_i^T=0$ for all $i\geq 0$.
It is important to add that both the index and the characteristic values do not depend on the particular choice of projector functions in the admissible sequence of matrix functions. They are also invariant with respect to regular transformations, cf.\ \cite{CRR}.

The main result in the framework of the projector-based analysis of nonlinear DAEs is given by \cite[Theorem 3.33]{CRR} that claims:
\begin{description}
 \item[\textbullet] The DAE \eqref{NT1} is regular on $\mathfrak G$ if all linearizations \eqref{NT3} along smooth reference functions residing in $\mathfrak G$ are regular DAEs, and vice versa.
\item[\textbullet] If the nonlinear DAE is regular on $\mathfrak G$ with  index $\mu$ and the characteristics \eqref{NT6}, all linearizations built from reference functions residing in $\mathfrak G$ inherit this.
\item[\textbullet] If all linearizations built from reference functions residing in $\mathfrak G$ are regular, then they feature a uniform index $\mu$ and uniform characteristics \eqref{NT3}. The nonlinear DAE has then the same index and characteristics.
\end{description}
This allows to trace back questions concerning the DAE properties to the linearizations.

\bigskip
We underline that the concept of regularity regions does not assume the existence of solutions. However, if a solutions resides in a regularity region, then for $d>0$ it is part of a regular flow with the canonical characteristics from the regularity region. In any case, also for $d=0$, the solution has no critical points. 
\bigskip

In the  dissection index concept in \cite{Jansen2014} similar results are reproduced by using smooth basis functions instead of the projector functions. For the linear parts the decomposition described in Section \ref{subs.dissection} above is applied, and this is combined with rules of the tractability framework to construct a matrix function sequence emulating that from the tractability concept.
This is theoretically much more intricately but maybe useful in practical realizations. However, when using basis functions instead of projector functions, it must also be taken into account that there are not necessarily global bases in the multidimensional case, e.g., \cite[Remark A.16]{CRR}.

Regarding linear DAEs, the dissection concept in Section \ref{subs.dissection} and the regular strangeness concept in Section \ref{subs.strangeness} are closely related in turn. 
In contrast to the basic reduction for linear DAEs in Section \ref{s.regular}, which encompasses the elimination of variables and a reduction in dimension, the original dimension is retained and all variables stay involved. This is one of the cornerstones for the adoption of the linearization concept. 
We quote from \cite[p.\ 65]{Jansen2014}: \textit{The index arises as we use the linearization concept of the Tractability Index  and the decoupling procedure of the Strangeness Index}. This way, the regular strangeness index also finds a variant for nonlinear DAEs by means of corresponding sequences of matrix functions and linearization.

\subsection{Regularity regions and perturbation index}\label{subs.pert}
%
The perturbation index of nonlinear DAEs is an immediate generalization of the version for linear DAEs.
We slightly extend \cite[Definition 5.3]{HairerWanner} to be valid also for nonautonomous DAEs, cf. also Definition \ref{d.perturbation} above:
\begin{definition}\label{d.perturbationN}
 The equation \eqref{N.0} has perturbation index $\mu_p=\nu\in \Natu$ along a solution $x_*:\mathcal [a,b]\rightarrow \Real^m$, if $\nu$ is the smallest integer such that, for all functions  $x:\mathcal [a,b]\rightarrow \Real^m$ having a defect
 \begin{align*}
  \delta(t):=f(t,x(t),x'(t)), \; t\in [a,b],
 \end{align*}
there exists an estimation
\begin{align*}
 |x(t)-x_*(t)|\leq c \{|x(a)-x_*(a)|+\max_{a\leq\tau\leq t}|\delta(\tau)|+\cdots+ \max_{a\leq\tau\leq t}|\delta^{(\mu_p-1)}(\tau)|\}, \; t\in\mathcal I,
\end{align*}
whenever the expression on the right-hand side is sufficiently small.
\end{definition}
Because of its significance, we adopt the authors' comment in \cite[p.\ 479]{HairerWanner} on this definition: \emph{We deliberately do not write ``Let $x(\cdot)$ be the solution of $f(t,x(t),x'(t))=\delta(t), t\in [a,b]$ ...'' in this definition, because the existence of such a solution for an arbitrarily given $\delta(\cdot)$ is not assured.}

Actually, we are confronted with a problem belonging to functional analysis, with mapping properties and the question of how solutions and their components, respectively, depend on perturbations and their derivatives. Some answers concerning linear DAEs are given by means of the projector-based analysis in \cite{CRR,Ma2014,HM2020}.
In the case of nonlinear DAEs, this most significant question has hardly played an adequate role to date. Among other things, it is  associated with the relationship of the differentiation index to the perturbation index and the controversies surrounding it, e.g., \cite{CaGear95,Simeon}, see also Examples \ref{e.CamGear}, \ref{e.simeon} below.
The attempts made in \cite{CaGear95} to clarify the relation between these two different index notions did not achieve the intended goal.
A number of additional index terms is introduced in  \cite{CaGear95} (so-called uniform and maximum indices), however, this is not helpful 
 because quite special solvability properties of the DAE are assumed in advance.
\medskip

The geometric reduction  approaches concentrate exclusively on the dynamic properties of unperturbed systems. The approaches via derivative arrays are based on the assumption that all derivatives are or can be calculated correctly. 
They are primarily intended to figure out and approximate a particular solution of an unperturbed DAE.

For linear DAEs, the nature of the sensitivity of the solutions with respect to perturbations $\delta$ is determined by the structure of the canonical subspace $N_{can}$, i.e., not only by the index, but just as much by the characteristic values, see Subsection \ref{subs.posedness}. 
The projector-based analysis and the decoupled system in the tractability framework allows a precise and detailled insight into the dependencies.
For regular linear index-$\mu$ DAEs, both the differentiation index and the perturbation index are equal to $\mu$, and the homogenous structure of $N_{can}$ ensures homogenous dependencies over the given interval. In contrast,
for linear almost-regular DAEs featuring differentiation index $\mu$, on subintervals the perturbation index and the differentiation index may both be lower than $\mu$. 

Of course, nonlinear DAEs are much more complicated.
As shown in the previous section, the sequences of admissible matrix functions for nonlinear DAEs allow the determination of regularity regions. All points where the required rank conditions are not fulfilled are critical points.

Regarding the related method  equivalencies  obtained for linear regular DAEs by Theorem \ref{t.Sum_equivalence}, the linearization concept of the previous section can be utilized in all these cases. It seems that on each regularity region, the perturbation index coincides with the differentiation index and the other ones as it is the case in the examples below. We emphasize again that regularity regions characterize the DAE without assuming the existence of solutions.

Naturally, the maximal possible regularity regions are bordered with critical points. Then the definition domain of the DAE may be decomposed into  maximal regularity regions. Each regularity region 
comprises solely regular points with the very same index and characteristics. But different regularity regions may feature different index and characteristics. 
Since the matrix function sequence is built from the partial Jacobians of the given data,
the same sequence evidently arises for the unperturbed and the perturbed DAE.

The solutions may reside in one of the regularity regions, but they may also cross the borders or stay there. If they cross the border of regularity regions with different index values, then the perturbation index of the related solution segments changes accordingly as in Example \ref{e.CamGear} below.

\subsection{Some further comments}\label{subs.comments}
In \cite{Mehrmann} it is pointed out that the requirements of Hypothesis \ref{SHypN} and that of  a well-defined  differentiation index are equivalent up to some (technical) smoothness requirements. Harmless critical points are not at all indicated. It may happen that the differentiation index is much lower up to one on partial segments.
For details we refer to \cite[Remark 4.29]{KuMe2006}.

A comparison of the regular differentiation index, the projector-based differentiation  index, and the geometric index shows full consistency with regard to the rank  conditions and a slight difference in regards to smoothness. All kind of critical points are excluded for regularity, but they are being detected in the course of the procedures.

\subsection{Nonlinear examples}\label{subs.Nexamples}
 We give a brief outlook for nonlinear DAEs considering some small, representative, and easy-to-follow examples and emphasize again the importance of taking account of all canonical characteristic values in addition to the index.
 
 The first two Examples \ref{e.exp} and \ref{e.sin-cos} have a positive degree of freedom and show expected singularities from a geometric point of view.
 The next Examples \ref{e.CamGear} and \ref{e.simeon}  are classics from literature and show harmless critical points as well as changing characteristics. 
 With the next Example \ref{e.Ricardo} we emphasize the fact that problems with harmless critical points do not allow the geometric reduction.
 Our last Example \ref{e.robotic_arm} shows the so-called robotic arm DAE, a problem with zero degree of freedom, which nevertheless has serious singularities.

\begin{example}[Singular index-one DAE with bifurcation and impasse points]\label{e.exp}
Consider the very simple autonomous  DAE
\begin{align}\label{ex.2}
\begin{matrix}   x_1'-\gamma x_1 &=&0,\\
(x_1)^2+(x_2)^2-1 &=& 0,
\end{matrix}\quad \Bigg\rbrace
\end{align}
their perturbed version
\begin{align}\label{ex.2_pert}
\begin{matrix}
x_1'-\gamma x_1 &=\delta_1,\\
(x_1)^2+(x_2)^2-1 &= \delta_2,
\end{matrix}\quad \Bigg\rbrace
\end{align}

and the associated functions
\begin{align*}
 f(t,x,x^1)=\begin{bmatrix}
           x^1_1-\gamma x_1-\delta_1(t)\\(x_1)^2+(x_2)^2-1-\delta_2(t)
          \end{bmatrix}, \quad  t\in \Real, x,x^1\in \Real^2,\\
  f_{x^1}(t,x,x^1)=\begin{bmatrix}
           1&0\\0&0
          \end{bmatrix},        
f_x(t,x,x^1)=\begin{bmatrix}
           -\gamma &0\\2x_1&2x_2
          \end{bmatrix},\quad \gamma\in \Real \;\text{  is a given parameter}.
\end{align*}

\begin{figure}[ht]
\includegraphics[width=6.5cm]{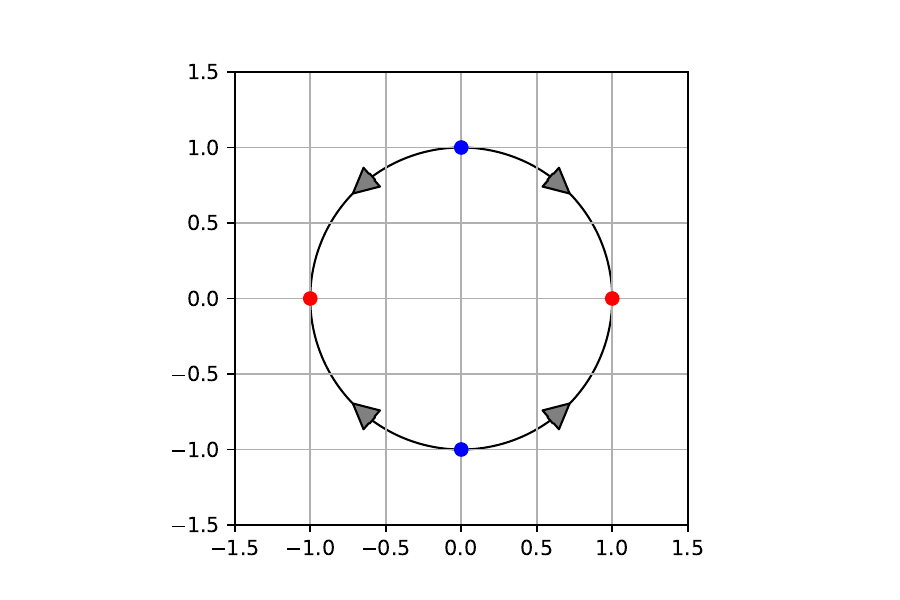}
\includegraphics[width=6.5cm]{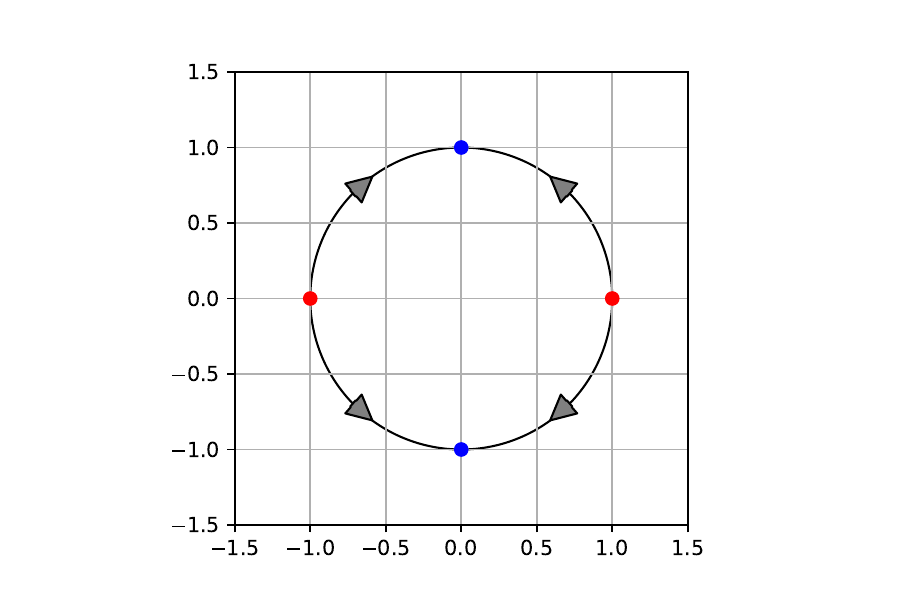}
\caption{Behavior of solutions for the DAE \eqref{ex.2} from Example \ref{e.exp}  for  $\gamma>0$ (left) or $\gamma<0$ (right), critical points (red) and stationary solutions (blue).}
\label{fig:CircleArrows-exp}
\end{figure}

Obviously, the points $\begin{bmatrix}
           0\\1
          \end{bmatrix}$ and $\begin{bmatrix}
           0\\-1
          \end{bmatrix}$ serve as stationary solutions of the autonomous DAE \eqref{ex.2}. 
Further, for each initial point $x_0\in \Real^2$ lying on the unit circle arc, except for the two points on the $x_1$-axis, there exists exactly one solution to the autonomous DAE \eqref{ex.2} passing through at $t_0=0$, namely:

\begin{itemize}
  \item if $\gamma<0$ and  $x_{0,2}>0$ then
  \begin{align*}
   x_*(t)=\begin{bmatrix}
         \exp{(\gamma t)} x_{0,1} \\
				\, \\
				\sqrt{1-\exp{(2\gamma t)}x_{0,1}^2}
          \end{bmatrix},\; t\in [0,\infty),\quad x(t)\xrightarrow{t\rightarrow \infty}\begin{bmatrix}
           0\\1
          \end{bmatrix},
  \end{align*}
\item if $\gamma<0$  and  $x_{0,2}<0$ then:
   \begin{align*}
   x_*(t)=\begin{bmatrix}
         \exp{(\gamma t)} x_{0,1} \\
				\, \\
				- \sqrt{1-\exp{(2\gamma t)}x_{0,1}^2}
          \end{bmatrix},\; t\in [0,\infty),\quad x(t)\xrightarrow{t\rightarrow \infty}\begin{bmatrix}
           0\\-1
          \end{bmatrix},
  \end{align*}
 \item if $\gamma>0$ and  $x_{0,2}>0$ then
  \begin{align*}
   x_*(t)=\begin{bmatrix}
         \exp{(\gamma t)} x_{0,1} \\ 
				\, \\
				\sqrt{1-\exp{(2\gamma t)}x_{0,1}^2}
          \end{bmatrix},\; t\in [0,t_f],
  \end{align*}
\item if $\gamma>0$  and  $x_{0,2}<0$ then: 
   \begin{align*}
   x_*(t)=\begin{bmatrix}
         \exp{(\gamma t)} x_{0,1} \\
				\, \\
				- \sqrt{1-\exp{(2\gamma t)}x_{0,1}^2}
          \end{bmatrix},\; t\in [0,t_f],
  \end{align*}
\end{itemize}
whereby, the final time $t_f$ of the existence intervals is determined by the equation $\exp(\gamma t_f)=1/|x_{0,1}|$ and
\begin{align*}
  x(t_f)=\begin{bmatrix}
          1\\0
          \end{bmatrix}\; \text{ for } x_{0,1}>0,  \;x(t_f)=\begin{bmatrix}
          -1\\0
          \end{bmatrix}\; \text {for } x_{0,1}<0.
\end{align*}
It is now evident that merely the two points  $\begin{bmatrix}
           1\\0
          \end{bmatrix}$ and $\begin{bmatrix}
           -1\\0
          \end{bmatrix}$
are critical. For $\gamma<0$, from each of these points, two solutions start, but, for  $\gamma>0$, these points are so-called impasse points, cf. Figure \ref{fig:CircleArrows-exp}.
From the geometric point of view the DAE has degree $s=1$ and the unit circle arc can be seen as the configuration space.
\medskip

In case of nontrivial perturbations $\delta_1, \delta_2$, the situation is on the one hand quite similar but on the other hand much more intricate. In particular, the configuration space becomes time-dependent, and we are confronted with different configuration spaces for different perturbations $\delta_2$, see Figures \ref{fig:Circle_growing} and \ref{fig:Circle_growing_sin}. The final time $t_f$ also depends on the perturbations and
seemingly the place of the critical points varies.  The solution representations allow the realization on correspondingly small time intervals $[a,b]$ that this is a DAE with perturbation index one apart from the critical points.

\begin{figure}[ht]
\includegraphics[width=7cm]{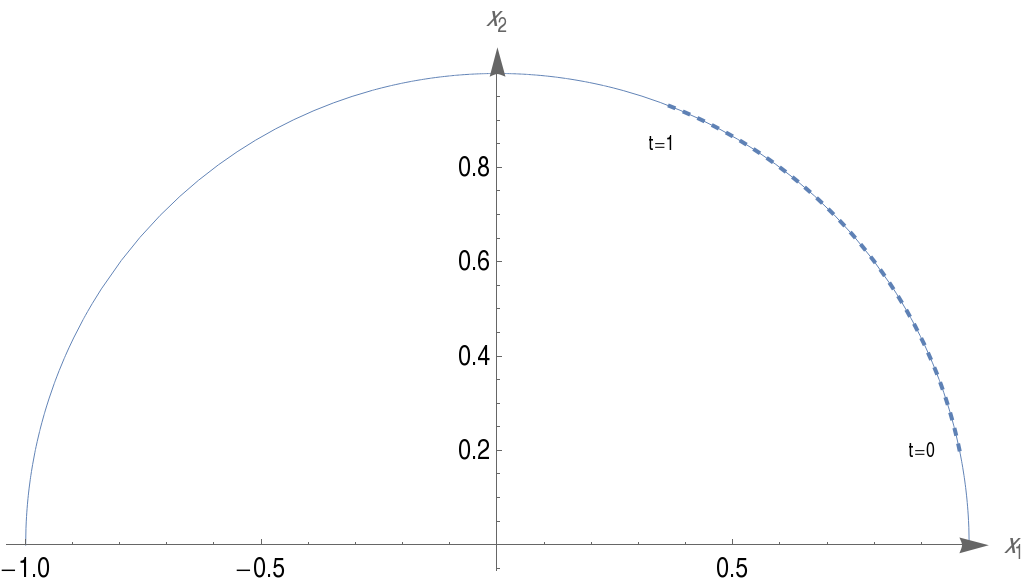} \hspace{0.3cm}
\includegraphics[width=7cm]{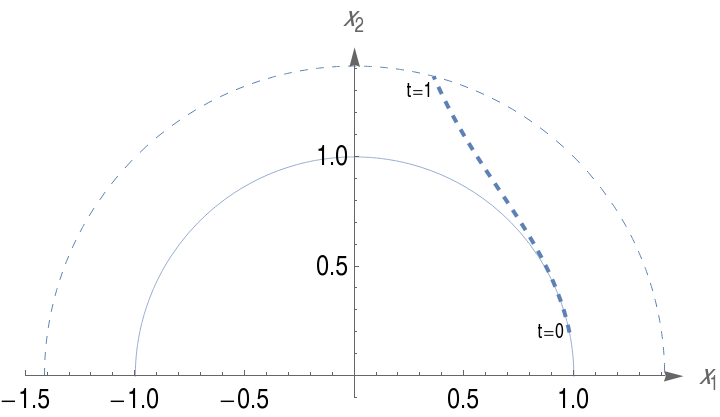}
\caption{Solution of the DAE \eqref{ex.2} from Example \ref{e.exp} for $\gamma=-1$ and initial value $x_{0,1}=0.98$  (left), as well as solution of \eqref{ex.2_pert} for $\delta_1(t)=0$, $\delta_2(t)=t^2$ (right), both for $t \in \left[0,1\right]$.}
\label{fig:Circle_growing}
\end{figure}
\begin{figure}[ht]
\includegraphics[width=7cm]{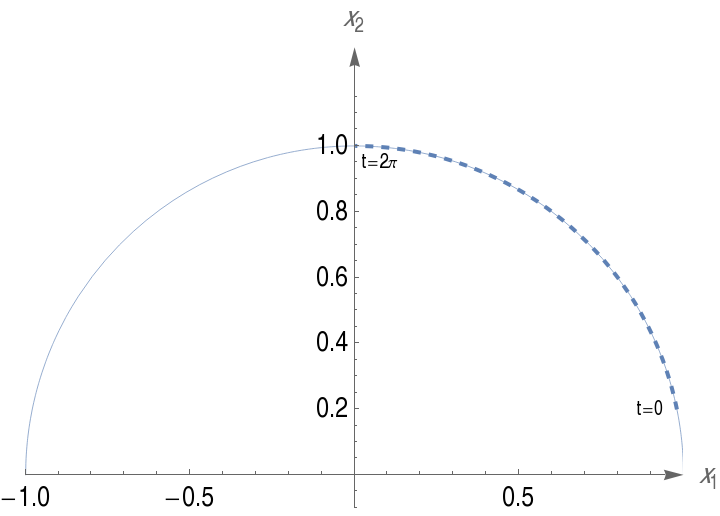}
\hspace{0.3cm}
\includegraphics[width=7cm]{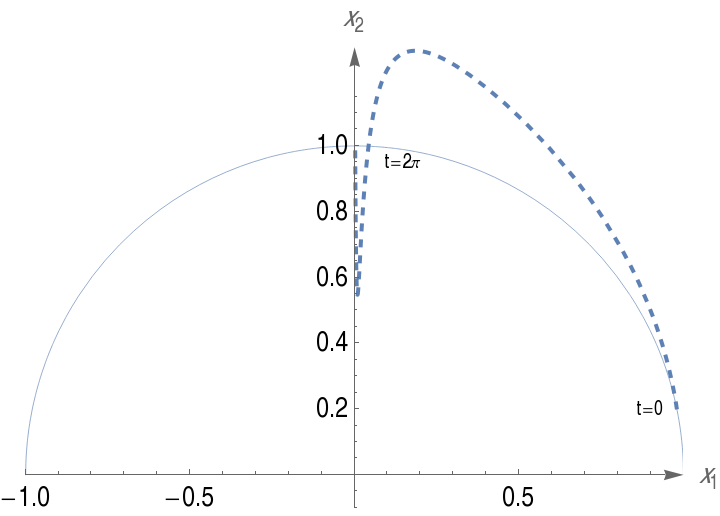}
\caption{Solution of the DAE \eqref{ex.2} from Example \ref{e.exp} for $\gamma=-1$ and initial value $x_{0,1}=0.98$  (left), as well as solution of \eqref{ex.2_pert} for $\delta_1(t)=0$,  $\delta_2(t)=0.7 \sin(t)$ (right), both for $t \in \left[0,2\pi\right]$.}
\label{fig:Circle_growing_sin}
\end{figure}

\medskip

Note that the partial derivatives $f_{x^1}$ and $f_x$ are independent of the perturbations  $\delta_1, \delta_2$.
Applying the projector-based analysis we form the matrix function
\begin{align*}
 G_1(t,x,x^1)=f_{x^1}(t,x,x^1)+f_x(t,x,x^1)Q_0=\begin{bmatrix}
                                1&0\\0&2x_2
                               \end{bmatrix}, \quad Q_0=\begin{bmatrix}0&0\\
           0&1
          \end{bmatrix}.
\end{align*}
Obviously, $G_1(t,x,x^1)$ is nonsingular if and only if $x_2\neq 0$.

On the other hand, the inflated system  $\mathfrak F_{[1]}=0$ yields the partial Jacobian  
\[
 \mathcal E_{[1]}(t,x,x^1,x^2)=\begin{bmatrix}
                   1&0&0&0\\0&0&0&0\\-\gamma&0&1&0\\2x_1&2x_2&0&0
                  \end{bmatrix},
\]
that
undergoes a rank drop from 3 for $x_2\neq 0$ to 2 for $x_2=0$. 
Now it becomes clear that $x_2=0$ indicates critical points, which splits $\Real^2$ into the two regularity regions
\[
 \mathcal G_{+}=\{x\in \Real^2:x_2>0\}\;\text{ and }\; \mathcal G_{-}=\{x\in \Real^2:x_2<0\} ,
\]
see Figure \ref{fig:RegReg1+2} (left). The border set consists of critical points,
\[
 \mathcal G_{crit}=\{x\in \Real^2:x_2=0\}.
\]
On each of these regularity regions, the DAE is said to be regular with index $\mu=\mu^T=\mu^{pbdiff}=\mu^{diff}=1$ and canonical characteristics $r=1$, $\theta_0=0$. This means that for all perturbed versions of our DAE, the intersection of the corresponding  configuration space with $\mathcal G_{crit}$ contain the singular points of the flow.

\end{example}

\begin{example}[Singular index-one DAE with critical-point-crossing solution]\label{e.sin-cos}
Given the value $\gamma=1$ or $\gamma=-1$,  let us have a closer look at the simple autonomous  DAE
\begin{align}\label{ex.1}
x_1'-\gamma x_2 &=0,\\
x_1^2+x_2^2-1 &= 0,\nonumber
\end{align}
This DAE  possesses obvious solutions, namely
\begin{itemize}
  \item if $\gamma=1$:
  \begin{align*}
   x_*(t)=\begin{bmatrix}
           \sin t\\\cos t
          \end{bmatrix},
x_{**}(t)=\begin{bmatrix}
           1\\0
          \end{bmatrix}, and \;
          x_{***}(t)=\begin{bmatrix}
           -1\\0
          \end{bmatrix}, \quad t\in [0, 2\pi],
  \end{align*}
\item if $\gamma=-1$:
  \begin{align*}
   x_*(t)=\begin{bmatrix}
           \cos t\\\sin t
          \end{bmatrix},
x_{**}(t)=\begin{bmatrix}
           1\\0
          \end{bmatrix}, and \;
          x_{***}(t)=\begin{bmatrix}
           -1\\0
          \end{bmatrix}, \quad t\in [0, 2\pi],
  \end{align*}
\end{itemize}
together with phase-shifted variants. It is evident that the first solution crosses the other ones  at $t=\frac{\pi}{2}$ and at $t=\frac{3\pi}{2}$, thus the points $\begin{bmatrix}
           1\\0
          \end{bmatrix}, \begin{bmatrix}
           -1\\0
          \end{bmatrix}$ appear to be singular, see Figure \ref{fig:CircleArrows-sin-cos}.
\medskip

Similarly as in the previous example, 
applying the projector-based analysis we form the matrix function
\begin{align*}
 G_1(t,x,x^1)=f_{x^1}(t,x,x^1)+f_x(t,x,x^1)Q_0=\begin{bmatrix}
                                1&-\gamma\\0&2x_2
                               \end{bmatrix}, \quad Q_0=\begin{bmatrix}0&0\\
           0&1
          \end{bmatrix}.
\end{align*}
Again, $G_1(x,p)$ is nonsingular if and only if $x_2\neq 0$ and,
on the other hand, the  array function 
\[
 \mathcal E_{[1]}(x)=\begin{bmatrix}
                   1&0&0&0\\0&0&0&0\\0&-\gamma&1&0\\2x_1&2x_2&0&0
                  \end{bmatrix},
\]
undergoes a rank drop from 3 for $x_2\neq 0$ to 2 for $x_2=0$.
It becomes clear that $x_2=0$ indicates critical points, which splits $\Real^2$ into the two regularity regions
\[
 \mathcal G_{+}=\{x\in \Real^2:x_2>0\}\;\text{ and }\; \mathcal G_{-}=\{x\in \Real^2:x_2<0\},
\]
see Figure \ref{fig:RegReg1+2} (left), and the border consisting of critical points
\[
  \mathcal G_{crit}=\{x\in \Real^2:x_2=0\}.
\]
From the geometric point of view the DAE has degree $s=1$ and the unit circle arc can be seen as the configuration space. 
On each of the regularity regions, the DAE is regular with index $\mu=\mu^T=\mu^{pbdiff}=\mu^{diff}=1$ and canonical characteristics $r=1$, $\theta_0=0$.
The intersection of the configuration space with $\mathcal G_{crit}$ contains the singular points of the flow.

Indeed, for instance, if we simulate the first solution $x_*$ from above,  then we start in a regularity region. However, at $t=\frac{\pi}{2}$, this solution crosses the  constant solution $x_{**}$, and at $t=\frac{3\pi}{2}$ the other constant solution $x_{***}$, which are flow singularities. In  terms of the characteristic-monitoring, at these times the canonical characteristic value $\theta_0$ changes, and  this indicates that the solution crosses the border of a regularity region.

\begin{figure}
\includegraphics[width=6.5cm]{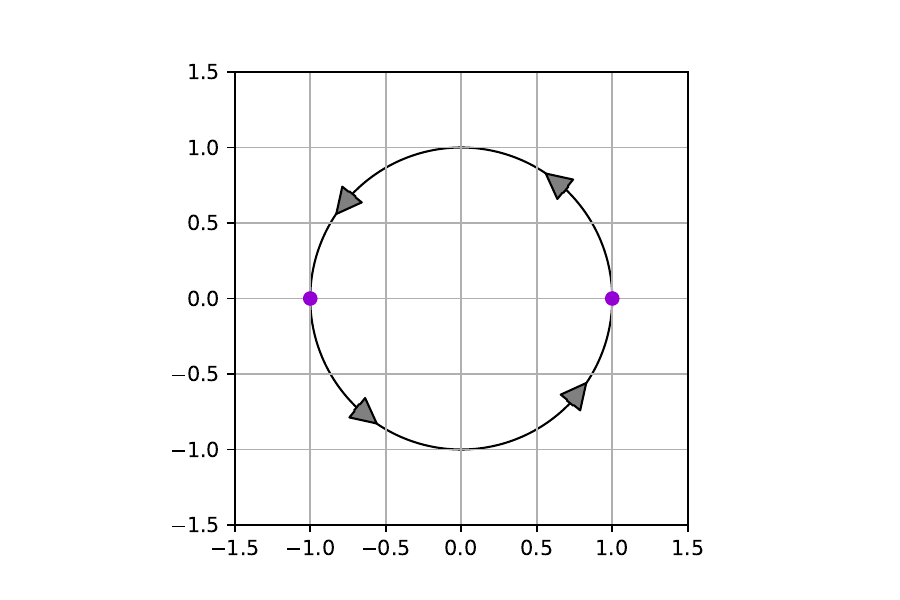}
\includegraphics[width=6.5cm]{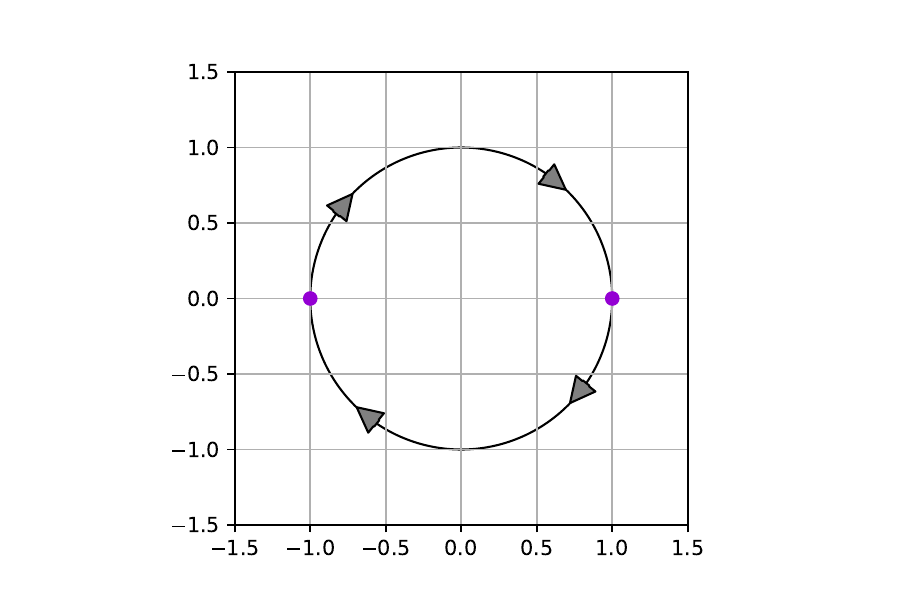}
\caption{Behavior of solutions for Example \ref{e.sin-cos} and critical points (violet) that are also stationary solutions.}
\label{fig:CircleArrows-sin-cos}
\end{figure}

\end{example}


\begin{figure}[h]
    \begin{minipage}[b]{0.45\textwidth}
        \includegraphics[width=\textwidth]{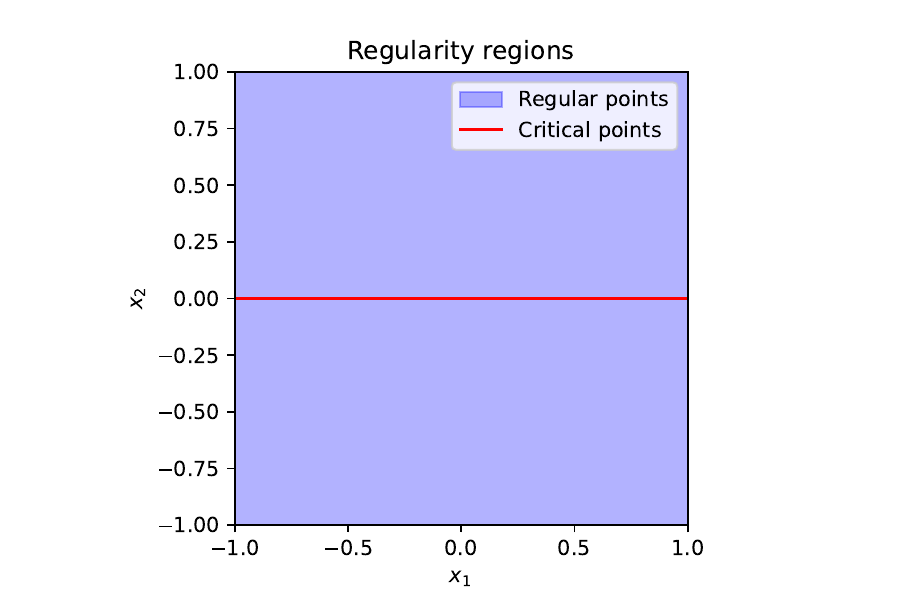}
        \caption*{Examples \ref{e.exp} and \ref{e.sin-cos}}
    \end{minipage}
    \hfill
    \begin{minipage}[b]{0.45\textwidth}
        \includegraphics[width=\textwidth]{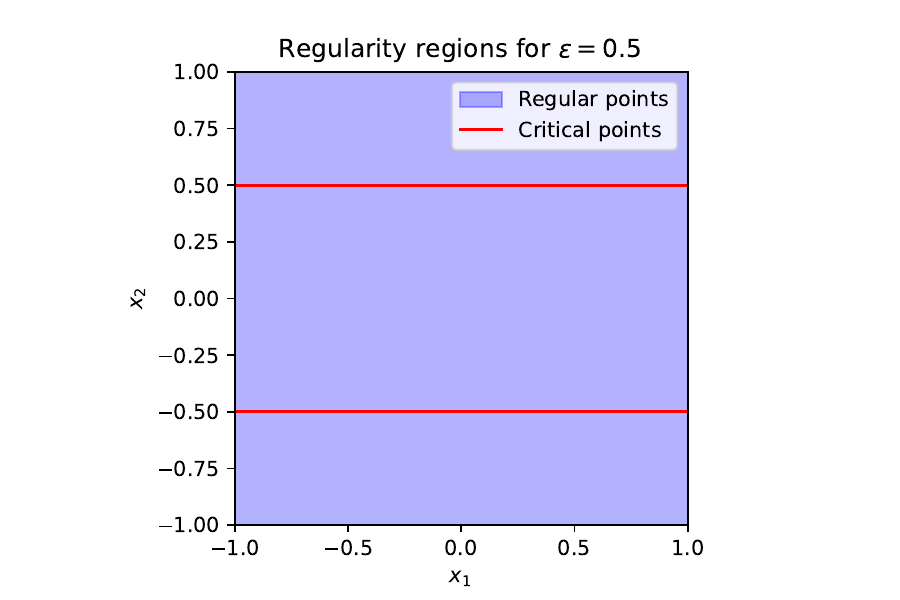}
        \caption*{Example \ref{e.CamGear}}
    \end{minipage}
		\caption{Regularity regions of three examples with $m=2$.}
    \label{fig:RegReg1+2}
\end{figure}

\begin{example}[Index change and harmless critical points]\label{e.CamGear}
We revisit now \cite[Example 12]{CaGear95} that plays its role in the early discussion concerning index notions.

For $\epsilon>0$, let $\gamma : \Real \rightarrow \Real$ be an infinitely differentiable function which has the property
\[
\begin{array}{ll}
\gamma(z)=0 & \text{ for } \left|z\right| \leq \epsilon, \\
\gamma(z)\neq 0 & \text{ else },
\end{array}
\]
and consider the DAE
\begin{align*}
\gamma(x_2)x_2'+x_1&= \delta_1, \\
x_2 &= \delta_2,
\end{align*}
by means of the associated given functions
\begin{align*}
 f(t,x,x^1)&=\begin{bmatrix}
           \gamma(x_2) x_2^1+x_1\\x_2
          \end{bmatrix}-\delta(t), \quad  t\in \Real,\; x,x^1\in \Real^2,\\
  f_{x^1}(t,x,x^1)&=\begin{bmatrix}
           0&\gamma(x_2)\\0&0
          \end{bmatrix}, \;       
f_x(t,x,x^1)=\begin{bmatrix}
           1&\gamma'(x_2)\\0&1
          \end{bmatrix}.
\end{align*}
The DAE is solvable for each arbitrary smooth perturbation $\delta$. The solutions are given by
\begin{align*}
 x_1=\delta_1-\gamma(\delta_2)\delta'_2,\\
 x_2=\delta_1,
\end{align*}
which indicates that the perturbation index is not greater than one and the dynamic degree of freedom is $d=0$.

Not surprisingly, using the projector-based analysis, we observe  three regularity regions showing different characteristics,
\[
 \mathcal G_{+}=\{x\in \Real^2:x_2>\epsilon\}\;\text{ and }\; \mathcal G_{-}=\{x\in \Real^2:x_2<-\epsilon\},
\]
and 
\[
  \mathcal G_{\epsilon}=\{x\in \Real^2:|x_2|<\epsilon\},
\]
see Figure \ref{fig:RegReg1+2} (right), and in detail
\begin{align*}
 \text{on}\quad \mathcal G_{+}:&\quad r=1, \quad  \theta_0=1,\quad  \theta_1=0,\quad \mu=2,\\
 \text{on}\quad \mathcal G_{\epsilon}:&\quad r=0, \quad  \theta_0=0,\quad \mu=1,\\
 \text{on}\quad \mathcal G_{-}:&\quad r=1, \quad  \theta_0=1, \quad  \theta_1=0,\quad \mu=2.
\end{align*}
The borders between these regularity regions consist of critical poitns. All these critical points are obviously harmless. If a solution fully resides in $\mathcal G_{+}$ or  $\mathcal G_{-}$ then the DAE has perturbation index $\mu_p=2$ along this solution. In contrast, if a solution fully resides in $\mathcal G_{\epsilon}$ then the DAE has perturbation index $\mu_p=1$ along this solution. Of course, there might be solutions crossing the borders and then the perturbation index changes accordingly along the solution.
\end{example}


\begin{example}[Campbell's counterexample] \label{e.simeon}
This is a special case of \cite[Example 10]{CaGear95}  which was picked out and discussed in the essay  \cite[p.\ 73]{Simeon}. It is about the relationship between the differentiation index and the perturbation index.
Consider the DAE
\begin{align*}
x_3x_2'+x_1 &=\delta_1,\\
x_3x_3' +x_2 &= \delta_2,\\
x_3&= \delta_3,
\end{align*}
and the associated functions
\begin{align*}
 f(t,x,x^1)&=\begin{bmatrix}
           x_3x_2^1+x_1\\x_3x_3^1+x_2\\x_3
          \end{bmatrix}-\delta(t), \quad  t\in \Real,\, x,x^1\in \Real^3,\\
  f_{x^1}(t,x,x^1)&=\begin{bmatrix}
           0&x_3&0\\0&0&x_3\\0&0&0
          \end{bmatrix},        
f_x(t,x,x^1)=\begin{bmatrix}
           1&0&x_2^1\\0&1&x_3^1\\0&0&1
          \end{bmatrix}.
\end{align*}
The DAE has obviously a unique solution to each arbitrary smooth perturbation $\delta$, namely
\begin{align*}
x_3&=\delta_3,\\
x_2&=\delta_2-\delta_3'\delta_3,\\
x_1&= \delta_1-\delta_3 (\delta_2-\delta_3'\delta_3)'=\delta_1-\delta_3\delta'_2+\delta_3(\delta'_3)^2+(\delta_3)^2\delta''_3,
\end{align*}
which clearly indicates perturbation index $\mu_p=3$ and zero dynamic degree of freedom.
In contrast, by Definition \ref{d.N1} (that is controversal) the unperturbed DAE with $\delta=0$ has differentiation index $\mu^{diff}=1$, with the underlying ODE $x'=0$ given on the single  point $x=0$. 
Note that the related array function $ \mathcal E_{[1]}$ in Definition \ref{d.N1a} undergoes a rank drop at $x_3=0$,
\[
 \mathcal E_{[1]}=\begin{bmatrix}
                   0&x_3&0&0&0&0\\
                   0&0&x_3&0&0&0\\
                   0&0&0&0&0&0\\
                   1&x_3^1&0&0&x_3&0\\
                   0&1&x_3^1&0&0&x_3\\
                   0&0&1&0&0&0
                  \end{bmatrix},\quad r_{[1]}=\rank \mathcal E_{[1]}=\Bigg\lbrace \begin{matrix}
                                                                4\;\text{ for } x_3\neq 0\\
                                                                 3\;\text{ for } x_3= 0
                                                               \end{matrix}\;.
\]
In particular, the rank function $r_{[1]}$ fails to be constant on each neighborhood of the origin, which would be necessary for an index-1 DAE in the sense of the precise Definition \ref{d.N1a}. 
According to our  understanding the  DAE has the regularity regions 
\begin{align*}
 \mathcal G_{+}=\{z\in \Real^3:x_3>0\},\quad  \mathcal G_{-}=\{z\in \Real^3:x_3<0\},
\end{align*}
and the critical point set
\begin{align*}
 \mathcal G_{crit}=\{z\in \Real^3:x_3=0\}.
\end{align*}
At the points of the regularity regions we form the admissible matrix functions from the projector-based framework,
\begin{align*}
A=f_{x^1},\; D=P=\begin{bmatrix}
	0 & 0 & 0 \\
	0 & 1 & 0 \\
	0 & 0 & 1 
\end{bmatrix},\;
 G_0=AD=\begin{bmatrix}
	0 & x_3 & 0 \\
	0 & 0 & x_3 \\
	0 & 0 & 0 
\end{bmatrix}, \; r_0^T=2,
 \; B_0=\begin{bmatrix}
           1&0&x_2^1\\0&1&x_3^1\\0&0&1
          \end{bmatrix},
\end{align*}
and 
\begin{align*}
Q_0&=\begin{bmatrix}
	1 & 0 & 0 \\
	0 & 0 & 0 \\
	0 & 0 & 0 
\end{bmatrix},\;
G_1=\begin{bmatrix}
	1 & x_3 & 0 \\
	0 & 0 & x_3 \\
	0 & 0 & 0 
\end{bmatrix}, \; r_1^T=2,\;
B_1=\begin{bmatrix}
           0&0&x_2^1\\0&1&x_3^1\\0&0&1
          \end{bmatrix},\\
 Q_1&=\begin{bmatrix}
	0 & -x_3 & 0 \\
	0 & 1 & 0 \\
	0 & 0 & 0 
\end{bmatrix},\;
G_2=\begin{bmatrix}
	1 & x_3 & 0 \\
	0 & 1 & x_3 \\
	0 & 0 & 0 
\end{bmatrix},\; r_2^T=2, \;
B_2=\begin{bmatrix}
           0&0&x_2^1\\0&0&x_3^1\\0&0&1
          \end{bmatrix},\\
Q_2&=\begin{bmatrix}
	0 & 0 & (x_3)^2 \\
	0 & 0 & -x_3 \\
	0 & 0 & 1 
\end{bmatrix},
G_3=\begin{bmatrix}
	1 & x_3 & 0 \\
	0 & 1 & x_3 \\
	0 & 0 & 1 
\end{bmatrix},\; r_2^T=3.
\end{align*}
Therefore, on both regularity regions, the DAE is regular with index $\mu^T=\mu_p=\mu^{diff}$ and canonical characteristics 
\begin{align*}
r=2, \quad \theta_0=1, \quad \theta_1= 1, \quad \theta_2=0, \quad d=0.
\end{align*}
All points from $\mathcal G_{crit}$ are harmless critical points as the above solution representation confirms.

It should also be added that the further array functions $\mathcal E_{[2]}$ and $\mathcal E_{[3]}$ feature constant ranks, and the differentiation index is well-defined and equal to one on the entire domain.
\end{example}


\begin{example}[Riaza's counterexample] \label{e.Ricardo}
The following example is part of the discussion whether problems with harmless critical points are accessible to treatment by geometric reduction from \cite{RaRh}. It is commented in \cite[p.\ 186]{RR2008} with the words: \emph{There is no way to apply the framework of \cite{RaRh} neither globally nor locally around the origin}.

We apply the perturbed version of the system \cite[(4.6), p.\ 186]{RR2008},
\begin{align*}
x_1'-\alpha(x_1,x_2, x_3) &=\delta_1, \\
x_1x_2' - x_3 &= \delta_2, \\
x_2 &= \delta_3,
\end{align*}
in which $\alpha$ denotes a smooth function.
We recognize that
\begin{align*}
x_2 &= \delta_3,\\
x_3 &= -\delta_2+x_1\delta_3', \\
x_1'-\alpha(x_1,\delta_3, -\delta_2+x_1\delta_3') &=\delta_1,
\end{align*}
such that it becomes clear that the DAE is solvable for all sufficiently smooth perturbations  $\delta$ and initial conditions for the first solution component. 

To apply the projector-based analysis we use the associated functions
\begin{align*}
 f(t,x,x^1)&=\begin{bmatrix}
           x^1_1-\alpha(x_1,x_2,x_3) \\x_1x_2^1-x_3\\x_2
          \end{bmatrix}, \quad  t\in \Real, x,x^1\in \Real^3,\\
  f_{x^1}(t,x,x^1)&=\begin{bmatrix}
           1&0&0\\0&x_1&0\\0&0&0
          \end{bmatrix},        \\
f_x(t,x,x^1)&=\begin{bmatrix}
           -\alpha_{x_1}(x_1,x_2,x_3)&  -\alpha_{x_2}(x_1,x_2,x_3)&  -\alpha_{x_3}(x_1,x_2,x_3)\\x_2^1&0&-1\\0&1&0
          \end{bmatrix}.
\end{align*}
Supposing $x_1\neq 0$ we form the admissible matrix functions. We drop the arguments of the functions whenever it is reasonable. We obtain 
\begin{align*}
Q_0=\begin{bmatrix}
     0&0&0\\0&0&0\\0&0&1
    \end{bmatrix},\;
 G_1=f_{x^1}+f_xQ_0=\begin{bmatrix}
           1&0&-\alpha_{x_3}\\0&x_1&-1\\0&0&0
          \end{bmatrix},\; r_1^T=\rank G_1=2,\;  \theta_0=1,\\
   Q_1=\begin{bmatrix}
     0&\alpha_{3}/x_1&0\\0&1&0\\0&1/x_1&0
    \end{bmatrix},\;
 G_2=G_1+B_1Q_1=\begin{bmatrix}
           1&*&*\\0&x_1&-1\\0&1&0
          \end{bmatrix}, \; r_2^T=\rank G_2=3,\; \theta_1=0.     
\end{align*}
Consequently, the DAE has the two regularity regions
\begin{align*}
 \mathcal G_{+}=\{z\in \Real^3:x_1>0\},\quad  \mathcal G_{-}=\{z\in \Real^3:x_1<0\}
\end{align*}
and the critical point set
\begin{align*}
 \mathcal G_{crit}=\{z\in \Real^3:x_1=0\}.
\end{align*}
Regarding the solvability properties we know the critical point to be harmless. The perturbation index is two around solutions residing in a regularity region. If a solution does not cross or touch the critical point set, then the perturbation index is two along this solution. Obviously, along reference functions $x_*$, with vanishing first components, the perturbation index reduces to one.
\end{example}

\begin{example}[DAE describing a two-link robotic arm]\label{e.robotic_arm}
The so-called robotic arm DAE is a well understood benchmark for higher-index DAEs 
on the form
\begin{align*}
\left(\begin{bmatrix}
			I_6 &  \\
			 & 0_2
		\end{bmatrix}x \right)'+b(x,t)=0.
\end{align*}
that is well described in literature, see \cite{CaKu2019}, \cite{ELMRoboticArm2020} and the references therein. It results from a tracking problem in robotics and presents two types of singularities. Without going into the details of the $m=8$ equations and variables with $r=6$, here we interpret them in terms of the characteristics $\theta_i$.   

The DAE describes a two-link robotic arm with an elastic joint moving on a horizontal plane, see Figure \ref{fig:RA} (left). The third variable $x_3$ of the equations corresponds to the rotation of the second link with respect to the first link.

\begin{figure}[htbp]
    \begin{minipage}[b]{0.45\textwidth}
        \includegraphics[width=\textwidth]{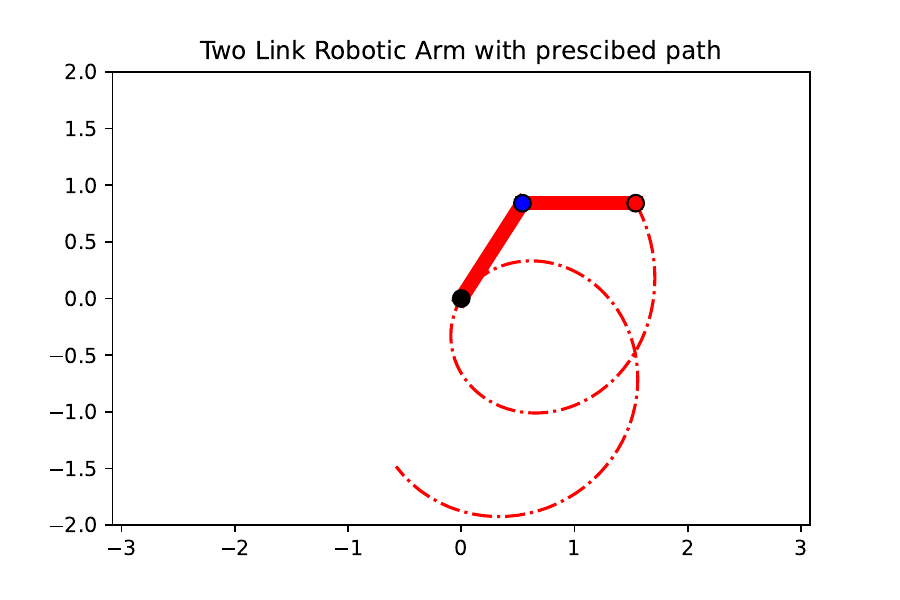}
        \caption*{The blue marker corresponds to the elastic joint of the two links. At the black marker the end of one link is fixed to the origin. The position of the red endpoint of the outer link is prescribed by a path. }
    \end{minipage}
    \hfill
    \begin{minipage}[b]{0.45\textwidth}
        \includegraphics[width=\textwidth]{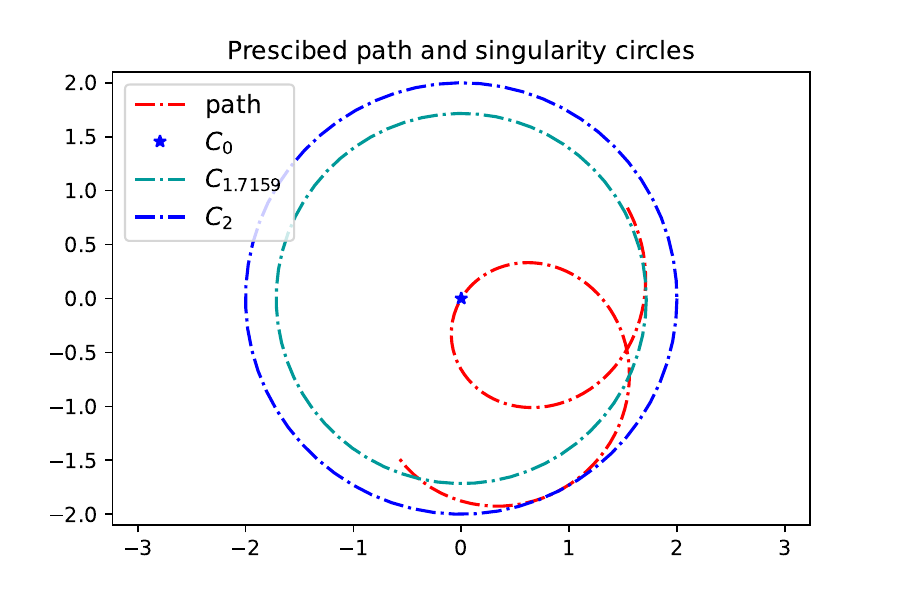}
        \caption*{If the prescribed path crossed a singularity circle $C_r$ with radius $r$, then singularities of the DAE arise.  $C_0$ and $C_2$ correspond to $\sin(x_3)=0$ and $\cos(x_3)=z_*$ leads to $C_{1.7159}$ for certain model parameters. }
    \end{minipage}
		\caption{As $x_3$ is the angle between the two links, the singularities can be interpreted geometrically, see \cite{ELMRoboticArm2020}, where also the original figures and the discussion of the parameters can be found.}
    \label{fig:RA}
\end{figure}

In \cite{ELMRoboticArm2020} it has been shown that critical points arise at
\[
\cos(x_3)=z_* \quad or \quad \sin(x_3)=0,
\]
whereas the constant value $z_*$ depends on the particular parameters of the model.
In the consequence, the original definition domain of the DAE $\mathcal I\times\mathcal D=\mathcal I\times\Real^8$ decomposes into an infinite number of regularity regions whose borders are  hyperplanes consisting of the corresponding critical points. Owing to \cite[Proposition 5.1]{ELMRoboticArm2020} the canonical characteristics are the same on all regularity regions. If the component $x_{*,3}$ of a solution $x_*$ crosses or touches such a critical hyperplane then this gives rise to a singular behavior. In case of the robotic arm, this happens
if the prescribed path crosses so-called singularity circles, see Figure \ref{fig:RA} (right).

In regularity regions, we obtain $\mu^T=\mu^{pbdiff}=5$ and
\begin{align*}
\theta_0 &=8-r_1^T=\rank T_1=2, \quad  & \theta_1 &=8-r_2^T= \rank T_2=2, \\\
\theta_2 &=8-r_3^T=\rank T_3=1, \quad  & \theta_3 &=8-r_4^T=\rank T_4=1, \\
\theta_4 &=8-r_5^T=\rank T_5=0, \quad  & d & =0.
\end{align*}
Indeed, monitoring this ranks is how the singularities $\cos(x_3)=z_*$ were detected, which to our knowledge had not been described before.
\end{example}


\section{Conclusions}

Until now, DAE literature has been rather heterogeneous since each approach uses quite different starting points, definitions, assumptions, leading to own results. 
The diversity of the frameworks made it difficult to compare them. Although
 a few equivalence statements were proven,  a general and rigorous framework was missing so far.

To get to our Main Theorem \ref{t.Sum_equivalence} for linear DAEs we revised and compiled many results from the literature and closed several gaps. By doing so, a characterization of regularity and almost regularity that is interpretable for all approaches resulted straight forward. 

This theorem with the canonical characteristics is, in our opinion, the common ground of all the considered approaches. For nonlinear DAEs, we worked out aspects for possible further investigations. 

\bibliography{CommonGround_DAEs}{}
\bibliographystyle{plain}

\section{Appendix}\label{s.Anhang}

\subsection{1-full matrix functions and rank estimations}\label{subs.A1}
The matrix $M\in \Real^{(s+1)m\,\times\, (s+1)m}$ which has block structure built from $m\times m$ matrices is said to be  \emph{$1$-full}, if there is a  nonsingular matrix  $T$ such that $TM=\begin{bmatrix}I_{m}&0\\0&H \end{bmatrix}$.

Let $M:\mathcal I\rightarrow \Real^{(s+1)m\,\times\, (s+1)m}$ be a continuous matrix function which has block structure built from $m\times m$ matrix functions. $M$ is said to be \emph{smoothly $1$-full}, if there is a pointwise nonsingular, continuous matrix function $T$ such that $TM=\begin{bmatrix}I_{m}&0\\0&H   \end{bmatrix}$.

\begin{lemma}\label{l.app}
Let $M:\mathcal I\rightarrow \Real^{(s+1)m\,\times\, (s+1)m}$ be a continuous matrix function which has block structure built from $m\times m$ matrix functions. 

The following assertions are equivalent:
\begin{description}
 \item[\textrm{(1)}] $M$ is smoothly $1$-full and has constant rank $r_{M}$.
 \item[\textrm{(2)}] $M$ has constant rank $r_{M}$ and  $M(t)$ is $1$-full pointwise for each $t\in \mathcal I$.
 \item[\textrm{(3)}] There is a continuous function $H:\mathcal I\rightarrow \Real^{sm\,\times\, sm}$  with constant rank $r_{H}$ such that
\begin{align}\label{a.kerM}
 \ker M=\{\begin{bmatrix}
           z\\w
          \end{bmatrix}\in \Real^{m}\times\Real^{sm}: z=0, w\in \ker H
\},
\end{align}
 \item[\textrm{(4)}] $M$ has constant rank $r_{M}$ and 
 \begin{align*}
  T_{[s]}\ker M={0},
 \end{align*}
with the truncation matrix $T_{[s]}=[I_{m}\, 0\cdots 0]\in\Real^{m\times (s+1)m}$.
\end{description}
\end{lemma}
\begin{proof}
{\textrm (1)}$\leftrightarrow${\textrm (2)}: The straight direction is trivial, the opposite direction  is provided, e.g., by \cite[Lemma 3.36]{KuMe2006}.

{\textrm (3)}$\leftrightarrow${\textrm (4)}: The straight direction is trivial, we immediately turn to the opposite one.
 Let $M$ have constant rank $r_{M}$ and $T_{[s]} \ker M=\{0\}$. Denote by $Q_{M}$ the  continuous orthoprojector function onto $\ker M$. Then, $Q_{M}$ must have the special form
\begin{align*}
 Q_{M}&=\begin{bmatrix}
          0&0\\0&K
         \end{bmatrix}:\mathcal I\rightarrow \Real^{(m+km)\times (m+km)},\\
         &K=K^*:\mathcal I\rightarrow \Real^{sm\times sm},\quad \rank K=\rank Q_{M}=(m+sm)-r_{M}.
\end{align*}
Let $Z:\mathcal I\rightarrow \Real^{sm \times (r_{M}-m)}$ denote a continuous basis of $(\im K)^{\perp}$ such that $\rank Z=r_{M}-m$. Then the assertion becomes true with
\begin{align*}
 H=\begin{bmatrix}
          Z^*\\0
         \end{bmatrix}:\mathcal I\rightarrow \Real^{(sm)\times (sm)},
         \quad \rank H=\rank Z^*=r_{M}-m.
\end{align*}

{\textrm (1)}$\leftrightarrow${\textrm (3)}: 
Since the straight direction is trivial again, we immediately turn to the opposite one. Let $H:\mathcal I\rightarrow \Real^{sm\,\times\, sm}$ be a continuous function with constant rank $r_{H}$ such that \eqref{a.kerM} is valid. Then $\ker M$ has dimension $\dim\ker H=sm-r_{H}$ which is constant. Therefore, $M$ has constant rank $r_{M}=(s+1)m-(sm-r_{H})=m+r_{H}$.

The projector-valued matrix functions
\begin{align*}
 W_{M}=I_{sm+m}-MM^{+} \quad \text{ and } \quad Q_{M}=I_{sm+m}-M^{+}M=\begin{bmatrix}
                                                 0&0\\0&I_{sm}-H^{+}H
                                                \end{bmatrix}
\end{align*}
are continuous and have constant rank $sm+m-r_{M}=:\rho$ both.  Then there are  continuous matrix functions $U_{W}, V_{W}, U_{Q}, V_{Q}$ being pointwise orthogonal such that (e.g. \cite[Theorem 3.9]{KuMe2006})
\begin{align*}
 Q_{M}= U_{Q}\begin{bmatrix}
                           \Sigma_{Q}&0\\0&0
                          \end{bmatrix} V^{T}_{Q},\quad
  W_{M}=U_{W}\begin{bmatrix}
                           \Sigma_{W}&0\\0&0
                          \end{bmatrix}V^{T}_{W},
\end{align*}
with nonsingular sigma blocks of size $\rho$. Set 
\begin{align*}
 \mathcal C=V_{Q}\begin{bmatrix}
                           \Sigma_{Q}^{-1}\Sigma_{W}^{-1}&0\\0&0
                          \end{bmatrix}U^{T}_{W}
\end{align*}
such that 
\begin{align*}
  Q_{M}\mathcal C W_{M}=U_{Q}\begin{bmatrix}
                           I_{\rho}&0\\0&0
                          \end{bmatrix}V^{T}_{W},\quad \im Q_{M}\mathcal CW_{M}= \im  Q_{M},\quad \ker  Q_{M}\mathcal CW_{M}=\ker  W_{M},
\end{align*}
and the matrix $ T= Q_{M}\mathcal C W_{M}+M^{+}$ is continuous and nonsingular. It follows that $TM=I_{sm+m}- Q_{M}=M^{+}M=\diag(I_{m},H^{+}H)$ which means that $M$ is smoothly $1$-full.
\end{proof}

\begin{lemma}\label{l.app2}
For $m\in \Natu$, a matrix function $E: \mathcal I\rightarrow \Real^{m\,\times \, m}$
and matrix functions
\[
\mathcal M_{[0]}(t) :=E(t), \quad  \mathcal M_{[k]}: \mathcal I\rightarrow \Real^{(k+1)m\,\times\, (k+1)m}, \quad k=1 , \ldots
\]
defined in a way that the structure
\begin{eqnarray}
\mathcal M_{[k+1]}(t) := \begin{bmatrix}
	\mathcal M_{[k]}(t) & 0 \\
	* & E(t)
\end{bmatrix}
\label{eq:Structure_M}
\end{eqnarray}
is given, for $r(t):= \rank E(t)$ and $r_{[k]}(t):=\rank \mathcal M_{[k]}(t)$ 
it holds
\[
r_{[k]}(t) + r(t) \leq r_{[k+1]}(t) \leq r_{[k]}(t) + m, \quad  t \in {\mathcal I}, \, k \geq 0.
\]
\end{lemma}

\begin{proof}
The structure \eqref{eq:Structure_M} obviously leads to
\[
 r_{[k+1]}(t) \geq r_{[k]}(t) + r(t).
\]
and
\[
r_{[k+1]}(t) = \dim \im \begin{bmatrix}
	\mathcal M_{[k]}(t) & 0 \\
	* & E
\end{bmatrix} \leq \dim \im \begin{bmatrix}
	\mathcal M_{[k]}(t)
\end{bmatrix}  +m= r_{[k]}(t)+m. 
\]
\end{proof}

\subsection{Continuous matrix function with rank changes}\label{subs.M}
We quote the following useful result from \cite[Proof of Theorem 10.5.2]{CaMe}:
\begin{theorem}\label{t.M}
 Let the matrix function  $M:[a, b]\rightarrow\Real^{m\times n}$ be continuous and let
 \begin{align*}
  \varphi=\{t_0\in [a, b]: \rank M(t)\;\text{ is not continuous at }\; t_0\}
 \end{align*}
denote the set of its rank-change points. 

Then the set $\varphi$ is closed and has no interior and 
there exist a collection $\mathfrak S$ of open intervals $\{(a^{\ell}, b^{\ell})\}_{{\ell}\in \mathfrak S}$, such that
 \begin{align*}
  \overline{\bigcup_{ {\ell} \in \mathfrak S}(a^{{\ell}}, b^{\ell})}=[a, b],\quad (a^{\ell_i}, b^{\ell_i})\cap \mathcal (a^{\ell_j}, b^{\ell_j})=\emptyset \quad\text{for}\quad \ell_i \neq \ell_j,
 \end{align*}
 and integers $r^{\ell}\geq 0$, ${\ell}\in \mathfrak S$, such that
 \begin{align*}
  \rank M(t)=r^{\ell} \quad \text{for all}\quad t\in (a^{\ell}, b^{\ell}),\quad {\ell} \in \mathfrak S.
 \end{align*}
\end{theorem}
As emphasided already in \cite{CaMe}, the set  $\varphi$ and in turn collection $\mathfrak S$ can be finite, countable, and also over-countable.

\subsection{Strictly block upper triangular matrix functions and array functions of them}\label{subs.A_strictly}
In this part, for given integers $\nu\geq 2, l\geq\nu, l_1\geq 1,\ldots,l_{\nu}\geq1$, such that $l=l_1+\ldots+l_{\nu}$,
we denote by $SUT=SUT(l,\nu,l_1,\ldots,l_{\nu})$ the set of all strictly upper triangular  matrix functions $N:\mathcal I\rightarrow \Real^{l\times l}$ showing the block structure
\begin{align*}
N=\begin{bmatrix}
   0&N_{12}&*&\cdots&*\\
   &0&N_{23}&*&*\\
   &&\ddots&\ddots&\vdots\\
   &&&&N_{\nu-1, \nu}\\
   &&&&0
   \end{bmatrix}, \quad N_{ij}=(N)_{ij}:\mathcal I\rightarrow \Real^{l_j\times l_i},\quad N_{ij}=0 \quad\text{for}\quad i\geq j.
	\end{align*}
If $l=\nu$ and $l_i=\cdots=l_{\nu}=1$ then $N$ is strictly upper triangular in the usual sense.

The following lemma collects some rules that can be checked by straightforward computations.
\begin{lemma}\label{l.SUT1}
 $N,\hat N\in SUT$ and $N_1,\ldots,N_k\in SUT$ imply
 \begin{description}
  \item[\textrm{(1)}] $N+\hat N \in SUT$.
  \item[\textrm{(2)}] $N\hat N \in SUT$, and the entries of the secondary diagonals are  
  \begin{align*}
   (N\hat N)_{i,i+1}&=0, \quad i=1,\ldots,\nu-1,\\
   (N\hat N)_{i,i+2}&=(N)_{i,i+1}(\hat N)_{i+1,i+2},\quad i=1,\ldots,\nu-2.
  \end{align*}
\item[\textrm{(3)}] $N^{\nu}=0$ and  $N_1\cdots N_k=0$ for $k\geq\nu$.
\item[\textrm{(4)}] $I-N$ remains nonsingular and $(I-N)^{-1}=I+N+\cdots+N^{\nu-1}$.
\item[\textrm{(5)}] $(I-\hat N)^{-1}N=N+\hat NN+\cdots+(\hat N)^{\nu-2}N$ and
 \begin{align*}
   ((I-\hat N)^{-1}N)_{i,i+1}=(N)_{i,i+1}, \quad i=1,\ldots,\nu-1.
  \end{align*}
 \end{description}
\end{lemma}

The following two subsets of $SUT$ are of special interest because they enable rank determinations and beyond.

\textbullet\; Supposing $l_1\geq \cdots\geq l_{\nu}$ we denote by $SUT_{column}\subset SUT$ the set  of all $N\in SUT$ having exclusively blocks $(N)_{i,i+1}$ with full column rank, that is 
\begin{align}\label{N.col}
 \rank (N)_{i,i+1}=l_{i+1},\quad i=1,\ldots,\nu-1.
\end{align}

\textbullet\; Supposing $l_1\leq \cdots\leq l_{\nu}$ we denote by $SUT_{row}\subset SUT$ the set  of all $N\in SUT$ having exclusively blocks $(N)_{i,i+1}$ with full row rank, that is 
\begin{align}\label{N.row}
 \rank (N)_{i,i+1}=l_{i},\quad i=1,\ldots,\nu-1.
\end{align}
\begin{lemma}\label{l.Ncol}
 Each $N\in SUT_{column}$ has constant rank $l-l_1$ and,   for $k\leq \nu-1$,  one has 
 \begin{align*}
  \ker N=\im \begin{bmatrix}
              I_{l_1}\\0
             \end{bmatrix},\quad 
\ker N^k=\im \begin{bmatrix}
              I_{l_1}&&   \\
              &\ddots&\\
							& & I_{l_k}\\
              0&&0
             \end{bmatrix}.
 \end{align*}
Moreover, for the product of any $k$ elements $N_1,\ldots, N_k\in SUT_{column}$ it holds that 
\begin{align*}
\ker N_1\cdots N_k=\im \begin{bmatrix}
              I_{l_1}&&   \\
              &\ddots& \\
							& & I_{l_k}\\
              0&&0
             \end{bmatrix}, \quad \rank N_1\cdots N_k = l-(l_1+\cdots+l_k).
 \end{align*}
\end{lemma}
\begin{proof}
  This follows from generalizing Lemma \ref{l.SUT1} (2) for products of several matrices and \eqref{N.col}. 
\end{proof}

\begin{lemma}\label{l.Nrow}
 Each $N\in SUT_{row}$ has constant rank $l-l_{\nu}$ and,  for $k\leq \nu-1$,  one has
 \begin{align*}
  \im N=\im \begin{bmatrix}
              I_{l_1}& &   \\
              &\ddots&  \\
							& & I_{l_{\nu-1}}\\
              0&&0
             \end{bmatrix},\quad 
\im N^k=\im \begin{bmatrix}
              I_{l_1}&  & \\
              &\ddots&   \\
							 & &  I_{l_{\nu-k}} \\
              0& & 0
             \end{bmatrix}.
 \end{align*}
Moreover, for the product of any $k$ elements $N_1,\ldots, N_k\in SUT_{row}$ it holds that 
\begin{align*}
\im N_1\cdots N_k=\im \begin{bmatrix}
              I_{l_1}&&   \\
              &\ddots& \\
							& & I_{l_{\nu-k}}\\
              0&&0
             \end{bmatrix}, \quad \rank N_1\cdots N_k = l_1+\cdots+l_{\nu-k}.
 \end{align*}
\end{lemma}
\begin{proof}
This follows from generalizing Lemma \ref{l.SUT1} (2) for products of several matrices and \eqref{N.row}. 
\end{proof}

Next we turn to the derivative array function\footnote{See \eqref{eq.arrayNk} in Section \ref{subs.SCFarrays}.} $\mathcal N_{[k]}:\mathcal I\rightarrow\Real^{(kl+l)\times(kl+l)}$ associated  with $N\in SUT$, that is,
\begin{align}\label{N.array}
\mathcal N_{[k]}:=
\begin{bmatrix}
 N&0&&&\cdots&0\\
 I+\alpha_{2,1} N^{(1)}&N&&&&\vdots\\
 \alpha_{3,1} N^{(2)}&I+\alpha_{3,2}N^{(1)}&N&&\\
\vdots& \ddots&\ddots&\ddots& &\\
 \vdots& &\ddots&\ddots&\ddots&0\\
 \alpha_{k+1,1}N^{(k)}&\cdots&&\alpha_{k+1,k-2}N^{(2)}&I+ \alpha_{k+1,k}N^{(1)}&N
\end{bmatrix}.
\end{align}

Since the derivatives $N^{(i)}$  inherit the basic structure of $SUT$, too, we rewrite
\begin{align}\label{N.M}
\mathcal N_{[k]}:=
\begin{bmatrix}
 N&0&&&\cdots&0\\
 I+M_{2,1}&N&&&&\vdots\\
 M_{3,1}&I+M_{3,2}&N&&\\
\vdots& \ddots&\ddots&\ddots& &\\
 \vdots& &\ddots&\ddots&\ddots&0\\
 M_{k+1,1}&\cdots&&M_{k+1,k-2}&I+ M_{k+1,k}&N
\end{bmatrix},
\end{align}
and keep in mind that  $M_{i,j}\in SUT$ for all $i$ and $j$.

\begin{lemma} \label{l.existence.Ntilde}
Given $N\in SUT$, there
exist a regular lower block triangular matrix function $\mathcal L_{[k]}$, such that
\[
\mathcal L_{[k]} \cdot \mathcal N_{[k]} = \begin{bmatrix}
 N&0&&&\cdots&0\\
 I&\tilde{N}_2&&&&\vdots\\
 0 &I&\tilde{N}_3&&\\
\vdots& \ddots&\ddots&\ddots& &\\
 \vdots& &\ddots&\ddots&\ddots& 0\\
 0&\cdots&&0&I&\tilde{N}_{k+1}
\end{bmatrix}= \mathcal{\tilde{N}}_{[k]},
\]
whereas the matrix functions $\tilde{N}_{2},\ldots,\tilde{N}_{k+1}$ again belong to $SUT$ and  inherit the secondary diagonal blocks of $N$, that means $(\tilde N_s)_{i,i+1}=N_{i,i+1}$, $i=1,\ldots,\nu-1$, $s=2,\ldots,k$.

It holds that 
\begin{align*}
 N\tilde N_2\cdots \tilde N_{k+1}=0\quad \text{if}\quad k\geq l-1.
\end{align*}
Moreover, for $i=1,\ldots,l$, as long as $i+k+1\leq l$,
it results that
\begin{align*}
 (N\tilde N_2\cdots \tilde N_{k+1})_{i,i+j}&=(N^{k+1})_{i,i+j}=0,\quad j=1,\ldots,k,\\
 (N\tilde N_2\cdots \tilde N_{k+1})_{i,i+k+1}&=(N^{k+1})_{i,i+k+1}.
\end{align*}
\end{lemma}
\begin{proof}
Since $I+M_{2,1}$ is nonsingular, with
the nonsingular lower block-triangular matrix function
\[
\mathcal L_{[k]}^1 =\begin{bmatrix}
	 I &0  & 0  & & \cdots & 0 \\
	 0 & (I+M_{2,1})^{-1} & 0 & & \cdots & 0 \\
	 0 & -M_{3,1}(I+M_{2,1} )^{-1} & I & 0  & \cdots& 0 \\
	0 & -M_{4,1}(I+M_{2,1})^{-1} & 0 &I  & \ddots & \vdots \\
	 \vdots & \vdots & \vdots & \ddots& \ddots & 0\\
	0 & -M_{k+1,1}(I+M_{2,1})^{-1} & 0 & \cdots & 0&I  
\end{bmatrix}
\]
we generate zero blocks  in the first column of  $\mathcal N_{[k]}$ such that 
\[
\mathcal L^{1}_{[k]} \cdot \mathcal N_{[k]} = \begin{bmatrix}
 N&0&&&\cdots&0\\
 I&\tilde{N}_2&&&&\vdots\\
 0&I+\hat{M}_{3,2}&N&&\\
\vdots& \ddots&\ddots&\ddots& &\\
 \vdots& &\ddots&\ddots&\ddots&0\\
 0&\hat{M}_{k+1,2}&\cdot &M_{k+1,k-2}&I+ M_{k+1,k}&N
\end{bmatrix},
\]
with $\tilde{N}_2=(I+M_{2,1})^{-1}N$. According to Lemma \ref{l.SUT1}, $\tilde{N}_1$ has the same second diagonal blocks than $N$ and  $I+\hat{M}_{3,2} $ is regular with the same pattern than $N$. If we make $k$ such elimination steps, we obtain
\[
\mathcal L^{k}_{[k]}  \cdot  \cdots \cdot \mathcal L^{1}_{[k]} \cdot \mathcal N_{[k]} = \mathcal{\tilde{N}}_{[k]}, 
\]
analogously to an LU-decomposition, and $\mathcal L_{[k]} := \mathcal L^{k}_{[k]}  \cdot  \cdots \cdot \mathcal L^{1}_{[k]}$.

The remaining assertions are straightforward consequences of the properties of matrix functions belongin to the set $SUT$.
\end{proof}
\begin{proposition}\label{prop.rank.Nk}
Given $N\in SUT$ the associated array function $\mathcal N_{[k]} $ has the nullspace
\begin{align*}
\ker \mathcal{{N}}_{[k]}
 &= \{ y=\begin{bmatrix}
         y_0\\\vdots\\y_k
        \end{bmatrix}
 \in \Real^{(k+1)l}: N\tilde{N}_2 \cdots \tilde{N}_{k+1} y_{k} = 0, \\
	& \quad \quad  \quad \quad
y_i=(-1)^{k+1-i}\tilde{N}_{i+1} \cdots \tilde{N}_{k+1}y_{k}, i=0,\ldots,k-1 \},
\end{align*}
and 
\begin{align*}
\dim \ker  \mathcal{{N}}_{[k]} &= \dim\ker N\tilde{N}_2 \cdots \tilde{N}_{k+1},\\
 \rank \mathcal N_{[k]} & = kl +\rank N\tilde{N}_2 \cdots \tilde{N}_{k+1}.
\end{align*}
Moreover, if $N$ belongs even to $SUT_{row}$ or to $SUT_{column}$, then
\begin{align}\label{Nk_rank}
 \rank \mathcal N_{[k]} & = kl +\rank N^{k+1}= \text{constant}.
\end{align}
\end{proposition}
\begin{proof}
Lemma \ref{l.existence.Ntilde} implies $\ker \mathcal{\tilde{N}}_{[k]}=\ker \mathcal{{N}}_{[k]}$ and $\rank \mathcal{\tilde{N}}_{[k]}=\rank \mathcal{{N}}_{[k]}$.
We evaluate the nullspace $\ker \mathcal{\tilde{N}}_{[k]}$,
\begin{align*}
\ker \mathcal{\tilde{N}}_{[k]} &= \{ y \in \Real^{(k+1)l}: Ny_0=0, y_{i}=-\tilde{N}_{i+2}y_{i+1} , i=0,\ldots,k-1 \}\\
  &= \{ y \in \Real^{(k+1)l}: N\tilde{N}_2 \cdots \tilde{N}_{k+1} y_{k} = 0, \\
	& \quad \quad  \quad \quad
	y_i=(-1)^{k+1-i}\tilde{N}_{i+1} \cdots \tilde{N}_{k+1}y_{k}, i=0,\ldots,k-1 \},
\end{align*}
which yields
\begin{align*}
\dim \ker  \mathcal{\tilde{N}}_{[k]} &= \dim\ker N\tilde{N}_2 \cdots \tilde{N}_{k+1},\\
 \rank \tilde{\mathcal N}_{[k]} & = kl +\rank N\tilde{N}_2 \cdots \tilde{N}_{k+1}.
\end{align*}
If, additionally, $N\in SUK_{column}$, then owing to Lemma \ref{l.Ncol} it results that
also $\dim\ker N\tilde{N}_2 \cdots \tilde{N}_{k+1}=\dim \ker N^{k+1} =l-(l_1+\cdots+l_{k+1})$ such that $\rank N\tilde{N}_2 \cdots \tilde{N}_{k+1}=\rank N^{k+1} =l_1+\cdots+l_{k+1}$, and hence
\[
  \rank \tilde{\mathcal N}_{[k]} = kl +\rank N^{k+1}.
\]
If, contrariwise, $N\in SUK_{row}$, then owing to Lemma \ref{l.Nrow} it results that
also $\rank N\tilde{N}_2 \cdots \tilde{N}_{k+1}=\rank N^{k+1} =l_1+\cdots+l_{\nu-(k+1)}$  and hence
\[
  \rank \tilde{\mathcal N}_{[k]} = kl +\rank N^{k+1}.
\]
\end{proof}
\begin{corollary}\label{c.Nk_-full} 
If  $N\in SUT$ then, for $k\geq\nu$,
 \begin{align*}
\ker \mathcal{{N}}_{[k]}
 = \{  y=\begin{bmatrix}
         y_0\\\vdots\\y_k
        \end{bmatrix} \in \Real^{(k+1)m}: y_0 = 0, y_i=\tilde{N}_{i+1} \cdots \tilde{N}_{k+1}y_{k}, i=1,\ldots,k-1 \}.
\end{align*}
\end{corollary}

\end{document}